\documentclass[10pt,francais,leqno]{smfart}

\usepackage[latin1]{inputenc}
\usepackage[T1]{fontenc}
\usepackage[francais]{babel}
\usepackage{enumitem}
\usepackage{lmodern}

\usepackage{amssymb,url,xspace,smfthm}
\usepackage{mathrsfs,euscript,color}

\usepackage[all]{xy}

\newcommand{\BibTeX}{{\scshape Bib}\kern-.08em\TeX}
\newcommand{\T}{\S\kern .15em\relax }
\newcommand{\AMS}{$\mathcal{A}$\kern-.1667em\lower.5ex\hbox
        {$\mathcal{M}$}\kern-.125em$\mathcal{S}$}

\theoremstyle{plain}

\newtheorem*{mapropo}{\textsc{Proposition}}
\newtheorem*{mapropo1}{\textsc{Proposition 1}}
\newtheorem*{mapropo2}{\textsc{Proposition 2}}
\newtheorem*{mapropo3}{\textsc{Proposition 3}}

\newtheorem*{monlem}{\textsc{Lemme}}
\newtheorem*{monlem1}{\textsc{Lemme 1}}
\newtheorem*{monlem2}{\textsc{Lemme 2}}
\newtheorem*{monlem3}{\textsc{Lemme 3}}

\newtheorem*{moncoro}{\textsc{Corollaire}}
\newtheorem*{moncoro1}{\textsc{Corollaire 1}}
\newtheorem*{moncoro2}{\textsc{Corollaire 2}}

\newtheorem*{marema}{\textsc{Remarque}}
\newtheorem*{marema1}{\textsc{Remarque 1}}
\newtheorem*{marema2}{\textsc{Remarque 2}}
\newtheorem*{marema3}{\textsc{Remarque 3}}

\newtheorem*{madefi}{\textsc{Définition}}

\newtheorem*{exemple}{\textsc{Exemple}}

\catcode`\@=11
\catcode`\=13
\catcode`\=13
\catcode`\=13
\catcode`\=13
\catcode`\=13
\catcode`\=13
\catcode`\=13
\catcode`\=13
\catcode`\=13
\catcode`\=13
\catcode`\=13
\def {\'e}
\def {\`e}
\def {\`a}
\def {\`u}
\def {\^e}
\def {\^a}
\def {\^o}
\def {\^{\i}}
\def {\^u}
\def {\c c}
\def {\"e}
\catcode`\:=13
\def :{~\string:}
\catcode`\;=13
\def ;{~\string;}
\catcode`\!=13
\def !{~\string!}

\def\ES#1{\EuScript{#1}}

\def\wt#1{\widetilde{#1}}
\def\bs#1{\boldsymbol{#1}}
\def\cad{c'est--\`a--dire\ }

\title[INT\'EGRALES ORBITALES SUR $GL(N,{\Bbb F}_q((t)))$]
{INT\'EGRALES ORBITALES SUR $GL(N,{\Bbb F}_q((t)))$}

\author{Bertrand Lemaire}
\address{Aix Marseille Universit, CNRS, Centrale Marseille, I2M, UMR 7373\\ 163 Avenue de Luminy, 
Case 901, 13288 Marseille, France}
\email{Bertrand.Lemaire@univ-amu.fr}
\thanks{L'auteur a bnfici d'une subvention de l'Agence Nationale de la Recherche, projet ANR--13--BS01--00120--02 FERPLAY}

\date{\today}
\begin{document}
\def\smfbyname{}
\small

% RESUME
\begin{abstract}Soit $F$ un corps local non archimédien de caractéristique $\geq 0$, et soit $G=GL(N,F)$, $N\geq 1$. Un élément $\gamma\in G$ est dit quasi--régulier si le centralisateur de $\gamma$ dans $M(N,F)$ est un produit 
d'extensions de $F$. Soit $G_{\rm qr}$ l'ensemble des éléments quasi--réguliers de $G$. Pour $\gamma\in G_{\rm qr}$, on note $\ES{O}_\gamma$ l'intégrale orbitale ordinaire sur $G$ associée à $\gamma$. On remplace ici le discriminant de Weyl $\vert D_G\vert$ 
par un facteur de normalisation $\eta_G: G_{\rm qr}\rightarrow {\Bbb R}_{>0}$ permettant d'obtenir les mêmes résultats que ceux prouvés par Harish--Chandra en caractéristique nulle: pour $f\in C^\infty_{\rm c}(G)$, l'intégrale orbitale normalisée $I^G(\gamma,f)=\eta_G^{1\over 2}(\gamma)\ES{O}_\gamma(f)$ est bornée sur $G$, et pour $\epsilon>0$ tel que $N(N-1)\epsilon <1$, la fonction $\eta_G^{-{1\over 2}-\epsilon}$ est localement intégrable sur $G$.
\end{abstract}

% ABSTRACT
\begin{altabstract}Let $F$ be a non--Archimedean local field of characteristic $\geq 0$, and let $G=GL(N,F)$, $N\geq 1$. An element $\gamma\in G$ is said to be quasi--regular if the centralizer of $\gamma$ in $M(N,F)$ is a product of field extensions of $F$. Let $G_{\rm qr}$ be the set of quasi--regular elements of $G$. For $\gamma\in G_{\rm qr}$, we denote by $\ES{O}_\gamma$ the ordinary orbital integral on $G$ associated with $\gamma$. In this paper, we replace the Weyl discriminant $\vert D_G\vert$ by a normalization factor $\eta_G: G_{\rm qr}\rightarrow {\Bbb R}_{>0}$ which allows us to obtain the same results as proven by Harish--Chandra in characteristic zero: for $f\in C^\infty_{\rm c}(G)$, the normalized orbital integral $I^G(\gamma,f)=\eta_G^{1\over 2}(\gamma)\ES{O}_\gamma(f)$ is bounded on $G$, and for $\epsilon>0$ such that $N(N-1)\epsilon <1$, the function $\eta_G^{-{1\over 2}-\epsilon}$ is locally integrable on $G$.
\end{altabstract}



\altkeywords{orbital integral, Weyl discriminant, simple stratum, minimal element}
\maketitle
\setcounter{tocdepth}{3}
\eject

\tableofcontents

%%%%%%%%%%%%%%%%%%%%%%%%%%%
\section{Introduction}

\vskip3mm
\noindent {\bf 1.1.} --- Il semble maintenant nécessaire d'établir la formule des traces (resp. tordue) pour les groupes réductifs connexes sur un 
corps de fonctions, puis d'essayer ensuite de la stabiliser, comme il a été fait pour les corps de nombres \cite{LW,MW}. La tâche s'annonce longue et laborieuse, et il n'est pas clair qu'il soit aujourd'hui possible de la mener à bien en toute généralité, \cad sans restriction sur la caractéristique du corps de base. Rappelons par exemple que les travaux de Ng\^o Bao Ch\^au sur le lemme fondamental supposent que la caractéristique du corps de base est grande par rapport au rang du groupe. C'est bien sûr du côté géométrique de la formule des traces que des phénomènes nouveaux apparaissent. Globalement, il est raisonnable d'espèrer que les arguments, une fois compris, soient plus simples pour les corps de fonctions que pour les corps de nombres. Localement en revanche, la théorie des intégrales orbitales en caractéristique $p>0$ pour un groupe réductif connexe quelconque est encore à écrire, et la formule des traces locale semble pour l'instant hors de portée (sauf si $p\gg 1$). 
Cet article est en quelque sorte une illustration de cette affirmation: d'un côté il ouvre la voie vers une formule des traces locale pour $GL(N,{\Bbb F}_q((t)))$, de l'autre il laisse imaginer la nature des difficultés à surmonter si on veut établir une telle formule en caractéristique $p>0$ pour un groupe plus général. 

\vskip3mm
\noindent {\bf 1.2.} --- 
Soit $F$ un corps commutatif localement compact non archimédien, de caractéristique quelconque, et soit $G$ un groupe réductif connexe défini sur $F$. On s'intéresse ici à la théorie des intégrales orbitales sur $G(F)$, dans le cas où ${\rm car}(F)=p>0$. Les résultats démontrés dans cet article concernent exclusivement le groupe $G=GL(N)$, et ne sont vraiment nouveaux que si $p$ divise $N$. Mais revenons, pour cette introduction seulement, au cadre général: $G$ quelconque et ${\rm car}(F)\geq 0$. On note $\mathfrak{o}$ l'anneau des entiers de $F$, $\mathfrak{p}$ son idal maximal, $\nu$ la valuation sur $F$ normalise par $\nu(F^\times)={\Bbb Z}$, et $\vert \;\vert$ la valeur absolue normalise sur $F$. On munit $G(F)$ de la topologie $\mathfrak{p}$--adique (\cad celle dfinie par $F$), et on fixe une mesure de Haar $dg$ sur $G(F)$. On note $l$ le rang de $G$, \cad la dimension des tores maximaux de $G$. On suppose $l\geq 1$. 
Un élément $\gamma\in G$ est dit (absolument) semisimple régulier si son centralisateur connexe $G_\gamma$ est un tore. On note $G_{\rm reg}\subset G$ le sous--ensemble des éléments semisimples réguliers. C'est un ouvert non vide de $G$, défini sur $F$. En effet, soit $\mathfrak{g}$ l'algèbre de Lie de $G$. Fixons une clôture algébrique $\overline{F}$ de $F$, et identifions $G$ à $G(\overline{F})$, $\mathfrak{g}$ à $\mathfrak{g}(\overline{F})$, etc. Pour $\gamma\in G$, on note $D_G(\gamma)\in \overline{F}$ le coefficient de $t^l$ dans le polynôme $\det_{\overline{F}}(t+1-{\rm Ad}_\gamma; \mathfrak{g})$. On a
$$
G_{\rm reg}= \{\gamma\in G: D_G(\gamma)\neq 0\}.
$$

Soit $\gamma\in G_{\rm reg}(F)$. Le tore $T=G_\gamma$ est défini sur $F$, et on peut munir $T(F)$ d'une mesure de Haar $dt=dg_\gamma$. L'orbite $\{g^{-1}\gamma g: g\in G(F)\}$ 
est fermée dans $G(F)$, et on note $\ES{O}_\gamma=\ES{O}_\gamma^G$ la distribution sur $G(F)$ définie par
$$
\ES{O}_\gamma(f)= \int_{T(F)\backslash G(F)}f(g^{-1}\gamma g) \textstyle{dg\over dt}, \quad f\in C^\infty_{\rm c}(G(F)).
$$
Soit $A_T=A_\gamma$ le sous--tore $F$--déployé maximal de $T$. On munit le groupe $A_T(F)$ de la mesure de Haar $da$ qui donne le volume $1$ au sous--groupe compact maximal $A_T(\mathfrak{o})$ de $A_T(F)$. Le groupe quotient $A_T(F)\backslash T(F)$ est compact, et on peut normaliser la mesure $dt$ en imposant la condition ${\rm vol}(A_T(F)\backslash T(F),\textstyle{dt\over da})=1$. On a donc
$$
\ES{O}_\gamma(f)= \int_{A_T(F)\backslash G(F)}f(g^{-1}\gamma g) \textstyle{dg\over da}, \quad f\in C^\infty_{\rm c}(G(F)).
$$
L'élement $D_G(\gamma)$ appartient à $F$, et en notant $\mathfrak{g}_\gamma\subset \mathfrak{g}$ le sous--espace formé des points fixes sous ${\rm Ad}_\gamma$ (qui co\"{\i}ncide avec l'algèbre de Lie de $T$), on a
$$
D_G(\gamma) = {\rm det}_F(1- {\rm Ad}_\gamma; \mathfrak{g}(F)/\mathfrak{g}_\gamma(F)).
$$
On note $I^G(\gamma,\cdot)$ la distribution sur $G(F)$ définie par
$$
I^G(\gamma,f)= \vert D_G(\gamma)\vert^{1\over 2} \ES{O}_\gamma(f),\quad f\in C^\infty_{\rm c}(G(F)).\leqno{(1)}
$$
On sait d'après Harish--Chandra que pour $f\in C^\infty_{\rm c}(G(F))$, l'application $\gamma \mapsto I^G(\gamma,f)$ est localement constante sur $T(F)\cap G_{\rm reg}$.

Le facteur de normalisation $\vert D_G\vert^{1\over 2}$ introduit en (1) est bien sûr motivé par la formule d'intégration de Weyl: pour toute fonction localement intégrable $\theta$ sur $G(F)$, définie sur $G_{\rm reg}(F)$ et telle que $\theta(g^{-1}\gamma g)= \theta(\gamma)$ pour $\gamma\in G_{\rm reg}(F)$ et $g\in G(F)$, et pour toute fonction $f\in C^\infty_{\rm c}(G(F))$, on a
$$
\int_{G(F)} \theta(g)f(g)dg = \sum_T\vert W^G(T)\vert^{-1}\int_{T(F)} \vert D_G(\gamma)\vert^{1\over 2} \theta(\gamma)I^G(\gamma,f)d\gamma,
\leqno{(2)}
$$
où $T$ parcourt les tores maximaux de $G$ définis sur $F$, pris modulo conjugaison par $G(F)$, $\vert W^G(T)\vert$ est le cardinal du groupe de Weyl $N_G(T)/T$ de $G$, et $d\gamma$ est la mesure de Haar normalisée sur $T(F)$.

%%%%%%%%%%%%%%%%%%%%%%%
\vskip3mm
\noindent {\bf 1.3.} --- On a aussi la variante sur $\mathfrak{g}(F)$ de l'intégrale orbitale normalisée définie en 1.2.(1). Pour $X\in \mathfrak{g}$, on note $D_\mathfrak{g}(X)$ le coefficient de $t^l$ dans le polynôme $\det_{\overline{F}}(t-{\rm ad}_X;\mathfrak{g})$, et on pose
$$
\mathfrak{g}_{\rm reg}= \{X\in \mathfrak{g}: D_\mathfrak{g}(X)\neq 0\}.
$$
C'est un ouvert non vide de $\mathfrak{g}$, défini sur $F$. Pour $X\in \mathfrak{g}_{\rm reg}(F)$, l'élément $D_\mathfrak{g}(X)$ appartient à $F$, et en posant $\mathfrak{g}_X=\ker({\rm ad}_X: \mathfrak{g}\rightarrow \mathfrak{g})$, on a
$$
D_\mathfrak{g}(X)={\rm det}_F(-{\rm ad}_X; \mathfrak{g}(F)/\mathfrak{g}_X(F)).
$$
Pour $X\in \mathfrak{g}_{\rm reg}(F)$, on note $I^\mathfrak{g}(X,\cdot)$ la distribution sur $\mathfrak{g}(F)$ définie par
$$
I^\mathfrak{g}(X,\mathfrak{f})= \vert D_\mathfrak{g}(X)\vert^{1\over 2}\int_{A_X(F)\backslash G(F)}\mathfrak{f}({\rm Ad}_{\smash{g^{-1}}}(X)) \textstyle{dg\over da},
\quad \mathfrak{f}\in C^\infty_{\rm c}(\mathfrak{g}(F)).\leqno{(1)}
$$
Ici $A_X$ est le sous--tore $F$--déployé maximal $A_T$ du centralisateur connexe $T=G_X$ de $X$ dans $G$, et $da$ est la mesure de Haar sur $A_X(F)$ qui donne le volume $1$ à $A_X(\mathfrak{o})$. On a bien sûr aussi l'analogue sur $\mathfrak{g}(F)$ de la 
formule d'intégration de Weyl 1.2.(2). 

%%%%%%%%%%%%%%%%%%%%%%%%%%
\vskip3mm
\noindent {\bf 1.4.} --- On suppose dans ce numéro que $F$ est de caractéristique nulle. Rappelons quelques résultats bien connus, dus à  Harish--Chandra. La somme sur $T$ dans 1.2.(2) est finie --- pour cela, il n'est pas nécessaire de déranger Harish--Chandra! ---, et d'après \cite[theo. 14]{HC1}, pour tout tore maximal $T$ de $G$ défini sur $F$ et toute fonction $f\in C^\infty_{\rm c}(G(F))$, l'application $\gamma \mapsto I^G(\gamma,f)$ est localement bornée sur $T(F)\cap G_{\rm reg}$: 
\begin{enumerate}[leftmargin=17pt]
\item[(1)]pour toute partie compacte $\Omega$ de $T(F)$, on a 
$\sup_{\gamma \in \Omega\cap G_{\rm reg}}\vert I^G(\gamma,f)\vert <+\infty$.
\end{enumerate}
Le résultat suivant \cite[theo.~15]{HC1} est indispensable pour l'étude des intégrales pondérées sur $G(F)$:
\begin{enumerate}[leftmargin=17pt]
\item[(2)]il existe $\epsilon >0$ tel que la fonction $\vert D_G\vert^{-{1\over 2}-\epsilon}$ est localement intégrable sur $G(F)$. 
\end{enumerate}
Compte--tenu de (1) et de 1.2.(2), Harish--Chandra déduit (2) de:
\begin{enumerate}[leftmargin=17pt]
\item[(3)]pour chaque tore maximal $T$ de $G$ défini sur $F$, il existe un $\epsilon >0$ tel que la fonction $\vert D_G\vert^{-\epsilon}$ est localement intégrable sur $T(F)$. 
\end{enumerate}

Revenons à la propriété (1). Il suffit pour l'obtenir de prouver que pour chaque $t\in T(F)$, il existe un voisinage ouvert compact $\Omega$ de $t$ dans $T(F)$ tel que $\sup_{\gamma \in \Omega\cap G_{\rm reg}}\vert I^G(\gamma,f)\vert <+\infty$. Notons $H=G_t$ le centralisateur connexe de $t$ dans $G$, et $\mathfrak{h}$ son algèbre de Lie. Soit $\omega_t$ l'ensemble des $\delta\in H(F)$ tels que 
$\det_F(1-{\rm Ad}_{t\delta}; \mathfrak{g}(F)/\mathfrak{h}(F))\neq 0$. Puisque $t$ est semisimple, on a la décomposition $\mathfrak{g}(F)= (1-{\rm Ad}_t)(\mathfrak{g}(F))\oplus \mathfrak{h}(F)$. De plus, $\omega_t$ est un voisinage ouvert et $H(F)$--invariant de $1$ dans $H(F)$, et l'application
$$
\delta: G(F)\times \omega_t\rightarrow (g,h)\mapsto g^{-1}thg
$$
est partout submersive. On peut donc appliquer le principe de submersion d'Harish--Chandra, et \og descendre\fg{} toute distribution $G(F)$--invariante au voisinage de $t$ dans $G(F)$ --- par exemple une intégrale orbitale --- en une distribution $H(F)$--invariante au voisinage de $1$ dans $H(F)$. On en déduit qu'il existe un voisinage ouvert compact $\ES{V}_t$ de $1$ dans $T(F)$ vérifiant la propriété: pour toute fonction $f\in C^\infty_{\rm c}(G(F))$, il existe une fonction $f^H\in C^\infty_{\rm c}(\ES{V}_t)$ telle que pour tout $\delta\in \ES{V}_t$ tel que $t\delta \in G_{\rm reg}$, on a l'égalité 
$$
\ES{O}_{t\delta}(f)= \ES{O}_\delta^H(f^H).
$$
Or $tH(F)\cap G_{\rm reg}\subset tH_{\rm reg}(F)$ et
$$
{\rm \det}_F(1-{\rm Ad}_{t\delta}; \mathfrak{g}(F)/\mathfrak{h}(F))= D_G(t\delta)D_H(\delta)^{-1},
$$ par conséquent l'égalité $\ES{O}_{t\delta}(f)= \ES{O}_\delta^H(f^H)$ s'écrit aussi
$$
I^G(t\delta,f)= \vert {\rm \det}_F(1-{\rm Ad}_{t\delta}; \mathfrak{g}(F)/\mathfrak{h}(F))\vert^{1\over 2} I^H(\delta,f^H).
$$
Comme l'application
$$\delta\mapsto 
 {\rm \det}_F(1-{\rm Ad}_{t\delta}; \mathfrak{g}(F)/\mathfrak{h}(F))
$$ est bornée au voisinage de $1$ dans $T(F)$, par récurrence sur la dimension de $G$ (et translation $\delta \mapsto t^{-1}\delta$ si $t\in Z(G)$), on est ramené pour obtenir (1) à prouver:
\begin{enumerate}[leftmargin=17pt]
\item[(4)]l'application $\gamma \mapsto I^G(\gamma,f)$ est bornée au voisinage de $1$ dans $T(F)$.
\end{enumerate}
Via l'application exponentielle, on peut passer à l'algèbre de Lie. En notant $\mathfrak{t}$ l'algèbre de Lie de $T$, (4) est impliqué par le résultat suivant \cite[theo.~13]{HC1}:
\begin{enumerate}[leftmargin=17pt]
\item[(5)]pour toute fonction $\mathfrak{f}\in C^\infty_{\rm c}(\mathfrak{g}(F))$, on a
$\sup_{\gamma \in \mathfrak{t}(F)\cap \mathfrak{g}_{\rm reg}} \vert I^{\mathfrak{g}}(\gamma,\mathfrak{f})\vert <+\infty$.
\end{enumerate}
Comme pour $(4) \Rightarrow (1)$, on obtient (5) par récurrence sur $\dim_{\overline{F}}(\mathfrak{g}_t)$ pour $t\in \mathfrak{t}(F)$, grâce à la propriété d'homogénéité des germes de Shalika \cite[\S 8]{HC2} (voir aussi \cite[17.14]{K}). Revenons  $G$. Pour toute fonction 
$f\in C^\infty_{\rm c}(G(F))$, l'application $T(F)\rightarrow {\Bbb C},\,\gamma \mapsto I^G(\gamma,f)$ obtenue en posant $I^G(\gamma,f)=0$ pour $\gamma\in T(F)\smallsetminus (T(F)\cap G_{\rm reg})$ est  support compact (c'est une consquence du lemme 39 de \cite{HC1}). On en dduit que la proprit (1) se renforce en:
\begin{enumerate}[leftmargin=17pt]
\item[(6)]pour toute fonction $f\in C^\infty_{\rm c}(G(F))$, on a
$\sup_{\gamma \in T(F)\cap G_{\rm reg}} \vert I^{G}(\gamma,f)\vert <+\infty$.
\end{enumerate}

Le passage de (5) à (1) via la submersion $\delta$ et l'application $f\mapsto f^H$ est appelé \og descente centrale au voisinage d'un élément semisimple\fg{} ou, plus simplement \cite{K}, \og descente semisimple\fg{}\footnote{Pour $G=GL(N)$ et $F$ de caractéristique $>0$, nous aurons à généraliser cette construction au voisinage d'éléments $t$ qui ne sont pas semisimples mais seulement d'orbite fermée --- voir 1.8.(2).}. Par descente semisimple, Harish--Chandra ramène aussi (3) au résultat suivant:
\begin{enumerate}[leftmargin=17pt]
\item[(7)]pour chaque tore maximal $T$ de $G$ défini sur $F$, il existe un $\epsilon >0$ tel que la fonction $\vert D_\mathfrak{g}\vert^{-\epsilon}$ est localement intégrable sur $\mathfrak{t}(F)$. 
\end{enumerate}
%

%%%%%%%%%%%%%%%%%%%%%%%%%%%%%%%%%%
\vskip3mm
\noindent {\bf 1.5.} --- On suppose maintenant que $F$ est de caractéristique $p>0$. Alors la généralisation des résultats rappelés en 1.4 se heurte à plusieurs obstacles. Parmi ceux--ci:
\begin{itemize}
\item les tores maximaux de $G$ définis sur $F$, modulo conjugaison par $G(F)$, peuvent former un ensemble infini. On ne peut donc pas se contenter d'une borne sur chaque $T$, comme en 1.4.(1), 1.4.(6) ou 1.4.(3), il faut en plus contrôler ces bornes de manière à ce que la somme sur les $T$ dans 1.2.(2) converge; 
\item la présence d'éléments $t\in G(F)$ qui ne sont pas semisimples (sur $\overline{F}$) mais néanmoins tels que l'orbite $\ES{O}_{G(F)}(\gamma)=\{g^{-1}tg:g\in G(F)\}\subset G(F)$ est fermée pour la topologie $\mathfrak{p}$--adique. Au voisinage de tels éléments, la descente centrale telle qu'on l'a rappelée en 1.4 ne fonctionne plus;
\item les classes de conjugaison unipotentes dans $G(F)$ peuvent former un ensemble infini, et la théorie des germes de Shalika ne s'applique pas dans ce cas (il faudrait l'crire autrement). De plus le passage à l'algèbre de Lie pose problème, car on ne dispose pas d'une application exponentielle comme en caractéristique nulle. 
\end{itemize} 
Pour s'en convaincre, il suffit de regarder les premiers exemples: les groupes $\ES{G}=GL(2,F)$ et $\ES{G}'= SL(2,F)$, avec $F={\Bbb F}_2((t))$. Les classes de conjugaison de tores maximaux de $\ES{G}$ sont classifiées par les classes d'isomorphisme d'extensions quadratiques séparables de $F$, qui sont en nombre infini. De plus, il y a aussi les extensions inséparables. Si $E\subset M(2,F)$ est une extension quadratique inséparable de $F$, et si $\gamma=\omega_E$ est une uniformisante de $E$, alors l'intégrale orbitale $\ES{O}_\gamma$ a les mêmes propriétés qu'une \og vraie\fg{} intégrale orbitale semisimple régulière (elliptique). Pourtant sur $\overline{F}$, l'élément $\gamma$ dégénère puisqu'il se décompose en $\gamma = zu$ avec $z\in Z(G;\overline{F})$ et $u\in G(\overline{F})$ unipotent. Dans $\ES{G}'$, on peut vérifier que tous les éléments sont séparables, mais les classes de conjugaison unipotentes non triviales sont classifiées par l'ensemble $F^\times /(F^\times)^2$, qui est infini. Quant à la descente centrale, les difficultés nouvelles apparaissent au voisinage des éléments inséparables qui sont contenus dans un sous--groupe de Levi propre (sur $F$). Pour que de tels éléments existent, il faut que le groupe ambiant soit un peu plus gros que $\ES{G}$ ou $\ES{G}'$: par exemple le groupe $GL(4,F)$, et l'élément $\gamma$ plongé diagonalement dans le sous--groupe de Levi $\ES{G}\times\ES{G}$ de $GL(4,F)$.

\vskip3mm
\noindent {\bf 1.6.} --- Revenons à ${\rm car}(F)\geq 0$ et dcrivons les rsultats contenus dans ce papier. Changeons de notations: dornavant, on fixe un entier $N\geq 1$, un $F$--espace vectoriel $V$ de dimension $N$, et on pose $\mathfrak{g}={\rm End}_F(V)$ et $G={\rm Aut}_F(V)$. Un lment $\gamma \in \mathfrak{g}$ est dit {\it ferm} si la $F$--algbre $F[\gamma]$ est un produit d'extensions de $F$, {\it pur} si la $F$--algbre $F[\gamma]$ est un corps, {\it quasi--rgulier} s'il est ferm et si $\dim_F(F[\gamma])=N$, {\it quasi--rgulier elliptique} s'il est quasi--rgulier et pur. Si de plus $\gamma$ est {\it sparable}, \cad si le polynme caractristique $\zeta_\gamma\in F[t]$ de $\gamma$ est produit de polynômes irréductibles et sparables sur $F$, alors il est ferm si et seulement s'il est (absolument) semisimple, et il est quasi--rgulier, resp. quasi--rgulier elliptique, si et seulement s'il est semisimple rgulier, resp. semisimple rgulier elliptique, au sens habituel (cf. 1.4). On note $\mathfrak{g}_{\rm qr}$, resp. $\mathfrak{g}_{\rm qre}$, l'ensemble des lments quasi--rguliers, resp. quasi--rguliers elliptiques, de $\mathfrak{g}$. Pour $\star= {\rm qr},\, {\rm qre}$, on pose $G_\star = G \cap \mathfrak{g}_\star$. Un lment $\gamma\in \mathfrak{g}$ est ferm si et seulement si son orbite $\ES{O}_G(\gamma)=\{g^{-1}\gamma g: g\in G\}$ est ferme dans $\mathfrak{g}$ (pour la topologie $\mathfrak{p}$--adique). 

Soit $\gamma\in G_{\rm qr}$. La $F$--algbre $F[\gamma]$ co\"{\i}ncide avec le centralisateur $\mathfrak{g}_\gamma= \{x\in \mathfrak{g}: \gamma x -x \gamma = 0\}$ de $\gamma$ dans $\mathfrak{g}$. \'Ecrivons $\mathfrak{g}_\gamma= E_1\times \cdots \times E_r$ pour des extensions $E_i$ de $F$, notons $A_\gamma$ le sous--tore dploy maximal $F^\times\times \cdots \times F^\times$ de $G_\gamma = G\cap \mathfrak{g}_\gamma$, et $M=M(\gamma)$ le centralisateur de $A_\gamma$ dans $G$. Alors $M$ est un produit de groupes linaires sur $F$, et $\gamma$ est quasi--rgulier elliptique dans $M$. Pour toute fonction $f\in C^\infty_{\rm c}(G)$, on dfinit comme suit l'intgrale normalise
$$
I^G(\gamma ,f)= \eta_G(\gamma)^{1\over 2}\ES{O}_\gamma(f).\leqno{(1)}
$$
On note $dg$, resp. $da$, la mesure de Haar sur $G$, resp. $A_\gamma$, qui donne le volume $1$  $GL(N,\mathfrak{o})$, resp. au 
sous--groupe compact maximal $\mathfrak{o}^\times\times \cdots \times \mathfrak{o}^\times$ de $A_\gamma$, et on pose
$$
\ES{O}_\gamma(f) = \int_{A_\gamma \backslash G}f(g^{-1} \gamma g) \textstyle{dg\over da}. 
$$
On pose
$$
\eta_G(\gamma) = \eta_{M}(\gamma) \vert {\rm det}_F( 1- {\rm Ad}_\gamma; \mathfrak{g}/\mathfrak{m})\vert,
$$
o $\mathfrak{m}=\mathfrak{m}(\gamma)$ est l'algbre de Lie de $M$, et on dfinit $\eta_M(\gamma)$ par produit  partir du cas elliptique suivant. Si $\gamma\in G_{\rm qre}$, alors $E=F[\gamma]$ est une extension de $F$ de degr $N$, et on pose
$$
\eta_G(\gamma)= q^{-f(\tilde{c}_F(\gamma) + e-1)}
$$
o $e$, resp. $f$, est l'indice de ramification, resp. le degr rsiduel, de l'extension $E/F$, et o $\tilde{c}_F(\gamma)$ est un invariant dfini comme suit. On commence par supposer que $\gamma$ appartient  l'anneau des entiers $\mathfrak{o}_E$ de $E$, et on pose
$$
\{x\in \mathfrak{o}_E: x \mathfrak{o}_E\subset \mathfrak{o}[\gamma]\} = \mathfrak{p}_E^{c_F(\gamma)}.
$$
Pour $z\in \mathfrak{o}\smallsetminus \{0\}$, on a $c_F(z\gamma)= e(N-1)\nu(z) + c_F(\gamma)$, ce qui permet de dfinir $\tilde{c}_F(\gamma)$ en gnral: on choisit $z\in F^\times$ tel que $z\gamma \in \mathfrak{o}_E$, et on pose $c_F(\gamma)= c_F(z\gamma)- e(N-1)\nu(z)$. On pose aussi $\tilde{c}_F(\gamma)= c_F(\gamma) - (N-1)\nu_E(\gamma)$. Par construction, on a $\tilde{c}_F(z\gamma)= \tilde{c}_F(\gamma)$ pour tout $z\in F^\times$. On vrifie que si $\gamma$ est sparable, \cad si l'extension $E/F$ est sparable, alors on a
$$
\tilde{c}_F(\gamma)= \textstyle{1\over f}(\nu(D_G(\gamma))-\delta),
$$
o $\delta$ est le discriminant de $E/F$, et donc
$$
\eta_G(\gamma)= \vert D_G(\gamma)\vert q^{\delta -f(e-1)}.
$$
On reconnait l'exposant de Swan $\delta - f(e-1)\geq 0$ de $E/F$.

On dfinit aussi la variante sur $\mathfrak{g}$ de l'intgrale orbitale normalise (1): pour $\gamma \in \mathfrak{g}_{\rm qr}$ et $\mathfrak{f}\in C^\infty_{\rm c}(\mathfrak{g})$, on pose
$$
I^\mathfrak{g}(\gamma,\mathfrak{f})= \eta_\mathfrak{g}(\gamma)^{1\over 2}\ES{O}_\gamma(\mathfrak{f}),\leqno{(2)}
$$
o la distribution $\ES{O}_\gamma$ sur $\mathfrak{g}$ est dfinie de la mme manire que celle sur $G$, et le facteur de normalisation $\eta_\mathfrak{g}(\gamma)$ est donn 
par $\eta_\mathfrak{g}(\gamma)= \vert {\rm det}_F(-{\rm ad}_\gamma; \mathfrak{g}/\mathfrak{m})\vert \eta_{\mathfrak{m}}(\gamma)$ avec $\mathfrak{m}=\mathfrak{m}(\gamma)$, et, si $\gamma\in \mathfrak{g}_{\rm qre}$, par $\eta_\mathfrak{g}(\gamma)= q^{-f(c_F(\gamma)-(e-1))}$.

%%%%%%%%%%%%%%%%
\vskip3mm
\noindent {\bf 1.7.} --- Les deux principaux résultats prouvés ici, qui généralisent ceux d'Harish--Chandra en caractéristique nulle (cf. 1.4), sont les suivants:
\begin{enumerate}[leftmargin=17pt]
\item[(1)]pour toute fonction $f\in C^\infty_{\rm c}(G)$, on a $\sup_{\gamma \in G_{\rm qr}}\vert I^G(\gamma,f)\vert < +\infty$,
\end{enumerate}
où l'intégrale orbitale normalisée $I^G(\gamma,f)= \eta_G(\gamma)^{1\over 2}\ES{O}_\gamma(f)$ est celle définie en 1.6.(1). Quant au facteur de normalisation $\eta_G: G_{\rm qr} \rightarrow {\Bbb R}_{>0}$, il vérifie:
\begin{enumerate}[leftmargin=17pt]
\item[(2)]pour tout $\epsilon >0$ tel que $N(N-1)\epsilon <1$, la fonction $G_{\rm qr}\rightarrow {\Bbb R}_{>0},\, \gamma \mapsto \eta_G(\gamma)^{-{1\over 2}-\epsilon}$ est localement intégrable sur $G$.
\end{enumerate}
En fait, on prouve d'abord la variante sur $\mathfrak{g}$ de ces deux résultats:
\begin{enumerate}[leftmargin=17pt]
\item[(3)]pour toute fonction $\mathfrak{f}\in C^\infty_{\rm c}(\mathfrak{g})$, on a $\sup_{\gamma \in \mathfrak{g}_{\rm qr}}\vert I^\mathfrak{g}(\gamma,\mathfrak{f})\vert < +\infty$;
\end{enumerate}
\begin{enumerate}[leftmargin=17pt]
\item[(4)]pour tout $\epsilon >0$ tel que $N(N-1)\epsilon <1$, la fonction $\mathfrak{g}_{\rm qr}\rightarrow {\Bbb R}_{>0},\, \gamma \mapsto \eta_\mathfrak{g}(\gamma)^{-{1\over 2}-\epsilon}$ est localement intégrable sur $\mathfrak{g}$.
\end{enumerate}

\vskip3mm
\noindent {\bf 1.8.} --- 
On l'a dit en 1.5, l'une des principales difficultés est ici la descente centrale au voisinage d'un élément fermé qui n'est pas semisimple. Soit $\beta\in G$ un élément fermé. Par descente parabolique standard, on se ramène facilement au cas où $\beta$ est pur. Posons $E=F[\beta]$, $d= {N\over [E:F]}$ et $\mathfrak{b}= {\rm End}_E(V)$. Ainsi $\mathfrak{b}$ est le commutant de $E$, \cad le centralisateur de $\beta$, dans $\mathfrak{g}$. Si l'extension $E/F$ est inséparable, l'intersection ${\rm ad}_\beta(\mathfrak{g})\cap \mathfrak{b}$ n'est pas nulle, par conséquent l'inclusion ${\rm ad}_\beta(\mathfrak{g})+ \mathfrak{b}\subset \mathfrak{g}$ est stricte, et la méthode d'Harish--Chandra (descente semisimple) ne fonctionne plus. On modifie cette méthode comme on l'a fait en \cite{L1} pour prouver l'intégrabilité locale des caractères. En gros, l'idée consiste à choisir un \og bon \fg{} supplémentaire de ${\rm ad}_\beta(\mathfrak{g})$ dans $\mathfrak{g}$. Pour cela on fixe une corestriction modérée $\bs{s}_0: A(E)\rightarrow E$ sur $A(E)={\rm End}_F(E)$ relativement à $E/F$, et un élément $\bs{x}_0\in \mathfrak{A}(E)= {\rm End}_\mathfrak{o}^0(\{\mathfrak{p}_E^i:i\in {\Bbb Z}\})$ tel que $\bs{s}_0(\bs{x}_0)=1$. On fixe aussi un $\mathfrak{o}$--ordre héréditaire $\mathfrak{A}$ dans $\mathfrak{g}$ normalisé par $E^\times$ (que l'on choisira minimal pour cette propriété), et une $(W,E)$--décomposition $\mathfrak{A}= \mathfrak{A}(E)\otimes_{\mathfrak{o}_E}\mathfrak{B}$ de $\mathfrak{A}$, avec $\mathfrak{B}=\mathfrak{A}\cap \mathfrak{b}$. Cette décomposition induit une $(W,E)$--décomposition $\mathfrak{g}= A(E)\otimes_E\mathfrak{b}$ de $\mathfrak{g}$, et on pose $\bs{x}= \bs{x}_0\otimes 1\in \mathfrak{A}$. Pour ces notions, dues à Bushnell--Kutzko \cite{BK}, on renvoie à \ref{des décompositions}. Le supplémentaire en question est le sous--espace $\bs{x}_0\otimes \mathfrak{b}= \bs{x}\mathfrak{b}$ de $\mathfrak{g}$. On en déduit qu'il existe un voisinage ouvert compact $\ES{V}$ de $0$ dans $\mathfrak{b}$ tel que l'application
$$
\delta: G \times \bs{x}\ES{V} \rightarrow G, \, (g,b)\mapsto g^{-1}(\beta + \bs{x}b)g\leqno{(1)}
$$
est partout submersive. On peut donc appliquer le principe de submersion d'Harish--Chandra, et \og descendre \fg{} toute distribution $G$--invariante $T$ au voisinage de $\beta$ dans $G$ en une distribution $\wt{\vartheta}_T$ sur $\bs{x}\ES{V}$. Mais cette dernière n'est pas invariante sous l'action du groupe $H={\rm Aut}_E(V)$ par conjugaison (d'ailleurs $\bs{x}\ES{V}$ lui--même n'est pas $H$--invariant). On peut cependant en déduire, par un procédé de recollement \cite{L1}, une distribution $H$--invariante $\theta_T$ sur $\mathfrak{b}$. Signalons que dans cette construction, l'élément $\bs{x}$ n'appartient pas à $\mathfrak{b}$, sauf si l'extension $E/F$ est modérément ramifiée (donc en particulier séparable), auquel cas on peut prendre $\bs{x}_0=1$. On est donc ramené à déterminer la distribution $\theta_T$ sur $\mathfrak{b}$ lorsque $T= \ES{O}_\gamma$ pour un élément $\gamma\in G_{\rm qr}$ de la forme $\gamma = \beta + \bs{x}b$ avec $b\in \ES{V}$, et aussi à calculer le facteur de normalisation $\eta_G(\gamma)$ pour un tel élément $\gamma$. C'est la partie la plus difficile de ce travail: elle occupe les sections \ref{descente centrale au voisinage d'un élément pur} et \ref{descente centrale: le cas général}. Dans la section 3, on prouve que si $\ES{V}$ est suffisamment petit, alors pour $b\in \ES{V}\cap \mathfrak{b}_{\rm qre}$, l'élément $\gamma$ appartient à $G_{\rm qre}$, et la distribution $\theta_{\ES{O}_\gamma}$ est égale $\lambda \ES{O}_b^\mathfrak{b}$ pour une constante $\lambda$ ne dépendant que de $\beta$ (et pas de $b$). Ici $\ES{O}_b^\mathfrak{b}$ est l'intégrale orbitale sur $\mathfrak{b}$ définie par $b$, normalisée comme plus haut en rempla\c{c}ant $\mathfrak{g}$ par $\mathfrak{b}$. De plus, le facteur de normalisation $\eta_G^{1\over 2}(\gamma)$ est égal à $\mu \eta_{\mathfrak{b}}^{1\over 2}(b)$ pour une constante $\mu$ ne dépendant elle aussi que de $\beta$ (et pas de $b$). On en déduit en particulier que pour toute fonction $f\in C^\infty_{\rm c}(G)$, il existe une fonction $f^\mathfrak{b}\in C^\infty_{\rm c}(\ES{V})$ telle que
$$
I^G(\beta + \bs{x}b,f) = I^\mathfrak{b}(b,f^\mathfrak{b}),\quad b\in \ES{V}\cap \mathfrak{b}_{\rm qre}.\leqno{(2)}
$$
Dans la section \ref{descente centrale: le cas général}, on prouve que cette construction est compatible aux applications \og terme constant \fg{} (sur $G$ et sur $\mathfrak{b}$), ce qui entra\^{i}ne que l'égalité (1) est vraie pour tout $b\in \ES{V}\cap \mathfrak{b}_{\rm qr}$.

Notons que la construction est relativement explicite. En particulier on ne se contente pas d'affirmer l'existence du voisinage $\ES{V}$, on en produit un qui est en quelque sorte optimal: l'ensemble ${^H\ES{V}}= \{h^{-1}bh: h\in H,\, b\in \ES{V}\}$ est le plus gros possible, et il est {\it fermé}  dans $\mathfrak{g}$. L'étude de la distribution $\theta_{\ES{O}_\gamma}$ consiste d'une part à prouver que pour $b,\,b'\in \ES{V}\cap \mathfrak{b}_{\rm qre}$, on a $\ES{O}_H(b)= \ES{O}_H(b')$ si et seulement si $\ES{O}_G(\beta + \bs{x}b)= \ES{O}_G(\beta + \bs{x}b')$, d'autre part à calculer la différentielle de la submersion $\delta$ en $(1,b)$ pour chaque $b\in \ES{V}\cap \mathfrak{b}_{\rm qre}$. On se ramène, par une récurrence assez compliquée, au cas où l'élément $b$ est le plus simple possible, \cad {\it $E$--minimal} au sens de Bushnell--Kutzko (cf. \ref{l'invariant k}). D'ailleurs, cette construction pourrait avoir des implications intéressantes, puisqu'on prouve au passage que tout élément quasi--régulier elliptique de $G$ admet une décomposition --- en un certain sens \og unique \fg{} --- en termes d'éléments quasi--réguliers elliptiques $F_i$--minimaux de ${\rm End}_{F_{i+1}}(F_i)$ pour une suite d'extensions $(F_0=F[\gamma], \ldots ,F_m)$ de $F$ (cf. \ref{une conséquence}). Notons que si la caractéristique résiduelle $p$ de $F$ ne divise pas $N$ (et aussi si $N=p$), toutes les extensions $F_i/F$ ($i>0$) sont modrment ramifies, et il est possible de les choisir de telle manière que $F_m\subset \cdots \subset F_1\subset F_0$. Mais c'est un cas trs particulier: si $p<N$ divise $N$, il n'est en gnral pas possible de les choisir de cette manire.

\vskip3mm
\noindent {\bf 1.9.} --- Par descente centrale, on est donc ramené à l'étude des intégrales orbitales $I^\mathfrak{g}(\gamma,\mathfrak{f})$ pour $\gamma\in \mathfrak{g}$ proche de $0$ dans $\mathfrak{g}$. On dispose pour cela des germes de Shalika associés aux orbites nilpotentes de $\mathfrak{g}$, lesquels vérifient une propriété d'homogénéité particulièrement utile. Cette étude fait l'objet de la section \ref{germes de shalika et résultats sur l'algèbre de lie}. En caractéristique nulle, Kottwitz a exposé la théorie des germes de Shalika de manière trs claire dans \cite{K}. Notre contribution est ici minime, puisque nous n'avons eu qu'à adapter son travail. La propriété d'homogénéité des germes de Shalika (normalisés) permet de leur associer des fonctions sur $\mathfrak{g}_{\rm qr}$. Par descente centrale et homogénéité, on prouve que ces fonctions sont localement bornées sur $\mathfrak{g}$. On en déduit que les intégrales orbitales normalisées sont elles aussi localement bornées sur $\mathfrak{g}$, puis, grâce à un argument de support relativement simple (cf. \ref{parties compactes modulo conjugaison}), qu'elles sont bornées sur $\mathfrak{g}$, \cad 1.7.(3). On prouve 1.7.(4) de la même manière. Dans la section \ref{résultats sur le groupe}, on en déduit les mêmes résultats sur $G$, \cad 1.7.(1) et 1.7.(2). 

\vskip3mm
\noindent {\bf 1.10.} --- Pour les caractères, on peut prouver un résultat analogue à 1.7.(1). Précisément, soit $\pi$ une représentation complexe lisse irréductible de $G$. À $\pi$ est associée une distribution 
$\Theta_\pi$ sur $G$, donnée par $\Theta_\pi(f)= {\rm trace}(\pi(f))$ pour toute fonction $f\in C^\infty_{\rm c}(G)$, où $\pi(f)$ est l'opérateur sur l'espace de $\pi$ défini par $\pi(f)= \int_G f(g)\pi(g)dg$. On sait que cette distribution $\Theta_\pi$ est localement constante sur $G_{\rm qr}$, et localement intégrable sur $G$ \cite{L1}: il existe une fonction localement constante 
$\theta_\pi: G_{\rm qr}\rightarrow {\Bbb C}$ telle que pour toute fonction $f\in C^\infty_{\rm c}(G)$, on a
$$
\Theta_\pi(f)= \int_G f(g)\theta_\pi(g)dg,
$$
l'intégrale étant absolument convergente. Notons que le caractère--distribution $\Theta_\pi$ dépend de la mesure de Haar $dg$ sur $G$ mais que la fonction caractère $\theta_\pi$ n'en dépend pas. 
Comme pour les intégrales orbitales, on peut prouver\footnote{Nous donnerons ailleurs une démonstration détaillée de ce résultat, le présent article étant déjà suffisamment long.} 
que la fonction caractère normalisée
$$
G_{\rm qr}\rightarrow {\Bbb C},\, \gamma \mapsto I^G(\pi,\gamma) = \eta_G(\gamma)^{1\over 2} \theta_\pi(\gamma)
$$
est localement bornée sur $G$. Compte--tenu des constructions de \cite{L1} --- en partie reprises ici (cf. 1.8) ---, 
on se ramène à prouver, par descente parabolique puis descente centrale au voisinage d'un élément pur de $G$, que les transformées de 
Fourier normalisées des intégrales orbitales nilpotentes sur $\mathfrak{g}$, qui sont des fonctions localement constantes sur $\mathfrak{g}_{\rm qr}$, 
sont bornées sur $\mathfrak{g}$. Or ces transformées de Fourier sont des fonctions bien plus faciles à calculer, et donc 
à majorer, que les germes de Shalika (cf. \cite{Ho}). 

\vskip3mm
\noindent {\bf 1.11.} --- Un joli papier de J.--P. Serre \cite{S1} est à l'origine de ce travail. L'auteur y prouve une \og formule de masse \fg{}, valable en toute caractéristique: $
\sum_{E}q^{-(\delta(E/F)-(N-1))}=N$, 
où $E/F$ parcourt les sous--extensions totalement ramifiées de degré $N$ de $F^{\rm sep}/F$ pour une clôture séparable de $F^{\rm sep}$ de $F$, et $\delta(E/F)$ est le discriminant de $E/F$. Cette formule de masse est rappelée dans la section \ref{des invariants}, et étendue à toutes les sous--extensions de degré $N$ de $F^{\rm sep}\!/F$. Jointe à 1.7.(1) et à la formule d'intégration de Weyl, elle entraîne l'intégrabilité locale de la fonction $\eta_G^{-{1\over 2}}:G_{\rm qr}\rightarrow {\Bbb R}_{>0}$ sur $G$. Si on remplace $G$ par le groupe multiplicatif $D^\times$ d'une algèbre à divison de centre $F$ et de degré $N^2$ sur $F$, alors 1.7.(1) est pratiquement immédiat par compacité. Pour $G=GL(N,F)$, on peut donc en déduire 1.7.(1) pour les fonctions $f\in C^\infty_{\rm c}(G)$ qui s'obtiennent par transfert à partir de celles sur $D^\times$ --- \cad les fonctions cuspidales sur $G$ ---, mais cette approche ne permet pas de traiter les autres fonctions (rappelons qu'en caractéristique nulle, une intégrale orbitale semisimple régulière elliptique s'écrit comme une combinaison linéaire de caractères de représentations elliptiques mais aussi de caractères pondérés). Curieusement, la formule de masse de Serre, qui semblait au départ cruciale pour l'étude de la fonction $\eta_G^{-{1\over 2}}$, s'est révélée au bout du compte inutile, puisque l'intégrabilité locale de la fonction 
$\eta_G^{-{1\over 2}}$ est impliquée par celle de la fonction $\eta_G^{-{1\over 2}-\epsilon}$, obtenue sans utiliser la formule de masse.  

Signalons que le facteur de normalisation $\eta_G^{1\over 2}$ apparaît pour le groupe $G=GL(2,F)$ dans le livre de 
Jacquet--Langlands \cite{JL}, l'un des (trop) rares textes sur la question écrit en caractéristique $\geq 0$.

\vskip1mm
Je remercie vivement le rapporteur pour sa lecture minutieuse du manuscrit.

%%%%%%%%%%%%%%%%%%%%%%%%%%
\section{Des invariants (rappels)}\label{des invariants}

%%%%%%%%%%%%%%%%%%%%%%%%%
\subsection{Extensions}\label{extensions}Soit $F$ un corps commutatif localement compact non archimédien, de caractéristique résiduelle $p$. On note $\mathfrak{o}$ l'anneau des entiers de $F$, $\mathfrak{p}$ l'idéal maximal de $\mathfrak{o}$, et $\kappa$ le corps résiduel $\mathfrak{o}/\mathfrak{p}$. Ce dernier est un corps fini de cardinal $q= p^r$ pour un entier $r\geq 1$ et un nombre premier $p$. On note $\nu$ la valuation sur $F$ normalisée par $\nu(F^\times)={\Bbb Z}$, et $\vert\; \vert$ la valeur absolue sur $F$ donnée par $\vert x \vert = q^{-\nu(x)}$. 

Soit $E$ une extension finie de $F$. On définit de la même manière, en les affublant d'un indice $E$, les objets $\mathfrak{o}_E$, $\mathfrak{p}_E$, $\kappa_E= \mathfrak{o}_E/\mathfrak{p}_E$, $q_E=\vert \kappa_E\vert$, (etc.). On pose aussi $U_E=U_E^0=\mathfrak{o}_E^\times$ et $U_E^k= 1+ \mathfrak{p}_E^k$ ($k\geq 1$). On note $e(E/F)$, resp. $f(E/F)$, l'indice de ramification, resp. le degré résiduel, de $E/F$. La valuation normalisée $\nu_F=\nu$ sur $F$ se prolonge de manière unique en une valuation sur $E$, que l'on note encore $\nu_F$. Cette dernière est reliée à la valuation normalisée $\nu_E$ sur $E$ par l'égalité $\nu_E= e(E/F)\nu_F$. On a donc
$$
\vert x \vert_E = q_E^{-\nu_E(x)}=q^{-[E:F]\nu_F(x)},\quad x\in E^\times.
$$

Supposons l'extension $E/F$ séparable. On note $N_{E/F}:E^\times \rightarrow F^\times$, resp. 
$T_{E/F}:E \rightarrow F$, l'homomorphisme norme, resp. trace. On a
$$
\nu_E(x)= {1\over f(E/F)}\nu(N_{E/F}(x)),\quad x\in E^\times.
$$
L'homomorphisme $T_{E/F}$ est surjectif, et la forme bilinéaire
$E\times E\rightarrow F,\, (x,y)\mapsto T_{E/F}(x)$ est non dégénérée. Soit $\ES{D}_{E/F}$ la différente de $E/F$, \cad l'inverse de l'idéal fractionnaire $\ES{D}_{E/F}^-$ de $\mathfrak{o}_E$ donné par
$$
\ES{D}_{E/F}^-=\{x\in E: {\rm Tr}_{E/F}(xy)\in \mathfrak{o},\, \forall y\in \mathfrak{o}_E\}.
$$
On note $\delta(E/F)$ l'exposant du discriminant de $E/F$, \cad l'entier $\geq 0$ donné par
$\delta(E/F)= \nu(N_{E/F}(x))$
pour un (i.e. pour tout) élément $x\in E^\times$ tel que $\ES{D}_{E/F}= x\mathfrak{o}_E$. On sait \cite[III, \S7, prop.~13]{S2} que
$$
\delta(E/F)\geq f(E/F)(e(E/F)-1)\leqno{(1)}
$$
avec égalité si et seulement si l'extension $E/F$ est {\it modérément ramifiée}, \cad si son indice de ramification $e(E/F)$ est premier à $p$. On note $\sigma(E/F)$ l'entier défini par
$$
\sigma(E/F)= {\delta(E/F)\over f(E/F)}-(e(E/F)-1).\leqno{(2)}
$$
On a donc $\sigma(E/F)\geq 0$ avec égalité si et seulement si $(e(E/F),p)=1$, o $(a,b)$ dsigne le plus grand commun diviseur de $a$ et $b$.

%%%%%%%%%%%%%%%%%%%%%%%%
\subsection{L'invariant $\tilde{k}_F(\gamma)$}\label{l'invariant k}
Soit $\gamma\neq 0$ un élément algbrique sur $F$. Alors $E=F[\gamma]$ est une extension finie de $F$, et l'on pose $e=e(E/F)$, $f=f(E/F)$ et $n=ef$.

L'ensemble $\{\mathfrak{p}_E^i: i\in {\Bbb Z}\}$ des idéaux fractionnaires de $\mathfrak{o}_E$ est une chaîne de $\mathfrak{o}$--réseaux dans $E$ (vu comme un $F$--espave vectoriel de dimension $n$). Pour chaque entier $k$, cette chaîne définit un $\mathfrak{o}$--réseau  $\mathfrak{P}^k(E)={\rm End}_{\mathfrak{o}}^k(\{\mathfrak{p}_E^i\})$ dans $A(E)={\rm End}_F(E)$, donné par
$$
\mathfrak{P}^k(E)=\{u\in A(E): u(\mathfrak{p}_E^i)\subset \mathfrak{p}_E^{i+k}, \, \forall i\in {\Bbb Z}\}.
$$
Alors $\mathfrak{A}(E)= \mathfrak{P}^0(E)$ est un $\mathfrak{o}$--ordre héréditaire dans $A(E)$, et c'est l'unique $\mathfrak{o}$--ordre héréditaire dans $A(E)$ normalisé par $E^\times$ (pour l'identification naturelle $E^\times\subset {\rm Aut}_F(E)$). De plus, 
$\mathfrak{P}(E) = \mathfrak{P}^1(E)$ est le radical de Jacobson de $\mathfrak{A}(E)$ --- c'est donc, en particulier, un idéal fractionnaire de $\mathfrak{A}(E)$ ---, et pour $k\in {\Bbb Z}$, on a $\mathfrak{P}^k(E)= \mathfrak{P}(E)^k$. Soit ${\rm ad}_\gamma: A(E)\rightarrow A(E)$ l'homomorphisme adjoint, donné par
$$
{\rm ad}_\gamma(u)= \gamma u-u \gamma,\quad u\in A(E).
$$
En \cite[1.4.5, 1.4.11]{BK} est défini un invariant $k_0(\gamma,\mathfrak{A}(E))\in {\Bbb Z}\cup\{-\infty\}$, que l'on note ici $k_F(\gamma)$. Rappelons sa définition. Si $E=F$, on pose $k_F(\gamma)=-\infty$; sinon, $k_F(\gamma)$ est le plus petit $k\in {\Bbb Z}$ vérifiant l'inclusion
$$
\mathfrak{P}^k(E)\cap {\rm ad}_\gamma(A(E))\subset {\rm ad}_\gamma(\mathfrak{A}(E)).
$$
On pose $n_F(\gamma)= -\nu_E(\gamma)\in {\Bbb Z}$ et
$$
\tilde{k}_F(\gamma)= k_F(\gamma)+n_F(\gamma)\in {\Bbb Z}\cup \{-\infty\}.\leqno{(1)}
$$
L'élément $\gamma$ est dit {\it $F$--minimal} si les deux conditions suivantes sont vérifiées:
\begin{itemize}
\item l'entier $n_F(\gamma)$ est premier à l'indice de ramification $e(E/F)$;
\item pour une (i.e. pour toute) uniformisante $\varpi$ de $F$, l'élément $\varpi^{-\nu_E(\gamma)}\gamma^{e} + \mathfrak{p}_E$ de $\kappa_E$ engendre l'extension $\kappa_E$ sur $\kappa$.
\end{itemize}
En particulier, tout élément de $F^\times$ est $F$--minimal. D'après \cite[1.4.15]{BK}, si $E\neq F$, on a:
\begin{enumerate}[leftmargin=17pt]
\item[(2)] $\tilde{k}_F(\gamma)\geq 0$ avec égalité si et seulement si $\gamma$ est $F$--minimal.
\end{enumerate}

%%%%%%%%%%%%%%%%%%%%%%%%%%%%%%%
\subsection{L'invariant $\tilde{c}_F(\gamma)$}\label{l'invariant c}
Continuons avec les hypothèses et les notations de \ref{l'invariant k}. Supposons de plus que $\gamma$ appartient à $\mathfrak{o}_E$. Alors $\mathfrak{o}[\gamma]$ est un sous--anneau de $\mathfrak{o}_E$, et le sous--ensemble $(\mathfrak{o}[\gamma]:\mathfrak{o}_E)\subset \mathfrak{o}_E$ --- appelé conducteur de $\mathfrak{o}[\gamma]$ dans $\mathfrak{o}_E$ --- défini par
$$
(\mathfrak{o}[\gamma]:\mathfrak{o}_E)=\{x\in \mathfrak{o}_E: x\mathfrak{o}_E \subset \mathfrak{o}[\gamma]\}
$$
est un idéal de $\mathfrak{o}_E $. On note $c_F(\gamma)\geq 0$ l'exposant de cet idéal, \cad que l'on pose
$$
(\mathfrak{o}[\gamma]:\mathfrak{o}_E)=\mathfrak{p}_E^{c_F(\gamma)}.
$$
Soit $\phi_\gamma\in F[t]$ le polynôme minimal de $\gamma$ sur $F$, et soit $\phi'_\gamma\in F[t]$ la dérivée de $\phi_\gamma$. D'après 
\cite[III, \S6, cor.~1]{S2}, on a:
\begin{enumerate}[leftmargin=17pt]
\item[(1)]si l'extension $E/F$ est séparable, alors $(\mathfrak{o}[\gamma],\mathfrak{o}_E)= \phi'_\gamma(\gamma) \ES{D}_{E/F}^-$.
\end{enumerate}
On en déduit que
\begin{enumerate}[leftmargin=17pt]
\item[(2)]pour $z\in \mathfrak{o}\smallsetminus \{0\}$, on a $c_F(z\gamma)= e(n-1)\nu(z) + c_F(\gamma)$.
\end{enumerate}
En effet, pour $z\in \mathfrak{o}\smallsetminus \{0\}$, on a $\phi'_{z\gamma}(z\gamma)= z^{n-1}\phi'_\gamma(\gamma)$. Si l'extension $E/F$ séparable --- par exemple si $F$ est de caractéristique nulle ---, d'après (1), on a
$$
c_F(z\gamma)- c_F(\gamma)= \nu_E(z^{n-1})= e(n-1)\nu(z),
$$
d'où (2). Si $F$ est de caractéristique $p$, le résultat se déduit de celui en caractéristique nulle par la théorie des corps locaux proches \cite{D}.  

On ne suppose plus que $\gamma$ appartient à $\mathfrak{o}_E$. Grâce à (2), on peut encore définir l'invariant $c_F(\gamma)$: 
on choisit un élément $z\in  \mathfrak{p}\smallsetminus \{0\}$ tel que $z\gamma\in \mathfrak{o}_E$, et on pose
$$
c_F(\gamma)= c_F(z\gamma)-e(n-1)\nu(z).
$$
D'après (2), l'entier $c_F(\gamma)$ est bien défini (\cad qu'il ne dépend pas du choix de $z$). On pose
$$
\tilde{c}_F(\gamma)= c_F(\gamma)+ (n-1)n_F(\gamma).\leqno{(3)}
$$
Par construction, on a
$$
\tilde{c}_F(z\gamma)= \tilde{c}_F(\gamma),\quad z\in F^\times.\leqno{(4)}
$$

% remarque 1
\begin{marema1}
{\rm 
Choisissons un élément $z\in F^\times$ tel que $z\gamma \in \mathfrak{o}_E\smallsetminus \mathfrak{p}_E^e$. Posons $\gamma'=z\gamma$. 
Puisque $c_F(\gamma')\geq 0$ et $n_F(z\gamma)=-\nu_E(\gamma')\geq 1-e$, on a
$$
\tilde{c}_F(\gamma) =\tilde{c}_F(\gamma')\geq -(n-1)(e-1).\eqno{\blacksquare}
$$
}
\end{marema1}

Soit ${\rm Ad}_\gamma:A(E)\rightarrow A(E)$ l'automorphisme $u\mapsto \gamma u \gamma^{-1}$. On pose
$$
D_F(\gamma)=\left\{\begin{array}{ll}
{\rm det}_F(1-{\rm Ad}_\gamma; A(E)/E) & \mbox{si $E\neq F$}\\
1 & \mbox{sinon}
\end{array}\right..
$$
On a
$$
D_F(\gamma^{-1})= (-1)^{n(n-1)}D_F(\gamma),
$$et $D_F(\gamma)\neq 0$ si et seulement si l'extension $E/F$ est séparable. On a aussi:
\begin{enumerate}[leftmargin=17pt]
\item[(5)]si l'extension $E/F$ est séparable, alors $ \tilde{c}_F(\gamma)= {1\over f}(\nu(D_F(\gamma))-\delta(E/F))$.
\end{enumerate}
Montrons (5). Supposons l'extension $E/F$ séparable et posons $\delta=\delta(E/F)$. On a
$$
D_F(\gamma)=N_{E/F}(\gamma^{1-n}\phi'_\gamma(\gamma)),
$$
d'où
$$
\nu(D_F(\gamma))= \nu\circ N_{E/F}(\gamma^{1-n}\phi'_\gamma(\gamma))= f\nu_E(\gamma^{1-n}\phi'_\gamma(\gamma)).
$$
D'autre part, d'après (1), on a
$$
c_F(\gamma)= \nu_E(\phi'_\gamma(\gamma))- {\delta \over f}.
$$
On en déduit que
$$
\nu(D_F(\gamma))=f(1-n)\nu_E(\gamma)+fc_F(\gamma)+\delta,
$$
puis que
$$
{1\over f}(\nu(D_F(\gamma)-\delta)= (1-n)\nu_E(\gamma) + c_F(\gamma)=\tilde{c}_F(\gamma).
$$

% remarque 2
\begin{marema2}
{\rm 
Supposons que l'extension $E/F$ est totalement ramifiée (mais pas forcment sparable), et soit $\varpi_E$ une uniformisante de $E$. C'est un élément $F$--minimal, et puisque $\mathfrak{o}[\varpi_E]=\mathfrak{o}_E$, on a $c_F(\varpi_E)=0$ et $\tilde{c}_F(\varpi_E)=1-n$. Si de plus l'extension $E/F$ est séparable, alors d'après (5), on a $\nu(D_F(\varpi_E))=\delta+1-n$.\hfill $\blacksquare$
}
\end{marema2}

%%%%%%%%%%%%%%%%%%%%%%%%%%
\subsection{La \og formule de masse\fg{} de Serre}\label{la formule de masse de Serre}
Pour toute extension finie $E$ de $F$, on note $w(E/F)$ le nombre de $F$--automorphismes de $E$. Si $F'/F$ la sous--extension non ramifiée maximale de $E/F$, on a donc $w(E/F)= f(E/F)w(E/F')$.

Fixons un entier $n\geq 1$ et une cl\^oture sparable $F^{\rm sep}$ de $F$. Soit $\bs{\ES{E}}(n)$ l'ensemble des sous--extensions de degré $n$ de $F^{\rm sep}\!/F$. Pour chaque entier $e\geq 1$ divisant $n$, soit $\bs{\ES{E}}_e(n)$ le sous--ensemble de $\bs{\ES{E}}(n)$ formé des extensions $E/F$ telles que $e(E/F)=e$. Lorsque $n$ est premier à $p$, l'ensemble $\bs{\ES{E}}_n(n)$ est fini de cardinal $n$. Si $F$ est de caractéristique $p$ et $p$ divise $n$, l'ensemble $\bs{\ES{E}}_n(n)$ est infini. En général, d'après Serre \cite{S1}, on a la {\it formule de masse}
$$
\sum_{E\in \bs{\ES{E}}_n(n)}q^{-\sigma(E/F)}=n,\leqno{(1)}
$$
où (rappel) $\sigma(E/F)= \delta(E/F)-(n-1)$. Notons $\ES{E}(n)$ l'ensemble des classes d'isomorphisme de sous--extensions 
de degré $n$ de $F^{\rm sep}/F$, et $\ES{E}_n(n)$ le sous--ensemble de $\ES{E}(n)$ formé des extensions qui sont totalement ramifiées. Pour chaque 
$E\in \ES{E}_n(n)$, il y a exactement ${n\over w(E/F)}$ éléments de $\bs{\ES{E}}_n(n)$ dans la classe $E$. La formule (1) s'écrit donc aussi
$$
\sum_{E\in \ES{E}_n(n)}{1\over w(E/F)}q^{-\sigma(E/F)}=1.\leqno{(2)}
$$

% remarque1
\begin{marema1}
{\rm Dans \cite{S1}, Serre donne deux preuves de la formule (2). La première consiste en gros à compter les polynômes d'Eisenstein de degré $n$ (cela devrait en principe pouvoir se déduire des résultats de Krasner, mais l'approche de Serre est plus directe). La seconde, très courte, est une simple application de la formule d'intégration de Hermann Weyl: on fixe une algèbre à division $D$ de centre $F$ et de dimension $n^2$ sur $F$. L'ensemble des uniformisantes de $D$ est un ouvert compact de $D^\times$ (c'est un espace principal homogène, à gauche et à droite, sous le groupe des unités de $D^\times$). La formule d'intégration de Weyl appliquée à la fonction caractéristique de cet ensemble donne le résultat.\hfill$\blacksquare$
}
\end{marema1}

Fixons un entier $e\geq 1$ divisant $n$, et posons $f={n\over e}$. Notons $F'/F$ la sous--extension non ramifiée de degré $f$ de $F^{\rm sep}$. Soit $\ES{E}'(e)$ l'ensemble des classes d'isomorphisme de sous--extensions de degré $e$ de $F^{\rm sep}\!/F'$, et soit $\ES{E}'_e(e)$ le sous--ensemble de $\ES{E}'(e)$ formé des extensions qui sont totalement ramifiées. D'après (2), on a
$$
\sum_{E\in \ES{E}'_e(e)}{1\over w(E/F')}q_{F'}^{-\sigma(E/F')}=1.
$$
Or $\ES{E}'_e(e)= \ES{E}_e(n)$, et pour $E\in \ES{E}_e(n)$, on a $w(E/F')= {1\over f}w(E/F)$ et $\delta(E/F')={1\over f}\delta(E/F)$, d'où 
$\sigma(E/F')=\sigma(E/F)$. On obtient
$$
\sum_{E\in \ES{E}_e(n)}{1\over w(E/F)}q_{F'}^{-\sigma(E/F)}={1\over f},
$$
avec
$$
q_{F'}^{-\sigma(E/F)}= q^{-f\sigma(E/F)}=q^{-(\delta(E/F)-(n-f))}.
$$
En sommant sur tous les entiers $e\geq 1$ divisant $n$, on obtient
$$
\sum_{E\in \ES{E}(n)}{1\over w(E/F)}q_E^{-\sigma(E/F)}=\sum_{e\vert n}{e\over n}.\leqno{(3)}
$$

% remarque 2
\begin{marema2}
{\rm La formule (2), ainsi que la formule (3) qui en est issue, s'appuient sur la propriété vérifiée par toute uniformisante 
$\varpi_E$ dans une extension totalement ramifiée $E/F$ de degré $n$ (remarque 2 de \ref{l'invariant c}): $\vert D_F(\varpi_E) \vert q^{\sigma(E/F)}= 1$. Plus généralement, pour $E\in \ES{E}(n)$, notant $E_*^\times $ l'ensemble des $\gamma\in E^\times$ tels que $F[\gamma]=E$, on définit la fonction
$$
E_*^\times \rightarrow {\Bbb R}_{>0},\, \gamma \mapsto \vert D_F(\gamma)\vert q_E^{\sigma(E/F)}.
$$
On peut la prolonger en une fonction sur $D^\times$, où $D$ est une algèbre à division de centre $F$ et de dimension $n^2$ sur $F$, et aussi --- plus intéressant encore! --- en une fonction sur $GL(n,F)$. C'est essentiellement l'étude de ce prolongement qui nous intéresse ici.\hfill $\blacksquare$
}
\end{marema2}

%%%%%%%%%%%%%%%%%%%%%%%%%%%%%%%%%%%%%%
\section{Descente centrale au voisinage d'un élément pur}\label{descente centrale au voisinage d'un élément pur}

%%%%%%%%%%%%%%%%%%%%%%%%%%%%%%%%%%
\subsection{\'Eléments quasi--réguliers elliptiques}\label{éléments qre}
On fixe un entier $N\geq 1$ et un $F$--espace vectoriel $V$ de dimension $N$. On pose $G={\rm Aut}_F(V)$ et $\mathfrak{g}={\rm End}_F(V)$, et on munit $G$ et $\mathfrak{g}$ de la topologie definie par $F$ (i.e. $\mathfrak{p}$--adique). On note $Z=F^\times$ le centre de $G$, et $\mathfrak{z}=F$ celui de $\mathfrak{g}$. Pour $\gamma\in \mathfrak{g}$, on note $\mathfrak{g}_\gamma =\{x\in \mathfrak{g}: \gamma x - x\gamma=0\}$ le centralisateur de $\gamma$ dans $\mathfrak{g}$, \cad le commutant dans $\mathfrak{g}$ de la sous--$F$--algèbre $F[\gamma]\subset \mathfrak{g}$, 
et $\det(\gamma)$ le déterminant ${\rm det}_F(v\mapsto \gamma v;V)$. Un élément $\gamma$ de $\mathfrak{g}$ est dit:
\begin{itemize}
\item {\it fermé} (ou {\it $F$--fermé}) si $F[\gamma]$ est un produit $E_1\times \cdots \times E_d$ d'extensions $E_i$ de $F$;
\item {\it pur} (ou {\it $F$--pur}) si $F[\gamma]$ est une extension de $F$;
\item {\it quasi--régulier} si $F[\gamma]$ est un produit $E_1\times \cdots \times E_d$ d'extensions $E_i$ de $F$ tel que $\sum_{i=1}^{d}[E_i:F]=N$;
\item {\it quasi--régulier elliptique} si $F[\gamma]$ est une extension de degré $N$ de $F$;
\end{itemize}

% remarque 1
\begin{marema1}
{\rm 
Un élément $\gamma\in \mathfrak{g}$ est fermé si et seulement si l'orbite $\{g^{-1}\gamma g:g\in G\}$ est fermée dans $\mathfrak{g}$ (pour la topologie $\mathfrak{p}$--adique) --- cf. \cite[2.3.2]{L2}. Parmi les éléments fermés $\gamma\in \mathfrak{g}$, il y a ceux qui sont (absolument) semisimples, \cad tels que $F[\gamma]$ est un produit $E_1\times \cdots \times E_d$ d'extensions {\it séparables} $E_i$ de $F$. On aurait pu aussi appeler {\it quasi--semisimples} les éléments fermés de $\mathfrak{g}$, mais comme il existe déjà une notion d'élément (absolument) quasi--semisimple différente de celle d'élément fermé introduite ici, on a préféré ne pas le faire. Pour la même raison, on a choisi d'abandonner la terminologie de Bourbaki (reprise dans \cite{L2}): rappelons que les éléments fermés, resp.  semisimples, du présent article sont dans Bourbaki appelés semisimples, resp. absolument semisimples.\hfill $\blacksquare$ 
}
\end{marema1}

Si $\gamma$ est un élément fermé de $\mathfrak{g}$, alors en notant $e_i$ l'idempotent primitif associé à $E_i$ dans la décomposition $F[\gamma]=E_1\times\cdots \times E_d$ et $V_i$ le sous--$F$--espace vectoriel $e_i(V)$ de $V$, on a $E_i\subset {\rm End}_F(V_i)$ et
$$
\mathfrak{g}_\gamma = \mathfrak{b}_1\times \cdots \times \mathfrak{b}_d,\quad \mathfrak{b}_i={\rm End}_{E_i}(V_i).
$$
En particulier on a $\sum_{i=1}^d [E_i:F]\dim_{E_i}(V_i)=N$, et $\gamma$ est pur, resp. quasi--régulier, si et seulement si $d=1$, resp. $\dim_{E_i}(V_i)=1$ pour $i=1,\ldots ,d$ (cette dernière condition est bien sûr équivalente à $\mathfrak{g}_\gamma=F[\gamma]$). Un élément quasi--régulier de $\mathfrak{g}$ est quasi--régulier elliptique si et seulement s'il est pur. Soit $\gamma$ un élément quasi--régulier de $\mathfrak{g}$. 
Posons $F[\gamma]=E_1\times \cdots \times E_d$ et écrivons $\gamma =(\gamma_1,\ldots ,\gamma_d)$, $\gamma_i\in E_i$. Pour $i=1,\ldots ,d$, $\gamma_i$ est un élément quasi--régulier elliptique de $A(E_i)={\rm End}_F(E_i)$. 
On note $\mathfrak{g}_{\rm qr}$ l'ensemble des éléments quasi--réguliers de $\mathfrak{g}$, et $\mathfrak{g}_{\rm qre}\subset \mathfrak{g}_{\rm qr}$ le sous--ensemble formé des éléments elliptiques. 
On pose $G_{\rm qr}= G \cap \mathfrak{g}_{\rm qr}$ et $G_{\rm qre}= G\cap \mathfrak{g}_{\rm qre}$. Notons que si $N=1$, alors $\mathfrak{g}_{\rm qre}=\mathfrak{g}_{\rm qr}=\mathfrak{g}$ et 
$G_{\rm qre}=G_{\rm re}= G$; en revanche si $N>1$, alors $\mathfrak{g}_{\rm qr} \supsetneq G_{\rm qr}$ et $\mathfrak{g}_{\rm qre}= G_{\rm qre}$. 
On définit comme suit une filtration décroissante $k\mapsto G_{\rm qre}^k$ de $G_{\rm qre}$: pour $k\in {\Bbb R}$, on pose
$$
G_{\rm qre}^k=\{\gamma\in G_{\rm qre}: \nu_F(\gamma)\geq k\}.
$$
Les \og sauts\fg de cette filtration sont les éléments de ${1\over N}{\Bbb Z}$: pour $k\in {\Bbb R}$, l'inclusion
$$\bigcup_{k'>k} G^{k'}_{\rm qre}\subset G_{\rm qre}^k$$ est stricte si et seulement si $k\in {1\over N}{\Bbb Z}$. 

On commence par quelques rappels sur les $\mathfrak{o}$--ordres héréditaires dans $\mathfrak{g}$ (cf. \cite[1.1]{BK}). On fixe un $\mathfrak{o}$--ordre héréditaire minimal (ou d'Iwahori) $\mathfrak{A}_{\rm min}$ dans $\mathfrak{g}$, et un $\mathfrak{o}$--ordre héréditaire maximal $\mathfrak{A}_{\rm max}$ dans $\mathfrak{g}$ contenant $\mathfrak{A}_{\rm min}$. Rappelons que $\mathfrak{A}_{\rm min}$ est le stabilisateur ${\rm End}_{\mathfrak{o}}^0(\ES{L})$ d'une chaîne de $\mathfrak{o}$--réseaux $\ES{L}=\{\ES{L}_i:i\in {\Bbb Z}\}$ dans $V$ telle que $\ES{L}_{i+1}\subsetneq \ES{L}_i$ et $\mathfrak{p}\ES{L}_i= \ES{L}_{i+N}$. Cette chaîne est unique à translation des indices près, et quitte à changer l'indexation, on peut supposer que $\mathfrak{A}_{\rm max}= {\rm End}_{\mathfrak{o}}^0(\{\ES{L}_{Ni}\})
\;(={\rm End}_{\mathfrak{o}}(\ES{L}_0))$. Pour chaque entier $e\geq 1$ divisant $N$, on note $\mathfrak{A}_e$ l'unique $\mathfrak{o}$--ordre héréditaire principal dans $\mathfrak{g}$ de période $e(\mathfrak{A}_e\vert \mathfrak{o})=e$ tel que $\mathfrak{A}_{\rm min}\subset \mathfrak{A}_e\subset \mathfrak{A}_{\rm max}$. On a donc $\mathfrak{A}_e={\rm End}_{\mathfrak{o}}^0(\{\ES{L}_{(N/e)i}\})$. 
Ces $\mathfrak{o}$--ordres héréditaires principaux $\mathfrak{A}_e$ dans $\mathfrak{g}$ sont dits {\it standards}. Ils forment un système de représentants des classes de $G$--conjugaison d'$\mathfrak{o}$--ordres héréditaires principaux dans $\mathfrak{g}$. 

Soit $e\geq 1$ un entier divisant $N$. On note $\mathfrak{P}_e$ le radical de Jacobson de $\mathfrak{A}_e$. Pour $k\in {\Bbb Z}$, on a 
donc $\mathfrak{p}^k\mathfrak{A}_e = \mathfrak{P}_e^{ke}$. On note $U^0_e=U^0(\mathfrak{A}_e)$ le sous--groupe (compact, ouvert) de $G$ défini par $U_e^0=\mathfrak{A}_e^\times$, $K_e=K(\mathfrak{A}_e)$ le normalisateur de $\mathfrak{A}_e$ dans $G$, et, pour chaque entier $k\geq 1$, $U_e^k=U^k(\mathfrak{A}_e)$ le sous--groupe distingué de $K_e$ défini par $U_e^k=1 + \mathfrak{P}_e^k$. Alors $U_e^0$ est l'unique sous--groupe compact maximal de $K_e$. Notons que pour $k\in {\Bbb Z}$, on a
$$
\{\det(\gamma): \gamma \in \mathfrak{P}_e^{k}\}= \mathfrak{p}^{k{N\over e}}.\leqno{(1)}
$$

Pour tous entiers $e,\,e'\geq 1$ tels que $e\vert N$ et $e'\vert e$, on a les inclusions
$$
\mathfrak{A}_{\rm min}\subset \mathfrak{A}_e\subset \mathfrak{A}_{e'}\subset \mathfrak{A}_{\rm max}
$$
et
$$
\mathfrak{P}_{\rm max}\subset \mathfrak{P}_{e'}\subset \mathfrak{P}_e \subset \mathfrak{P}_{\rm min}.
$$
Plus généralement, en posant $a=e/e'$, pour $k\in {\Bbb Z}$, on a les inclusions
$$
\mathfrak{P}_e^{ak}\subset \mathfrak{P}_{e'}^k\subset \mathfrak{P}_e^{a(k-1)+1}.\leqno{(2)}
$$

Soit $\gamma\in G_{\rm qre}$. Posons $E= F[\gamma]$. L'inclusion $E\subset \mathfrak{g}$ identifie $V$ à un $E$--espace vectoriel de dimension $1$, et le choix d'un vecteur non nul $v\in V$ identifie $\mathfrak{g}={\rm End}_F(V)$ à $A(E)$. Soit $\mathfrak{A}_\gamma$ l'$\mathfrak{o}$--ordre héréditaire dans $\mathfrak{g}$ correspondant à $\mathfrak{A}(E)$ via cette identification. Il est principal, de période $e(\mathfrak{A}_\gamma\vert \mathfrak{o})=e(E/F)$, et ne dépend pas du choix du vecteur $v$: c'est l'unique $\mathfrak{o}$--ordre héréditaire dans $\mathfrak{g}$ normalisé par $E^\times$. On note $\mathfrak{P}_\gamma$ le radical de Jacobson de $\mathfrak{A}_\gamma$, et on pose $U_\gamma^0= \mathfrak{A}_\gamma^\times$, $K_\gamma = K(\mathfrak{A}_\gamma)$, et $U_\gamma^k= 1 + \mathfrak{P}_\gamma^k$ ($k\geq 1$). Notons que
$$
\nu_E(\gamma)= k \Leftrightarrow \gamma\in \mathfrak{P}_\gamma^k\smallsetminus \mathfrak{P}_\gamma^{k+1}.\leqno{(3)}
$$
Si $\mathfrak{A}_\gamma = \mathfrak{A}_{e(E/F)}$, on dit que $\gamma$ est {\it en position standard}.

Pour $a\in {\Bbb R}$, on note $\lfloor a\rfloor$ la partie entière de $x$, \cad le plus petit grand entier inférieur ou égal à $a$, et $\lceil a \rceil = -\lfloor -a \rfloor$ le plus petit entier suprieur ou gal  $a$. Pour une partie 
$\mathfrak{X}$ de $\mathfrak{g}$, on pose ${^G{\mathfrak{X}}}=\{g^{-1}\gamma g: g\in G,\, x\in \mathfrak{X}\}$.

% lemme 1
\begin{monlem1}
{\rm Pour $k\in {1\over N}{\Bbb Z}$, on a
$$
G_{\rm qre}^k= \bigcup_{e\vert N}G_{\rm qre} \cap {^G({\mathfrak{P}_e^{{\lfloor ek - {e\over N} \rfloor}+1}})}.
$$
}\end{monlem1}

\begin{proof}Commen\c{c}ons par l'inclusion $\subset$. Soit $\gamma\in G_{\rm qre}^k$. Quitte à remplacer $\gamma$ par $g^{-1}\gamma g$ pour un élément $g\in G$, on peut supposer $\gamma$ en position standard. Posons $E=F[\gamma]$, $e=e(E/F)$ et $f=f(E/F)\;(={N\over e})$. On a $\nu_E(\gamma)=e\nu_F(\gamma) \geq ek\in {1\over f}{\Bbb Z}$. Comme $\nu_E(\gamma)\in {\Bbb Z}$, on obtient
$$
\nu_E(\gamma)\geq \lfloor ek + \textstyle{f-1\over f}\rfloor = \lfloor ek-\textstyle{1\over f}\rfloor +1.
$$
D'où l'inclusion $\subset$, d'après (3).

Montrons l'inclusion $\supset$. Soit $\gamma\in {^G({\mathfrak{P}_e^{{\lfloor ek - {e\over N} \rfloor}+1}})}$ pour un entier $k\in {\Bbb Z}$ et un entier $e\geq 1$ divisant $N$. Posons $f={N\over e}$. D'après (1), on a 
$$
\nu_F(\gamma)= \textstyle{1\over N}\nu(\det(\gamma))\geq \textstyle{1\over e}({\lfloor ek - {1\over f} \rfloor}+1).
$$
\'Ecrivons $ek= r +{t\over f}$ avec $r,\,t\in {\Bbb Z}$ et $0\leq t \leq f-1$. On a donc 
${\lfloor ek - {1\over f} \rfloor}= r + {\lfloor {t-1\over f}\rfloor}$. Comme $\nu_F(\gamma)\in {1\over e}{\Bbb Z}$, on obtient 
que $\nu_F(\gamma)\geq {1\over e}(r+1)>k$ si $t>0$ et $\nu_F(\gamma)\geq {r\over e}=k$ si $t=0$. D'où l'inclusion $\supset$.
\end{proof}

% lemme 2
\begin{monlem2}
Pour $k\in {1\over N}+ {\Bbb Z}$, on a
$$
G_{\rm qre}^k= G_{\rm qre} \cap {^G(\mathfrak{P}_{\rm min}^{Nk})}.
$$
\end{monlem2}

\begin{proof}
Puisque $\mathfrak{P}_{\rm min}^{Nk}=\mathfrak{P}_N^{{\lfloor Nk - {N\over N}\rfloor}+1}$, d'après le lemme 1, 
on a l'inclusion $\supset$. 
\'Ecrivons $k= {1\over N}+c$ avec $c\in {\Bbb Z}$. Pour chaque entier $e\geq 1$ divisant $N$, on a ${\lfloor ek -{e\over N} \rfloor} +1=ec+1$ et
$$
{\mathfrak{P}_{e}^{{\lfloor ek - {e\over N} \rfloor}+1}}= \mathfrak{P}_e^{ec+1}=\mathfrak{p}^c \mathfrak{P}_e\subset \mathfrak{p}^c\mathfrak{P}_{\rm min}= \mathfrak{P}_{\rm min}^{Nc+1}= \mathfrak{P}_{\rm min}^{Nk}.
$$
D'où, à nouveau d'après le lemme 1, l'inclusion $\subset$.
\end{proof}

On définit, comme on l'a fait pour $G$, une filtration décroissante $k\mapsto \mathfrak{g}_{\rm qre}^k$ de $\mathfrak{g}_{\rm qre}$: pour $k\in {\Bbb R}$, on pose
$$
\mathfrak{g}_{\rm qre}^k=\{\gamma\in \mathfrak{g}_{\rm qre}: \nu_F(\gamma)\geq k\}.
$$
Notons que si $N=1$, on a $\mathfrak{g}_{\rm qre}^k= \mathfrak{p}^{\lceil k \rceil}$.

% remarque 2
\begin{marema2}
{\rm Pour un entier $n\geq 1$ et un polynme unitaire $\zeta\in F[t]$ de degr $n$, disons $\zeta(t)= \sum_{i=0}^na_i t^i$, on note $C(\zeta)\in M(n,F)$ la matrice compagnon de $\zeta$ dfinie par 
$$
C(\zeta)= \left(\begin{array}{ccccc}
0 & 1 & 0 &\cdots & 0\\
\vdots & \ddots & \ddots & \ddots & \vdots \\
\vdots &  & \ddots & \ddots & 0\\
0 & \cdots & \cdots & 0& 1\\
-a_0 & -a_1   &\cdots &\cdots &-a_{n-1}
\end{array}
\right).
$$
Via le choix d'un $\mathfrak{o}$--base de $\ES{L}$ (cf. \cite[1.1.7]{BK}), identifions $\mathfrak{g}$  $M(N,F)$ et 
$\mathfrak{A}_{\rm min}$  la sous--$\mathfrak{o}$--algbre de ${\rm M}(N,\mathfrak{o})$ forme des matrices triangulaires suprieures modulo $\mathfrak{p}$. Alors $\mathfrak{P}_{\rm min}$ est la sous--$\mathfrak{o}$--algbre de ${\rm M}(N,\mathfrak{o})$ forme des matrices {\it strictement} triangulaires suprieures modulo $\mathfrak{p}$. Pour $\gamma\in \mathfrak{g}$, notons $\zeta_{\gamma,1},  \ldots  , \zeta_{\gamma,r}$ les invariants de similitude (diviseurs élémentaires) de $\gamma$:
\begin{itemize}
\item pour $i=1,\ldots ,r=r_\gamma$, $\zeta_{\gamma,i}\in F[t]$ est un polynme unitaire de degr $N_i=N_{\gamma,i}$;
\item pour $i= 2,\ldots ,r$, $\zeta_{\gamma,i}$ divise $\zeta_{\gamma,i-1}$;
\item $\zeta_{\gamma,1}$ est le polynme minimal de $\gamma$;
\item $\zeta_\gamma = \prod_{i=1}^r \zeta_{\gamma,i}$ est le polynme caractristique de $\gamma$.
\end{itemize} 
On note $\tilde{\gamma}\in \mathfrak{g}$ la matrice diagonale par blocs de taille $N_1,\ldots ,N_r$, dfinie par
$$
\tilde{\gamma}= {\rm diag}(C(\zeta_{\gamma,1}),\ldots ,C(\zeta_{\gamma,r})).
$$
D'aprs le thorme de dcomposition de Frobenius, un lment $\gamma'\in \mathfrak{g}$ est dans la classe de $G$--conjugaison de 
$\gamma$ si et seulement si $\tilde{\gamma}=\tilde{\gamma}'$, i.e. si et seulement si 
$r_{\gamma'}= r$ et $\zeta_{\gamma'\!,i}= \zeta_{\gamma,i}$ pour $i=1,\ldots ,r$. 
\'Ecrivons $\zeta_\gamma(t) = \sum_{j=0}^N a_{\gamma,j}t^j$. Pour $j=0,\ldots ,N-1$, on a
$$
a_{g^{-1}\gamma g,j} = a_{\gamma,j}, \quad g\in G,\leqno{(4)}
$$
et 
$$
a_{z\gamma,j} = z^{N-j} a_{\gamma,j}, \quad z\in Z=F^\times.\leqno{(5)}
$$
On en dduit que pour $k\in {\Bbb Z}$, on a
$$
\{\gamma\in \mathfrak{g}: \nu (a_{\gamma,j})\geq (N-j)k+ 1,\; j=0,\ldots ,N-1\}= {^G(\mathfrak{P}_{\rm min}^{Nk+1})}.\leqno{(6)}
$$
En effet, puisque pour une uniformisante $\varpi$ de $F$, on a ${^G(\mathfrak{P}_{\rm min}^{Nk +1})}= \varpi^k ({^G(\mathfrak{P}_{\rm min})})$, d'aprs (5), il suffit de vrifier (6) pour $k=0$. Si $\gamma\in \mathfrak{P}_{\rm min}$, alors $\overline{\gamma}= \gamma\; ({\rm mod}\,\mathfrak{p})$ est un lment nilpotent  de $M(N,\kappa)$, et comme $\zeta_\gamma\;({\rm mod}\,\mathfrak{p})\in \kappa[t]$ co\"{\i}ncide avec le polynme caractristique $\zeta_{\overline{\gamma}} =t^N\in \kappa[t]$ de $\overline{\gamma}$, les coefficients $a_{\gamma,j}$ ($j=0,\ldots ,N-1$) appartiennent  $\mathfrak{p}$. D'aprs (4), on a donc l'inclusion $\supset $ dans (6) pour $k=0$. Rciproquement, si $\gamma\in \mathfrak{g}$ est tel que $\nu(a_{\gamma,j})\geq 1$ pour $j=0,\ldots ,N-1$, alors $\zeta_\gamma \;({\rm mod}\,\mathfrak{p})=t^N$. Puisque l'anneau $\mathfrak{o}[t]$ est factoriel et que les polynmes $\zeta_{\gamma,1},\ldots , \zeta_{\gamma,r_\gamma}\in F[t]$ sont unitaires (donc en particulier primitifs), ils appartiennent tous  $\mathfrak{o}[t]$ et on a $\zeta_{\gamma,i}\;({\rm mod}\,\mathfrak{p})=t^{N_{\gamma,i}}$ pour $i=1,\ldots , r_\gamma$, par consquent le représentant standard $\tilde{\gamma}$ appartient  $\mathfrak{P}_{\rm min}$. On a donc aussi l'inclusion $\subset$ dans (6) pour $k=0$. 
\hfill $\blacksquare$
}
\end{marema2}

Notons que l'on peut retrouver le lemme 2  partir de (6). De plus, d'après (6), pour tout $k\in {\Bbb Z}$, on a:
\begin{enumerate}[leftmargin=17pt]
\item[(7)]${^G(\mathfrak{P}_{\rm min}^{Nk +1})}$ est un voisinage ouvert {\it ferm} et $G$--invariant de $0$ dans $\mathfrak{g}$.
\end{enumerate}
%

%%%%%%%%%%%%%%%%%%%%
\subsection{Parties compactes modulo conjugaison}\label{parties compactes modulo conjugaison}Une partie $X$ de $G$, resp. $\mathfrak{g}$, est dite {\it compacte modulo conjugaison} (dans $G$, resp. $\mathfrak{g}$) s'il existe une partie compacte $\Omega$ de $G$, resp. $\mathfrak{g}$, telle que $X$ est contenu dans ${^G\Omega}= \{g^{-1}\gamma g: g\in G,\,\gamma \in \Omega\}$. La caractérisation des parties compactes modulo conjugaison (dans $\mathfrak{g}$) donnée dans \cite[5.2]{K} ne fonctionne plus ici, du fait de la présence possible d'éléments fermés qui ne sont pas semisimples. On procède donc autrement\footnote{Les résultats contenus dans ce numéro sont bien connus, mais comme nous n'avons pas trouvé de référence utilisable, nous avons préféré les redémontrer ici.}.

Pour chaque entier $n\geq 1$, on note $F[t]_n$ la variété $\mathfrak{p}$--adique formée des polynômes unitaires de degré $n$, et 
$q_n:F[t]_n \rightarrow F^{n}$ l'isomorphisme de variétés $\mathfrak{p}$--adiques donné par
$$
q_n(\zeta)= (a_{n-1},\ldots ,a_0),\quad \zeta(t)= \textstyle{\sum_{i=0}^n}a_it^i.
$$
On note $F^\times \times F^n\rightarrow F^n,\, (z,\bs{a})\mapsto z\cdot \bs{a}$ l'action de $F^\times$ sur $F^n$ déduite de l'action 
de $F^\times$ sur $F[t]_n$ donnée par $(z,\zeta)\mapsto z\cdot\zeta$ avec $(z\cdot\zeta)(t)= z^n\zeta(z^{-1}t)$. On a donc
$$
z\cdot \bs{a}= (za_{n-1},z^2a_{n-2},\ldots ,z^na_0),\quad \bs{a}=(a_{n-1},\ldots ,a_0)\in F^n,\; z\in F^\times.
$$ 
On note aussi $F[t]_n^*$ la sous--variété $\mathfrak{p}$--adique ouverte de $F[t]_n$ formée des polynômes qui ne sont pas divisibles par $t$, et $q_n^*: F[t]_n^*\rightarrow  F^{n-1}\times F^\times$ l'isomorphisme de variétés $\mathfrak{p}$--adiques déduit de $q_n$ par restriction.

Reprenons les notations de la remarque 2 de \ref{éléments qre}, ainsi que l'identification $\mathfrak{g}=M(N,F)$ donnée par le choix d'un $\mathfrak{o}$--base de $\ES{L}$. Soit $\pi=\pi_\mathfrak{g}: \mathfrak{g}\rightarrow F[t]_N$ l'application polynôme caractéristique $\gamma \mapsto \zeta_\gamma$. Elle est surjective, et elle se restreint en une application surjective $\pi_G: G\rightarrow F[t]_N^*$. On dispose d'une section $\sigma :F[t]_N \rightarrow \mathfrak{g}$ donnée par l'application matrice compagnon $\zeta \mapsto C(\zeta)$. On a $\sigma(F[t]_N^*)\subset G$. D'après \cite[2.3]{L2}, pour $\zeta\in F[t]_N^*$, la fibre $\pi^{-1}(\zeta)$ au dessus de $\zeta$ --- qui est une partie fermée dans $G$ --- est une union finie de $G$--orbites, égale à la fermeture $\overline{\ES{O}_G(C(\zeta))}$ de $\ES{O}_G(C(\zeta))$ dans $G$ (pour la topologie $\mathfrak{p}$--adique). Parmi ces $G$--orbites, $\ES{O}_G(C(\zeta))$ est l'unique de dimension maximale, et il y en a une seule de dimension minimale, qui est l'unique $G$--orbite fermée (dans $G$) contenue dans $\pi^{-1}(\zeta)$. De même, pour $\zeta \in F[t]_N$, la fibre $\pi^{-1}(\zeta)$ au--dessus de $\zeta$ est une union finie de $G$--orbites, égale à la fermeture $\overline{\ES{O}_G(C(\zeta))}$ de $\ES{O}_G(C(\zeta))$ dans $\mathfrak{g}$.

Rappelons que pour $\gamma \in \mathfrak{g}$, on a défini dans la remarque 2 de \ref{éléments qre} un représentant standard $\tilde{\gamma}\in \ES{O}_G(\gamma)$. Pour $\zeta\in F[t]_N$, on note $R_\zeta$ l'ensemble (fini) des représentants standards $\tilde{\gamma}$ des éléments $\gamma\in \pi^{-1}(\zeta)$. On a donc $\pi^{-1}(\zeta)= \coprod_{\tilde{\gamma}\in R_\zeta} \ES{O}_G(\tilde{\gamma})$.

On note $\Pi= \Pi_\mathfrak{g}: \mathfrak{g}\rightarrow F^N$ l'application composée $q_N\circ \pi$, et $\Pi_G: G\rightarrow F^{N-1}\times F^\times$ sa restriction à $G$.

% lemme 1
\begin{monlem1}
Soit $X$ une partie de $\mathfrak{g}$. Les trois conditions suivantes sont équivalentes:
\begin{enumerate}
\item[(i)] $X$ est une partie compacte modulo conjugaison (dans $\mathfrak{g}$);
\item[(ii)] il existe une partie compacte $\Omega$ de $\mathfrak{g}$ telle que $X$ est contenu dans la fermeture 
$\overline{^G\Omega}$ de ${^G\Omega}$ dans $\mathfrak{g}$ (pour la topologie $\mathfrak{p}$--adique);
\item[(iii)] $\Pi_\mathfrak{g}(X)$ est une partie bornée dans $F^N$.
\end{enumerate}
\end{monlem1}

\begin{proof}
On a clairement $(i)\Rightarrow (ii)$, et puisque l'application $\Pi$ est continue, on a aussi $(ii)\Rightarrow (iii)$. Il suffit donc de prouver $(iii)\Rightarrow (i)$. Supposons que $\Pi(X)$ est une partie bornée dans $F^N$. Alors il existe un entier $k$ tel que $\Pi(X)$ est contenu dans la partie ouverte compacte $\varpi^k\cdot \mathfrak{o}^N$ de $F^N$ (pour l'action de $F^\times$ sur $F^N$ introduite plus haut), où $\varpi$ est une uniformisante de $F$. Pour $\gamma\in \Pi^{-1}(\varpi^k\cdot \mathfrak{o}^N)$, l'expression des coefficients $a_{\gamma,i}$ du polynôme caractéristique de $\gamma$ en termes des fonctions symétriques élémentaires des valeurs propres $\lambda_1,\ldots ,\lambda_N\in \overline{F}$ de $\gamma$ (chaque valeur propre étant comptée un nombre de fois égal à sa multiplicité), entraîne la relation
$$
\nu_F(\lambda_i)\geq k, \quad i=1,\ldots ,N.\leqno{(1)}
$$
En effet, supposons qu'il existe une valeur propre $\lambda_i$ telle que $\nu_F(\lambda_i)<k$, et notons $d\geq 1$ le nombre de valeurs propres $\lambda_j$ qui vérifient cette propriété. Alors ou bien $d=N$, ce qui contredit l'inégalité $\nu(a_{\gamma,0})\geq kN$; ou bien $d<N$, et en supposant que les $\lambda_i$ sont ordonnés de telle manière que $\nu(\lambda_j)<k$ pour $j=1,\ldots ,d$, on a les relations
$$
\nu_F(\lambda_1\cdots \lambda_d)< kd,\quad \nu(a_{N-d})\geq kd.
$$
Cela entraîne l'existence d'un \og mot \fg{} de longueur $d$, disons $\lambda_{i_1}\cdots \lambda_{i_d}$ avec $i_j\in \{1,\ldots ,N\}$, $i_j \neq i_{j'}$ si $j\neq j'$, et $\{i_1,\ldots ,i_d\}\neq \{1,\ldots ,d\}$, tel que
$$
\nu_F(\lambda_{i_1}\cdots \lambda_{i_d})= \nu(\lambda_1\cdots \lambda_d).
$$
Cela n'est possible que s'il existe une valeur propre $\lambda_j$, $j\geq d+1$, telle que $\nu(\lambda_j)<k$; contradiction.  
De (1), on déduit que pour tout $\zeta\in q_N^{-1}(\varpi^k\cdot \mathfrak{o}^N)$ et tout polynôme $\xi\in F[t]_n$ divisant $\zeta$, 
on a $q_n(\xi)\in \varpi^k\cdot \mathfrak{o}^{n}$. Posant $k' = \inf \{k,Nk\}$, cela entraîne que le représentant standard $\tilde{\gamma}\in \ES{O}_G(\gamma)$ appartient à $\mathfrak{p}^{k'} M(N,\mathfrak{o})= \mathfrak{P}_{\rm max}^{k'}$. On a donc
$$
X\subset \Pi^{-1}(\varpi^k\cdot \mathfrak{o}^N) \subset {^G(\mathfrak{P}_{\rm max}^{k'})},
$$
et le lemme est démontré.
\end{proof}

% remarque 1
\begin{marema1}
{\rm Du lemme 1, on déduit en particulier que l'on peut choisir des voisinages ouverts fermés et $G$--invariants du cône nilpotent dans $\mathfrak{g}$ aussi petits que l'on veut, ce que l'on savait déjà d'après \ref{éléments qre}.(7). Précisément, si $\Lambda$ est un $\mathfrak{o}$--réseau dans $\mathfrak{g}$, alors il existe un voisinage ouvert fermé et $G$--invariant $X$ de $\Pi^{-1}(0^N)$ dans $\mathfrak{g}$ tel que $X\subset {^G\Lambda}$. En effet, d'après le lemme 1, on a  
$\Pi^{-1}(\mathfrak{o}^N)\subset {^G(\mathfrak{A}_{\rm max})}$. Soit un entier $k\geq 0$ tel que $\mathfrak{P}_{\rm max}^{k}\subset \Lambda$. Alors $X= \Pi^{-1}(\varpi^k\cdot \mathfrak{o}^N)$ convient.
\hfill $\blacksquare$
}
\end{marema1} 

Avant de poursuivre, rappelons quelques faits élémentaires sur la structure de $F[t]_n^*$, $n\geq 1$. Pour $\zeta\in F[t]_n^*$, on peut écrire $\zeta= \prod_{i=1}^r f_i^{d_i}$ pour des polynômes irréductibles $f_i \in F[t]_{n_i}^*$ et des entiers $d_i\geq 1$ tels que $\sum_{i=1}^r n_id_i=n$. Pour chaque entier $k$, notons $\ES{V}_k(\zeta)$ l'ensemble des polynômes $\zeta'\in F[t]_n$ tels que, posant $\zeta(t)= \sum_{i=0}^n a_it^i$ et $\zeta'=\sum_{i=1}^n a'_i t^i$, on a $a'_i-a_i\in \mathfrak{p}^k$ pour $i=0,\ldots ,n-1$. C'est un voisinage ouvert compact de $\zeta$ dans $F[t]_n$. Puisque 
$a_{\zeta,0}\neq 0$, la condition $k>\nu(a_{\zeta,0})$ assure que $\ES{V}_k(\zeta)$ est contenu dans $F[t]_n^*$. D'après \cite[2.5]{L2}, si l'entier $k$ est suffisamment grand, alors pour tout $\zeta'\in \ES{V}_k(f)$, on a:
\begin{itemize}
\item $\zeta'= \prod_{i=1}^r \zeta'_i$ pour des polynômes $\zeta'_i\in F[t]_{n_id_i}^*$;
\item pour $i=1,\ldots ,r$ et pour toute composante irréductible $f'_i$ de $\zeta'_i$, le degré de $\zeta'_i$ est divisible par $n_i = {\rm deg}(f_i)$.
\end{itemize}
La première propriété est une 
conséquence du lemme de Hensel, et la seconde est une variante du lemme de Krasner (cf. \cite[2.5.1]{L2}). On note $k_\zeta$ le plus petit entier $k> \nu(a_{\zeta,0})$ vérifiant ces deux propriétés. On en déduit que pour 
tout entier $m$, il existe un plus petit entier 
$k_\zeta(m)\geq k_\zeta$ tel que pour tout $\zeta'\in \ES{V}_{k_\zeta(m)}(\zeta)$, on a:
\begin{itemize}
\item pour tout polynôme unitaire $h'$ divisant $\zeta'$, il existe un unique polynôme unitaire $h$ divisant $\zeta$ tel que $h'\in \ES{V}_m(h)$.
\end{itemize}
Le voisinage ouvert compact $\ES{V}_{k_\zeta(m)}(\zeta)$ de $\zeta$ dans $F[t]_n$ est contenu dans $F[t]_n^*$, par conséquent $h'$ appartient à $F[t]_{{\rm deg}(h)}^*$. Si de plus on suppose, ce qui est toujours possible, que pour tout polynôme unitaire $h$ divisant $\zeta$, on a $m>\nu(a_{h,0})$, alors le voisinage ouvert compact $\ES{V}_m(h)$ de $h$ dans $F[t]_{{\rm deg}(h)}$ est contenu dans $F[t]_{{\rm deg}(h)}^*$.

\vskip1mm
Le lemme 2 est la version sur $G$ du lemme 1. Le lemme 3 est une conséquence de la preuve du lemme 2, qui nous servira plus loin. 

% lemme 2
\begin{monlem2}
Soit $X$ une partie de $G$. Les trois conditions suivantes sont équivalentes:
\begin{enumerate}
\item[(i)] $X$ est une partie compacte modulo conjugaison (dans $G$);
\item[(ii)] il existe une partie compacte $\Omega$ de $G$ telle que $X$ est contenu dans la fermeture 
$\overline{^G\Omega}$ de ${^G\Omega}$ dans $G$ (pour la topologie $\mathfrak{p}$--adique);
\item[(iii)] $\Pi_G(X)$ est une partie bornée dans $ F^{N-1}\times F^\times $.
\end{enumerate}
\end{monlem2}

\begin{proof}
Comme pour le lemme 1, les implications $(i)\Rightarrow (ii) \Rightarrow (iii)$ sont claires. Il suffit donc de 
prouver $(iii)\Rightarrow (i)$. Supposons que $\Pi_G(X)$ est une partie bornée dans $F^{N-1}\times F^\times$. Pour $\zeta\in 
F[t]_N^*$, notons $m_\zeta$ le plus petit entier entier $m\geq 1$ tel que pour tout $\gamma\in \pi^{-1}(\zeta)$, on a $\tilde{\gamma} + \mathfrak{P}_{\rm max}^m\subset \tilde{\gamma}U^1_{\rm max}$ avec $U^1_{\rm max}=1+\mathfrak{P}_{\rm max}$. Un tel $m_\zeta$ existe car l'ensemble $R_\zeta= \{\tilde{\gamma}: \gamma\in \pi^{-1}(\zeta)\}$ est fini. On peut recouvrir $\pi(X)$ par des ouverts compacts de $F[t]_N^*$ de la forme $\ES{V}_k(\zeta)$ avec $\zeta\in F[t]_N^*$ et $k\geq k_\zeta(m_\zeta)$:
$$
\pi(X) \subset \bigcup_{j=1}^s \ES{V}_{k_j}(\zeta_j), \quad \zeta_j\in F[t]_N^*,\; k_j \geq k_{\zeta_j}(m_{\zeta_j}).
$$
Il suffit de vérifier que pour un tel ouvert compact $\ES{V}_k(\zeta)\subset F[t]_N^*$, la partie $\pi^{-1}(\ES{V}_k(\zeta))$ est compacte modulo conjugaison dans $G$. Cela résulte des définitions: pour $\gamma'\in \pi^{-1}(\ES{V}_k(\zeta))$, la condition $k\geq k_\zeta(m_\zeta)$ assure l'existence d'un élément $\gamma \in \pi^{-1}(\zeta)$ --- bien déterminé à conjugaison près dans $G$ --- tel que
$$
\tilde{\gamma}' \in \tilde{\gamma} + \mathfrak{P}_{\rm max}^{m_\zeta} \subset \tilde{\gamma}U^1_{\rm max}.
$$
On a donc l'inclusion
$$
\pi^{-1}(\ES{V}_k(\zeta)) \subset{^G(\textstyle{\bigcup_{\tilde{\gamma}\in R_\zeta}}\tilde{\gamma}U^1_{\rm max})},
$$
et le lemme est démontré.
\end{proof}

% lemme 3
\begin{monlem3}
Soit $\beta\in G$ un élément fermé, et soit $J$ un sous--groupe ouvert compact de $G$. Il existe un voisinage ouvert fermé et $G$--invariant $X$ de $\beta$ dans $G$ tel que $X\subset {^G(\beta J)}$. 
\end{monlem3}

\begin{proof}D'après la preuve du lemme 2, pour tout polynôme $\zeta \in F[t]_N^*$ et tout système de représentants $\{\gamma_1,\ldots ,\gamma_s\}\subset G$ des $G$--orbites $\ES{O}_G(\gamma)$ contenues dans $\pi^{-1}(\zeta)$, il existe une partie $X$ ouverte fermée et $G$--invariante dans $G$ telle que $X \subset \bigcup_{j=1}^s {^G(\gamma_j J)}$. En effet, pour $j=1,\ldots , s$, on écrit $\gamma_j = g_j^{-1} \tilde{\gamma}_j g_j$ avec $g_j\in G$. Posons $J' = \bigcap_{j=1}^s g_j J g_j^{-1}$, et choisissons un entier $m\geq 1$ tel que pour $j=1,\ldots ,s$, on a $\tilde{\gamma}_j + \mathfrak{P}_{\rm max}^m \subset \tilde{\gamma}_jJ'$. Enfin choisissons un entier $k \geq k_\zeta(m)$ et prenons $X= \pi^{-1}(\ES{V}_k(\zeta))$. Alors (d'après la preuve du lemme 2) on a
$$
X \subset \bigcup_{j=1}^s {^G(\tilde{\gamma}_j + \mathfrak{P}_{\rm max}^m)} \subset \bigcup_{j=1}^s {^G(\tilde{\gamma}_j J')}
\subset \bigcup_{j=1}^s {^G(\gamma_jJ)}.
$$
Prenons $\zeta = \zeta_\beta$ et $\gamma_1=\beta$. Pour toute $G$--orbite $\ES{O}\subset \pi^{-1}(\zeta)$, puisque la fermeture $\overline{\ES{O}}$ de $\ES{O}$ dans $G$ contient $\ES{O}_G(\beta)$, le voisinage ouvert compact $\beta J$ de $\beta$ dans $G$ rencontre $\ES{O}$. On peut donc, pour $j=2, \ldots ,s$, choisir l'élément $\gamma_j$ dans $\beta J$. D'où le lemme.
\end{proof}

% remarque 2
\begin{marema2}
{\rm Si $\Omega$ est une partie ouverte compacte de $\mathfrak{g}$, resp. $G$, pour que ${^G\Omega}$ soit fermé dans $\mathfrak{g}$, resp. $G$, il suffit que ${^G\Omega} = \Pi^{-1}(\Pi(\Omega))$. 
Cette égalité est vérifiée si et seulement si pour tout $\gamma\in \Omega$, il existe un élément fermé (dans $\mathfrak{g}$) $\gamma'\in \Omega \cap \overline{\ES{O}_G(\gamma)}$. En effet, si ${^G\Omega}= \Pi^{-1}(\Pi(\Omega))$, alors pour $\gamma\in \Omega$, on a $\overline{\ES{O}_G(\gamma)}\subset \Pi^{-1}(\Pi(\gamma))\subset {^G\Omega}$. D'autre part, si $\gamma'\in \Omega$ est un élément fermé (dans $\mathfrak{g}$), pour tout voisinage ouvert compact $\ES{V}_{\gamma'}$ de $\gamma'$ dans $\mathfrak{g}$ contenu dans $\Omega$, on a $\Pi^{-1}(\Pi(\gamma'))\subset {^G(\ES{V}_{\gamma'})}\subset {^G\Omega}$.\hfill $\blacksquare$
}
\end{marema2}

%%%%%%%%%%%%%%%%%%%%%
\subsection{Des $(W,E)$--décompositions}\label{des décompositions}Pour un $\mathfrak{o}$--ordre héréditaire (pas forcément principal) $\mathfrak{A}$ dans $\mathfrak{g}$, on note $\mathfrak{P}={\rm rad}(\mathfrak{A})$ son radical de Jacobson, 
$K(\mathfrak{A})$ son normalisateur dans $G$, et $\{U^k(\mathfrak{A}):k\geq 0\}$ la suite de sous--groupes ouverts compacts distingués de $K(\mathfrak{A})$ définie par
$$
U^0(\mathfrak{A})=U(\mathfrak{A})=\mathfrak{A}^\times,
$$
$$
U^k(\mathfrak{A})= 1+ \mathfrak{P}^k,\quad k\geq 1.
$$

Soit $E/F$ une extension telle que $E\subset \mathfrak{g}$. Posons $\mathfrak{b}={\rm End}_E(V)$. Soit $W$ un sous--$F$--espace vectoriel de $V$ tel que l'application naturelle $E\otimes_F W\rightarrow V$ (induite par l'inclusion $W\subset V$ et par l'action de $E$ sur $V$ donnée par l'inclusion $E\subset \mathfrak{g}$) est un isomorphisme. Cet isomorphisme induit un isomorphisme de $(A(E),\mathfrak{b})$--bimodules
$$
\tau_W= \tau_{W,E}: A(E)\otimes_E \mathfrak{b}\buildrel\simeq\over{\longrightarrow} \mathfrak{g}\leqno{(1)}
$$
appelé {\it $(W,E)$--décomposition de $\mathfrak{g}$} (cf. \cite[1.2.6]{BK}). Prcisons l'isomorphisme (1). L'extension $E$ de $F$ s'iden\-tifie naturellement à un sous--corps maximal de $A(E)$, et le choix de $W$ induit un homomorphisme injectif de $F$--algèbres
$$
\iota_W=\iota_{W,E}:A(E)\rightarrow \mathfrak{g}
$$
qui prolonge l'inclusion $E\subset \mathfrak{g}$. 
En considérant $A(E)$ comme un $(E,E)$--bimodule, on a une identification naturelle $A(E)=A(E)\otimes_E E$. D'autre part, l'isomorphisme de $E$--espaces vectoriels $E\otimes_FW\simeq V$ induit un isomorphisme de $E$--algèbres
$$
E\otimes_F{\rm End}_F(W) ={\rm End}_E(E\otimes_FW)\simeq \mathfrak{b}.
$$
En le combinant avec l'isomorphisme de $F$--algèbres
$$
A(E)\otimes_F {\rm End}_F(W) = {\rm End}_F(E\otimes_FW)\simeq \mathfrak{g},
$$
on obtient l'isomorphisme (1). Concrtement, pour $a\in A(E)$, $e\in E$ et $w\in W$, on a
$$
\iota_W(a)(e\otimes w) =a(e)\otimes w,
$$
et pour $b\in \mathfrak{b}$, on a
$$
\tau_W(a\otimes b) = \iota_W(a)b.
$$

% remarque 1
\begin{marema1}
{\rm Soit $E'/F$ une extension telle que $E'\subset A(E)$. On suppose que $E'$ est un sous--corps maximal de $A(E)$. Le sous--$F$--espace vectoriel $W'=F$ de $E$ engendre $E$ sur $E'$ (\cad que l'application naturelle $E'\otimes_F W'\rightarrow E$ est un isomorphisme), et l'homomorphisme injectif de $F$--algèbres $\iota=\iota_{W'\!,E'}:A(E')\rightarrow A(E)$ est un isomorphisme. D'autre part, comme l'application naturelle $E'\otimes_F W \rightarrow V$ est un isomorphisme, on a aussi une $(W,E')$--décomposition de $\mathfrak{g}$
$$
\tau_{W,E'}: A(E')\otimes_{E'} \mathfrak{b}'\buildrel \simeq\over{\longrightarrow} \mathfrak{g},\quad \mathfrak{b}'={\rm End}_{E'}(W).
$$
Posons $\mathfrak{a}={\rm End}_F(W)$. L'isomorphisme $\iota$ se prolonge naturellement en un isomorphisme de $F$--algèbres
$$
\tau_{W,E,E'}: A(E')\otimes_{E'}\mathfrak{b}'= A(E')\otimes_F \mathfrak{a} \buildrel \simeq\over{\longrightarrow}
A(E)\otimes_F\mathfrak{a}= A(E)\otimes_E\mathfrak{b}.
$$
Il vérifie
$$
\tau_{W,E}\circ \tau_{W,E,E'}= \tau_{W,E'},
$$
et c'est le seul isomorphisme de $F$--algbres $A(E')\otimes_{E'}\mathfrak{b}'\buildrel\simeq\over{\longrightarrow} A(E)\otimes_E\mathfrak{b}$ qui vrifie 
l'galit ci--dessus.\hfill $\blacksquare$
}
\end{marema1}

Rappelons qu'une {\it corestriction modérée sur $\mathfrak{g}$ relativement à $E/F$} est un homomorphisme de $(\mathfrak{b},\mathfrak{b})$--bimodules $\bs{s}:\mathfrak{g}\rightarrow \mathfrak{b}$ tel que $\bs{s}(\mathfrak{A})=\mathfrak{A}\cap \mathfrak{b}$ pour tout $\mathfrak{o}$--ordre héréditaire $\mathfrak{A}$ dans $\mathfrak{g}$ normalisé par $E^\times$. D'après \cite[1.3.4]{BK}, une telle corestriction modérée $\bs{s}$ sur $\mathfrak{g}$ existe, et elle est unique à multiplication près par un élément de $\mathfrak{o}_E^\times$. De plus, pour tout $\mathfrak{o}$--ordre héréditaire $\mathfrak{A}$ dans $\mathfrak{g}$ normalisé par $E^\times$, de radical de Jacobson $\mathfrak{P}$, on a
$$
\bs{s}(\mathfrak{P}^k)= \mathfrak{P}^k\cap \mathfrak{b},\quad k\in {\Bbb Z}.
$$
D'après \cite[1.3.9]{BK}, si $\bs{s}_E:A(E)\rightarrow E$ est une corestriction modérée sur $A(E)$ relativement à $E/F$, alors --- pour 
l'identification $\mathfrak{g}=A(E)\otimes_E \mathfrak{b}$ donnée par (1) --- $\bs{s}= \bs{s}_E\otimes {\rm id}_{\mathfrak{b}}$ est une corestriction modérée sur $\mathfrak{g}$ relativement à $E/F$, et d'après la propriété d'unicité, toute corestriction modérée sur $\mathfrak{g}$ relativement à $E/F$ est de cette forme.

Fixons un $\mathfrak{o}$--ordre héréditaire $\mathfrak{A}$ dans $\mathfrak{g}$ normalisé par $E^\times$, et posons $\mathfrak{B}= \mathfrak{A}\cap \mathfrak{b}$. C'est un $\mathfrak{o}_E$--ordre héréditaire dans $\mathfrak{b}$. 
Posons $\mathfrak{P}={\rm rad}(\mathfrak{A})$ et $\mathfrak{Q}={\rm rad}(\mathfrak{B})$. D'après \cite[1.2.4]{BK}, on a
$$
\mathfrak{Q}^k=\mathfrak{P}^k\cap \mathfrak{b},\quad k\in {\Bbb Z},
$$
et les périodes $e(\mathfrak{A}\vert \mathfrak{o})$ et $e(\mathfrak{B}\vert \mathfrak{o}_E)$ sont reliées par l'égalité
$$
e(\mathfrak{B}\vert \mathfrak{o}_E)={e(\mathfrak{A}\vert\mathfrak{o})\over e(E/F)}. 
$$
\'Ecrivons $\mathfrak{A}={\rm End}_{\mathfrak{o}}^0(\ES{L})$ pour une chaîne de $\mathfrak{o}$--réseaux $\ES{L}=\{L_i\}$ dans $V$. 
Puisque $\mathfrak{A}$ est normalisé par $E^\times$, chaque $\mathfrak{o}$--réseau $L_i$ est en fait un $\mathfrak{o}_E$--réseau. 
Supposons de plus que $W$ est engendré sur $F$ par une $\mathfrak{o}_E$--base de $\ES{L}$ (cf. \cite[1.1.7]{BK}). 
Alors d'après \cite[1.2.10]{BK}, pour $k\in {\Bbb Z}$, l'isomorphisme (1) se restreint en un isomorphisme de $(\mathfrak{A}(E),\mathfrak{B})$--bimodules
$$
\mathfrak{A}(E)\otimes_{\mathfrak{o}_E} \mathfrak{Q}^k\buildrel\simeq\over{\longrightarrow}\mathfrak{P}^k\leqno{(2)}
$$
En particulier, pour $k=0$, on a un 
isomorphisme de $(\mathfrak{A}(E),\mathfrak{B})$--bimodules
$$
\mathfrak{A}(E)\otimes_{\mathfrak{o}_E}\mathfrak{B}\buildrel\simeq\over{\longrightarrow} \mathfrak{A}\leqno{(3)}
$$
appelé {\it $(W,E)$--décomposition de $\mathfrak{A}$}. 

% remarque 2
\begin{marema2}
{\rm L'application $\mathfrak{B}_1\mapsto \mathfrak{A}(E)\otimes_{\mathfrak{o}_E}\mathfrak{B}_1$ est une bijection de l'ensemble des $\mathfrak{o}_E$--ordres héréditaires dans $\mathfrak{b}$ sur l'ensemble des $\mathfrak{o}$--ordres héréditaires $\mathfrak{A}_1$ dans $\mathfrak{g}$ de la forme $\mathfrak{A}_1= {\rm End}_{\mathfrak{o}}^0(\ES{L}_1)$ pour une chaîne de $\mathfrak{o}_E$--réseaux $\ES{L}_1$ possédant une $\mathfrak{o}_E$--base qui engendre $W$ sur $F$, de bijection réciproque $\mathfrak{A}_1\mapsto \mathfrak{A}_1\cap \mathfrak{b}$. 
}
\end{marema2}

% remarque 3
\begin{marema3}
{\rm 
Soit $E'/F$ une extension telle que $E'\subset A(E)$. On suppose que $E'$ est un sous--corps maximal de $A(E)$, et que $E'^\times$ normalise $\mathfrak{A}(E)$. Alors $\mathfrak{A}(E)$ est l'unique $\mathfrak{o}$--ordre héréditaire dans $A(E)$ normalisé par $E'^\times$. On a donc $e(E'/F)=e(E/F)$ et $f(E'/F)= f(E/F)$, et pour $k\in {\Bbb Z}$, le $\mathfrak{o}$--réseau $\mathfrak{p}_E^k$ dans $E$ est un $\mathfrak{o}_{E'}$--réseau (c'est donc un $\mathfrak{o}_{E'}$--module libre de rang $1$). De plus, avec les notations de la remarque 2, l'isomorphisme de $F$--algèbres $\iota:A(E')\buildrel\simeq\over{\longrightarrow} A(E)$ se restreint en un isomorphisme $\mathfrak{A}(E')\buildrel\simeq\over{\longrightarrow} \mathfrak{A}(E)$. Identifions $E'$ au sous--corps $\tau_{W,E}(E'\otimes 1)$ de $\tau_{W,E}(A(E)\otimes_E\mathfrak{b})=\mathfrak{g}$. On a donc aussi l'identification
$$
E'= \tau_{W,E'}(E'\otimes 1)\subset \tau_{W'\!,E'}(A(E')\otimes_{E'}W)=\mathfrak{g}.
$$
Chaque $\mathfrak{o}_E$--réseau $L_i$ dans $V$ de la chaîne $\ES{L}=\{L_i\}$ définissant $\mathfrak{A}$ est aussi un $\mathfrak{o}_{E'}$--réseau. Par conséquent l'$\mathfrak{o}$--ordre héréditaire $\mathfrak{A}$ dans $\mathfrak{g}$ est normalisé par $E'^\times$, et posant $\mathfrak{B}'= \mathfrak{A}\cap \mathfrak{b}'$, la $(W,E')$--décomposition $\tau_{W,E'}: A(E')\otimes_{E'}\mathfrak{b}'\buildrel \simeq\over{\longrightarrow}\mathfrak{g}$ de $\mathfrak{g}$ se restreint en une $(W,E')$--décomposition $\mathfrak{A}(E')\otimes_{\mathfrak{o}_{E'}}\mathfrak{B}'\buildrel\simeq\over{\longrightarrow} \mathfrak{A}$ de $\mathfrak{A}$. On en déduit que l'isomorphisme de $F$--algbres $\tau_{W,E,E'}: A(E')\otimes_{E'} \mathfrak{b}'\buildrel\simeq\over{\longrightarrow} A(E)\otimes_E\mathfrak{b}$ se restreint en un isomorphisme
$$
\mathfrak{A}(E')\otimes_{\mathfrak{o}_{E'}}\mathfrak{B}'\buildrel\simeq\over{\longrightarrow} \mathfrak{A}(E)\otimes_{\mathfrak{o}_E}\mathfrak{B}.\eqno{\blacksquare}
$$
}\end{marema3}

%%%%%%%%%%%%%%%%%%%%%%%%%%%%%%
\subsection{Une submersion}\label{une submersion}Soit $\beta\in \mathfrak{g}$ un élément pur. Posons $E=F[\beta]$ et $\mathfrak{b}={\rm End}_E(V)$. Fixons un $\mathfrak{o}$--ordre héréditaire $\mathfrak{A}$ dans $\mathfrak{g}$ normalisé par $E^\times$, posons $\mathfrak{B}=\mathfrak{A}\cap \mathfrak{b}$, et identifions $\mathfrak{A}$ à $\mathfrak{A}(E)\otimes_{\mathfrak{o}_E}\mathfrak{B}$ via le choix d'une $(W,E)$--décomposition de $\mathfrak{A}$ --- cf. \ref{des décompositions}.(3). Posons $\mathfrak{P}={\rm rad}(\mathfrak{A})$ et $\mathfrak{Q}={\rm rad}(\mathfrak{B})$. On a donc les identifications
$$
\mathfrak{g}=A(E)\otimes_E\mathfrak{b},
$$
$$
\mathfrak{P}^k= \mathfrak{A}(E)\otimes_{\mathfrak{o}_E}\mathfrak{Q}^k,\quad k\in {\Bbb Z}.
$$

Pour $k\in {\Bbb Z}$, on pose
$$
\mathfrak{N}_k(\beta,\mathfrak{A})=\{x\in \mathfrak{A}: {\rm ad}_\beta(x)\in \mathfrak{P}^k\}\subset \mathfrak{A}.
$$
C'est un $(\mathfrak{B},\mathfrak{B})$--biréseau dans $\mathfrak{g}$, qui vérifie $\mathfrak{N}_k(\beta,\mathfrak{A})\cap \mathfrak{b}= \mathfrak{B}$. D'après \cite[1.4.4]{BK}, on a
$$
\bigcap_{k\in {\Bbb Z}}\mathfrak{N}_k(\beta,\mathfrak{A})=\mathfrak{B},
$$
et pour $k$ suffisamment grand, on a $\mathfrak{N}_k(\beta,\mathfrak{A})\subset \mathfrak{B}+ \mathfrak{P}$. Si $E\neq F$, on note $k_0(\beta,\mathfrak{A})$ le plus grand entier $k$ tel que $\mathfrak{N}_k(\beta,\mathfrak{A})\not\subset \mathfrak{B} + \mathfrak{P}$; sinon, on pose $k_0(\beta,\mathfrak{A})=-\infty$. De manière équivalente \cite[1.4.11.(iii)]{BK}, si $E\neq F$, $k_0(\beta,\mathfrak{A})$ est le plus petit entier $k$ tel que $\mathfrak{P}^k\cap {\rm ad}_\beta(\mathfrak{g})\subset {\rm ad}_\beta(\mathfrak{A})$. Pour tous entiers $k,\,r$ tels que $k\geq k_0(\beta,\mathfrak{A})$ et $r\geq 1$, on a l'égalité \cite[1.4.9]{BK}
$$
\mathfrak{N}_{k+r}(\beta,\mathfrak{A})= \mathfrak{B} + \mathfrak{Q}^r\mathfrak{N}_k(\beta,\mathfrak{A}).\leqno{(1)}
$$
Rappelons qu'en \ref{l'invariant k}, on a posé $k_F(\beta)= k_0(\beta,\mathfrak{A}(E))$. D'après \cite[1.4.13]{BK}, on a
$$
\mathfrak{N}_{e(\mathfrak{B}\vert \mathfrak{o}_E)k}(\beta, \mathfrak{A})=\mathfrak{N}_k(\beta,\mathfrak{A}(E))\otimes_{\mathfrak{o}_E}\mathfrak{B},\quad k\in {\Bbb Z},
$$
et
$$
k_0(\beta,\mathfrak{A})= e(\mathfrak{B}\vert\mathfrak{o}_E)k_F(\beta).
$$
Soit $\nu_{\mathfrak{A}}:\mathfrak{g}\rightarrow {\Bbb Z}$ la \og valuation sur $\mathfrak{g}$\fg définie par la filtration $\{\mathfrak{P}^k:k\in {\Bbb Z}\}$, donnée par
$$
\nu_{\mathfrak{A}}(x)=k \Leftrightarrow x\in \mathfrak{P}^k\smallsetminus \mathfrak{P}^{k+1}.
$$
On a
$$
\nu_{\mathfrak{A}}(\beta)= \nu_E(\beta)e(\mathfrak{B}\vert \mathfrak{o}_E)\;(=-n_F(\beta)e(\mathfrak{B}\vert \mathfrak{o}_E)),
$$
par conséquent si $E\neq F$, on a $k_0(\beta,\mathfrak{A})\geq \nu_{\mathfrak{A}}(\beta)$ si et seulement si $k_F(\beta)\geq \nu_E(\beta)$, avec égalité si et seulement si $\beta$ est $F$--minimal \cite[1.4.15]{BK} (cf. \ref{l'invariant k}). 

Fixons une corestriction modérée $\bs{s}:\mathfrak{g}\rightarrow \mathfrak{b}$ sur $\mathfrak{g}$ relativement à $E/F$, et un élément $\bs{x}\in \mathfrak{A}$ tel que $\bs{s}(\bs{x})=1$. Puisque $\bs{s}(\mathfrak{g})=\mathfrak{b}$ et $\ker (\bs{s})={\rm ad}_\beta(\mathfrak{g})$, on a la décomposition
$$
\mathfrak{g}= {\rm ad}_\beta(\mathfrak{g})\oplus \bs{x}\mathfrak{b}.\leqno{(2)}
$$
Pour tout entier $k\geq k_0(\beta,\mathfrak{A})$, la décomposition (2) se précise en \cite[1.4.7]{BK}
$$
\mathfrak{P}^k = {\rm ad}_\beta(\mathfrak{N}_k(\beta,\mathfrak{A}))\oplus \bs{x}\mathfrak{Q}^k.\leqno{(3)}
$$

% remarque
\begin{marema}
{\rm 
L'ensemble des $\bs{x}\in \mathfrak{A}$ tels que $\bs{s}(\bs{x})=1$ est un espace principal homogène sous ${\rm ad}_\beta(\mathfrak{g})\cap \mathfrak{A}$. Or ${\rm ad}_\beta(\mathfrak{g})={\rm ad}_\beta(A(E))\otimes_E\mathfrak{b}$, et comme $\mathfrak{B}$ est un $\mathfrak{o}_E$--module libre, on a
$$
{\rm ad}_\beta(\mathfrak{g})\cap \mathfrak{A}= {\rm ad}_\beta(\mathfrak{g})\cap (\mathfrak{A}(E)\otimes_{\mathfrak{o}_E}\mathfrak{B})=
({\rm ad}_\beta(A(E))\cap \mathfrak{A}(E))\otimes_{\mathfrak{o}_E}\mathfrak{B}.
$$
\'Ecrivons $\bs{s}=\bs{s}_0\otimes {\rm id}_{\mathfrak{b}}$, où $\bs{s}_0:A(E)\rightarrow E$ est une corestriction modérée sur $A(E)$ relativement à $E/F$. L'ensemble des $\bs{x}_0\in \mathfrak{A}(E)$ tels que $\bs{s}_0(\bs{x}_0)=1$ est un espace principal homogène sous ${\rm ad}_\beta(A(E))\cap \mathfrak{A}(E)$, et pour un tel $\bs{x}_0$, on a $\bs{s}(\bs{x}_0\otimes 1)=1$. \hfill $\blacksquare$ }
\end{marema}

% proposition
\begin{mapropo}
On suppose $E\neq F$. L'application
$$
\delta: G\times \mathfrak{Q}^{k_0(\beta,\mathfrak{A})+1}\rightarrow G, \, (g,b)\mapsto g^{-1}(\beta + \bs{x}b)g
$$
est partout submersive.
\end{mapropo}

\begin{proof}
Posons $k_0=k_0(\beta,\mathfrak{A})$. Puisque $\delta(g,b)=g^{-1}\delta(1,b)g$, il suffit de prouver que pour tout 
$b_1\in \mathfrak{Q}^{k_0+1}$, la différentielle $d\delta_{(1,b_1)}$ de $\delta$ en le point $(1,b_1)$ est surjective. Fixons un élément $b_1
\in \mathfrak{Q}^{k_0+1}$. En identifiant l'espace tangent à $G\times \mathfrak{Q}^{k_0+1}$ au point $(1,b_1)$ à $\mathfrak{g}\times \mathfrak{b}$, et l'espace tangent à $G$ au point $\gamma_1= \beta + \bs{x}b_1$ à $\mathfrak{g}$, la différentielle $d\delta_{(1,b_1)}: \mathfrak{g}\times \mathfrak{b}\rightarrow \mathfrak{g}$ s'écrit
$$
d\delta_{1,b_1}(y,b)= \gamma_1y -y\gamma_1 + \bs{x}b.
$$
Pour $i\in {\Bbb Z}$ et $y\in \mathfrak{Q}^{i-k_0}\mathfrak{N}_{k_0}(\beta,\mathfrak{A})=\mathfrak{N}_{k_0}(\beta,\mathfrak{A})\mathfrak{Q}^{i-k_0}$, on a
$$
\gamma_1y - y\gamma_1\equiv {\rm ad}_\beta(y) \quad ({\rm mod}\;\mathfrak{P}^{i+1}).
$$
Comme d'après (3), on a la décomposition ${\rm ad}_\beta(\mathfrak{N}_{k_0}(\beta,\mathfrak{A})\oplus \bs{x}\mathfrak{Q}^{k_0}=\mathfrak{P}^{k_0}$, on obtient que
$$
d\delta_{(1,b_1)}(\mathfrak{Q}^{i-k_0}\mathfrak{N}_{k_0}(\beta,\mathfrak{A})\times \mathfrak{Q}^{i})) + \mathfrak{P}^{i+1}= \mathfrak{P}^i,\quad i\in {\Bbb Z}.
$$
Par approximations successives, on en déduit que
$$
d\delta_{(1,b_1)}(\mathfrak{Q}^{i-k_0}\mathfrak{N}_{k_0}(\beta,\mathfrak{A})\times \mathfrak{Q}^{i}))= \mathfrak{P}^i,\quad i\in {\Bbb Z}.
$$
Par conséquent $d\delta_{(1,b_1)}(\mathfrak{g}\times \mathfrak{b})=\mathfrak{g}$ et la proposition est démontrée. 
\end{proof}

%%%%%%%%%%%%%%%%%%%%
\subsection{Raffinement}\label{raffinement}
Une {\it strate dans $\mathfrak{g}$} est par définition un quadruplet $[\mathfrak{A},n,r,\gamma]$, où $\mathfrak{A}$ est un $\mathfrak{o}$--ordre héréditaire dans $\mathfrak{g}$, $n$ et $r$ sont deux entiers tels que $n>r$, et $\gamma$ est un élément de $\mathfrak{g}$ tel que $\nu_\mathfrak{A}(\gamma)\geq -n$. Une telle strate équivaut donc à la donnée d'un élément $\gamma+\mathfrak{P}^{-r}$ dans le groupe quotient $\mathfrak{P}^{-n}/\mathfrak{P}^{-r}$, o on a pos $\mathfrak{P}={\rm rad}(\mathfrak{A})$. D'ailleurs deux strates $[\mathfrak{A},n,r,\gamma]$ et $[\mathfrak{A}',n',r', \gamma']$ dans $\mathfrak{g}$ sont dites {\it quivalentes} si $\mathfrak{A}=\mathfrak{A}'$, $n'=n$, $r'=r$, et si $\gamma'-\gamma \in \mathfrak{P}^{-r}$. On rappelle la définition de strate {\it pure}, resp. {\it simple}, dans $\mathfrak{g}$ \cite[1.5.5]{BK}:

% remarque
\begin{madefi}
{\rm Une strate $[\mathfrak{A},n,r,\gamma]$ dans $\mathfrak{g}$ est dite:
\begin{itemize}
\item {\it pure} si l'élément $\gamma$ est pur, $F[\gamma]^\times$ normalise $\mathfrak{A}$, et $\nu_{\mathfrak{A}}(\gamma)= -n$;
\item {\it simple} si elle est pure, et si $r<-k_0(\gamma,\mathfrak{A})$.
\end{itemize}
}
\end{madefi}

% remarque 1
\begin{marema1}
{\rm Soit $[\mathfrak{A},n,n-1,\gamma]$ une strate simple dans $\mathfrak{g}$. On a $\nu_{\mathfrak{A}}(\gamma)=-n$ 
et $n-1<-k_0(\gamma,\mathfrak{A})$. Si de plus $F[\gamma]\neq F$, comme on a aussi $-n \leq k_0(\gamma,\mathfrak{A})$, cette inégalité est une égalité. Dans tous les cas, l'élément $\gamma$ est $F$--minimal. \hfill $\blacksquare$
}
\end{marema1}

\`A une strate dans $\mathfrak{g}$ de la forme $[\mathfrak{A},n,n-1,\gamma]$ est associ comme suit un polynme caractristique $\phi_\gamma=\phi_{[\mathfrak{A},n,n-1,\gamma]}\in \kappa[t]$. 
Rappelons sa définition \cite[2.3]{BK}. On pose $e=e(\mathfrak{A}\vert \mathfrak{o})$, $\mathfrak{P}={\rm rad}(\mathfrak{A})$, on choisit une uniformisante $\varpi$ de $F$, et on note $\delta= (e, n)\geq 1$ le plus grand diviseur commun de $e$ et $n$. Alors $y_\gamma = \varpi^{n/\delta}\gamma^{e/\delta}+ \mathfrak{P}$ est un élément de $\mathfrak{A}/\mathfrak{P}$, qui ne dépend que de la classe d'équivalence de la strate $[\mathfrak{A},n,n-1, \gamma]$. Si $\mathfrak{A}= {\rm End}_{\mathfrak{o}}^0(\ES{L})$ pour une chane de $\mathfrak{o}$--rseaux $\ES{L}=\{L_i\}$ dans $V$, on a les identifications
$$
\mathfrak{A}/\mathfrak{P}= \coprod_{i=0}^{e-1}{\rm End}_{\kappa}(L_i/L_{i+1})\subset 
{\rm End}_{\kappa}(L_0/ \mathfrak{p}L_0),
$$
et on note $\phi_\gamma\in \kappa[t]$ le polynôme caractéristique de $y_\gamma\in {\rm End}_{\kappa}(L_0/ \mathfrak{p}L_0)$ --- à ne pas confondre avec le polynôme caractéristique $\zeta_\gamma\in F[t]$ du $F$--endomorphisme $\gamma$ de $V$. Tout comme l'élément $y_\gamma$, il ne dépend que de la classe d'équivalence de la strate $[\mathfrak{A},n,n-1,\gamma]$.

Soit $[\mathfrak{A},n,r,\beta]$ une strate simple dans $\mathfrak{g}$. On a donc $n= -\nu_{\mathfrak{A}}(\beta)$ et
$$
r<\inf \{-k_0(\beta,\mathfrak{A}),n\}.
$$
Posons $E=F[\beta]$, $\mathfrak{b}={\rm End}_E(V)$, et notons $\mathfrak{B}$ l'$\mathfrak{o}_E$--ordre héréditaire $\mathfrak{A}\cap \mathfrak{b}$ dans $\mathfrak{b}$. Puisque $k_0(\beta,\mathfrak{A})=k_F(\beta)e(\mathfrak{B}\vert \mathfrak{o}_E)$ et $n = n_F(\beta)e(\mathfrak{B}\vert \mathfrak{o}_E)$, on a
$$
{r\over e(\mathfrak{B}\vert \mathfrak{o}_E)}< \inf(-k_F(\beta),n_F(\beta)\}.
$$
Posons $\mathfrak{P}={\rm rad}(\mathfrak{A})$ et $\mathfrak{Q}={\rm rad}(\mathfrak{B})$. 
Fixons une corestriction modérée $\bs{s}:\mathfrak{g}\rightarrow \mathfrak{b}$ sur $\mathfrak{g}$ relativement à $E/F$. La proposition suivante est due à Bushnell--Kutzko 
\cite[2.2.3]{BK}, et son corollaire est prouvé dans \cite[5.3.2]{L2}.

% proposition 1
\begin{mapropo1}
Soit $[\mathfrak{B},r,r-1,b]$ une strate simple dans $\mathfrak{b}$ telle que $E[b]\;(=F[\beta, b])$ est un sous--corps maximal de $\mathfrak{b}$. Soit $\gamma = \beta + y$ pour un $y\in \mathfrak{P}^{-r}$ tel que $\bs{s}(y)=b$. 
La strate $[\mathfrak{A},n,r-1,\gamma]$ dans $\mathfrak{g}$ est simple, et l'extension $F[\gamma]/F$ vérifie $e(F[\gamma]/F)=
e(E[b]/F)$ et $f(F[\gamma]/F)=f(E[b]/F)$. En particulier, $F[\gamma]$ est un sous--corps maximal de $\mathfrak{g}$. De plus, on a
$$
k_F(\gamma)=\left\{\begin{array}{ll}
-r=k_E(b)& \mbox{si $E[b]\neq E$}\\
k_0(\beta, \mathfrak{A})& \mbox{sinon}
\end{array}\right..
$$
\end{mapropo1}

% corollaire 1
\begin{moncoro1}
Soit $\bs{s}_b: \mathfrak{b}\rightarrow E[b]$ une corestriction modérée sur $\mathfrak{b}$ relativement à $E[b]/E$. Il existe une corestriction modérée $\bs{s}_\gamma: \mathfrak{g}\rightarrow F[\gamma]$ sur $\mathfrak{g}$ relativement à $F[\gamma]/F$ telle que pour tout $k\in {\Bbb Z}$ et tout $y\in \mathfrak{P}^k$, on a
$$
\bs{s}_\gamma(y) \equiv \bs{s}_b \circ \bs{s}(y) \quad ({\rm mod}\; \mathfrak{P}^{k+1}).
$$
\end{moncoro1}

% remarque 2
\begin{marema2}
{\rm 
D'après le corollaire 1, pour $k\in {\Bbb Z}$, on a les égalités
$$
\mathfrak{p}_{F[\gamma]}^k+ \mathfrak{P}^{k+1}= \mathfrak{p}_{E[b]}^k+ \mathfrak{P}^{k+1}
$$
et
$$
\ker(s_\gamma\vert_{\mathfrak{P}^k})+ \mathfrak{P}^{k+1} = \ker(\bs{s}_b\circ \bs{s}\vert_{\mathfrak{P}^k})+ \mathfrak{P}^{k+1}.
$$
Via les identifications naturelles
$$
(\mathfrak{p}_{F[\gamma]}^k+ \mathfrak{P}^{k+1})/\mathfrak{P}^{k+1}= \mathfrak{p}_{F[\gamma]}^k/
\mathfrak{p}_{F[\gamma]}^{k+1},
$$
on a donc une identification (naturelle) $\mathfrak{p}_{F[\gamma]}^k/\mathfrak{p}_{F[\gamma]}^{k+1}= 
\mathfrak{p}_{E[b]}^k/\mathfrak{p}_{E[b]}^{k+1}$, 
et cette dernière co\"{\i}ncide avec celle donnée par 
l'isomorphisme $\mathfrak{p}_{E[b]}^k/\mathfrak{p}_{E[b]}^{k+1}\buildrel\simeq\over{\longrightarrow} \mathfrak{p}_{F[\gamma]}^k/\mathfrak{p}_{F[\gamma]}^{k+1}$ déduit, par restriction et passage aux quotients, de l'application $E[b]\rightarrow F[\gamma],\,  y\mapsto \bs{s}_\gamma(\tilde{\bs{x}}y)$; où $\tilde{\bs{x}}$ est un élément de $\mathfrak{A}$ tel que $\bs{s}_b\circ \bs{s}(\tilde{\bs{x}})=1$.
\hfill $\blacksquare$
}
\end{marema2}

La strate $[\mathfrak{A},n,r-1, \beta + y]$ est un {\it raffinement} de la strate simple $[\mathfrak{A},n,r,\beta]$ dans $\mathfrak{g}$, de {\it strate dérivée associée} la strate simple $[\mathfrak{B},r,r-1,b]$ dans $\mathfrak{b}$. Nous allons voir plus loin (\ref{approximation}) que tout élément $\gamma\in G_{\rm qre}$ définit une strate simple $[\mathfrak{A}_\gamma,n,r-1,\gamma]$ dans $\mathfrak{g}$ avec $r=-k_F(\gamma)$, que l'on peut réaliser comme un raffinement de la forme ci--dessus.

Identifions $\mathfrak{A}$ à $\mathfrak{A}(E)\otimes_{\mathfrak{o}_E}\mathfrak{B}$ via le choix d'une $(W,E)$--décomposition de $\mathfrak{A}$ --- cf. \ref{des décompositions}.(3). Soit $\bs{s}_0: A(E)\rightarrow E$ la corestriction modérée sur $A(E)$ relativement à $E/F$ telle que $\bs{s}= \bs{s}_0\otimes {\rm id}_{\mathfrak{b}}$. Fixons 
un élément $\bs{x}_0\in \mathfrak{A}(E)$ tel que $\bs{s}_0(\bs{x}_0)=1$. Posons $\bs{x}=\bs{x}_0\otimes 1\in \mathfrak{A}$. On a donc $\bs{s}(\bs{x})=1$. La proposition suivante est une simple variante de la proposition 1: le choix particulier de $y=\bs{x}_0\otimes b$ permet de supprimer 
l'hypothèse que $E[b]$ est un sous--corps maximal de $\mathfrak{b}$. 

% proposition 2
\begin{mapropo2}
Soit $[\mathfrak{B},r,r-1,b]$ une strate simple dans $\mathfrak{b}$, et soit 
$\gamma=\beta +\bs{x}_0\otimes b$. La strate $[\mathfrak{A},n,r-1,\gamma]$ dans $\mathfrak{g}$ est simple, et 
l'extension $F[\gamma]/F$ vérifie $e(F[\gamma]/F)= e(E[b]/F)$ et $f(F[\gamma]/F)=f(E[b]/F)$.
De plus on a
$$
k_0(\gamma,\mathfrak{A})=\left\{\begin{array}{ll}
-r =k_0(b,\mathfrak{B}) & \mbox{si $E[b]\neq E$}\\
k_0(\beta,\mathfrak{A})& \mbox{sinon}\end{array}\right..
$$
\end{mapropo2}

\begin{proof} On a $\gamma =\beta + \bs{x}b$. Posons $E_1=E[b]$, $\mathfrak{a}_1={\rm End}_E(E_1)$ et $\mathfrak{b}_1={\rm End}_{E_1}(V)$. Soit $\mathfrak{A}_1$ l'$\mathfrak{o}_E$--ordre héréditaire ${\rm End}_{\mathfrak{o}_E}^0(\{\mathfrak{p}_{E_1}^i\})$ dans $\mathfrak{a}_1$, et soit $\mathfrak{B}_1$ l'$\mathfrak{o}_{E_1}$--ordre héréditaire $\mathfrak{B}\cap \mathfrak{b}_1$ dans $\mathfrak{b}_1$. Identifions $\mathfrak{B}$ à $\mathfrak{A}_1\otimes_{\mathfrak{o}_{E_1}}\mathfrak{B}_1$ via le choix d'une $(W_1,E_1)$--décomposition de $\mathfrak{B}$. On a donc les identifications
$$
\mathfrak{g}=A(E)\otimes_E\mathfrak{a}_1\otimes_{E_1}\mathfrak{b}_1,\quad 
\mathfrak{A}= \mathfrak{A}(E)\otimes_{\mathfrak{o}_E}\mathfrak{A}_1\otimes_{\mathfrak{o}_{E_1}}\mathfrak{B}_1.
$$
D'autre part, en identifiant $\mathfrak{A}(E_1)$ à $\mathfrak{A}(E)\otimes_{\mathfrak{o}_E} \mathfrak{A}_1$ via le choix d'une $(X,E)$--décomposition de $\mathfrak{A}(E_1)$, on a aussi les identifications
$$
\mathfrak{g}= A(E_1)\otimes_{E_1} \mathfrak{b}_1,\quad \mathfrak{A}= \mathfrak{A}(E_1)\otimes_{\mathfrak{o}_{E_1}}\mathfrak{B}_1.
$$

Soit $\bs{s}_1:A(E_1)\rightarrow \mathfrak{a}_1$ la corestriction modérée $\bs{s}_0\otimes {\rm id}_{\mathfrak{a}_1}$ sur $A(E_1)$ relativement à $E/F$. L'élément $\bs{x}_1= \bs{x}_0\otimes 1$ de $\mathfrak{A}(E)\otimes_{\mathfrak{o}_E}\mathfrak{A}_1$ vérifie $\bs{s}_1(\bs{x}_1)=1$, et l'élément $\bs{x}_1\otimes 1$ de $ \mathfrak{A}(E_1)\otimes_{\mathfrak{o}_{E_1}}\mathfrak{B}_1$ co\"{\i}ncide avec $\bs{x}$. \'Ecrivons $b=a_1\otimes 1$ avec $a_1\in \mathfrak{a}_1$, et posons $\gamma_1= \beta + \bs{x}_1 a_1\in A(E_1)$. Posons $e_1=e(\mathfrak{B}_1\vert \mathfrak{o}_{E_1})$, $n_1=n/e_1$ et $r_1=r/e_1$. Puisque $k_0(b, \mathfrak{B})=e_1k_E(b)$, la strate $[\mathfrak{A}_1,r_1,r_1-1,a_1]$ dans $\mathfrak{a}_1$ est simple, et $a_1$ est $E$--minimal. On a
$$
k_0(\beta,\mathfrak{A}(E_1))=k_F(\beta)e(\mathfrak{A}_1\vert \mathfrak{o}_E)=k_F(\beta)e(E_1/E).
$$
Comme $e(\mathfrak{B}\vert\mathfrak{o}_E)= e(E_1/E)e_1$, puisque $r<\inf \{-k_0(\beta,\mathfrak{A}),n\}$, on a
$$
r_1<\inf \{-k_0(\beta,\mathfrak{A}_1),n_1\}.
$$
En particulier, la strate $[\mathfrak{A}(E_1),n_1,r_1,\beta]$ dans $A(E_1)$ est simple. Puisque $E_1$ est un sous--corps maximal de $A(E_1)$, d'après la proposition 1, la strate $[\mathfrak{A}(E_1),n_1,r_1-1, \gamma_1]$ dans $A(E_1)$ est simple, et l'extension $E'_1=F[\gamma_1]$ de $F$ vérifie $e(E'_1/F)=e(E_1/F)$ et $f(E'_1/F)=f(E_1/F)$. 
De plus on a
$$
k_0(\gamma_1,\mathfrak{A}(E_1))= \left\{\begin{array}{ll}
-r_1=k_0(b,\mathfrak{A}_1) & \mbox{si $E_1\neq E$}\\
k_0(\beta,\mathfrak{A}(E_1)) & \mbox{sinon}\end{array}\right..
$$
L'élément $\gamma_1\otimes 1$ de $A(E_1)\otimes_{E_1}\mathfrak{b}_1$ co\"{\i}ncide avec $\gamma$, et l'extension $K=F[\gamma]$ de $F$ est isomorphe à $E'_1$. On a donc donc $e(K/F)=e(E_1/F)$ et $f(K/F)=f(E_1/F)$. 
D'autre part, comme on a  $k_0(\gamma,\mathfrak{A})=e_1k_0(\gamma_1,\mathfrak{A}(E_1))$, $k_0(\beta,\mathfrak{A})=e_1k_0(\beta,\mathfrak{A}(E_1))$ et  $k_0(b,\mathfrak{B})= e_1 k_0(b,\mathfrak{A}_1)$, on a aussi
$$
k_0(\gamma,\mathfrak{A})=\left\{\begin{array}{ll}
-r =k_0(b,\mathfrak{B}) & \mbox{si $E_1\neq E$}\\
k_0(\beta,\mathfrak{A})& \mbox{sinon}\end{array}\right..
$$
La strate $[\mathfrak{A},n,r-1,\gamma]$ dans $\mathfrak{g}$ est pure, donc simple, et la proposition est démon\-trée.
\end{proof}

% remarque 3
\begin{marema3}
{\rm 
Sous les hypothèses de la proposition 2, si $E_1=E$, on a $k_0(b,\mathfrak{B})=-\infty$ et $k_0(\gamma, \mathfrak{A})=k_0(\beta,\mathfrak{A})$, avec $k_0(\beta,\mathfrak{A})>-(r-1)$ si $E\neq F$, et $k_0(\beta,\mathfrak{A})=-\infty$ sinon. \hfill $\blacksquare$
}
\end{marema3}

% corollaire 2
\begin{moncoro2}
Soit $\bs{s}_b: \mathfrak{b}\rightarrow {\rm End}_{E[b]}(V)$ une corestriction modérée sur $\mathfrak{b}$ relativement à $E[b]/E$. Il existe une corestriction modérée $\bs{s}_\gamma: \mathfrak{g}\rightarrow {\rm End}_{F[\gamma]}(V)$ sur $\mathfrak{g}$ relativement à $F[\gamma]/F$ telle que pour tout $k\in {\Bbb Z}$ et tout $y\in \mathfrak{P}^k$, on a
$$
\bs{s}_\gamma(y) \equiv \bs{s}_b \circ \bs{s}(y) \quad ({\rm mod}\; \mathfrak{P}^{k+e});
$$
où on a posé $e=e(\mathfrak{A}\cap {\rm End}_{E[b]}(V)\vert \mathfrak{o}_{E[b]})$. 
\end{moncoro2}

\begin{proof}
Continuons avec les notations de la démonstration de la proposition 2. La corestriction modérée $\bs{s}_b: \mathfrak{b}\rightarrow \mathfrak{b}_1$ sur $\mathfrak{b}$ relativement à $E_1/E$ s'écrit $\bs{s}_b = \bs{s}_{a_1}\otimes {\rm id}_{\mathfrak{b}_1}$, où $\bs{s}_{a_1}: \mathfrak{a}_1\rightarrow E_1$ est une corestriction modérée sur $\mathfrak{a}_1$ relativement à $E_1/E$. Soit $\bs{s}_1: A(E_1)\rightarrow \mathfrak{a}_1$ la corestriction modérée $\bs{s}_0\otimes {\rm id}_{\mathfrak{a}_1}$ sur $A(E_1)= A(E)\otimes_E\mathfrak{a}_1$ relativement à $E/F$. D'après le corollaire 1, il existe une corestriction modérée $\bs{s}_{\gamma_1}: A(E_1)\rightarrow F[\gamma_1]$ sur $A(E_1)$ relativement à $F[\gamma_1]/F$ telle que pour tout $k\in {\Bbb Z}$ et tout $y_1\in \mathfrak{P}^k(E_1)$, on a
$$
\bs{s}_{\gamma_1}(y_1)\equiv \bs{s}_{a_1}\circ \bs{s}_1(y_1)\quad ({\rm mod}\;\mathfrak{P}^{k+1}(E_1))
$$
avec
$$
\bs{s}_{a_1}\circ \bs{s}_1= \bs{s}_0\otimes \bs{s}_{a_1}: A(E)\otimes_{E}\mathfrak{a}_1= A(E_1)\rightarrow E_1.
$$
On a les identifications
$$
\mathfrak{g}= A(E_1)\otimes_{E_1} \mathfrak{b}_1,\quad \mathfrak{A}= \mathfrak{A}(E_1)\otimes_{\mathfrak{o}_{E_1}}\mathfrak{B}_1.
$$
Elles sont données par une $(X_1,E_1)$--décomposition 
$\tau_{X_1,E_1}: A(E_1)\otimes_{E_1}\mathfrak{b}_1 \buildrel\simeq\over{\longrightarrow} \mathfrak{g}$ de $\mathfrak{g}$ qui se restreint 
en une $(X_1,E_1)$--décomposition $\mathfrak{A}(E_1)\otimes_{\mathfrak{o}_{E_1}}\mathfrak{B}_1
\buildrel\simeq\over{\longrightarrow} \mathfrak{A}$ de $\mathfrak{A}$.
Posons $E'_1=F[\gamma_1]$. D'aprs la remarque 3 de \ref{des dcompositions}, le groupe $E'^\times_1$ normalise $\mathfrak{A}(E_1)$, et on a $e(E'_1/F)=e(E_1/F)$ et $f(E'_1/F)=f(E_1/F)$. Posons $\mathfrak{b}'_1={\rm End}_{E'_1}(V)$ et $\mathfrak{B}'_1= \mathfrak{A}\cap \mathfrak{b}'_1$. Puisque $E'_1$ est un sous--corps maximal de $A(E_1)$, le sous--$F$--espace vectoriel $F\subset E_1$ engendre $E_1$ sur $E'_1$, et l'inclusion $E'_1\subset A(E_1)$ se prolonge en un isomorphisme de $F$--algèbres 
$\iota: A(E'_1)\buildrel\simeq\over{\longrightarrow} A(E_1)$ --- cf. la remarque 1 de \ref{des décompositions}. 
Posons $\mathfrak{Q}_1={\rm rad}(\mathfrak{B}_1)$ et 
$\mathfrak{Q}'_1= {\rm rad}(\mathfrak{B}'_1)$. 
D'après la remarque 3 de \ref{des décompositions}, la $(X_1,E'_1)$--décomposition $\tau_{X_1,E'_1}: A(E'_1)\otimes_{E'_1}\mathfrak{b}'_1\buildrel\simeq\over{\longrightarrow} \mathfrak{g}$ de $\mathfrak{g}$ se restreint en une 
$(X_1,E'_1)$--décomposition $\mathfrak{A}(E'_1)\otimes_{\mathfrak{o}_{E'_1}}\mathfrak{B}_1\buildrel\simeq\over{\longrightarrow} \mathfrak{A}$ de $\mathfrak{A}$, et $\iota$ se prolonge naturellement en un isomorphisme de $F$--algèbres
$$
\tau_{X_1,E,E'}:A(E'_1)\otimes_{E'_1}\mathfrak{b}'_1 \buildrel \simeq\over{\longrightarrow} A(E_1)\otimes_{E_1} \mathfrak{b}_1\leqno{(1)}
$$
qui est compatible à $\tau_{X_1,E'_1}$ et $\tau_{X_1,E_1}$. De plus, pour chaque $k\in {\Bbb Z}$, ce dernier se restreint en un isomorphisme
$$
\mathfrak{A}(E'_1)\otimes_{\mathfrak{o}_{E'_1}}\mathfrak{Q}'^k_1 \buildrel \simeq\over{\longrightarrow} \mathfrak{A}(E_1)\otimes_{\mathfrak{o}_{E_1}} \mathfrak{Q}^k_1.\leqno{(2)}
$$
En particulier, pour $k=0$, on a un isomorphisme
$$
\mathfrak{A}(E'_1)\otimes_{\mathfrak{o}_{E'_1}}\mathfrak{B}'_1 \buildrel \simeq\over{\longrightarrow} \mathfrak{A}(E_1)\otimes_{\mathfrak{o}_{E_1}} \mathfrak{B}_1.\leqno{(3)}
$$
Identifions $E'_1$ au sous--corps $\tau_{X_1,E_1}(E'_1\otimes 1)= \tau_{X_1,E'_1}(E'_1\otimes 1)$ de $\mathfrak{g}$. 
Rappelons que $\gamma =\gamma_1\otimes 1$. On a donc $E'_1=F[\gamma]$. 
Notons $\bs{s}_\gamma: \mathfrak{g}\rightarrow \mathfrak{b}'_1$ l'application 
$((\bs{s}_{\gamma_1}\circ \iota)\otimes{\rm id}_{\mathfrak{b}'_1})\circ \tau_{X_1,E'_1}^{-1}$ sur $\mathfrak{g}$. C'est une corestriction modre sur $\mathfrak{g}$ relativement à $E'_1/F$. Pour $k\in {\Bbb Z}$ et $y'_1\otimes b'_1 \in \mathfrak{A}(E'_1)\otimes_{\mathfrak{o}_{E'_1}}\mathfrak{Q}'^k_1$, posant 
$y=\tau_{X_1,E'_1}(y'_1\otimes b'_1)$ et $y_1= \iota (y'_1)$, on a
\begin{eqnarray*}
\bs{s}_\gamma (y)&= &\bs{s}_{\gamma_1}(y_1)b'_1\\
& \equiv & \tau_{X_1,E'_1}(\iota^{-1} (\bs{s}_{a_1}\circ \bs{s}_1(y_1))\otimes b'_1)\quad ({\rm mod}\;\tau_{X_1,E'_1}(\mathfrak{P}(E'_1)\otimes_{\mathfrak{o}_{E'_1}}\mathfrak{Q}'_1)).
\end{eqnarray*}
Comme on a
$$
\tau_{X_1,E'_1}(\mathfrak{A}(E'_1)\otimes_{\mathfrak{o}_{E'_1}}\mathfrak{Q}'^k_1) = \mathfrak{P}^k
$$
et
$$
\tau_{X_1,E'_1}(\mathfrak{P}(E'_1)\otimes_{\mathfrak{o}_{E'_1}}\mathfrak{Q}'^k_1) = 
\tau_{X_1,E'_1}(\mathfrak{A}(E'_1)\otimes_{\mathfrak{o}_{E'_1}}\mathfrak{Q}'^{k+e}_1)=\mathfrak{P}^{k+e},
$$
le corollaire est dmontr.
\end{proof}

Continuons avec les hypothses de la proposition 2: $[\mathfrak{B},r,r-1,b]$ est une strate simple dans $\mathfrak{b}\;(={\rm End}_E(V))$ et $\gamma= \beta + \bs{x}_0 \otimes b$. Posons $E_1=E[b]$, $\mathfrak{b}_1={\rm End}_{E_1}(V)$, $\mathfrak{B}_1= \mathfrak{A}\cap \mathfrak{b}_1$ et $\mathfrak{Q}_1= {\rm rad}(\mathfrak{B}_1)$. Posons $E'_1=F[\gamma]$. On sait que $E'^\times_1$ normalise $\mathfrak{A}$, et que $e(E'_1/F)=e(E_1/F)$ et $f(E'_1/F)= f(E_1/F)$. Posons $\mathfrak{b}'_1={\rm End}_{E'_1}(V)$, $\mathfrak{B}'_1= \mathfrak{A}\cap \mathfrak{b}'_1$ et $\mathfrak{Q}'_1= {\rm rad}(\mathfrak{B}'_1)$. Posons aussi $e=e(\mathfrak{B}_1\vert\mathfrak{o}_{E_1})\;(=e(\mathfrak{B}'_1\vert \mathfrak{o}_{E'_1})$. D'après le corollaire 2, pour $k\in {\Bbb Z}$, on a les égalités
$$
\mathfrak{Q}'^k_1 + \mathfrak{P}^{k+e}= \mathfrak{Q}^k_1 + \mathfrak{P}^{k+e}
$$
et
$$
\ker(\bs{s}_\gamma\vert_{\mathfrak{P}^k})+ \mathfrak{P}^{k+e}= \ker(\bs{s}_b\circ \bs{s}\vert_{\mathfrak{P}^k})+ \mathfrak{P}^{k+e}.
$$
Comme dans la remarque 2, on en déduit une identification naturelle $\mathfrak{Q}'^k_1/\mathfrak{Q}'^{k+e}_1= \mathfrak{Q}^k_1/\mathfrak{Q}^{k+e}_1$, 
qui co\"{\i}ncide avec celle donnée par l'isomorphisme de $\kappa_{E_1}\;(=\kappa_{E'_1})$--espaces vectoriels 
$$
\mathfrak{Q}^k_1/\mathfrak{Q}^{k+e}_1 \buildrel\simeq \over{\longrightarrow} \mathfrak{Q}'^k_1/\mathfrak{Q}'^{k+e}_1\leqno{(4)}
$$ déduit (par restriction et passage aux quotients) de l'application $\mathfrak{b}_1 \rightarrow \mathfrak{b}'_1,\, b_1 \mapsto \bs{s}_\gamma(\tilde{\bs{x}} b_1)$; où $\tilde{\bs{x}}$ est un élément de $\mathfrak{A}$ tel que $\bs{s}_b\circ \bs{s}(\tilde{\bs{x}})=1$. 

% proposition 3
\begin{mapropo3}
Soit un entier $s\leq r-1$, et soit $[\mathfrak{B}_1,s,s-1,c]$ une strate simple dans $\mathfrak{b}_1$ telle que $E_1[c]\;(=E[b,c])$ est un sous--corps maximal de $\mathfrak{b}_1$. Soit $\bar{\gamma} = \gamma + z$ pour un lment $z\in \mathfrak{P}^{-s}$ tel que $\bs{s}_b\circ \bs{s}(z)=c$. La strate $[\mathfrak{A},n,s-1,\bar{\gamma}]$ dans $\mathfrak{g}$ est simple, et l'extension $F[\bar{\gamma}]/F$ vrifie $e(F[\bar{\gamma}]/F)= e(E_1[c]/F)$ et $f(F[\bar{\gamma}]/F)= f(E_1[c]/F)$. En particulier, $F[\bar{\gamma}]$ est un sous--corps maximal de $\mathfrak{g}$. De plus on a
$$
k_F(\bar{\gamma})=\left\{\begin{array}{ll}
-s=k_{E_1}(c)& \mbox{si $E_1[c]\neq E_1$}\\
k_0(\gamma, \mathfrak{A})& \mbox{sinon}
\end{array}\right..
$$
\end{mapropo3}

\begin{proof}Puisque $s\leq r-1$, d'aprs la proposition 2, la strate $[\mathfrak{A},n,s,\gamma]$ dans $\mathfrak{g}$ est simple. D'aprs le corollaire 2, il existe une corestriction modre $s_{\gamma}: \mathfrak{g}\rightarrow \mathfrak{b}'_1$ sur $\mathfrak{g}$ relativement  $E'_1/F$ telle que pour tout $k\in {\Bbb Z}$ et tout $y\in \mathfrak{P}^k$, on a
$$
\bs{s}_{\gamma}(y) \equiv \bs{s}_b \circ \bs{s}(y) \quad ({\rm mod}\; \mathfrak{P}^{k+e}).
$$
L'lment $c'= s_{\gamma}(z)$ appartient  $\mathfrak{Q}'^{-s}_1$, et si $E_1[c]\neq E_1$, les entiers $s$ et $e$ sont premiers entre eux. D'après (4), les polynmes caractristiques $\phi_c\in \kappa_{E_1}[t]$ et $\phi_{c'}\in 
\kappa_{E'_1}[t]$ associs aux strates $[\mathfrak{B}_1,s,s-1,c]$ et $[\mathfrak{B}'_1,s,s-1,c']$, co\"{\i}ncident. Par conséquent 
la strate $[\mathfrak{B}'_1,s,s-1,c']$ dans $\mathfrak{b}'_1$ est simple. Comme $E'_1[c']$ est un sous--corps maximal de $\mathfrak{b}'_1$, on peut appliquer la proposition 1. Remarquons que $e(E'_1[c']/F)= e(E_1[c]/F)$ et $f(E'_1[c']/F)=f(E_1[c]/F)$, et que $k_{E'_1}(c')= k_{E_1}(c)$. D'o le résultat.
\end{proof}

%%%%%%%%%%%%%%%%%%%%%%%%
\subsection{Approximation}\label{approximation}Soit un élément $\gamma\in G_{\rm qre}$. On suppose que $\gamma$ {\it n'est pas $F$--minimal}. Posons $n= n_F(\gamma)\;(=-\nu_{F[\gamma]}(\gamma))$ et $r=-k_F(\gamma)$. Puisque $\gamma$ n'est pas $F$--minimal, on a $r>n$. Posons $\mathfrak{A}=\mathfrak{A}_\gamma$ et $\mathfrak{P}=\mathfrak{P}_\gamma$. On a donc $n=-\nu_{\mathfrak{A}}(\gamma)$. La strate $\bs{S}_\gamma=[\mathfrak{A},n,r,\gamma]$ dans $\mathfrak{g}$ est pure, et d'après 
\cite[2.4.1]{BK}, elle est équivalente à une strate simple $[\mathfrak{A},n,r,\beta]$. Par définition, $\beta$ est un élément de $\gamma + \mathfrak{P}^{-r}$. D'après loc.~cit., $e(F[\beta]/F)$ divise $e(F[\gamma]/F)$ et $f(F[\beta]/F)$ divise $f(F[\gamma]/F)$, et parmi les strates pures $[\mathfrak{A},n,r,\beta']$ dans $\mathfrak{g}$ qui sont équivalentes à $\bs{S}_\gamma$, les strates simples sont précisément celles qui minimisent le degré de l'extension $F[\beta']/F$, \cad qui vérifient $[F[\beta']:F]=[F[\beta]:F]$. De plus (loc.~cit.), pour toute strate simple $[\mathfrak{A},n,r,\beta']$ dans $\mathfrak{g}$ équivalente à $\bs{S}_\gamma$, on a:
\begin{itemize}
\item $e(F[\beta']/F)=e(F[\beta]/F)$ et $f(F[\beta']/F)=f(F[\beta]/F)$;
\item $k_F(\beta')=k_F(\beta)$.
\end{itemize}
Enfin (loc.~cit.), si $\bs{s}_\beta: \mathfrak{g}\rightarrow \mathfrak{g}_\beta={\rm End}_{F[\beta]}(V)$ est une corestriction modérée sur $\mathfrak{g}$ relativement à $F[\beta]/F$, alors la strate $[\mathfrak{A}\cap \mathfrak{g}_\beta,r,r-1,\bs{s}_\beta(\gamma-\beta)]$ dans $\mathfrak{g}_\beta$ est équivalente à une strate simple. C'est cette dernière assertion que l'on précise dans ce numéro.

Posons $E=F[\beta]$ et $\mathfrak{b}={\rm End}_E(V)$. Soit $\mathfrak{B}$ l'$\mathfrak{o}_E$--ordre héréditaire $\mathfrak{A}\cap \mathfrak{b}$ dans $\mathfrak{b}$, et soit $\mathfrak{Q}={\rm rad}(\mathfrak{B})$. On fixe une corestriction modérée $\bs{s}:\mathfrak{g}\rightarrow \mathfrak{b}$ sur $\mathfrak{g}$ relativement à $E/F$, et un élément $\bs{x}\in \mathfrak{A}$ tel que $\bs{s}(\bs{x})=1$. 

On pose $k_0=k_0(\beta,\mathfrak{A})\;(<-r)$ et $\mathfrak{N}_{k_0}=\mathfrak{N}_{k_0}(\beta,\mathfrak{A})$. 

% lemme 1
\begin{monlem1}
On a
$$\beta+ \mathfrak{P}^{-r}= \{g^{-1}(\beta + \bs{x} b)g: g\in 1+\mathfrak{Q}^{-r-k_0}\mathfrak{N}_{k_0},\,b\in \mathfrak{Q}^{-r}\}.
$$
\end{monlem1}

\begin{proof}L'inclusion $\supset$ est claire, puisque d'après \cite[1.5.8]{BK}, $1+\mathfrak{Q}^{-r-k_0}\mathfrak{N}_{k_0}$ est contenu dans le $G$--entrelacement de la strate simple $[\mathfrak{A}, n,r,\beta]$ dans $\mathfrak{g}$. Pour l'inclusion $\subset$, on procède par approximations successives. D'après les relations (1) et (3) de \ref{une submersion}, pour $i>k_0$, on a
$$
\mathfrak{P}^{i}= {\rm ad}_\beta(\mathfrak{Q}^{i-k_0}\mathfrak{N}_{k_0})\oplus \bs{x}\mathfrak{Q}^{i}.
$$
Soit $X \in \mathfrak{P}^{-r}$, et soit $a= -r-k_0>0$. \'Ecrivons $X={\rm ad}_\beta(y)+\bs{x}b$ avec $y\in \mathfrak{Q}^{a}\mathfrak{N}_{k_0}$ et $b\in \mathfrak{Q}^{-r}$. On a
\begin{eqnarray*}
(1+y)(\beta +X)(1+y)^{-1}&\equiv& \beta + X -({\rm ad}_\beta(y)+ {\rm ad}_X(y))(1+y)^{-1}\\
&\equiv & \beta + \bs{x}b \quad ({\rm mod}\; \mathfrak{P}^{-r+a}).
\end{eqnarray*}
Posons $g=(1+y)$, et écrivons $g(\beta +X)g^{-1}= \beta + \bs{x}b + X'$ pour un élément $X'\in \mathfrak{P}^{-r+a}$. \'Ecrivons $X'= {\rm ad}_\beta(y')+ \bs{x}b'$ avec $y'\in \mathfrak{Q}^{2a}\mathfrak{N}_{k_0}$ et $b'\in \mathfrak{Q}^{-r+a}$. On obtient de même
$$
(1+y')(\beta + X')(1+y')^{-1}\equiv \beta + \bs{x}b'\quad ({\rm mod}\; \mathfrak{P}^{-r+3a}),
$$
et donc
$$
(1+y')(\beta + \bs{x}b + X')(1+y')^{-1}\equiv \beta + \bs{x}(b+b') \quad ({\rm mod}\; \mathfrak{P}^{-r+2a}).
$$
Posant $g'=1+y'$, $g_1=g'g$ et $b_1=b+b'$, on a donc
$$
g_1(\beta+X)g_1^{-1}\equiv \beta + \bs{x}b_1\quad ({\rm mod}\; \mathfrak{P}^{-r+2a}).
$$
L'élément $g_1$ appartient à $1+\mathfrak{Q}^{a}\mathfrak{N}_{k_0}$ et l'élément $b_1$ appartient à $\mathfrak{Q}^{-r}$. Posons
$$\Omega=\{g^{-1}(\beta + \bs{x} b)g: g\in 1+\mathfrak{Q}^{a}\mathfrak{N}_{k_0},\,b\in \mathfrak{Q}^{-r}\}.$$
On a montré que pour tout entier $j\geq 1$, on a l'inclusion
$$
\beta + \mathfrak{P}^{-r}\subset \Omega +\mathfrak{P}^{-r+aj}.
$$
Comme $\Omega$ est ouvert (d'après la proposition de \ref{une submersion}) et compact dans $G$, pour $j$ suffisamment grand, on a l'égalité $\Omega + \mathfrak{P}^{-r+aj}=\Omega$. D'où le lemme.
\end{proof}

Identifions $\mathfrak{g}$ à $A(E)\otimes_E \mathfrak{b}$ via le choix d'une $(W,E)$--décomposition $\mathfrak{A}(E)\otimes_{\mathfrak{o}_E}\mathfrak{B}\buildrel\simeq\over{\longrightarrow}\mathfrak{A}$ de $\mathfrak{A}$ --- cf. \ref{une submersion} ---, et écrivons $\bs{s}=\bs{s}_0\otimes {\rm id}_{\mathfrak{b}}$, où $\bs{s}_0:A(E)\rightarrow E$ est une corestriction modérée sur $A(E)$ relativement à $E/F$. Fixons un élément $\bs{x}_0\in \mathfrak{A}(E)$ tel que $\bs{s}_0(\bs{x}_0)=1$, et prenons pour $\bs{x}$ l'élément $\bs{x}_0\otimes 1$. D'après le lemme 1, il existe un élément $g\in 1+\mathfrak{Q}^{-r-k_0}\mathfrak{N}_{k_0}$ tel que $\gamma \in g^{-1}(\beta+\bs{x}\mathfrak{Q}^{-r})g$. \'Ecrivons $\gamma = g^{-1}(\beta + \bs{x}b)g$ avec $b\in \mathfrak{Q}^{-r}$. Puisque $g\in U^1(\mathfrak{A})$ et $xb\in \mathfrak{P}^{-r}$, on a $g^{-1}\bs{x}bg\in \mathfrak{P}^{-r}$ et $\beta'=g^{-1}\beta g\in \gamma + \mathfrak{P}^{-r}$. La strate $[\mathfrak{A},n,r,\beta']$ dans $\mathfrak{g}$ est simple et équivalente à $[\mathfrak{A},n,r,\gamma]$, et quitte à remplacer $\beta$ par $\beta'$, $E$ par $E'=F[\beta']$, 
$\mathfrak{b}$ par $\mathfrak{b}'={\rm End}_{E'}(V)$, l'identification $\mathfrak{g}=A(E)\otimes_E\mathfrak{b}$ par 
$\mathfrak{g}= A(E')\otimes_{E'}\mathfrak{b}'$ (par transport de structure via ${\rm Int}_{\smash{g^{-1}}}$), $\bs{x}_0$ par $g^{-1}\bs{x}_0 g$ et $b$ par $g^{-1}bg$, on peut supposer que $\gamma = \beta + \bs{x}_0\otimes b$.

On définit comme on l'a fait pour $\mathfrak{g}$, en rempla\c{c}ant le corps de base $F$ par $E$, les notions d'éléments quasi--réguliers et quasi--réguliers elliptiques de $\mathfrak{b}$. On note $\mathfrak{b}_{\rm qr}$, resp. $\mathfrak{b}_{\rm qre}$, l'ensemble des éléments quasi--réguliers, resp. quasi--réguliers elliptiques, de $\mathfrak{b}$.

% lemme 2
\begin{monlem2}
L'élément $b$ est quasi--régulier elliptique dans $\mathfrak{b}$ et la strate $[\mathfrak{B},r,r-1,b]$ dans $\mathfrak{b}$ est simple. En particulier, $b$ est $E$--minimal.
\end{monlem2}

\begin{proof}D'après \cite[2.4.1.(iii)]{BK}, la strate $[\mathfrak{B},r,r-1,b]$ dans $\mathfrak{b}$ est équivalente à une strate simple, 
disons $[\mathfrak{B},r,r-1,c]$. Posons $E_1=E[c]$. Si $c$ est quasi--régulier elliptique dans $\mathfrak{b}$, \cad si $E_1$ est un sous--corps maximal de $\mathfrak{b}$, alors d'après \cite[2.2.2]{BK}, la strate $[\mathfrak{B},r,r-1,b]$ dans $\mathfrak{b}$ est simple et $E[c]$ est un sous--corps maximal de $\mathfrak{b}$. Dans ce cas $b$ est un élément quasi--régulier elliptique de $\mathfrak{b}$ et il est $E$--minimal. 

Supposons (par l'absurde) que le sous--corps $E_1\subset \mathfrak{b}$ n'est pas maximal. Puisque la strate $[\mathfrak{A},n,r,\beta]$ dans $\mathfrak{g}$ est simple, on a $r<\inf \{-k_0(\beta,\mathfrak{A}),n\}$. On peut donc appliquer la proposition 2 de \ref{raffinement}:
en posant $\gamma'= \beta +\bs{x}_0\otimes c$, la strate $[\mathfrak{A},n,r-1,\gamma']$ dans $\mathfrak{g}$ est simple, $
e(F[\gamma']/F)=e(E_1/F)$ et $f(F[\gamma']/F)=f(E_1/F)$, et
$$
k_0(\beta + \bs{x}_0\otimes c,\mathfrak{A})=\left\{\begin{array}{ll}
-r =k_0(c,\mathfrak{B}) & \mbox{si $E_1\neq E$}\\
k_0(\beta,\mathfrak{A})& \mbox{sinon}\end{array}\right..
$$
En particulier $[F[\gamma']:F]$ est strictement 
inférieur à $[F[\gamma]:F]$, ce qui est impossible puisque les strates $[\mathfrak{A},n,r-1,\gamma]$ et $[\mathfrak{A},n,r-1,\gamma']$ dans $\mathfrak{g}$ sont simples et équivalentes. Le sous--corps $E_1\subset \mathfrak{b}$ est donc forcément maximal, ce qui achève la démonstration du lemme.
\end{proof}

Le lemme 2 est le point de départ du procédé d'approximation des éléments quasi--réguliers elliptiques de $\mathfrak{g}$ par des éléments minimaux. 

Soit un élément $\gamma\in G_{\rm qre}$. Posons $\mathfrak{A}=\mathfrak{A}_\gamma$ et $\mathfrak{P}=\mathfrak{P}_\gamma$,  $n=n_F(\gamma)$ et $r= \inf(-k_F(\gamma),n-1)$. On note $\bs{S}_\gamma$ la strate pure $[\mathfrak{A},n,r,\gamma]$ dans $\mathfrak{g}$. Si $\gamma$ est $F$--minimal, i.e. si $r=n-1$, alors la strate $\bs{S}_\gamma$ est simple et il n'y a plus rien à faire. Sinon, on écrit $\gamma= \beta + \bs{x}_\beta\otimes b$ comme ci--dessus: $\beta$ est un élément pur de $\mathfrak{g}$ telle que la strate $[\mathfrak{A},n,r,\beta]$ dans $\mathfrak{g}$ est simple et équivalente à $\bs{S}_\gamma$; $b$ est un élément quasi--régulier elliptique de $\mathfrak{b}={\rm End}_{F[\beta]}(V)$ tel que la strate $[\mathfrak{A}\cap \mathfrak{b},r,r-1,b]$ dans $\mathfrak{b}$ est simple; $\bs{x}_\beta$ est un élément de $\mathfrak{A}(F[\beta])$ tel que $\bs{s}_\beta(\bs{x}_\beta)=1$ pour une corestriction modérée $\bs{s}_\beta:A(F[\beta])\rightarrow F[\beta]$ sur $A(F[\beta])$ relativement à $F[\beta]/F$; $\bs{x}_\beta\otimes b$ est un élément de $\mathfrak{P}^{-r}$ pour l'identification $\mathfrak{g}=A(F[\beta])\otimes_{F[\beta]}\mathfrak{b}$ donnée par le choix d'une $(W,F[\beta])$--décomposition de $\mathfrak{A}$. On pose $\gamma_1=\beta$, $F_1=F[\gamma_1]$, $\mathfrak{g}_1= A(F_1)$ et $\bs{x}_1=\bs{x}_{\gamma_1}$. On a donc $\gamma =\gamma_1+ \bs{x}_1\otimes b$. L'élément $\gamma_1$ est quasi--régulier elliptique dans $\mathfrak{g}_1$. Il définit comme ci--dessus une strate pure $\bs{S}_{\gamma_1}=[\mathfrak{A}_1,n_1,r_1,\gamma_1]$ dans $\mathfrak{g}_1$. Ici $\mathfrak{A}_1=\mathfrak{A}(F_1)$,  $n_1=n_F(\gamma_1)$ et $r_1= \inf(-k_F(\gamma_1),n_1-1))$. Si $\gamma_1$ est $F$--minimal, i.e. si $r_1=n_1-1$, on s'arrête là. Sinon, en rempla\c{c}ant $\gamma$ par $\gamma_1$ dans la construction précédente, on écrit $\gamma_1= \gamma_2 + \bs{x}_2\otimes b_1$ comme ci--dessus. Puisque $[F_1:F]< N$, le processus s'arrête au bout d'un nombre fini d'étapes. D'où la

% proposition
\begin{mapropo}
Soit $\gamma\in G_{\rm qre}$. Il existe un entier $m\geq 0$ et des éléments $\gamma_0,\ldots ,\gamma_m$ tels que:
\begin{itemize}
\item $\gamma_0=\gamma$;
\item (si $i<m$) $\gamma_{i+1}$ est un lment $F$--pur de $A(F[\gamma_i])$;
\item $\gamma_m$ est $F$--minimal;
\end{itemize} 
la suite $\{\gamma_0,\gamma_1,\ldots ,\gamma_m\}$ vérifiant les propriétés suivantes. Pour $i=0,\ldots ,m$, posons:
\begin{itemize}
\item $F_i=F[\gamma_i]$, $n_i=n_F(\gamma_i)$, $r_i=-k_F(\gamma_i)$;
\item $\mathfrak{g}_i=A(F_i)$, $\mathfrak{A}_i=\mathfrak{A}(F_i)$ et $\mathfrak{P}_i=\mathfrak{P}(F_i)\;(= {\rm rad}(\mathfrak{A}_i))$;
\item $\bs{S}_{\gamma_i}=[\mathfrak{A}_i,n_i,r_i,\gamma_i]$ --- une strate pure dans $\mathfrak{g}_i$;
\item (si $i<m$) $\mathfrak{b}_i={\rm End}_{F_{i+1}}(F_i)$, $\mathfrak{B}_i= \mathfrak{A}_i\cap \mathfrak{b}_i$ et $\mathfrak{Q}_i= {\rm rad}(\mathfrak{B}_i)$.
\end{itemize}
On identifie $\mathfrak{g}$ à $\mathfrak{g}_0$ via le choix d'un vecteur $v\in V\smallsetminus \{0\}$ (on a donc $\mathfrak{A}_\gamma = \mathfrak{A}_0$, cf. \ref{éléments qre}), 
et si $i<m$, on identifie $\mathfrak{g}_i$ à $\mathfrak{g}_{i+1}\otimes_{F_{i+1}}\mathfrak{b}_i$ via le choix d'une $(W_i,F_{i+1})$--décomposition $\mathfrak{A}_{i+1}\otimes_{\mathfrak{o}_{F_{i+1}}}\!\!\mathfrak{B}_i\buildrel\simeq\over{\longrightarrow}\mathfrak{A}_i$ de $\mathfrak{A}_i$ (on a un isomorphisme de $F$--espaces vectoriels $F_{i+1}\otimes_FW_i\simeq F_i$). 
La strate $\bs{S}_{\gamma_i}$ dans $\mathfrak{g}_i$ est équivalente à une strate simple $[\mathfrak{A}_i,n_i,r_i,\gamma_{i+1}]$ avec
$$
\gamma_i=\gamma_{i+1}+\bs{x}_{i+1}\otimes b_i
$$ pour un élément $b_i\in (\mathfrak{b}_i)_{\rm qre}$ tel que la strate $[\mathfrak{B}_i,r_i,r_i-1,b_i]$ dans $\mathfrak{b}_i$ est simple, et un élément $\bs{x}_{i+1}\in \mathfrak{A}_{i+1}$ tel que $\bs{s}_{\gamma_{i+1}}(\bs{x}_{i+1})=1$, où $\bs{s}_{\gamma_{i+1}}:\mathfrak{g}_{i+1}\rightarrow F_{i+1}$ est une corestriction modérée sur $\mathfrak{g}_{i+1}$ relativement à $F_{i+1}/F$. On a donc les décompositions
$$
\mathfrak{g}= \mathfrak{g}_m \otimes_{F_m} \mathfrak{b}_{m-1}\otimes_{F_{m-1}}\mathfrak{b}_{m-2}\otimes \cdots \otimes_{F_1}\mathfrak{b}_0\leqno{(1)}
$$
et (en identifiant $\mathfrak{g}_{i+1}$ à la sous--$F$--algèbre $\mathfrak{g}_{i+1}\otimes 1$ de $\mathfrak{g}_i$ )
$$
\gamma =\gamma_m+ \bs{x}_{m} b_{m-1}+ \bs{x}_{m-1} b_{m-2}+\cdots + \bs{x}_1 b_0.\leqno{(2)}
$$
\end{mapropo}

% définition
\begin{madefi}
{\rm Soit $\gamma\in G_{\rm qre}$. Toute suite $(\gamma_0=\gamma,\gamma_1,\ldots ,\gamma_m)$ vérifiant les conditions de la proposition est appelée {\it suite d'approximation minimale de $\gamma$}. \`A une telle suite sont associées:
\begin{itemize}
\item une suite $(F_0, \ldots ,F_m)$ d'extensions $F_i/F$, $F_i=F[\gamma_i]$ --- pour $i=0,\ldots ,m$, on pose $\mathfrak{g}_i=A(F_i)$ et 
$\mathfrak{A}_i = \mathfrak{A}(F_i)$, et (si $i<m$) $\mathfrak{b}_i={\rm End}_{F_{i+1}}(F_i)$ et $\mathfrak{B}_i= \mathfrak{A}_i\cap \mathfrak{b}_i$;
\item une suite $(\bs{x}_1,\ldots ,\bs{x}_m)$ d'éléments $\bs{x}_i\in \mathfrak{A}_i$ tels que $\bs{s}_{\gamma_i}(\bs{x}_i)=1$ pour une corestriction modérée $\bs{s}_{\gamma_i}: \mathfrak{g}_i\rightarrow F_i$ sur $\mathfrak{g}_i$ relativement à $F_i/F$;
\item une suite $(b_0,\ldots ,b_{m-1})$ d'éléments $b_i\in \mathfrak{b}_i$ tels que $\bs{x}_{i+1}\otimes 
b_i=\gamma_i-\gamma_{i+1}$, où pour $i=0,\ldots ,m-1$, on a identifié $\mathfrak{g}_i$ à $\mathfrak{g}_{i+1}\otimes_{F_{i+1}}\mathfrak{b}_i$ via le choix d'une $(W_i,F_{i+1})$--décomposition 
$\mathfrak{A}_{i+1}\otimes_{\mathfrak{o}_{F_{i+1}}}\!\! \mathfrak{B}_i \buildrel \simeq \over{\longrightarrow} \mathfrak{A}_i$ de $\mathfrak{A}_i$. 
\end{itemize}
La suite $(\bs{x}_1,\ldots ,\bs{x}_m)$ est appelée {\it suite des correcteurs} de la suite $(\gamma_0,\ldots ,\gamma_m)$, et la suite 
$(b_0,\ldots ,b_{m-1})$ est appelée {\it suite dérivée} de la suite $(\gamma_0,\ldots ,\gamma_m)$. La suite des correcteurs $(\bs{x}_1,\ldots ,\bs{x}_m)$ est définie via le choix des corestrictions modérées $\bs{s}_{\gamma_i}: \mathfrak{g}_{i}\rightarrow F_{i}$. Si, pour $i=0,\ldots ,m-1$, on note $\tilde{\bs{s}}_i$ la corestriction modérée $\bs{s}_{\gamma_{i+1}}\otimes{\rm id}_{\mathfrak{b}_i} : \mathfrak{g}_i\rightarrow \mathfrak{b}_i$ sur $\mathfrak{g}_i$ relativement à $F_{i+1}/F$, alors la suite dérivée $(b_1,\ldots ,b_m)$ est donnée par $b_i = \tilde{\bs{s}}_i(\gamma_i-\gamma_{i+1})$. 
}
\end{madefi}

% remarque 1
\begin{marema1}
{\rm Fixé $\gamma\in G_{\rm qre}$, la suite d'approximation minimale $(\gamma_0,\ldots ,\gamma_m)$ de $\gamma$ n'est pas unique, la suite d'extensions $(F_0,\ldots ,F_m)$ de $F$ définie par $(\gamma_0,\ldots ,\gamma_m)$ n'est pas non plus unique, mais les invariants suivants le sont:
\begin{itemize}
\item l'entier $m\geq0$, appelé \og longueur \fg{} de la suite d'approximation minimale de $\gamma$, ou simplement \og longueur\fg de $\gamma$;
\item les entiers $n_i= n_F(\gamma_i)$ et $r_i=-k_F(\gamma_i)$;
\item les entiers $e_i=e(F_i/F)$ et $f_i=f(F_i/F)$.
\end{itemize}
L'élément $F$--minimal $\gamma_m$ peut être central, \cad que l'on peut avoir $F_m=F$. La longueur d'un lment de 
$G_{\rm qre}$ est infrieure ou gale au nombre de facteurs premiers de $N$ (compts avec multiplicit). Les éléments 
de longueur $0$ sont les éléments $F$--minimaux. 

Notons qu'il n'est en général pas possible de choisir la suite d'extensions $(F_0,\ldots ,F_m)$ de $F$ telle que $F_m\subset \cdots \subset F_0$. En effet (pour $i<m$), 
on a l'inclusion $F_{i+1} \subset F_i$ si et seulement si l'lment $\bs{x}_{i+1}\otimes b_i= \gamma_i - \gamma_{i+1}$ est dans $F_i$ (et donc en particulier commute  $\gamma_{i+1}$), 
ce qui n'est possible que si l'extension $F_{i+1}/F$ est {\it modrment ramifie} \cite[2.2.6]{BK}. D'autre part l'extension 
$F_{i+1}/F$ est modrment ramifie si et seulement si on peut prendre $\bs{x}_{i+1}=1$ \cite[1.3.8]{BK}. En dfinitive, on peut choisir la suite d'extensions $(F_0,\ldots,F_m)$ telle que 
$F_m\subset \cdots \subset F_0$ si et seulement si on peut prendre la suite des correcteurs $(\bs{x}_1,\ldots,\bs{x}_m)$ égale à $(1,\ldots ,1)$, ce qui n'est possible 
que si toutes les extensions $F_{i+1}/F$ ($i=0,\ldots ,m-1$) sont {\it modérément ramifiées}, 
\cad (puisque $e(F_{i+1}/F)$ divise $e(F_i/F)$) si $e_1=e(F_1/F)$ est premier  la caractristique rsiduelle $p$ de $F$. 
C'est donc toujours possible si $p$ ne divise pas $N$, et aussi si $N=p$ (car dans ce cas ou bien $\gamma$ est $F$--minimal, ou bien $m=1$ et $F_1=F$). 
Si $p<N$ divise $N$, on construit facilement un contre--exemple (voir ci--dessous).

Si $(\gamma_0,\ldots ,\gamma_m)$ est une suite d'approximation minimale de $\gamma\in G_{\rm qre}$ de suite des correcteurs $(\bs{x}_1,\ldots ,\bs{x}_m)$ et de suite dérivée $(b_0,\ldots ,b_{m-1})$, alors 
pour $k=0,\ldots ,m$, $(\gamma_k,\ldots ,\gamma_m)$ est une suite d'approximation minimale de $\gamma_k\in A(F_k)^\times_{\rm qre}$ de suite des correcteurs $(\bs{x}_{k+1},\ldots ,\bs{x}_m)$ et de suite dérivée $(b_k,\ldots ,b_{m-1})$. 
\hfill $\blacksquare$
}
\end{marema1}

% Exemple
\begin{exemple}
{\rm 
Si $N$ est premier, les lments de $G_{\rm qre}$ sont de longueur $0$ ou $1$. Les lements de longueur $1$ sont de la forme 
$\gamma = z +\gamma'$ avec $z\in F^\times$ et $\gamma'\in \mathfrak{g}_{\rm qre}$ tels que $n_F(\gamma')= - k_F(\gamma)$ et 
$\nu_{F[\gamma]}(z)= - n_F(\gamma)<k_F(\gamma)$. Alors $(\gamma,z)$ est une suite d'approximation minimale de $\gamma$. Supposons 
maintenant que $N= p_1p_2$ avec $p_i$ premier ($p_1=p_2$ est permis). Soit $\beta\in G$ un lment pur tel que $\beta\notin G_{\rm qre}$. 
Posons $E=F[\beta]$ et $\mathfrak{b}= {\rm End}_E(V)$. Comme $\beta \in A(E)^\times_{\rm qre}$ et 
$[E:F]\in \{1,p_1,p_2\}$, si $\beta$ n'est pas $F$--minimal (ce qui implique $\beta\notin F^\times$), alors toute suite d'approximation minimale de $\beta$ est de 
la forme $(\beta,z)$ pour un $z\in F^\times$. Soit $b\in \mathfrak{b}_{\rm qre}$ un lment $E$--minimal tel que
$$
- \nu_E(b) \;\left( = {n_E(b) \over e(E[b]/E)}\right) <\inf (-k_F(\beta),n_F(\beta)).
$$
Posons $K= E[b]= F[\beta,b]$. 
Soit $\mathfrak{A}$ l'unique $\mathfrak{o}$--ordre hrditaire 
dans $\mathfrak{g}$ normalis par $K^\times$, et soit $\mathfrak{B}= \mathfrak{A}\cap \mathfrak{b}$. On a donc $e(E[b]/E) = e(\mathfrak{B} \vert \mathfrak{o}_E)$ et
$$
- n_E(b) < \inf (-k_0(\mathfrak{A},\beta), - \nu_{\mathfrak{A}}(\beta)).
$$Choisissons une $(W,E)$--dcomposition 
$\mathfrak{A}(E)\otimes_{\mathfrak{o}_E}\mathfrak{B} \buildrel \simeq \over{\longrightarrow} \mathfrak{A}$ de $\mathfrak{A}$ et un lment $\bs{x}_0\in \mathfrak{A}(E)$ tel que 
$\bs{s}_0(\bs{x}_0)=1$ pour une corestriction modre $\bs{s}_0: A(E) \rightarrow E$ sur $A(E)$ relativement  $E/F$ (on peut prendre $\bs{x}_0=1$ si et seulement si 
l'extension $E/F$ est modrment ramifie). Alors l'lment $\gamma = \beta + \bs{x}_0\otimes b$ appartient  
$G_{\rm qre}$, et il est de longeur $1$ ou $2$: si $\beta$ est $F$--minimal, alors $(\gamma,\beta)$ est une suite d'approximation minimale de $\gamma$; et si $(\beta,z)$ est une suite d'approximation 
minimale de $\beta$, alors $(\gamma,\beta,z)$ est une suite d'approximation minimale de $\gamma$. Tous les lments de $G_{\rm qre}$ qui ne sont pas $F$--minimaux sont 
obtenus de cette manire. 
}
\end{exemple}

% remarque 2
\begin{marema2}
{\rm  
Soit $\gamma\in G_{\rm qre}$, et soit $(\gamma_0,\ldots ,\gamma_m)$ une suite d'approximation minimale de $\gamma$ de suite des correcteurs $(\bs{x}_1,\ldots ,\bs{x}_m)$ et de suite dérivée $(b_0,\ldots ,b_{m-1})$. \'Ecrivons
$$
\gamma =\gamma_m+ \bs{x}_{m} b_{m-1}+ \bs{x}_{m-1} b_{m-2}+\cdots + \bs{x}_1 b_0
$$
comme en (2). Pour $i=0,\ldots ,m$, rappelons que l'on a posé $\mathfrak{g}_i=A(F_i)$, $\mathfrak{A}_i = \mathfrak{A}(F_i)$, 
et (si $i<m$) $\mathfrak{b}_i={\rm End}_{F_{i+1}}(F_i)$ et $\mathfrak{B}_i = \mathfrak{A}_i\cap \mathfrak{b}_i$, et que l'on a identifié $\mathfrak{g}_i$ à $\mathfrak{g}_{i+1}\otimes_{F_{i+1}}\mathfrak{b}_i$ via le choix d'une $(W_i,F_{i+1})$--décomposition 
$\mathfrak{A}_{i+1}\otimes_{\mathfrak{o}_{F_{i+1}}}\!\! \mathfrak{B}_i \buildrel \simeq \over{\longrightarrow} \mathfrak{A}_i$ de $\mathfrak{A}_i$. Si 
$\tilde{\bs{t}}_i: \mathfrak{g}_i\rightarrow \mathfrak{b}_i$ est une corestriction modérée sur $\mathfrak{g}_i$ relativement à $F_{i+1}/F$, alors on a $\tilde{\bs{t}}_i(\gamma_i-\gamma_{i+1})= u_i b_i$ pour un élément $u_i\in \mathfrak{o}_{F_{i+1}}^\times$. La corestriction modérée $\tilde{\bs{s}}_i = \bs{s}_{\gamma_{i+1}}\otimes {\rm id}_{\mathfrak{b}_i}$ sous--jacente à la définition de $(\gamma_0,\ldots ,\gamma_m)$ est donc celle qui est normalisée par $u_i=1$. Soit $s_i^{i+1}: \mathfrak{b}_i\rightarrow F_i$ la corestriction modérée sur $\mathfrak{b}_i$ relativement à $F_i/F_{i+1}$ telle que
$$
\bs{s}_i^{i+1}\circ \tilde{\bs{s}}_{i}=\bs{s}_{\gamma_i}.
$$
Soit un élément $\bs{x}_{i}^{i+1}\in \mathfrak{B}_i$ tel que $\bs{s}_i^{i+1}(\bs{x}_i^{i+1})=1$. Alors l'élément $\bs{y}_i=\bs{x}_{i+1}\otimes \bs{x}_i^{i+1}$ de $\mathfrak{A}_{i+1}\otimes_{\mathfrak{o}_{F_{i+1}}}\!\mathfrak{B}_i=\mathfrak{A}_i$ vérifie $\bs{s}_{\gamma_i}(y_i)=1$. 
Par conséquent l'élément $\bs{x}_i - \bs{y}_i$ appartient à $\ker (\bs{s}_{\gamma_i})\cap \mathfrak{A}_i={\rm ad}_{\gamma_i}(\mathfrak{g}_i)\cap \mathfrak{A}_i$. On pourrait essayer de s'arranger --- mais nous ne le ferons pas ici --- pour que la suite des correcteurs vérifie la condition supplémentaire: $\bs{x}_i = \bs{x}_{i+1} \otimes \bs{x}_i^{i+1}$ pour $i=1,\ldots ,m-1$.\hfill$\blacksquare$.
}
\end{marema2}

%%%%%%%%%%%%%%
\subsection{Le résultat principal}\label{le résultat principal}On reprend la situation de \ref{une submersion}. 
Soit $\beta\in \mathfrak{g}$ un élément pur. Posons $E=F[\beta]$ et $\mathfrak{b}={\rm End}_E(V)$. Fixons une corestriction modérée $\bs{s}_0:A(E)\rightarrow E$ sur $A(E)$ relativement à $E/F$, et un élément $\bs{x}_0\in \mathfrak{A}(E)$ tel que $\bs{s}_0(\bs{x}_0)=1$. Notons $H$ le groupe $\mathfrak{b}^\times ={\rm Aut}_E(V)$. 

Fixons aussi un $\mathfrak{o}_E$--ordre hrditaire minimal $\underline{\mathfrak{B}}$ dans $\mathfrak{b}$, et notons $\underline{\mathfrak{A}}$ l'$\mathfrak{o}$--ordre héréditaire dans $\mathfrak{g}$ normalis par $E^\times$ tel que $\underline{\mathfrak{A}}\cap \mathfrak{b}= \underline{\mathfrak{B}}$. Posons $\underline{\mathfrak{P}}={\rm rad}(\underline{\mathfrak{A}})$, $\underline{\mathfrak{Q}}={\rm rad}(\underline{\mathfrak{B}})\;(=\underline{\mathfrak{P}}\cap \mathfrak{b})$ et $d = {N\over [E:F]} \;(=e(\underline{\mathfrak{B}}\vert \mathfrak{o}_E)$. Identifions $\mathfrak{g}$ à $A(E)\otimes_E\mathfrak{b}$ via le choix d'une $(\underline{W},E)$--décomposition $\mathfrak{A}(E)\otimes_{\mathfrak{o}_E}\underline{\mathfrak{B}}\buildrel \simeq \over{\longrightarrow}\underline{\mathfrak{A}}$ de $\underline{\mathfrak{A}}$. 
Enfin posons $\bs{x}=\bs{x}_0\otimes 1$ et $\bs{s}=\bs{s}_0\otimes {\rm id}_{\mathfrak{b}}:\mathfrak{g}\rightarrow \mathfrak{b}$. 

% remarque 1
\begin{marema1}
{\rm 
Pour tout $\mathfrak{o}$--ordre hrditaire $\mathfrak{A}$ dans $\mathfrak{g}$ normalis par $E^\times$ et tel que $\mathfrak{A}\cap \mathfrak{b}$ contient $\underline{\mathfrak{B}}$, l'identification $\mathfrak{g}= A(E)\otimes_E \mathfrak{b}$ induit par restriction une identification $\mathfrak{A}= \mathfrak{A}(E) \otimes_{\mathfrak{o}_E} (\mathfrak{A}\cap \mathfrak{b})$. Mais ce ne sont pas les seuls $\mathfrak{A}$ qui vrifient cette proprit. \'Ecrivons $\underline{\mathfrak{B}} = {\rm End}_{\mathfrak{o}_E}^0(\ES{R})$ pour une chane de $\mathfrak{o}_E$--rseaux $\ES{R}= \{\ES{R}_i:i\in {\Bbb Z}\}$ dans $V$ de priode $d$. Alors $\underline{W}$ est le sous--$F$--espace vectoriel de $V$ engendr par une $\mathfrak{o}_E$--base $\{\underline{w}_1,\ldots ,\underline{w}_d\}$ de $\ES{R}$. Choisissons une uniformisante $\varpi_E$ 
de $E$. Soit $\bs{W}_{\!H}$ le sous--groupe de $H$ form des lments qui permutent la base $\{\underline{w}_i\}$, et soit 
$\bs{D}_{H}= \bs{D}_{H,\varpi_E}$ le sous--groupe de $H$ form des matrices diagonales (par rapport  la base $\{\underline{w}_i\}$) de la forme 
${\rm diag}(\varpi_E^{a_1},\ldots ,\varpi_E^{a_d})$ avec $a_i\in {\Bbb Z}$. 
Enfin soit $\wt{\bs{W}}_{\!H} = \wt{\bs{W}}_{\!H,\varpi_E}$ le sous--groupe de $H$ engendr par 
$\bs{W}_{\!H}$ et $\bs{D}_{H}$. Alors $\wt{\bs{W}}_{\!H} = \bs{W}_{\!H} \rtimes \bs{D}_H$, et on a la  dcomposition d'Iwahori
$$
H = U(\underline{\mathfrak{B}}) \wt{\bs{W}}_{\! H} U(\underline{\mathfrak{B}}).
$$
Par construction, pour $w\in \bs{W}_{\!H}$, on a $w \bs{x} w^{-1} = \bs{x}$. D'autre part, d'aprs la remarque 2 de \ref{des dcompositions}, 
pour tout $\mathfrak{o}$--ordre hrditaire $\mathfrak{A}$ dans $\mathfrak{g}$ normalis par $E^\times$ et tel que $w^{-1} (\mathfrak{A}\cap \mathfrak{b})w  \supset \underline{\mathfrak{B}}$ 
pour un $w\in \wt{\bs{W}}_{\!H}$, l'identification $A(E)\otimes_E\mathfrak{b} = \mathfrak{g}$ induit par restriction une identification 
$$
\mathfrak{A}(E) \otimes_{\mathfrak{o}_E} (\mathfrak{A \cap \mathfrak{b})} = \mathfrak{A}.\eqno{\blacksquare}
$$
}
\end{marema1}

Comme en \ref{éléments qre}, pour $k\in {\Bbb R}$, on pose
$$
\mathfrak{b}_{\rm qre}^k= \{h\in \mathfrak{b}_{\rm qre}: \nu_E(h)\geq k\}.
$$
Les \og sauts\fg de cette filtration $k\mapsto \mathfrak{b}_{\rm qre}^k$ de $\mathfrak{b}_{\rm qre}$ sont les éléments de ${1\over d}{\Bbb Z}$. 
Notons que si $d=1$, i.e. si $E$ est un sous--corps maximal de $\mathfrak{g}$, alors $\mathfrak{b}=E$, et pour $k\in {\Bbb Z}$, on a $\mathfrak{b}_{\rm qre}= \mathfrak{p}_E^k$.  
Posons $\underline{k}_0=k_0(\beta,\underline{\mathfrak{A}})$ et $\underline{n} = -\nu_{\underline{\mathfrak{A}}}(\beta)$. 
On a donc $\underline{k}_0= k_F(\beta)d$ et $\underline{n}= n_F(\beta)d$, et d'après le lemme 2 de \ref{éléments qre}, si $k_F(\beta)\neq -\infty$ (\cad si $E\neq F$), on a
$$
\mathfrak{b}_{\rm qre}^{k_F(\beta)+{1\over d}}= \mathfrak{b}_{\rm qre}\cap {^H\!(\underline{\mathfrak{Q}}^{\underline{k}_0 +1})}.\leqno{(1)}
$$
Notons que si $k_F(\beta)\neq -\infty$ et si $b\in \mathfrak{b}_{\rm qre}^{k_F(\beta) + {1\over d}}$ est tel que l'unique $\mathfrak{o}_E$--ordre hrditaire $\mathfrak{B}_b$ dans $\mathfrak{b}$ normalis par $E[b]^\times$ vrifie l'inclusion $ \underline{\mathfrak{B}}\subset \mathfrak{B}_b$, 
alors $b\in \underline{\mathfrak{Q}}^{\underline{k}_0+1}$. En effet, posant 
$r= n_E(b) \;(= -\nu_{E[b]}(b))$ et $e= e(E[b]/E)$, puisque $-{r\over e} > k_F(\beta)$, on a $-r \geq ek_F(\beta)+ 1$. Notant $\mathfrak{Q}_b$ le radical de Jacobson de $\mathfrak{B}_b$, on a donc 
$b \in \mathfrak{p}_E^{k_F(\beta)} \mathfrak{Q}_b$. Or $\underline{\mathfrak{Q}}^{\underline{k}_0 +1} = \mathfrak{p}_E^{k_F(\beta)} \underline{\mathfrak{Q}}$ et $\mathfrak{Q}_b \subset \underline{\mathfrak{Q}}$ (car $\underline{\mathfrak{B}} \subset \mathfrak{B}_b$).  

Pour $\gamma\in \mathfrak{g}$, on note $\ES{O}_G(\gamma)$ l'orbite $\{g^{-1}\gamma g: g\in G\}\subset G$. De même, pour $b\in \mathfrak{b}$, on note $\ES{O}_H(b)$ l'orbite $\{h^{-1}b h: h\in H\}\subset \mathfrak{b}$. Le lemme suivant a t prouv en \cite[5.4.2]{L2} grce  la proposition de \ref{une submersion}, 
comme consquence du principe de submersion d'Harish--Chandra (prcisment, de la construction 
de l'application $T \mapsto \vartheta_T$ que nous rappellerons en \ref{le principe de submersion}). Notons que 
dans loc.~cit., on travaille avec un $\mathfrak{o}_E$--ordre maximal dans $\mathfrak{b}$ au lieu de l'$\mathfrak{o}_E$--ordre hrditaire minimal 
$\underline{\mathfrak{B}}$, mais cela n'a aucune incidence sur le rsultat. 

% lemme
\begin{monlem}
Soient $b,\, b'\in \underline{\mathfrak{Q}}^{\underline{k}_0+1}$. On a
$$
\ES{O}_H(b')=\ES{O}_H(b)\Rightarrow \ES{O}_G(\beta + \bs{x}_0\otimes b')=\ES{O}_G(\beta + \bs{x}_0\otimes b)
$$
\end{monlem}

La proposition suivante généralise un résultat obtenu dans la démonstration de 
\cite[5.4.3]{L2}, où on se limitait aux éléments de $\mathfrak{b}_{\rm qre}$ qui sont $E$--minimaux. 

% proposition
\begin{mapropo}
On suppose $E\neq F$.
\begin{enumerate}
\item[(i)]Soit $b\in \mathfrak{b}_{\rm qre}\cap \underline{\mathfrak{Q}}^{\underline{k}_0+1}$. 
L'élément $\gamma=\beta + \bs{x}_0\otimes b$ appartient à $G_{\rm qre}$, on a $e(F[\gamma]/F)=e(E[b]/F)$ et $f(F[\gamma]/F)=f(E[b]/F)$, et
$$k_F(\gamma)= \left\{\begin{array}{ll} k_E(b) & \mbox{si $[E:F]<N$} \\
k_F(\beta)& \mbox{sinon}
\end{array}\right..
$$
\item[(ii)]Soient $b,\,b'\in\mathfrak{b}_{\rm qre}\cap \underline{\mathfrak{Q}}^{\underline{k}_0+1}$. On a
$$\ES{O}_G(\beta + \bs{x}_0\otimes b')=\ES{O}_G(\beta + \bs{x}_0\otimes b)\Leftrightarrow \ES{O}_H(b')=\ES{O}_H(b).
$$
\end{enumerate}
\end{mapropo}

\begin{proof}
Puisque $E\neq F$, on a $k_F(\beta) \geq -n_F(\beta)$ avec égalité si et seulement si $\beta$ est $F$--minimal. 
Commen\c{c}ons par prouver (i). Puisqu'il existe un $h\in H$ tel que $h\mathfrak{B}_b h^{-1}$ contient $\underline{\mathfrak{B}}$, 
quitte  remplacer $b$ par $hbh^{-1}$, on peut grce au lemme supposer que $\underline{\mathfrak{B}}\subset \mathfrak{B}_b$. Notons tout d'abord que si $b=0$, ce qui n'est possible que si $E$ est un sous--corps maximal de $\mathfrak{g}$, alors il n'y a rien à démontrer. On peut donc supposer $b\neq 0$. Posons $E_0=E[b]$ et $\mathfrak{B}=\mathfrak{B}_b$. 
Soit $\mathfrak{A}$ l'$\mathfrak{o}$--ordre héréditaire dans $\mathfrak{g}$ normalisé par $E^\times$ tel que $\mathfrak{A} \cap \mathfrak{b}= \mathfrak{B}$. C'est l'unique 
$\mathfrak{o}$--ordre hrditaire dans $\mathfrak{g}$ normalis par $E_0^\times$, et puisque $\underline{\mathfrak{B}}\subset \mathfrak{B}$, 
on a l'identification $\mathfrak{A}=\mathfrak{A}(E)\otimes_{\mathfrak{o}_E}\mathfrak{B}$. Posons $r=n_E(b)\;(=-\nu_{E_0}(b))$, 
$k_0=k_0(\beta, \mathfrak{A})$ et $n= -\nu_{\mathfrak{A}}(\beta)$. On a $k_0=k_F(\beta)e(E_0/E)$ et $n=n_F(\beta)e(E_0/E)$, et $k_0\geq -n$ avec égalité si et seulement si $\beta$ est $F$--minimal. Puisque
$$
{-r\over e(E_0/E)}= \nu_E(b) \geq k_F(\beta)+ {1\over d},
$$
on a
$$
-r\geq k_F(\beta)e(E_0/E) +{e(E_0/E)\over d}=k_0 + {1\over f(E_0/E)} >k_0.
$$
Considérons la strate pure $[\mathfrak{B},r,r-1,b]$ dans $\mathfrak{b}$. Si elle est simple, \cad si $b$ est $E$--minimal, comme $r< -k_0 \leq n$, on peut appliquer la proposition 1 de \ref{raffinement}: la strate $[\mathfrak{A}, n, r-1, \gamma]$ dans $\mathfrak{g}$ est simple, $e(F[\gamma]/F)=e(E_0/F)$ et $f(F[\gamma]/F)=f(E_0/F)$, et
$$
k_0(\gamma,\mathfrak{A})=\left\{\begin{array}{ll}
-r =k_0(b,\mathfrak{B})=k_E(b) & \mbox{si $E_0\neq E$}\\
k_0=k_F(\beta)& \mbox{sinon}\end{array}\right..
$$
Dans ce cas on a $[F[\gamma]/F]=[E_0:F]=N$, \cad que $\gamma$ est quasi--régulier elliptique dans $\mathfrak{g}$, et $E_0=E$ si et seulement si $[E:F]=N$.

Supposons maintenant que $b$ n'est pas $E$--minimal. On a donc $E_0\neq E$ et $k_E(b)>-r$. Posons $s=-k_E(b)$ et considérons la strate pure $[\mathfrak{B},r,s,b]$ dans $\mathfrak{b}$. Soit $\mathfrak{Q}={\rm rad}(\mathfrak{B})$. D'après les lemmes 1 et 2 de \ref{approximation}, on peut écrire $b$ sous la forme $b=b_1 + \bs{y}_1\otimes c$ avec:
\begin{itemize}
\item $b_1$ est un élément $E$--pur de $\mathfrak{B}$ tel que la strate $[\mathfrak{B},r,s,b_1]$ dans $\mathfrak{b}$ est simple et équi\-valente à $[\mathfrak{B},r,s,b]$;
\item $c$ est un élément quasi--régulier elliptique de $\mathfrak{b}_1={\rm End}_{E_1}(V)$, $E_1=E[b_1]$,  tel que la strate $[\mathfrak{B}_1,s,s-1,c]$ dans $\mathfrak{b}_1$ est simple, où on a posé $\mathfrak{B}_1=\mathfrak{B}\cap \mathfrak{b}_1$;
\item $\bs{y}_1$ est un élément de $\mathfrak{A}_1={\rm End}_{\mathfrak{o}_E}^0(\{\mathfrak{p}_{E_1}^i\})$ tel que $\bs{t}_1(\bs{y}_1)=1$ pour une corestriction modérée $\bs{t}_1:\mathfrak{a}_1\rightarrow E_1$ sur $\mathfrak{a}_1={\rm End}_E(E_1)$ relativement à $E_1/E$;
\item $\bs{y}_1\otimes c$ est un élément de $\mathfrak{Q}^{-s}$ pour l'identification $\mathfrak{b}=\mathfrak{a}_1\otimes_{E_1}\mathfrak{b}_1$ donnée par le choix d'une $(W_1,E_1)$--décomposition $\mathfrak{A}_1\otimes_{\mathfrak{o}_{E_1}}\mathfrak{B}_1
\buildrel \simeq\over{\longrightarrow} \mathfrak{B}$ de $\mathfrak{B}$.
\end{itemize}
De plus $e(E[b]/E)= e(E_1[c]/E)$ et $f(E[b]/E)=f(E_1[c]/E)$, et puisque $E_1[c]\neq E_1$ (car $[E_1:E]<d$ et $c$ est quasi--régulier elliptique dans $\mathfrak{b}_1$), on a
$$
k_E(b)=k_0(b,\mathfrak{B})=-s= k_0(c,\mathfrak{B}_1)= k_{E_1}(c).
$$
On a les identifications
$$
\mathfrak{g}=A(E)\otimes_E\mathfrak{a}_1\otimes_{E_1}\mathfrak{b}_1,\quad 
\mathfrak{A}= \mathfrak{A}(E)\otimes_{\mathfrak{o}_E}\mathfrak{A}_1\otimes_{\mathfrak{o}_{E_1}}\mathfrak{B}_1.
$$
D'autre part, en identifiant $A(E_1)$ à $A(E)\otimes_E \mathfrak{a}_1$ via le choix d'une $(X,E)$--décomposition
$\mathfrak{A}(E)\otimes_{\mathfrak{o}_E} \mathfrak{A}_1\buildrel\simeq \over{\longrightarrow}\mathfrak{A}(E_1)$
de $A(E_1)$, on a aussi les identifications
$$
\mathfrak{g}= A(E_1)\otimes_{E_1}\mathfrak{b}_1,\quad \mathfrak{A}= \mathfrak{A}(E_1)\otimes_{\mathfrak{o}_{E_1}}\mathfrak{B}_1.\leqno{(2)}
$$
L'élément $b_1$ s'écrit $b_1=a_1\otimes 1$ avec $a_1\in \mathfrak{a}_1$. Notons $\gamma_1$ l'élément $\beta + \bs{x}_0\otimes a_1$ de $A(E)\otimes_{E_1}\mathfrak{a}_1$, $G_1$ le groupe $A(E_1)^\times ={\rm Aut}_F(E_1)$, et $H_1$ le groupe $\mathfrak{a}_1^\times = {\rm Aut}_E(E_1)$. Posons $N_1= [E_1:F]$, $e_1=e(\mathfrak{B}_1\vert \mathfrak{o}_{E_1})$, $r_1 = {r\over  e_1}$, $k_1={k_0\over  e_1}$ et $n_1= {n\over  e_1}$. On a $k_1= k_0(\beta, \mathfrak{A}(E_1))$, $n_1 = -\nu_{\mathfrak{A}(E_1)}(\beta)$, et puisque $\nu_{\mathfrak{B}}(b_1)= -r$, on a
$$
\nu_{\mathfrak{A}_1}(a_1)= {\nu_{\mathfrak{B}}(b_1)\over e_1}=-r_1.
$$
La strate $[\mathfrak{A}_1,r_1,r_1-1,a_1]$ dans $\mathfrak{a}_1$ est simple, et comme $r_1<-k_1\leq n_1$, on peut appliquer la proposition 1 de \ref{raffinement}: la strate $[\mathfrak{A}(E_1), n_1, r_1-1, \gamma_1]$ dans $A(E_1)$ est simple, $e(F[\gamma_1]/F)=e(E_1/F)$ et $f(F[\gamma_1]/F)=f(E_1/F)$, et puisque $E_1\neq E$, on a
$$
k_0(\gamma_1,\mathfrak{A}(E_1))=
-r_1 =k_0(a_1,\mathfrak{A}_1)=k_E(a_1).
$$
En particulier on a $[F[\gamma_1]/F]=[E_1:F]=N_1$, l'élément $\gamma_1$ est quasi--régulier elliptique dans $A(E_1)$, et
$$
k_F(\gamma_1)=k_E(a_1).
$$ 

Il s'agit maintenant de remonter à $G$. Notons $\beta_1$ l'élément $\gamma_1\otimes 1$ de $A(E_1)\otimes_{\mathfrak{o}_{E_1}}\mathfrak{b}_1$, 
et $\bs{x}_1$ l'élément $\bs{x}_0\otimes \bs{y}_1$ de $\mathfrak{A}(E)\otimes_{\mathfrak{o}_E}\mathfrak{A}_1= \mathfrak{A}(E_1)$. \'Ecrivons
$$
\gamma = \beta + \bs{x}_0\otimes b = \beta+ \bs{x}_0\otimes (b_1+ \bs{y}_1\otimes c)= \beta_1+ \bs{x}_1\otimes c.
$$
Soit $\bs{s}:\mathfrak{g}\rightarrow \mathfrak{b}$ la corestriction modérée $\bs{s}_0\otimes {\rm id}_{\mathfrak{b}}$ sur 
$\mathfrak{g}=A(E)\otimes_{E}\mathfrak{b}$ relativement à $E/F$, et soit $\bs{s}_{b_1}: \mathfrak{b}\rightarrow \mathfrak{b}_1$ la corestriction modérée 
$\bs{t}_1\otimes {\rm id}_{\mathfrak{b}_1}$ sur $\mathfrak{b}= \mathfrak{a}_1\otimes_{E_1}\mathfrak{b}_1$ relativement à $E_1/E$. 
On a $\bs{s}_{b_1}\circ \bs{s}(\bs{x}_1\otimes c)=c$. 
D'autre part on a $k_0(\beta_1,\mathfrak{A})=k_F(\beta_1)e_1$, et comme $E_1[c]$ est un sous--corps maximal de $\mathfrak{b}_1$ tel que $E_1[c]^\times$ normalise $\mathfrak{B}_1$, on a $e_1=e(E_1[c]/E_1)$. Puisque
$$
s=-k_E(b)<r < -k_0 <n
$$
et
$$
r= r_1e_1= -k_E(a_1)e_1 = -k_F(\beta_1)e_1=-k_0(\beta_1,\mathfrak{A}),
$$
on peut appliquer la proposition 1 de \ref{raffinement}: la strate $[\mathfrak{A}, n, s-1, \gamma]$ dans $\mathfrak{g}$ est simple, $e(F[\gamma]/F)=e(E_1[c]/F)$ et $f(F[\gamma]/F)=f(E_1[c]/F)$, et puisque $E_1[c]\neq E_1$, on a
$$
k_0(\gamma,\mathfrak{A})=
-s =k_0(c,\mathfrak{B}_1)=k_{E_1}(c).
$$
Comme $e(E_1[c]/E)=e(E[b]:E]$ et $f(E_1[c]/E)=f(E[b]/E)$, on a $e(F[\gamma]/F)=e(E[b]]/F)$ et $f(F[\gamma]/F)=f(E[b]/F)$. Comme 
$k_{E_1}(c)=k_E(b)$, on a aussi
$$
k_0(\gamma,\mathfrak{A})=k_E(b).
$$
Cela achève la démonstration du point (i).

\vskip1mm 
Prouvons (ii). L'implication $\Leftarrow$ est une consquence du lemme. 
Prouvons l'implication $\Rightarrow$. Comme pour le point (i), on peut grce au lemme supposer que $b$ et $b'$ vrifient les inclusions 
$\underline{\mathfrak{B}}\subset \mathfrak{B}_b$ et $\underline{\mathfrak{B}}\subset \mathfrak{B}_{b'}$. 
Posons $\gamma = \beta + \bs{x}_0\otimes b$ et $\gamma' = \beta + \bs{x}_0\otimes b'$. On suppose que $\ES{O}_G(\gamma')= \ES{O}_G(\gamma)$. On suppose aussi dans un premier temps que 
$b\neq 0$ et $b'\neq 0$. Posons $E_0=E[b]$ et $\mathfrak{B}=\mathfrak{B}_b$. Soit $\mathfrak{A}=\mathfrak{A}(E)\otimes_{\mathfrak{o}_E}\mathfrak{B}$ l'unique $\mathfrak{o}$--ordre héréditaire dans $\mathfrak{g}$ normalisé par $E_0^\times$. Posons $\mathfrak{Q}={\rm rad}(\mathfrak{B})$ et $\mathfrak{P}={\rm rad}(\mathfrak{A})$. Posons $r=n_E(b)\;(=-\nu_{E_0}(b))$, 
$k_0=k_0(\beta, \mathfrak{A})$ et $n= -\nu_{\mathfrak{A}}(\beta)$. On a vu (cf. le début de la démonstration du point (i)) que $-r> k_0\geq -n$. Considérons la strate pure $\bs{S}_b=[\mathfrak{B},r,r-1,b]$ dans $\mathfrak{b}$. En rempla\c{c}ant $b$ par $b'$, on définit de la même manière 
$E'_0=E[b']$, $\mathfrak{B}'= \mathfrak{B}_{b'}$, $\mathfrak{A}'= \mathfrak{A}(E)\otimes_{\mathfrak{o}_E}\mathfrak{B}'$, $\mathfrak{Q}'$ et $\mathfrak{P}'$, les entiers $r'$, $k'_0$ et $n'$, et la strate pure $\bs{S}_{b'}=[\mathfrak{B}',r',r'-1,b']$ dans $\mathfrak{b}$. La strate pure $\bs{S}_\gamma=[\mathfrak{A},n,r-1,\gamma]$, resp. $\bs{S}_{\gamma'}=[\mathfrak{A}',n',r'-1,\gamma']$, dans $\mathfrak{g}$ est un raffinement de la strate simple $\bs{S}=[\mathfrak{A},n,r,\beta]$, resp. $\bs{S}'=[\mathfrak{A}, n',r',\beta]$, de strate dérivée $\bs{S}_b$, resp. $\bs{S}_{b'}$. Posons $\mathfrak{N}_{k_0}= \mathfrak{N}_{k_0}(\beta,\mathfrak{A})$ et $\mathfrak{N}'_{k'_0}= \mathfrak{N}_{k'_0}(\beta, \mathfrak{A}')$. D'après \cite[1.5.12]{BK}, le $G$--entrelacement formel
$$
\ES{I}_G(\bs{S},\bs{S}')=\{g\in G: g^{-1}(\beta + \mathfrak{P}^{-r})g\cap (\beta + \mathfrak{P}'^{-r'})\neq \emptyset\}
$$
co\"{\i}ncide avec
$$
(1+ \mathfrak{Q}^{-r-k_0}\mathfrak{N}_{k_0})H(1+ \mathfrak{Q}'^{-r'-k'_0}\mathfrak{N}'_{k'_0}).
$$
Soit $g\in G$ tel que $\gamma'= g^{-1}\gamma g$. Puisque $g\in \ES{I}_G(\bs{S},\bs{S}')$, on peut écrire $
g= (1+ y)h (1+y')$ avec $y\in \mathfrak{Q}^{-r-k_0}\mathfrak{N}_{k_0}$, $y'\in \mathfrak{Q}^{-r'-k'_0}\mathfrak{N}'_{k'_0}$ et $h\in H$. Posons $\bs{x}=\bs{x}_0\otimes 1$ et $\bs{s}=\bs{s}_0\otimes {\rm id}_{\mathfrak{b}}:\mathfrak{g}\rightarrow \mathfrak{b}$. L'égalité $\gamma g - g \gamma'=0$ s'écrit
$$
(\beta + \bs{x}b) (1+y)h(1+y') - (1+y)h(1+y')(\beta + \bs{x}b')=0,
$$
\cad
$$
(1+y)^{-1}(\beta + \bs{x}b)(1+y)h -h(1+y')(\beta+ \bs{x}b')(1+y')^{-1}=0.
$$
Posant $a= -r-k_0>0$ et $a'=-r'-k'_0>0$, on en déduit que
$$
({\rm ad}_\beta(y)+ \bs{x}b)h -h ({\rm ad}_\beta(-y')+ \bs{x}b')\equiv 0 \quad ({\rm mod}\; \mathfrak{P}^{-r +a }h +
h\mathfrak{P}'^{-r'+a})
$$
En appliquant $\bs{s}$ (rappelons que $\bs{\mathfrak{s}}:\mathfrak{g}\rightarrow \mathfrak{b}$ est un homomorphisme de $\mathfrak{b}\times \mathfrak{b}$--bimodule et que 
$\bs{s}(\mathfrak{P}^k)= \mathfrak{Q}^k$ et $\bs{s}(\mathfrak{P}')=\mathfrak{Q}'^k$ pour tout $k\in {\Bbb Z}$) , on obtient
$$
bh - hb' \equiv 0\quad ({\rm mod}\; \mathfrak{Q}^{-r+a}h+ h \mathfrak{Q}'^{-r'+a'}),
$$
\cad
$$
h^{-1}(b+ \mathfrak{Q}^{-r+a})h \cap (b' + \mathfrak{Q}'^{-r'+a'})\neq \emptyset.
$$
En particulier, $h$ appartient au $H$--entrelacement formel $\ES{I}_H(\bs{S}_b,\bs{S}_{b'})$. Posons $e=e(\mathfrak{B}\vert \mathfrak{o}_E)\;(=e(E_0/E))$ et $e'= e(\mathfrak{B}'\vert \mathfrak{o}_E)\;(=e(E'_0/E))$.

% remarque 2
\begin{marema2}
{\rm 
Soit $\phi_b\in \kappa_E[t]$, resp. $\phi_{b'}\in \kappa_E[t]$, le polynôme caractéristique de la strate $\bs{S}_b$, resp. $\bs{S}_{b'}$. Comme les strates pures $\bs{S}_b$ et $\bs{S}_{b'}$ dans $\mathfrak{b}$ sont équivalentes à des strates simples, d'après \cite[2.6.3]{BK}, on a ${r\over e}= {r'\over e'}$ et $\phi_b = \phi_{b'}$. Si la strate $\bs{S}_b$ est simple, \cad si $b$ est $E$--minimal, alors puisque $E_0=E[b]$ est un sous--corps maximal de $\mathfrak{b}$, le polynôme caractéristique $\phi_b\in \kappa_E[t]$ se factorise en $\phi_b=\phi_0^e$ pour un polynôme $\phi_0\in \kappa_E[t]$ irréductible sur $\kappa_E$. De l'égalité $\phi_{b'}= \phi_0^e$, on déduit que $e'$ divise $e$. Mais puisque $b$ est $E$--minimal, 
l'entier $r=n_E(b)$ est premier  l'indice de ramification $e$, par consquent l'galit 
$r= {e\over e'}r'$ n'est possible que si $e'=e$ et $r'=r$.\hfill $\blacksquare$ }
\end{marema2}

D'après la remarque 2, on a $e=e'$ si $b$ est $E$--minimal. En fait d'après le point (i), 
l'égalité $e=e'$ est toujours vérifiée, que $b$ soit $E$--minimal ou non, \cad que la strate $\bs{S}_b$ dans $\mathfrak{b}$ soit simple ou non: 
on a $e= {e(F[\gamma]/F)\over e(E/F)}$ et $e'= {e(F[\gamma']/F)\over e(E/F)}$, et comme par hypothèse $\gamma$ et $\gamma'$ sont conjugués dans $G$, on a $e(F[\gamma]/F)=e(F[\gamma']/F)$. Les $\mathfrak{o}_E$--ordres héréditaires principaux $\mathfrak{B}$ et $\mathfrak{B}'$ dans $\mathfrak{b}$ sont donc conjugués par un élément de $H$, et quitte  remplacer 
$b'$ par un lment $b''\in \ES{O}_H(b')$ tel que $\mathfrak{B}_{b''}= \mathfrak{B}$ --- ce qui, d'aprs le lemme, laisse inchange l'orbite $\ES{O}_G(\gamma')$ ---,  
on peut supposer que $\mathfrak{B}'=\mathfrak{B}$. Comme les strates pures $\bs{S}_b=[\mathfrak{B}, r,r-1,b]$ et $\bs{S}_{b'}=[\mathfrak{B},r,r-1,b']$ s'entrelacent dans $H$, d'après \cite[2.6.1]{BK}, il existe un élement $u\in U(\mathfrak{B})= \mathfrak{B}^\times$ tel que $b'\in u^{-1}bu + \mathfrak{Q}^{-r+1}$. Quitte à remplacer $b'$ par $ub'u^{-1}$, 
on peut donc supposer que les strates $\bs{S}_b$ et $\bs{S}_{b'}$ sont équivalentes. Si la strate $\bs{S}_b$ est simple, puisque $E_0=E[b]$ est un sous--corps maximal de $\mathfrak{b}$, la strate $\bs{S}_{b'}$ l'est aussi. 

Supposons que la strate $\bs{S}_b$ n'est pas simple, et montrons que l'on peut se ramener au cas où elle l'est. Soit une strate simple $[\mathfrak{B},r,r-1,b_1]$ dans $\mathfrak{b}$ équivalente à $\bs{S}_b$. L'élément $b_1$ est $E$--minimal, et l'on a $b,\,b'\in b_1 + \mathfrak{Q}^{-r+1}$. Posons $E_1=E[b_1]$, $\mathfrak{a}_1={\rm End}_E(E_1)$ et $\mathfrak{b}_1={\rm End}_{E_1}(V)$. Soit $\mathfrak{A}_1$ l'$\mathfrak{o}_E$--ordre héréditaire ${\rm End}_{\mathfrak{o}_E}^0(\{\mathfrak{p}_{E_1}^i\})$ dans $\mathfrak{a}_1$, et 
soit $\mathfrak{B}_1$ l'$\mathfrak{o}_{E_1}$--ordre héréditaire $\mathfrak{B}\cap \mathfrak{b}_1$ dans $\mathfrak{b}_1$. Posons $\mathfrak{Q}_1= {\rm rad}(\mathfrak{B}_1)$. Identifions $\mathfrak{b}$ à $\mathfrak{a}_1\otimes_{E_1}\mathfrak{b}_1$ via le choix d'une $(W_1,E_1)$--décomposition $\mathfrak{A}(E_1)\otimes_{\mathfrak{o}_{E_1}} \mathfrak{B}_1 \buildrel \simeq \over{\longrightarrow} \mathfrak{B}$ de $\mathfrak{B}$. On a les identifications 
$$
\mathfrak{g}= A(E)\otimes_A \mathfrak{a}_1\otimes_{E_1} \mathfrak{b}_1,\quad
\mathfrak{A}= \mathfrak{A}(E)\otimes_{\mathfrak{o}_E}\mathfrak{A}_1 \otimes_{\mathfrak{o}_{E_1}}\mathfrak{B}_1.
$$
Fixons une corestriction modérée $\bs{t}_1: \mathfrak{a}_1 \rightarrow E_1$ sur $\mathfrak{a}_1$ relativement à $E_1/E$, et un élément $\bs{y}_1\in \mathfrak{A}_1$ tel que $\bs{t}_1(\bs{y}_1)=1$. Puisque $k_0(b_1,\mathfrak{B})=-r$, d'après le lemme 1 de \ref{approximation}, il existe des éléments $v,\,v'\in 1 + \mathfrak{Q}_1\mathfrak{N}_{-r}(b_1, \mathfrak{B})$ et $c,\,c'\in \mathfrak{B}_1^{-r+1}$ tels que $v^{-1}bv = b_1 + \bs{y}_1\otimes c$ et $v'^{-1}b'v'= b_1+ \bs{y}_1\otimes c'$. Quitte à remplacer $b$ par $vbv^{-1}$ et $b'$ par $v'b'v'^{-1}$, on peut supposer que $b= b_1 + \bs{y}_1\otimes c$ et $b' = b_1+ \bs{y}_1 \otimes c'$. L'élément $c$, resp. $c'$, est quasi--régulier elliptique dans $\mathfrak{b}_1$. En effet, si ce n'est pas le cas, $c$ est contenu dans une sous--$E_1$--algèbre parabolique propre $\mathfrak{p}_1$ de $\mathfrak{b}_1$, et $b$ appartient à la sous--$E$--algèbre parabolique propre $\mathfrak{a}_1\otimes_{E_1}\mathfrak{p}_1$ de $\mathfrak{b}$, ce qui contredit le fait que $b$ est quasi--régulier elliptique dans $\mathfrak{b}$. \'Ecrivons $b_1=a_1\otimes 1$ avec $a_1\in \mathfrak{a}_1$. Identifions $A(E_1)$ à $A(E)\otimes_E \mathfrak{a}_1$ via le choix d'une $(X,E)$--décomposition
$\mathfrak{A}(E)\otimes_{\mathfrak{o}_E} \mathfrak{A}_1\buildrel\simeq \over{\longrightarrow}\mathfrak{A}(E_1)$
de $\mathfrak{A}(E_1)$. On a les identifications
$$
\mathfrak{g}= A(E_1)\otimes_{E_1}\mathfrak{b}_1,\quad \mathfrak{A}= \mathfrak{A}(E_1)\otimes_{\mathfrak{o}_{E_1}}\mathfrak{b}_1.
$$
Notons $\gamma_1$ l'élément $\beta + \bs{x}_0\otimes a_1$ de $A(E)\otimes_{E_1}\mathfrak{a}_1$. Posons $N_1= [E_1:F]$, $e_1=e(\mathfrak{B}_1\vert \mathfrak{o}_{E_1})$, $r_1 = {r\over e_1}$, $k_1={k_0\over e_1}$ et $n_1= {n\over e_1}$. On a $k_1= k_0(\beta, \mathfrak{A}(E_1))$, $n_1 = -\nu_{\mathfrak{A}(E_1)}(\beta)$, et puisque $\nu_{\mathfrak{B}}(b_1)= -r$, on a $\nu_{\mathfrak{A}_1}(a_1)=-r_1$. 
La strate $[\mathfrak{A}_1,r_1,r_1-1,a_1]$ dans $\mathfrak{a}_1$ est simple, et comme $r_1<-k_1\leq n_1$, on peut appliquer la proposition  1 de \ref{raffinement}: la strate $[\mathfrak{A}(E_1), n_1, r_1-1, \gamma_1]$ dans $A(E_1)$ est simple, $e(F[\gamma_1]/F)=e(E_1/F)$ et $f(F[\gamma_1]/F)=f(E_1/F)$, et
$$
k_0(\gamma_1,\mathfrak{A}(E_1))=\left\{\begin{array}{ll}
-r_1 =k_0(a_1,\mathfrak{A}_1)=k_E(a_1)& \mbox{si $E_1\neq E$}\\
k_1=k_0(\beta,\mathfrak{A}(E_1))& \mbox{sinon}
\end{array}\right..
$$
En particulier, on a $[F[\gamma_1]/F]=N_1$, et $\gamma_1$ est quasi--régulier elliptique dans $A(E_1)$. Notons $\beta_1$ l'élément $\gamma_1\otimes 1$ de $A(E_1)\otimes_{\mathfrak{o}_{E_1}}\mathfrak{b}_1$, 
et $\bs{x}_1$ l'élément $\bs{x}_0\otimes \bs{y}_1$ de $\mathfrak{A}(E)\otimes_{\mathfrak{o}_E}\mathfrak{A}_1= \mathfrak{A}(E_1)$. \'Ecrivons
$$
\gamma = \beta + \bs{x}_0\otimes b = \beta+ \bs{x}_0\otimes (b_1+ \bs{y}_1\otimes c)= \beta_1+ \bs{x}_1\otimes c
$$
et (de la même manière)
$$
\gamma'=\beta_1 + \bs{x}_1\otimes c'.
$$
Soit $\bs{s}_1:A(E_1)\rightarrow E_1$ la corestriction modérée sur $A(E_1)= A(E)\otimes_E\mathfrak{a}_1$ relativement à $E_1/F$ donnée par $\bs{s}_1 = \bs{s}_0\otimes \bs{t}_1$. On a donc $\bs{s}_1(\bs{x}_1)= 1$. On distingue deux cas: 
\begin{itemize}
\item {\it Premier cas}: $E_1=E$. En ce cas $F[\beta_1]=F[\beta]$ et $k_0(\beta_1,\mathfrak{A})=k_0(\beta,\mathfrak{A})$, et le passage de l'écriture $\gamma = \beta + \bs{x}_0\otimes b$, resp. $\gamma=\beta+ \bs{x}_0\otimes b'$, à l'écriture $\gamma = \beta_1 + \bs{x}_1\otimes c$, resp. $\gamma'= \beta_1+ \bs{x}_1\otimes  c'$, a pour effet de faire croître la valuation de $b$, resp. $b'$. On a en effet $n_E(b)=n_E(b')= r$, et par construction on obtient $n_E(c)=n_E(c') <r$.
\item {\it Deuxième cas}: $E_1\neq E$. En ce cas le passage de l'écriture $\gamma = \beta + \bs{x}_0\otimes b$, resp. $\gamma=\beta+ \bs{x}_0\otimes b'$, à l'écriture $\gamma = \beta_1 + \bs{x}_1\otimes c$, resp. $\gamma'= \beta_1+ \bs{x}_1\otimes c'$, a pour effet de faire croître le degré de l'extension $F[\beta]/F$. On a donc $[F[\beta_1,c]:F[\beta_1]]< [F[\beta,b]: F[\beta]]$. 
\end{itemize}
Si l'élément $c$ est $F[\beta_1]$--minimal, alors l'élément $c'$ l'est aussi (même argument que plus haut pour $b$ et $b'$) et on s'arrête là: pour montrer que $b$ et $b'$ sont conjugués dans $H={\rm Aut}_E(V)$, il suffit de montrer que $c$ et $c'$ le sont dans $H_1={\rm Aut}_{E_1}(V)$. Sinon, on refait la même opération avec le couple $(c,c')$. Le processus s'arrête au bout d'un nombre fini de fois. En effet, le second cas ne peut se produire qu'un nombre fini de fois par un argument de dimension. Quant au premier cas, supposons par l'absurde qu'il se produise un nombre infini de fois. Compte--tenu du fait que le second cas ne peut se produire qu'un nombre fini de fois, cela implique qu'il existe un élément pur $\delta\in \mathfrak{g}$ avec  
$[F[\delta]:F]<N$ --- ce $\delta$ est le $\beta_1$ du processus obtenu lors de la dernire occurrence du cas 2 --- tel que pour tout entier $k$, l'intersection $\ES{O}_G(\gamma)\cap 
(\delta U_{F[\delta]} + \mathfrak{P}^k)$ n'est pas vide. Cela signifie que $\delta U_{F[\delta]}$ rencontre la fermeture $\overline{\ES{O}_G(\gamma)}$ de l'orbite $\ES{O}_G(\gamma)$ dans $G$. Or cette dernière est fermée, et comme tout élément de $\ES{O}_G(\gamma)$ est quasi--régulier (elliptique), on a forcément $\ES{O}_G(\gamma)\cap F[\delta]^\times= \emptyset$; contradiction. En définitive, on s'est ramené au cas où $b$ est $E$--minimal, et donc à celui où les strates $\bs{S}_b$ et $\bs{S}_{b'}$ dans $\mathfrak{b}$ sont simples et équivalentes, ce que l'on suppose désormais. 

Prouvons que $b$ et $b'$ sont conjugués dans $H$. Pour cela, montrons que pour chaque entier $j\geq 1$, il existe un élément $u_j\in U(\mathfrak{B})$ tel que $u_j b' u_j^{-1} - b \in \mathfrak{Q}^{-r+j}$. Le cas $j=1$ ayant déjà été traité (on peut prendre $u_1=1$ puisqu'on a déjà conjugué $b'$ dans $U(\mathfrak{B})$ de manière à ce que les strates $\bs{S}_b$ et $\bs{S}_{b'}$ dans $\mathfrak{b}$ soient équivalentes), on procède par récurrence sur $j$. Fixons un entier $j\geq 1$ et supposons qu'un tel $u_j$ existe. On a déjà posé $E_0=E[b]$. Posons $K_0= F[\gamma]$. La strate $[\mathfrak{A}, n,r-1,\gamma]$ dans $\mathfrak{g}$ est simple, $K_0$ est un sous--corps maximal de $\mathfrak{g}$, et $k_F(\gamma)$ est donné par le point (i). Soit $\bs{s}_b:\mathfrak{b}\rightarrow E_0$ une corestriction modérée sur $\mathfrak{b}$ relativement à $E_0/E$. D'après le corollaire de \ref{raffinement}, il existe une corestriction modérée $s_\gamma:\mathfrak{g}\rightarrow K_0$ sur $\mathfrak{g}$ relativement à $K_0/F$ telle que pour tout $k\in {\Bbb Z}$ et tout $y\in \mathfrak{P}^k$, on a
$$
\bs{s}_\gamma(y)\equiv \bs{s}_b\circ  \bs{s}(y)\quad ({\rm mod}\; \mathfrak{P}^{k+1}).
$$
\'Ecrivons $u_j b' u_j^{-1} = b +c$ avec $u_j\in U(\mathfrak{B})$ et $c\in \mathfrak{Q}^{-r+j}$. Puisque $\ES{O}_G(\gamma')=\ES{O}_G(\gamma)$, d'après le lemme, il existe un $g\in G$ tel que $g^{-1}\gamma g= \beta + \bs{x}_0\otimes (b+c)= \gamma + \bs{x}c$. Posons $t= \nu_{\mathfrak{A}}(g)$. Comme
$$
{\rm ad}_\gamma(g)\equiv 0\quad ({\rm mod}\;\mathfrak{P}^{t-r+j}),
$$
d'après \cite[2.1.1]{BK}, l'élément $g$ appartient à $\mathfrak{p}_{K_0}^t\mathfrak{N}_{-r+j}(\gamma,\mathfrak{A})$. D'autre part, on a $k_F(\gamma)\leq -r$ (avec égalité si et seulement si $E_0\neq E$) d'où, d'après \ref{une submersion}.(1),
$$
\mathfrak{N}_{-r-j}(\gamma,\mathfrak{A})= \mathfrak{o}_{K_0}+ \mathfrak{p}_{K_0}^{-r+j-k_F(\gamma)}\mathfrak{N}_{k_F(\gamma)}(\gamma,\mathfrak{A})\subset \mathfrak{o}_{K_0}+ \mathfrak{P}^j.
$$ 
\'Ecrivons $g= \alpha + y$ avec $\alpha\in \mathfrak{p}_{K_0}^t$ et $y\in \mathfrak{P}^{t+j}$. Puisque $\nu_{\mathfrak{A}}(g)=t$, l'élément $\alpha$ appartient à $\mathfrak{p}_{K_0}^t\smallsetminus \mathfrak{p}_{K_0}^{t+1}$, et on a
$$
0= \gamma(\alpha+y)- (\alpha+y)(\gamma + \bs{x}c)\equiv {\rm ad}_\gamma(y) - \alpha\bs{x}c\quad({\rm mod}\; \mathfrak{P}^{t-r+2j}).
$$
En appliquant $\bs{s}_\gamma$, on obtient que $\bs{s}_\gamma(\alpha \bs{x}c)=\alpha \bs{s}_\gamma(\bs{x}c)$ appartient à 
$\mathfrak{p}_{K_0}^{t-r+2j}$, et donc que $\bs{s}_\gamma(\bs{x}c)$ appartient à $\mathfrak{p}_{K_0}^{-r+2j}\subset \mathfrak{p}_{K_0}^{-r+j+1}$. On en déduit que
$$
\bs{s}_b\circ \bs{s}(\bs{x}c)\equiv 0\quad ({\rm mod}\; \mathfrak{P}^{-r+j+1}),
$$
et donc que $\bs{s}_b\circ \bs{s}(\bs{x}c)=\bs{s}_b(c)$ appartient à $\mathfrak{P}^{-r+j+1}\cap \mathfrak{b}= \mathfrak{Q}^{-r+j+1}$. 
Puisque $k_E(b)\leq -r$ (avec égalité si $E_0\neq E$), d'après \cite[1.4.10]{BK}, il existe un $a\in \mathfrak{P}_{E_0}^j\mathfrak{N}_{-r}(b,\mathfrak{B})\;(\subset \mathfrak{Q}^j)$ tel que
$$
c\equiv {\rm ad}_b(a)\quad ({\rm mod}\; \mathfrak{Q}^{-r+j+1}).
$$
Alors on a
$$
(1+a)^{-1}b(1+a)\equiv b+c\quad ({\rm mod}\;\mathfrak{Q}^{-r+j+1}),
$$
et puisque $(1+a)\in U(\mathfrak{B})$, en posant $u_{j+1}=(1+a)u_j$, on obtient que $u_{j+1} b' u_{j+1}^{-1}- b$ appartient à $\mathfrak{Q}^{-r+j+1}$. L'hypothèse de récurrence est donc vraie au cran $j+1$. 
Pour tout entier $j\geq 1$, on a donc
$$
\ES{O}_H(b') \cap (b + \mathfrak{Q}^{-r+j})\neq \emptyset.
$$
Cela implique que $b$ appartient  la fermeture $\overline{\ES{O}_H(b')}$ de l'orbite $\ES{O}_H(b')$ dans $\mathfrak{b}$. 
Puisque cette dernire est fermée dans $\mathfrak{b}$, on a l'égalité $\ES{O}_H(b')=\ES{O}_H(b)$. 

On a prouvé l'implication $\Rightarrow $ du point (ii) dans le cas où $b\neq 0$ et $b'\neq 0$. Si $b=b'=0$, il n'y a rien à démontrer. Reste à prouver que si $b\neq 0$, alors $b'\neq 0$. Supposons par l'absurde que $b\neq 0$ et $b'=0$. Puisque $b'=0\in \mathfrak{b}_{\rm qre}$, $E$ est un sous--corps maximal de $\mathfrak{g}$, et $E[b]=E=\mathfrak{b}$. Posons $r=-\nu_E(b)$. Soit $\mathfrak{A}=\mathfrak{A}_\beta$ l'unique $\mathfrak{o}$--ordre héréditaire dans $\mathfrak{g}$ normalisé par $E^\times$. Puisque $k_F(\gamma)=k_F(\beta)<-r$ et $\gamma'=\beta$, les strates $\bs{S}=[\mathfrak{A}, n,r, \gamma]$ et $\bs{S'}=[\mathfrak{A},n,r,\gamma']$ dans $\mathfrak{g}$ sont simples et équivalentes. Comme elles s'entrelacent dans $G$, on montre comme plus haut qu'il existe un élément $h\in H= E^\times$ tel que
$$
bh - hb' \equiv 0 \quad ({\rm mod}\; \mathfrak{p}_E^{-r + a}h + h\mathfrak{p}_E^{-r+a})
$$
avec $a = -r-k_F(\beta)>0$. Mais cela signifie que $b\in \mathfrak{p}_E^{-r+a}$, ce qui est impossible puisque $\nu_E(b)=-r$. 

Cela achève la démonstration de la proposition.
\end{proof}

% remarque 3
\begin{marema3}
{\rm 
Le lemme a t utilis plusieurs fois dans la preuve de la proposition: pour l'implication $\Leftarrow $ du point (ii) bien sr, mais 
aussi pour passer des lments de $\mathfrak{b}_{\rm qre} \cap \underline{\mathfrak{Q}}^{\underline{k}_0+1}$ aux lments 
$b\in \mathfrak{b}_{\rm qre}^{k_F(\beta)+ {1\over d}}$ tels que $\underline{\mathfrak{B}}\subset \mathfrak{B}_b$. Il a aussi une conséquence que nous utiliserons plus loin:
\begin{enumerate}[leftmargin=17pt]
\item[(3)]l'ensemble ${^G(\beta + \bs{x}_0\otimes \underline{\mathfrak{Q}}^{\underline{k}_0+1})}$ est ouvert {\it fermé} et $G$--invariant dans $G$.
\end{enumerate}
En effet, posons $\mathfrak{X}= {^G(\beta + \bs{x}_0\otimes \underline{\mathfrak{Q}}^{\underline{k}_0+1})}$. D'aprs la proposition de \ref{une submersion}, c'est 
un ouvert (clairement $G$--invariant) de $G$. 
Soit $\gamma = \beta + \bs{x}_0\otimes b$ pour un lment $b\in \underline{\mathfrak{Q}}^{\underline{k}_0+1}$. Puisque 
d'après \ref{éléments qre}.(7), l'ensemble ${^H(\underline{\mathfrak{Q}}^{\underline{k}_0+1})}$ est ouvert fermé et $H$--invariant dans $\mathfrak{b}$, on peut choisir un élément fermé (dans $\mathfrak{b}$) $b'$ qui appartient  $\underline{\mathfrak{Q}}^{\underline{k}_0+1}\cap \overline{\ES{O}_H(b)}$. D'après le lemme, l'élément $\gamma'= \beta + \bs{x}_0 \otimes b'$ appartient à $\overline{\ES{O}_G(\gamma)}$. 
Montrons que $\gamma'$ est ferm (dans $\mathfrak{g}$). Si $b'\in \mathfrak{b}_{\rm qre}$, c'est vrai puisque d'aprs la proposition, $\gamma'\in \mathfrak{g}_{\rm qre}$. Sinon, 
quitte  conjuguer $b'$ dans $H$, on peut supposer qu'il existe une dcomposition $V= V_1\times \cdots \times V_s$ o $V_i$ est un sous--$E$--espace vectoriel de $V$ 
de dimension $d_i$, telle que, en posant $\mathfrak{b}_i= {\rm End}_E(V_i)$ et $\mathfrak{m}_*= \mathfrak{b}_1 \times \cdots \times \mathfrak{b}_s$, 
$\mathfrak{g}_i= {\rm End}_F(V_i)$ et $\mathfrak{m}= \mathfrak{g}_1 \times \cdots \times \mathfrak{g}_s$, on a:
\begin{itemize}
\item $\underline{\mathfrak{B}}\cap \mathfrak{m}_* = \underline{\mathfrak{B}}_1\times \cdots \times \underline{\mathfrak{B}}_s$ o $\underline{\mathfrak{B}}_i$ est un $\mathfrak{o}_E$--ordre hrditaire 
minimal dans $\mathfrak{b}_i$;
\item $\underline{\mathfrak{A}}\cap \mathfrak{m}= \underline{\mathfrak{A}}_1\times \cdots \times \underline{\mathfrak{A}}_s$ o $\underline{\mathfrak{A}}_i$ est l'unique $\mathfrak{o}$--ordre hrditaire dans 
$\mathfrak{g}_i$ normalis par $E^\times$ tel que $\underline{\mathfrak{A}}_i\cap \mathfrak{b}_i= \underline{\mathfrak{B}}_i$;
\item la $(\underline{W},E)$--décomposition $\mathfrak{A}(E)\otimes_{\mathfrak{o}_E}\underline{\mathfrak{B}}\buildrel \simeq \over{\longrightarrow}\underline{\mathfrak{A}}$ de 
$\underline{\mathfrak{A}}$ se restreint en un isomorphisme $\mathfrak{A}(E)\otimes_{\mathfrak{o}_E}(\underline{\mathfrak{B}}\cap \mathfrak{m}_*)\buildrel \simeq 
\over{\longrightarrow}\underline{\mathfrak{A}}\cap \mathfrak{m}$;
\item l'lment $b'$ appartient  $\mathfrak{m}_*$, et pour chaque $i$, 
la composante $b_i$ de $b$ sur $\mathfrak{b}_i$ appartient  $(\mathfrak{b}_i)_{\rm qre}\cap \underline{\mathfrak{Q}}_i^{\underline{k}_{i,0}+1}$ o $\underline{\mathfrak{Q}}_i$ est le radical 
de Jacobson de $\underline{\mathfrak{B}}_i$ et $\underline{k}_{i,0}= k_F(\beta)d_i \;(= k_0(\beta, \underline{\mathfrak{A}}_i))$.
\end{itemize}
Pour des dtails sur ces dcompositions, voir plus loin (\ref{descente centrale au voisinage d'un élément pur (suite)}). Ainsi l'lment $\gamma'$ appartient 
 $\mathfrak{m}$, et pour chaque $i$, la composante $\gamma'_i= \beta + \bs{x}_0\otimes b'_i$ de $\gamma'$ sur $\mathfrak{g}_i$ appartient  $(\mathfrak{g}_i)_{\rm qre}$. 
Donc $\gamma'$ est ferm dans $\mathfrak{m}$, et aussi dans $\mathfrak{g}$. On a prouv que pour tout $\gamma\in \beta + \bs{x}_0 \otimes \underline{\mathfrak{Q}}^{\underline{k}_0+1}$, il existe 
un lment ferm (dans $\mathfrak{g}$) $\gamma' $ qui appartient  $(\beta + \bs{x}_0 \otimes \underline{\mathfrak{Q}}^{\underline{k}_0+1}) \cap \overline{\ES{O}_G(\gamma)}$. D'après la remarque 2 de \ref{parties compactes modulo conjugaison}, cela prouve (3). \hfill$\blacksquare$
}
\end{marema3}

%%%%%%%%%%%%%%%%%%%
\subsection{Une conséquence du résultat principal}\label{une conséquence}Soit un élément $\gamma\in G_{\rm qre}$, et soit $(\gamma_0=\gamma,\gamma_1,\ldots ,\gamma_m)$ une suite d'approximation minimale de $\gamma$, de suite des correcteurs 
$(\bs{x}_1,\ldots ,\bs{x}_m)$ et de suite dérivée $(b_0,\ldots ,b_{m-1})$ --- cf. la définition de \ref{approximation}. On note $(F_0,\ldots ,F_m)$ la suite d'extensions de $F$ définie par $F_i=F[\gamma_i]$, et pour $i=0,\ldots ,m$, on note $n_i$, $r_i$, $e_i$, $f_i$ les entiers définis comme dans la remarque 1 de \ref{approximation}. On pose $\mathfrak{g}_i={\rm End}_F(F_i)$ et $\mathfrak{A}_i= 
\mathfrak{A}(F_i)$, et (si $i<m)$ $\mathfrak{b}_i={\rm End}_{F_{i+1}}(F_i)$ et $\mathfrak{B}_i= \mathfrak{A}_i\cap \mathfrak{b}_i$. On a identifié $\mathfrak{g}$ à $\mathfrak{g}_0$ via le choix d'un vecteur non nul $v\in V$, et pour $i=0,\ldots ,m-1$, on a identifié $\mathfrak{g}_i$ à $\mathfrak{g}_{i+1}\otimes_{F_{i+1}}\!\mathfrak{b}_i$ via le choix d'une $(W_i,F_{i+1})$--décomposition 
$\mathfrak{A}_{i+1}\otimes_{\mathfrak{o}_{F_{i+1}}}\!\!\mathfrak{B}_i \buildrel\simeq\over{\longrightarrow} \mathfrak{A}_i$ de $\mathfrak{A}_i$. Avec ces identifications, on a l'égalité 
$\bs{x}_{i+1}\otimes b_i = \gamma_i -\gamma_{i+1}$. On a aussi l'égalité (comme $F_{i+1}$--espaces vectoriels) $F_{i+1} \otimes_F W_i =F_i$.

Soit aussi un autre élément $\gamma'\in G_{\rm qre}$, et soit $(\gamma'_0=\gamma, \gamma'_1,\ldots ,\gamma'_{m'})$ une suite d'approximation minimale de $\gamma'$, de suite des correcteurs $(\bs{x}'_1,\ldots ,\bs{x}'_m)$ et de suite dérivée $(b'_0,\ldots ,b'_{m-1})$. Elle définit  comme ci--dessus une suite d'extensions $(F'_0,\ldots ,F'_{m'})$, des entiers $n'_i$, $r'_i$, $e'_i$, $f'_i$, des algèbres $\mathfrak{g}'_i$ et $\mathfrak{A}'_i$, et (si $i<m$) $\mathfrak{b}'_i$ et $\mathfrak{B}'_i$. On a identifié $\mathfrak{g}$ à $\mathfrak{g}'_0$ via le choix d'un vecteur non nul $v'\in V$, et pour $i=0,\ldots ,m-1$, on a identifié $\mathfrak{g}'_i$ à $\mathfrak{g}'_{i+1}\otimes_{F'_{i+1}}\!\mathfrak{b}'_i$ via le choix d'une $(W'_i,F'_{i+1})$--décomposition 
$\mathfrak{A}'_{i+1}\otimes_{\mathfrak{o}_{F'_{i+1}}}\!\!\mathfrak{B}'_i \buildrel\simeq\over{\longrightarrow} \mathfrak{A}'_i$ de $\mathfrak{A}'_i$.

% définitions
\begin{madefi}
{\rm Les suites $(\gamma'_0,\ldots ,\gamma'_{m'})$ et $(\gamma_0,\ldots ,\gamma_m)$ sont dites {\it ($G$--)équivalentes} si les conditions suivantes sont vérifiées:
\begin{itemize}
\item $m'=m$;
\item pour $i=0,\ldots ,m$, on a $n'_i=n_i$ et $r'_i=r_i$;
\item il existe une suite d'isomorphismes de $F$--espaces vectoriels $\iota_i: F_i \buildrel \simeq \over{\longrightarrow} F'_i$ ($i=0,\ldots ,m$) qui sont {\it compatibles} au sens où (pour $i<m$), 
en identifiant $F_{i+1}$ au sous--$F$--espace vectoriel $F_{i+1}\otimes 1$ de $F_{i+1}\otimes_F W_i=F_i$ et $F'_{i+1}$ au sous--$F$--espace vectoriel $F'_{i+1}\otimes 1$ de $F'_{i+1}\otimes_{F} W'_i=F'_i$, on a $\iota_{i+1}=\iota_i\vert_{F_{i+1}}$, et tels que pour $i>0$, $\iota_i$ est un {\it isomorphisme de $F$--extensions}. On note $\bs{\alpha}_i: \mathfrak{g}_i\rightarrow \mathfrak{g}'_i$ l'isomorphisme de $F$--algèbres donné  par $\bs{\alpha}_i(g)= \iota_i \circ g \circ \iota_i^{-1}$. Pour $i=0,\ldots ,m-1$, $\bs{\alpha}_i$ induit par restriction un isomorphisme de $F_{i+1}$--algèbres $\bs{\beta}_i: \mathfrak{b}_i\rightarrow \mathfrak{b}'_i$ (pour la structure de $F_{i+1}$--algèbre sur $\mathfrak{b}'_i$ déduite de l'isomorphisme de $F$--extensions $\iota_{i+1}=F_{i+1}\buildrel\simeq\over{\longrightarrow} F'_{i+1}$). Par construction, les $\bs{\alpha}_i$ sont compatibles aux identifications $\mathfrak{g}_i= \mathfrak{g}_{i+1}\otimes_{F_{i+1}}\mathfrak{b}_i$ et $\mathfrak{g}'_i =\mathfrak{g}'_{i+1}\otimes_{F'_{i+1}}\mathfrak{b}'_i$, au sens où pour $i=0,\ldots ,m-1$, on a $\bs{\alpha}_i = \bs{\alpha}_{i+1}\otimes \bs{\beta}_i$;
\item Pour $i=0,\ldots ,m-1$, on a $\bs{\alpha}_{i+1}(\bs{x}_{i+1})= \bs{x}'_{i+1}$ et $\bs{\beta}_i(b_i)\in \ES{O}_{H'_i}(b'_i)$ avec $H'_i=(\mathfrak{b}'_i)^\times$;
\item $\bs{\alpha}_m(\gamma_m)\in \ES{O}_{G'_m}(\gamma'_m)$ avec $G'_m=(\mathfrak{g}'_m)^\times$. 
\end{itemize}
}
\end{madefi}

% remarque 1
\begin{marema1}
{\rm 
Si les suites d'approximation minimale $(\gamma'_0,\ldots ,\gamma'_{m'})$ et $(\gamma_0,\ldots ,\gamma_m)$ sont équivalentes, alors pour $i=0,\ldots ,m'=m$, on a $e'_i=e_i$ et $f'_i=f_i$. En effet pour $i>0$, c'est une conséquence de l'existence de l'ismorphisme de $F$--extensions $\iota_i: F_i\buildrel\simeq\over{\longrightarrow} F'_i$. Pour $i=0$, on a $\gamma = \gamma_1 + \bs{x}_1\otimes b_0$ et $\gamma'= \gamma'_1 + \bs{x}'_1\otimes b'_0$, $e(F_0/F) = e(F_1[b_0]/F)$ et $e(F'_0/F)= e(F'_1[b'_0]/F)$, $f(F_0/F) = f(F_1[b_0]/F)$ et $f(F'_0/F)= f(F'_1[b'_0]/F)$. Or, puisque $\bs{\beta}_0(b_0)\in \ES{O}_{H'_0}(b'_0)$, on a
$$
e(F_1[b_0]/F)= e(F_1[b_0]/F_1)e_1 = e(F'_1[b'_0]/F'_1)e'_1= e(F'_1[b'_0]/F).
$$
De la même manière, on obtient $f(F_1[b_0]/F)= f(F'_1[b'_0]/F)$. \hfill $\blacksquare$
}
\end{marema1}

% remarque 2
\begin{marema2}
{\rm 
Pour $i=0,\ldots ,m-1$, soit $\tilde{\bs{s}}_i:\mathfrak{g}_i\rightarrow \mathfrak{b}_i$ la corestriction modérée sur $\mathfrak{g}_i$ relativement à $F_{i+1}/F$ normalisée par $\tilde{\bs{s}}_i(\gamma_i-\gamma_{i+1})=b_i$ --- cf. la remarque 1 de \ref{approximation}. De même, pour $i=0,\ldots ,m'-1$, soit $\tilde{\bs{s}}'_i:\mathfrak{g}'_i\rightarrow \mathfrak{b}'_i$ la corestriction modérée sur $\mathfrak{g}'_i$ relativement à $F'_{i+1}/F$ normalisée par $\tilde{\bs{s}}'_i(\gamma'_i-\gamma'_{i+1})=b'_i$. Si les suites $(\gamma_0,\ldots ,\gamma_m)$ et $(\gamma'_0,\ldots ,\gamma'_{m'})$ sont équivalentes, alors pour $i=0,\ldots ,m-1$, l'application $\tilde{\bs{t}}'_i= \bs{\beta}_i \circ \tilde{\bs{s}}_i \circ \bs{\alpha}_i^{-1}: \mathfrak{g}'_i\rightarrow \mathfrak{b}'_i$ est une corestriction modérée sur $\mathfrak{g}'_i$ relativement à $F'_{i+1}/F$. On a donc $\tilde{\bs{t}}'_i= u'_{i+1} \tilde{\bs{s}}'_i$ pour un élément $u'_{i+1}\in \mathfrak{o}_{\smash{F'_{i+1}}}^\times$. Puisque $\gamma'_i-\gamma'_{i+1} = \bs{x}'_{i+1}\otimes b'_i$ avec $\bs{x}'_{i+1} = \bs{\alpha}_{i+1}(\bs{x}_{i+1}) $ et 
$b'_i = h'^{-1}_i \bs{\beta}_i(b_i) h'_i$ pour un élément $h'_i\in H'_i$, posant $h_i= \bs{\beta}_i^{-1}(h'_i)\in H_i = \mathfrak{b}_i^\times$, on obtient
$$
\tilde{\bs{t}}'_i(\gamma'_i-\gamma'_{i+1})= \bs{\beta}_i \circ \tilde{\bs{s}}_i(\bs{x}_{i+1}\otimes h_i^{-1}b_ih_i))= \bs{\beta}_i(h_i^{-1}b_i h_i)= b'_i = \tilde{\bs{s}}'_i(\gamma'_i-\gamma'_{i+1}).
$$
Par conséquent $u'_{i+1}= 1$ et $\tilde{\bs{t}}'_i = \tilde{\bs{s}}'_i$. En d'autres termes, les corestrictions modérées $\tilde{\bs{s}}_i$ et $\tilde{\bs{s}}'_i$ sont compatibles aux isomorphismes $\bs{\alpha}_i$ et $\bs{\beta}_i$: on a l'égalité $\tilde{\bs{s}}'_i \circ \bs{\alpha}_i= \bs{\beta}_i\circ \tilde{\bs{s}}_i:\mathfrak{g}_i\rightarrow \mathfrak{b}'_i$.\hfill $\blacksquare$
}
\end{marema2}

% proposition
\begin{mapropo}
On a $\ES{O}_G(\gamma')= \ES{O}_G(\gamma)$ si et seulement s'il existe des suites d'approximation minimale 
$(\gamma'_0,\ldots ,\gamma'_{m'})$ de $\gamma'$ et $(\gamma_0,\ldots ,\gamma_m)$ de $\gamma$ qui sont équivalentes.
\end{mapropo}

\begin{proof}Si $\ES{O}_G(\gamma')=\ES{O}_G(\gamma)$, on écrit $\gamma'= g^{-1}\gamma g$ pour un $g\in G$. Alors toute 
suite d'approximation minimale $(\gamma_0,\ldots ,\gamma_m)$ de $\gamma$ définit, par transport de structure via 
l'automorphisme ${\rm Int}_{g^{-1}}$ de $G$, une suite d'approximation minimale de $\gamma'$ qui lui est équivalente. 

En sens inverse, on raisonne par récurrence sur la longueur $m$ des suites d'approximation minimale. Pour $m=0$, l'assertion est claire: si 
les suites d'approximation minimale $(\gamma'=\gamma'_0)$ et $(\gamma=\gamma_0)$ sont équivalentes, alors (par définition) $\gamma'$ et $\gamma$ sont conjugués dans $G$. 

Supposons $m\geq 1$ et l'assertion que l'on veut prouver vraie pour les suites d'approximation minimale de longueur $\leq m-1$. Supposons aussi que les suites d'approximation minimale $(\gamma'_0,\ldots ,\gamma'_{m'})$ de $\gamma'$ et $(\gamma_0,\ldots ,\gamma_m)$ de $\gamma$ sont équivalentes. On a donc $m'=m$. Posons $F_0= F[\gamma]$, $F'_0= F[\gamma']$, $F_1= F[\gamma_1]$ et $F'_1= F[\gamma'_1]$. Par hypothèse, on a un isomorphisme de $F$--espaces vectoriels $\iota_0: F_0 \buildrel \simeq\over {\longrightarrow} F'_0$ induisant par restriction un isomorphisme de $F$--extensions $\iota_1: F_1\buildrel\simeq \over{\longrightarrow} F'_1$ (pour les identifications $F_1= F_1\otimes 1\subset F_0$ et $F'_1= F'_1\otimes 1 \subset F'_0$). Pour $i=0,\,1$, 
posons $\mathfrak{g}_i = A(F_i)$ et $\mathfrak{g}'_i = A(F'_i)$, et notons $\bs{\alpha}_i: \mathfrak{g}_i\rightarrow \mathfrak{g}'_i$ l'isomorphisme de $F$-algèbres induit par $\iota_i$. Posons aussi $\mathfrak{b}_0= {\rm End}_{F_1}(F_0)$, et notons $\bs{\beta}_0: \mathfrak{b}_0\rightarrow \mathfrak{b}'_0$ l'isomorphisme de $F_1$--algèbres induit par $(\iota_0,\iota_1)$. 
On a des identifications $\mathfrak{g}_0= \mathfrak{g}_1\otimes_{F_1}\mathfrak{b}_0$ et $\mathfrak{g}'_0= \mathfrak{g}'_1\otimes_{F'_1}\mathfrak{b}'_0$. \'Ecrivons $\gamma= \gamma_1 + \bs{x}_1\otimes b_0$ et $\gamma'= \gamma'_1 + \bs{x}'_1 \otimes b'_0$. Par hypothèse, on a $\bs{\alpha}_0= \bs{\alpha}_1\otimes \bs{\beta}_0$, $\bs{\alpha}_1(\bs{x}_1)= \bs{x}'_1$ et $\bs{\beta}_0(b_0)\in \ES{O}_{H'_0}(b'_0)$ avec $H'_0= (\mathfrak{b}'_0)^\times$. Identifions $F'_0$ à $F_0$ via $\iota_0$. Cela revient aussi à identifier $F'_1$ à $F_1$, $\mathfrak{g}'_1$ à $\mathfrak{g}_1$, $\bs{x}'_1$ à $\bs{x}_1$, et 
$\mathfrak{b}'_0$ à $\mathfrak{b}_0$. Posons $G_1= {\rm Aut}_F(F_1)$. Comme les suites d'approximation minimale $(\gamma_1,\ldots ,\gamma_m)$ de $\gamma_1$ et $(\gamma'_1,\ldots ,\gamma'_m)$ de $\gamma'_1$ sont $G_1$--équivalentes et de longueur $m-1$, d'après l'hypothèse de récurrence, il existe un élément $g_1\in G_1$ tel que $\gamma'_1 = g_1^{-1}\gamma_1 g_1$. Quitte à remplacer 
$\gamma'_1$ par $g_1\gamma'_1 g_1^{-1}$, ce qui revient à remplacer $\gamma'$ par $g\gamma'g^{-1} $ avec $g=g_1\otimes 1\in G$, et aussi la suite d'approximation minimale $(\gamma'_0=\gamma', \ldots ,\gamma'_m)$ de $\gamma'$ par celle s'en déduisant par transport de structure via l'automorphisme ${\rm Int}_g$ de $G$, on peut supposer que $\gamma'_1=\gamma_1$. On peut alors appliquer le point (ii) de la proposition de \ref{le résultat principal}: les éléments $\gamma= \gamma_1+ \bs{x}_1\otimes b_0$ et $\gamma'=\gamma_1+ \bs{x}_1\otimes b'_0$ sont conjugués dans $G$. \end{proof}

%%%%%%%%%%%%%%%
\subsection{Le principe de submersion}\label{le principe de submersion}Reprenons les hypothèses et les notations de \ref{le résultat principal}. En particulier $E=F[\beta]$ pour un élément pur $\beta\in \mathfrak{g}$, et $H={\rm Aut}_E(V)$. On a les identifications
$$
\mathfrak{g}= A(E)\otimes_E\mathfrak{b}, \quad \underline{\mathfrak{P}}^k= \mathfrak{A}(E)\otimes_{\mathfrak{o}_E}\underline{\mathfrak{Q}}^k\quad (k\in {\Bbb Z}).
$$
On a fixé une corestriction modérée $\bs{s}_0: A(E)\rightarrow E^\times$ relativement à $E/F$, et un élément $\bs{x}_0\in \mathfrak{A}(E)$ tel que $\bs{s}_0(\bs{x}_0)=1$. On pose $\bs{s}=\bs{s}_0\otimes {\rm id}_{\mathfrak{b}}$ et $\bs{x}=\bs{x}_0\otimes 1$. 

Comme dans la proposition de \ref{le résultat principal}, on suppose $E\neq F$. On veut descendre une distribution $G$--invariante au voisinage de $\beta$ dans $G$ en une distribution $H$--invariante au voisinage de $0$ dans $\mathfrak{b}$. On reprend pour cela la construction de \cite{L1} (voir aussi \cite[5.4]{L2}), qui est une variante du principe de submersion d'Harish--Chandra. D'après la proposition de \ref{une submersion}, l'application
$$
\delta:G\times \bs{x}\underline{\mathfrak{Q}}^{\underline{k}_0+1}\rightarrow G, \;(g,\bs{x}b)\mapsto g^{-1}(\beta + \bs{x} b)g
$$
est partout submersive (pour les constructions à suivre, on a préféré remplacer $\underline{\mathfrak{Q}}^{\underline{k}_0+1}$ par $\bs{x}\underline{\mathfrak{Q}}^{\underline{k}_0+1}$). Fixons une mesure de Haar $dg$ sur $G$ et une mesure de Haar $\mathfrak{d}b'$ sur $\mathfrak{b}$. D'après le principe de submersion d'Harish--Chandra, il existe une unique application linéaire surjective
$$
C^\infty_{\rm c}(G\times \bs{x}\underline{\mathfrak{Q}}^{\underline{k}_0+1})\rightarrow C^\infty_{\rm c}({\rm Im}(\delta)),\, 
\phi\mapsto \phi^\delta,
$$
telle que, pour toute fonction $\phi\in C^\infty_{\rm c}(G\times \bs{x}\underline{\mathfrak{Q}}^{\underline{k}_0+1})$ et toute fonction 
$f\in C^\infty({\rm Im}(\delta))$, on a
$$
\int\!\!\!\int_{G\times \mathfrak{b}}\phi(g,\bs{x}b')f(\delta(g,\bs{x}b'))dg\mathfrak{d}b'= \int_G\phi^\delta(g)f(g)dg.
$$
On déduit (cf. \cite[2.3.1]{L1}) que pour toute distribution {\it $G$--invariante}, \cad invariante par conjugaison, $T$ sur $G$, il existe une unique distribution $\wt{\vartheta}_T$ sur $\bs{x}\underline{\mathfrak{Q}}^{\underline{k}_0+1}$ telle que pour toute fonction $\phi\in 
C^\infty_{\rm c}(G\times \bs{x}\underline{\mathfrak{Q}}^{\underline{k}_0+1})$, on a
$$
\langle \phi_\delta , \wt{\vartheta}_T\rangle = \langle \phi^\delta,T\rangle,\leqno{(1)}
$$
où la fonction $\phi_\delta\in C^\infty_{\rm c}(\bs{x}\underline{\mathfrak{Q}}^{\underline{k}_0+1})$ est donnée par
$$
\phi_\delta(\bs{x}b')=\int_G\phi(g,\bs{x}b')dg.
$$
Bien sûr, si $T$ et $T'$ sont deux distributions $G$--invariantes sur $G$ qui co\"{\i}ncident sur l'ouvert ${\rm Im}(\delta)$ de $G$, alors les distributions $\wt{\vartheta}_T$ et $\wt{\vartheta}_{T'}$ sur $\bs{x}\underline{\mathfrak{Q}}^{\underline{k}_0+1}$ sont égales. 

\`A une distribution $G$--invariante $T$ sur $G$, on associe comme en \cite{L1} une distribution $H$--invariante $\vartheta_T$ sur ${^H\!(\underline{\mathfrak{Q}}^{\underline{k}_0+1})}$. Rappelons la construction. Soit une fonction $\varphi\in C^\infty_{\rm c}({^H\!(\underline{\mathfrak{Q}}^{\underline{k}_0+1})})$. Elle se décompose en
$$
\varphi= \sum_{h\in H} \varphi_h\leqno{(2)}
$$
avec $\varphi_h\in C^\infty_{\rm c}(h\,\underline{\mathfrak{Q}}^{\underline{k}_0+1}h^{-1})$ et $\varphi_h=0$ sauf pour un nombre fini de $h$. Pour $h\in H$, on note $\wt{\varphi}_h\in C^\infty_{\rm c}(\bs{x}\underline{\mathfrak{Q}}^{\underline{k}_0+1})$ la fonction $(\varphi_h\circ {\rm Ad}_h)\circ \bs{s}$ sur $\bs{x}\underline{\mathfrak{Q}}^{\underline{k}_0+1}$, et on pose
$$
\langle \varphi, \vartheta_T\rangle = \sum_{h\in H} \langle \wt{\varphi}_h, \wt{\vartheta}_T\rangle.\leqno{(3)}
$$
D'après \cite{L1} (voir aussi \cite[5.4]{L2}), la quantité $\langle \varphi, \vartheta_T\rangle$ ne dépend pas de la décomposition (2) choisie, et la distribution $\vartheta_T$ sur ${^H\!(\underline{\mathfrak{Q}}^{\underline{k}_0+1})}$ ainsi définie est $H$--invariante. Le {\it support}  de cette distribution $\vartheta_T$ est par définition l'ensemble des $b\in {^H\!(\underline{\mathfrak{Q}}^{\underline{k}_0+1})}$ tels que 
pour tout voisinage ouvert $U_b$ de $b$ dans ${^H\!(\underline{\mathfrak{Q}}^{\underline{k}_0+1})}$, la restriction de $\vartheta_T$ à $U_b$ n'est pas nulle. C'est une partie fermée de ${^H\!(\underline{\mathfrak{Q}}^{\underline{k}_0+1})}$, que l'on note ${\rm Supp}(\vartheta_T)$. D'aprs la remarque 2 de \ref{lments qre}, ${^H\!(\underline{\mathfrak{Q}}^{\underline{k}_0+1})}$ est une partie ferme dans $\mathfrak{b}$. Par consquent ${\rm Supp}(\vartheta_T)$ est aussi une partie ferme dans $\mathfrak{b}$, et on peut prolonger $\vartheta_T$ par $0$ sur $\mathfrak{b}\smallsetminus 
{\rm Supp}(\vartheta_T)$. On obtient ainsi une distribution $H$--invariante sur $\mathfrak{b}$, de support 
${\rm Supp}(\vartheta_T)$, que l'on note $\theta_T$.  

On peut aussi, comme en \cite[2.3]{L1}, se restreindre aux distributions sur $\mathfrak{b}$ à support dans un voisinage ouvert fermé et $H$--invariant $\Omega$ de $0$ dans $\mathfrak{b}$ contenu dans ${^H\!(\underline{\mathfrak{Q}}^{\underline{k}_0+1})}$. Fixons un tel voisinage $\Omega$ (on peut choisir $\Omega$ aussi petit que l'on veut --- cf. la remarque de \ref{parties compactes modulo conjugaison}). 
Pour une distribution $G$--invariante $T$ sur $G$, on note $\theta_T^\Omega$ la distribution $H$--invariante sur $\mathfrak{b}$ à support dans $\Omega$ définie par
$$
\langle \mathfrak{f}, \theta_T^\Omega\rangle = \langle \mathfrak{f}\vert_\Omega, \vartheta_T\rangle,\quad \mathfrak{f}\in C^\infty_{\rm c}(\mathfrak{b}).\leqno{(4)}
$$
Bien sûr si ${\rm Supp}(\vartheta_T)\subset \Omega$, on a $\theta_T^\Omega=\theta_T$.

\vskip1mm
Rappelons que d'après \cite[4.3.5]{L2}, on a le

% lemme
\begin{monlem}
Soit $\mathfrak{A}$ un $\mathfrak{o}$--ordre héréditaire dans $\mathfrak{g}$ normalisé par $E^\times$, et soit $\mathfrak{B}=\mathfrak{A}\cap \mathfrak{b}$ . Posons $\mathfrak{P}={\rm rad}(\mathfrak{A})$ et $\mathfrak{Q}={\rm rad}(\mathfrak{B})$. 
Pour tout entier $i> k_0=k_0(\beta,\mathfrak{A})$, on a
$$
\{g\in G: g^{-1}\beta g \in \beta + \mathfrak{P}^{i}\}=H(1+ \mathfrak{Q}^{i-k_0}\mathfrak{N}_{k_0}(\beta,\mathfrak{A})).
$$
\end{monlem} 

Fixons une mesure de Haar $dz$ sur le centre $Z=F^\times$ de $G$. On peut prendre pour $dz$ la mesure qui donne le volume $1$ au sous--groupe compact maximal $U_F$ de $Z$, 
mais ce n'est pas vraiment nécessaire. Pour $\gamma\in G_{\rm qre}$, on note $dg_\gamma$ la mesure de Haar sur $G_\gamma=F[\gamma]^\times$ telle que ${\rm vol}(Z \backslash G_\gamma, \textstyle{dg_\gamma \over dz})=1$, \cad celle telle que
$$
e(F[\gamma]/F){\rm vol}(U_{F[\gamma]},dg_\gamma)={\rm vol}(U_F,dz),
$$
et on note $\ES{O}_\gamma=\ES{O}_\gamma^G$ la distribution ($G$--invariante) sur $G$ définie par
$$
\ES{O}_\gamma(f)= \int_{G_\gamma\backslash G}f(g^{-1}\gamma g)\textstyle{dg\over dg_\gamma},\quad f\in C^\infty_{\rm c}(G).
$$
On a donc
$$
\ES{O}_\gamma(f) =  \int_{Z\backslash G}f(g^{-1}\gamma g)\textstyle{dg\over dz},\quad f\in C^\infty_{\rm c}(G).
$$
On définit aussi une constante $v_F(\gamma)=v_F(\gamma,\textstyle{dg\over dg\gamma})>0$ par la formule
$$
v_F(\gamma)= {\rm vol}\!\left(F[\gamma]^\times(1+ \mathfrak{p}_{F[\gamma]}\mathfrak{N}_{k_F(\gamma)}(\gamma,\mathfrak{A}_\gamma)),
\textstyle{dg\over dg_\gamma}\right).\leqno{(5)}
$$
Fixons une mesure de Haar $dh$ sur $H$, et une mesure de Haar $dz_E$ sur le centre $Z_H=E^\times$ de $H$. On peut prendre pour $dh$ la mesure $\mathfrak{d}^\times b'= \textstyle{ \mathfrak{d}b' \over \vert \det(b')\vert_E^d}$ associée à $\mathfrak{d}b'$, et pour $dz_E$ la mesure telle que ${\rm vol}(F^\times \backslash E^\times, {dz_E\over dz})=1$, 
mais ce n'est pas nécessaire pour l'instant --- voir la proposition de \ref{IO normalisées}. De la même manière, pour $b \in \mathfrak{b}_{\rm qre}$, on note $dh_b$ la mesure de Haar sur $H_b= E[b]^\times$ telle que ${\rm vol}(Z_H\backslash H_b, {dh_b\over dz_E})=1$, \cad telle que $e(E[b]/E){\rm vol}(U_{E[b]},dh_b)= {\rm vol}(U_E,dz_E)$, 
et $\ES{O}_b^{\mathfrak{b}}$ la distribution ($H$--invariante) sur $\mathfrak{b}$ définie par
$$
\ES{O}_b^{\mathfrak{b}}(\mathfrak{f})=\int_{H_b\backslash H}\mathfrak{f}(h^{-1}b h)\textstyle{dh\over dh_b},\quad \mathfrak{f}\in C^\infty_{\rm c}(\mathfrak{b}).
$$
On a donc
$$
\ES{O}_b^{\mathfrak{b}}(\mathfrak{f})=\int_{Z_H\backslash H}\mathfrak{f}(h^{-1}b h)\textstyle{dh\over dz_E},\quad \mathfrak{f}\in C^\infty_{\rm c}(\mathfrak{b}).
$$
On définit aussi une constante $v_E(b)= v_E(b, \textstyle{dh\over dh_b})>0$ par la formule
$$
v_E(b)={\rm vol}\!\left(E[b]^\times(1+ \mathfrak{p}_{E[b]}\mathfrak{N}_{k_E(b)}(b,\mathfrak{B}_b)), \textstyle{dh\over dh_b}\right).\leqno{(6)}
$$
Notons que si $E[b]=E$, alors $E[b]^\times(1+ \mathfrak{p}_{E[b]}\mathfrak{N}_{k_E(b)}(b,\mathfrak{B}_b))=E^\times\;(=H)$ et $v_E(b)$ n'est autre que le rapport des mesures ${dh\over dh_b}$, \cad (compte--tenu de la normalisation de $dh_b$) ${\rm vol}(U_E,dh){\rm vol}(U_E,dz_E)^{-1}$. 

% remarque 1
\begin{marema1}
{\rm Pour $\gamma\in G_{\rm qre}$, la distribution $f\mapsto v_F(\gamma)^{-1}\ES{O}_\gamma(f)$ sur $G$ ne dépend pas du choix des  
mesures de Haar. D'après le lemme, pour $\gamma\in G_{\rm qre}$, la constante $v_F(\gamma)$ est le volume (pour la mesure $\textstyle{dg\over dg_\gamma}$ sur $F[\gamma]^\times \backslash G$) de l'ensemble des $g\in G$ tels que $g^{-1}\gamma g$ appartient à $\gamma + \mathfrak{P}_\gamma^{k_F(\gamma)+1}=\gamma U^{\tilde{k}_F(\gamma)+1}(\mathfrak{A}_\gamma)$. 
Ce lemme est à la base de ce que, dans \cite{L2}, nous avons maladroitement appelé la normalisation \og $J$\fg des intégrales orbitales sur $G$. Cette normalisation \og $J$\fg consiste à choisir les mesures $dg$ et $dg_\gamma$ (pour $\gamma\in G_{\rm qre})$ de telle manière que le facteur de normalisation $v_F(\gamma)$ soit égal à $1$. Nous y reviendrons plus loin (\ref{IO normalisées}). \hfill $\blacksquare$
}
\end{marema1} 

On a posé $d= {N\over [E:F]}$. Pour $b\in \mathfrak{b}_{\rm qre}\cap \underline{\mathfrak{Q}}^{\underline{k}_0+1}$, on 
pose
$$
k_F(\beta,b)\;(= k_F(\beta + \bs{x}_0\otimes b)) = \left\{\begin{array}{ll}
k_E(b)& \mbox{si $E[b]\neq E$}\\
k_F(\beta) & \mbox{sinon}
\end{array}\right..
$$
On note $\mathfrak{A}_{\beta,b}$ l'unique $\mathfrak{o}$--ordre héréditaire dans $\mathfrak{g}$ normalisé par $F[\beta,b]^\times =E[b]^\times$, et $\mathfrak{B}_b$ l'unique $\mathfrak{o}_E$--ordre héréditaire dans $\mathfrak{b}$ normalisé par $E[b]^\times$. On a donc $\mathfrak{A}_{\beta,b}\cap \mathfrak{b}= \mathfrak{B}_b$, et $\mathfrak{A}_{\beta,b}$ est aussi l'unique $\mathfrak{o}$--ordre héréditaire dans $\mathfrak{g}$ normalisé par $F[\beta + \bs{x}_0\otimes b]^\times$. On pose $\mathfrak{P}_{\beta,b}={\rm rad}(\mathfrak{A}_{\beta,b})$, $\mathfrak{Q}_b={\rm rad}(\mathfrak{B}_b)$, 
$$
n_F(\beta,b)\;(=-\nu_{\mathfrak{A}_{\beta,b}}(\beta)) = n_F(\beta)e(\mathfrak{B}_b\vert \mathfrak{o}_E),
$$
et
$$
\tilde{k}_F(\beta,b)\;(= \tilde{k}_F(\beta + \bs{x}_0\otimes b)) = n_F(\beta,b)+ k_F(\beta,b)\geq 0.
$$
Soit
$$
I_G^H(\beta,b)= {{\rm vol}(\mathfrak{Q}_b^{k_F(\beta,b)+1},\mathfrak{d}b')\over 
{\rm vol}(U^{\tilde{k}_F(\beta,b)+1}(\mathfrak{A}_{\beta,b}),dg)}.\leqno{(7)}
$$

% proposition
\begin{mapropo}
Soit $b\in \mathfrak{b}_{\rm qre}\cap \underline{\mathfrak{Q}}^{\underline{k}_0+1}$ et soit $\gamma= \beta + \bs{x}_0\otimes b$. 
On a ${\rm Supp}(\vartheta_{\ES{O}_\gamma})= \ES{O}_H(b)$, et la distribution $H$--invariante 
$\theta_{\ES{O}_\gamma}$ sur $\mathfrak{b}$ est donnée par
$$
\theta_{\ES{O}_\gamma} = I_G^H(\beta,b){v_F(\gamma)\over v_E(b)} \ES{O}_{b}^{\mathfrak{b}}.
$$
\end{mapropo}

\begin{proof}Commen\c{c}ons par prouver que le support de la distribution 
$\vartheta_{\ES{O}_\gamma}$ sur ${^H\!(\underline{\mathfrak{Q}}^{\underline{k}_0+1})}$ est égal à $\ES{O}_H(b)$. Soit $\ES{K}$ un sous--groupe ouvert compact de $G$, et soit $f_{\ES{K}}$ la fonction caractéristique de $\ES{K}$ divisée par ${\rm vol}(\ES{K},dg)$. Soit $\varphi\in C^\infty_{\rm c}({^H\!(\underline{\mathfrak{Q}}^{\underline{k}_0+1})})$. Décomposons 
$\varphi$ en $\varphi= \sum_{h\in H} \varphi_h$ comme en (2). Par définition de $\vartheta_{\ES{O}_\gamma}$, on a
$$
\langle \varphi , \vartheta_{\ES{O}_\gamma} \rangle = \sum_{h\in H} \ES{O}_\gamma((f_{\ES{K}}\otimes \wt{\varphi}_h)^\delta)
=\ES{O}_\gamma\!\left(\textstyle{\sum_{h\in H}} (f_{\ES{K}}\otimes \wt{\varphi}_h)^\delta\right),
$$
puis, par linéarité de l'application $\phi \mapsto \phi^\delta$,
$$
\langle \varphi , \vartheta_{\ES{O}_\gamma} \rangle = \ES{O}_\gamma\!\left([f_{\ES{K}}\otimes (\textstyle{\sum_{h\in H}} \wt{\varphi}_h)]^\delta\right).
$$
Si de plus $\varphi>0$, alors $[f_{\ES{K}}\otimes (\textstyle{\sum_{h\in H}} \wt{\varphi}_h)]^\delta>0$, et notant $Y\subset \underline{\mathfrak{Q}}^{\underline{k}_0+1}$ le support de la fonction 
$\textstyle{\sum_{h\in H}}\varphi_h\circ {\rm Ad}_h$, on a $\langle \varphi , \vartheta_{\ES{O}_\gamma} \rangle\neq 0$ si et seulement si
$$
\delta(\ES{K}\times \bs{x}Y)\cap \ES{O}_G(\gamma)\neq \emptyset,
$$
\cad si et seulement si
$$
\{\beta + \bs{x}_0\otimes b': b'\in Y\}\cap \ES{O}_G(\gamma)\neq \emptyset.
$$
Puisque $\vartheta_{\ES{O}_\gamma}$ est une distribution $H$--invariante, son support contient l'orbite $\ES{O}_H(b)$, et il est égal à cette orbite si et seulement si pour tout $b'\in \underline{\mathfrak{Q}}^{\underline{k}_0+1}$, on a
$$
\beta + \bs{x}_0\otimes b' \in \ES{O}_G(\gamma)\Rightarrow b'\in \ES{O}_H(b).\leqno{(8)}
$$
Si l'élément $b'$ n'est pas quasi--régulier elliptique (dans $\mathfrak{b}$), alors il est contenu dans une sous--$E$--algèbre parabolique propre de $\mathfrak{b}$, et $\gamma'=\beta + \bs{x}_0\otimes b'$ est contenu dans une sous--$F$--algèbre parabolique propre de $\mathfrak{g}$, ce qui entraîne que $\gamma'\notin \ES{O}_G(\gamma)$. D'autre part si $b'\in \mathfrak{b}_{\rm qre}\cap \underline{\mathfrak{Q}}^{\underline{k}_0+1}$, l'implication (8) est vraie d'après la proposition de \ref{le résultat principal}. On a donc prouvé que le support de la distribution $\vartheta_{\ES{O}_\gamma}$ est égal à $\ES{O}_H(b)$. 

On vient de voir que la distribution $H$--invariante $\theta_{\ES{O}_\gamma}$ sur $\mathfrak{b}$ vérifie $\theta_{\ES{O}_\gamma}= \alpha \ES{O}_b^{\mathfrak{b}}$ pour une constante $\alpha >0$. Calculons cette constante $\alpha$. Posons $\mathfrak{A}=\mathfrak{A}_{\beta,b}$, $\mathfrak{B}=\mathfrak{B}_b$, $\mathfrak{P}=\mathfrak{P}_{\beta,b}$ et $\mathfrak{Q}=\mathfrak{Q}_b$. Posons $s= \inf\{-k_E(b), n_E(b)\}$. On a donc $-s =k_E(b)$ si $E[b]\neq E$, et $-s = \nu_E(b)$ sinon. Soit 
$\varphi$, resp. $\wt{\varphi}$, la fonction caractéristique de $b+ \mathfrak{Q}^{-s+1}$, resp. $\bs{x}(b+ \mathfrak{Q}^{-s+1})$. Si $E[b]\neq E$, d'après le lemme, on a
$$
\ES{O}_b^{\mathfrak{b}}(\varphi)=v_E(b).\leqno{(9)}
$$
Si $E[b]=E$, l'égalité (9) reste vraie. En effet dans ce cas, on a $E^\times(1+ \mathfrak{p}_{E}\mathfrak{N}_{k_E(b)}(b,\mathfrak{B}))=E^\times $ et
$$
\ES{O}_b^{\mathfrak{b}}(\mathfrak{f})=\mathfrak{f}(b)\textstyle{dh\over dh_b}= \mathfrak{f}(b){{\rm vol}(U_E,dh)\over {\rm vol}(U_E,dz_E)},\quad \mathfrak{f}\in C^\infty_{\rm c}(\mathfrak{b});
$$ 
en particulier pour $\mathfrak{f}=\varphi$, puisque $\varphi(b)=1$, on a bien $\ES{O}_b^{\mathfrak{b}}(\varphi)=v_E(b)$.

Posons $k_0=k_0(\beta,\mathfrak{A})$, et notons $\ES{K}$ le sous--groupe 
$1 + \mathfrak{Q}^{-s-k_0+1}\mathfrak{N}_{k_0}(\beta,\mathfrak{A})$ de $U^1(\mathfrak{A})$. D'après la preuve de \cite[5.4.3]{L2}, pages 73--74, qui utilise \cite[5.3.4]{L2}, on a
$$
\delta(\ES{K}\times \bs{x}(b+ \mathfrak{Q}^{-s+1}))=\gamma + \mathfrak{P}^{-s+1}
$$
et pour toute fonction $f\in C^\infty_{\rm c}(\gamma + \mathfrak{P}^{-s+1})$, on a
$$
\int_{\ES{K} \times \mathfrak{Q}^{-s+1}}f\circ\delta(g,\bs{x}(b+b'))dg\mathfrak{d}b'=\int_{\ES{K}\times \mathfrak{Q}^{-s+1}}f(\gamma + {\rm Ad}_\beta(1-g)+ \bs{x}b')dg\mathfrak{d}b'.
$$
Notons $\bs{1}_{\ES{K}}$ la fonction caractéristique de $\ES{K}$, et prenons pour $f$ la fonction caractéristique de $\gamma + \mathfrak{P}^{-s+1}$. On obtient
$$
(\bs{1}_{\ES{K}}\otimes \wt{\varphi})^\delta = {{\rm vol}(\ES{K},dg){\rm vol}(\mathfrak{Q}^{-s+1},\mathfrak{d}b') \over {\rm vol}(\gamma + \mathfrak{P}^{-s+1},dg)}f.
$$
Posons $n=n_F(\beta,b)$. Puisque $n= -\nu_{\mathfrak{A}}(\gamma)$, on a $\gamma + \mathfrak{P}^{-s+1}= \gamma U^{n-s+1}(\mathfrak{P})$ et
$$
{\rm vol}(\gamma + \mathfrak{P}^{-s+1},dg)={\rm vol}(U^{n-s+1}(\mathfrak{P}),dg).
$$
D'autre part, on a
$$
(\bs{1}_{\ES{K}}\otimes \wt{\varphi})_\delta={\rm vol}(\ES{K},dg) \wt{\varphi}.
$$
D'après (1), on obtient
$$
\ES{O}_\gamma(f)= {{\rm vol}(U^{n-s+1}(\mathfrak{P}),dg)\over {\rm vol}(\mathfrak{Q}^{-s+1},\mathfrak{d}b')}
\langle \wt{\varphi}, \wt{\vartheta}_{\ES{O}_\gamma}\rangle.
$$
Or on a
$$
\langle \wt{\varphi}, \wt{\vartheta}_{\ES{O}_\gamma}\rangle = \langle \varphi, \theta_{\ES{O}_\gamma}\rangle =\alpha \ES{O}_b^{\mathfrak{b}}(\varphi),
$$
d'où
$$
{{\rm vol}(\mathfrak{Q}^{-s+1},\mathfrak{d}b')\over {\rm vol}(U^{n-s+1}(\mathfrak{P}),dg)}\ES{O}_\gamma(f) = \alpha \ES{O}_b^{\mathfrak{b}}(\varphi).\leqno{(10)}
$$
Enfin, à nouveau d'après le lemme, on a
$$
\ES{O}_\gamma(f) = {\rm vol}\!\left(F[\gamma]^\times(1+ \mathfrak{p}_{F[\gamma]}^{-s-k_F(\gamma)+1}\mathfrak{N}_{k_F(\gamma)}(\gamma,\mathfrak{A})),\textstyle{dg\over dg_\gamma}\right).\leqno{(11)}
$$
Si $E[b]\neq E$, alors $-s = k_F(\gamma)=k_E(b)$ et $\ES{O}_\gamma(f)=v_F(\gamma)$, et grâce à (9) et (10), on obtient la valeur annoncée pour la constante $\alpha$. Reste à traiter le cas $E[b]=E$. En ce cas, on a $-s = \nu_E(b)$ et $k_F(\gamma)=k_F(\beta)$, et posant 
$\mathfrak{N}=\mathfrak{N}_{k_F(\beta)}(\gamma,\mathfrak{A})$, le volume à droite de l'égalité (11) est égal à
$$
v_F(\gamma){[\mathfrak{p}_{F[\gamma]}: \mathfrak{p}_{F[\gamma]}^{-s-k_F(\beta)+1}]\over 
[\mathfrak{p}_{F[\gamma]}\mathfrak{N}: \mathfrak{p}_{F[\gamma]}^{-s-k_F(\beta)+1}\mathfrak{N}]},
$$
ou encore à
$$
v_F(\gamma){[\mathfrak{o}_{E}: \mathfrak{p}_{E}^{-s-k_F(\beta)}]\over 
[\mathfrak{A}: \mathfrak{P}^{-s-k_F(\beta)}]}.
$$
Le terme à gauche de l'égalité (10) est donc égal à
$$
v_F(\gamma){[\mathfrak{o}_{E}: \mathfrak{p}_{E}^{-s-k_F(\beta)}]\over 
[\mathfrak{A}: \mathfrak{P}^{-s-k_F(\beta)}]}{{\rm vol}(\mathfrak{p}_E^{-s+1},\mathfrak{d}b')\over {\rm vol}(U^{n-s+1}(\mathfrak{A}),dg)},
$$
ou encore à
$$
v_F(\gamma){{\rm vol}(\mathfrak{p}_E^{k_F(\beta)+1},\mathfrak{d}b')\over {\rm vol}(U^{n+k_F(\beta)+1}(\mathfrak{A}),dg)}=v_F(\gamma)I_G^H(\beta,b).
$$
On conclut grâce à (9).
\end{proof}

% corollaire
\begin{moncoro}
Supposons que $\beta$ est quasi--régulier elliptique (dans $G$), i.e. que $E$ est un sous--corps maximal de $\mathfrak{g}$ (on a donc $\mathfrak{b}=E$, i.e. $d=1$). Alors pour $b\in \mathfrak{p}_E^{k_F(\beta)+1}$, posant $\gamma= \beta + \bs{x}_0\otimes b$, on a
$$
\theta_{\ES{O}_\gamma}= {{\rm vol}(\mathfrak{p}_E^{k_F(\beta)+1},\mathfrak{d}b')v_F(\beta)\over {\rm vol}(U^{\tilde{k}_F(\beta)+1}(\mathfrak{A}_\beta),dg)} \bs{\delta}_b,
$$
où $\bs{\delta}_b$ désigne la mesure de Dirac au point $b$.
\end{moncoro}

\begin{proof}
Soit $b\in \mathfrak{p}_E^{k_F(\beta)+1}$, et soit $\gamma=\beta+ \bs{x}_0\otimes b$. Puisque $E[b]=E=\mathfrak{b}$, on  
a $k_F(\gamma)= k_F(\beta)$. De plus, $\mathfrak{A}= \mathfrak{A}_\gamma$ est l'unique $\mathfrak{o}$--ordre héréditaire dans $\mathfrak{g}$ normalisé par $E^\times$, et d'après \cite[2.1.3]{BK}, on a 
$\mathfrak{p}_{F[\gamma]}\mathfrak{N}_{k_F(\gamma)}(\gamma,\mathfrak{A})= 
\mathfrak{p}_{E}\mathfrak{N}_{k_F(\beta)}(\beta,\mathfrak{A})$, d'où $v_F(\gamma)=v_F(\beta)$. Enfin pour $\mathfrak{f}\in C^\infty_{\rm c}(E)$, d'après la démonstration de la proposition, on a $\theta_b^{\mathfrak{b}}(\mathfrak{f})= \mathfrak{f}(b)v_E(b)$. D'où le corollaire.\end{proof}

% remarque 2
\begin{marema2}
{\rm 
La proposition a pour conséquence que pour toute fonction $f\in C^\infty_{\rm c}(G)$ à support contenu dans l'ouvert 
${\rm Im}(\delta)$ de $G$, il existe une fonction $f^{\mathfrak{b}}\in C^\infty_{\rm c}(\underline{\mathfrak{Q}}^{\underline{k}_0+1})$ telle que pour tout 
$b\in \mathfrak{b}_{\rm qre}\cap \underline{\mathfrak{Q}}^{\underline{k}_0+1}$, 
posant $\gamma = \beta + \bs{x}_0\otimes b$, on a l'égalité
$$
\ES{O}_{\beta + \bs{x}_0\otimes b}(f) = I_G^H(\beta,b) {v_F(\gamma)\over v_E(b)} \ES{O}_b^{\mathfrak{b}}(f^{\mathfrak{b}}).\leqno{(12)}
$$
En effet, puisque l'application $C^\infty_{\rm c}(G\times \bs{x}\underline{\mathfrak{Q}}^{\underline{k}_0+1})\rightarrow C^\infty_{\rm c}({\rm Im}(\delta)),\, 
\phi'\mapsto \phi'^\delta$ est surjective, il suffit d'écrire $f =  \phi^\delta$ pour une fonction $\phi \in C^\infty_{\rm c}(G\times \bs{x}\underline{\mathfrak{Q}}^{\underline{k}_0+1})$ et de prendre pour $f^{\mathfrak{b}}$ la fonction $b'\mapsto \phi_\delta(\bs{x}b')$. D'ailleurs, sans l'hypothèse que le support de $f$ est contenu dans ${\rm Im}(\delta)$, en rempla\c{c}ant $f$ par $f_\Xi = f\vert_{\Xi}$ pour un voisinage  ouvert fermé et $G$--invariant $\Xi$ de $\beta$ dans ${\rm Im}(\delta)$ --- un tel voisinage existe d'après le lemme 3 de \ref{parties compactes modulo conjugaison} ---, on obtient le même résultat pourvu que dans l'égalité (12), on se limite aux éléments $b\in \mathfrak{b}_{\rm qre}\cap \underline{\mathfrak{Q}}^{\underline{k}_0+1}$ tels que $\gamma= \beta + \bs{x}_0\otimes b$ appartient à $\Xi$. Comme d'après \ref{le résultat principal}.(3), l'ensemble ${\rm Im}(\delta)= {^G(\beta + \bs{x}\underline{\mathfrak{Q}}^{\underline{k}_0+1})}$ est ouvert {\it fermé} et $G$--invariant dans $G$, cette limitation n'en est pas une: 
on peut prendre $\Xi= {\rm Im}(\delta)$. 
\hfill $\blacksquare$
}
\end{marema2}

% remarque 3
\begin{marema3}
{\rm 
D'après le corollaire et la remarque 2, pour toute fonction $f\in C^\infty_{\rm c}(G)$, la fonction
$$
G_{\rm qre}\rightarrow {\Bbb C},\, \gamma \mapsto \ES{O}_\gamma(f)
$$
est localement constante.
\hfill $\blacksquare$ 
}
\end{marema3} 

%%%%%%%%%%%%%%%
\subsection{Intégrales orbitales normalisées}\label{IO normalisées}
Pour $\gamma\in G_{\rm qre}$, considrons maintenant le facteur de normalisation $\mu_F(\gamma)$ défini par
$$
\mu_F(\gamma) = \left\{\begin{array}{ll}{{\rm vol}\left(F[\gamma]^\times (1+ \mathfrak{p}_{F[\gamma]}\mathfrak{N}_{k_F(\gamma)}(\gamma,\mathfrak{A}_\gamma)),\textstyle{dg\over dg_\gamma}\right)\over {\rm vol}\left(F[\gamma]^\times U^{\tilde{k}_F(\gamma)+1}(\mathfrak{A}_\gamma),\textstyle{dg\over dg_\gamma}\right)}& \mbox{si $N>1$}\\
1 & \mbox{sinon}\end{array}\right..\leqno{(1)}
$$ 
Notons que la quantité $\mu_F(\gamma)$ ne dépend pas de la mesure $G$--invariante $\textstyle{dg\over dg_\gamma}$ sur l'espace quotient $G_\gamma\backslash G$ utilisée pour la définir. Pour $\gamma\in G_{\rm qre}$, on a:
\begin{enumerate}[leftmargin=17pt]
\item[(2)]$\mu_F(\gamma)\geq 1$ avec égalité si et seulement si $\gamma$ est minimal.
\end{enumerate}
En effet, c'est évident si $N=1$, et si $N>1$, posant $k= k_F(\gamma)$ et $\tilde{k}=\tilde{k}_F(\gamma)$, on a l'inclusion
$$
\mathfrak{o}_{F[\gamma]}+\mathfrak{P}_\gamma^{\tilde{k}}\subset \mathfrak{N}_k(\gamma,\mathfrak{A}_\gamma),
$$
et $\mu_F(\gamma)$ n'est autre que l'indice $[\mathfrak{N}_k(\gamma,\mathfrak{A}_\gamma): \mathfrak{o}_{F[\gamma]}+\mathfrak{P}_\gamma^{\tilde{k}}]$. En particulier $\mu_F(\gamma)\geq 1$. Si $\tilde{k}=0$, i.e. si $\gamma$ est minimal, alors $\mathfrak{N}_k(\gamma,\mathfrak{A}_\gamma)= \mathfrak{A}_\gamma$ et cet indice vaut $1$. Si $\tilde{k}\geq 1$, alors $\mathfrak{P}_\gamma^{\tilde{k}}\subset \mathfrak{P}_\gamma$, et comme $\mathfrak{N}_k(\gamma,\mathfrak{A}_\gamma)\not\subset \mathfrak{o}_{F[\gamma]}+ \mathfrak{P}_\gamma$, l'inclusion $\mathfrak{o}_{F[\gamma]}+\mathfrak{P}_\gamma^{\tilde{k}}\subset \mathfrak{N}_k(\gamma,\mathfrak{A}_\gamma)$ est stricte.

Pour $\gamma\in G_{\rm qre}$, posons $e_\gamma =e(F[\gamma]/F)$, $f_\gamma = f(F[\gamma]/F)$ et (si l'extension $E/F$ est séparable) $\delta_\gamma = \delta(F[\gamma]/F)$. On a donc $e_\gamma f_\gamma =N$. Soit $\eta_G: G_{\rm qre}\rightarrow {\Bbb R}_{>0}$ la fonction définie par
$$
\eta_G(\gamma)= q^{-f_\gamma (\tilde{c}_F(\gamma) + e_\gamma -1)},\leqno{(3)}
$$
où l'invariant $\tilde{c}_F(\gamma)$ a été défini en \ref{l'invariant c}. D'après \ref{l'invariant c}.(5), on a:
\begin{enumerate}[leftmargin=17pt]
\item[(4)]si l'extension $F[\gamma]/F$ est séparable, alors $\eta_G(\gamma)= \vert D_F(\gamma)\vert q^{\delta_\gamma -(N-f_\gamma)}$.
\end{enumerate}

% lemme 1
\begin{monlem1}
Soit $\gamma\in G_{\rm qre}$. On a
$$
\eta_G(\gamma)\mu_F(\gamma) = 1.
$$
\end{monlem1}

\begin{proof}Si $N=1$, alors $\eta_G(\gamma)=\mu_F(\gamma)=1$ et il n'y a rien à démontrer. On suppose donc $N>1$. 
Soit $K=F[\gamma]$. Posons $\mathfrak{A}=\mathfrak{A}_\gamma$, $\mathfrak{P}=\mathfrak{P}_\gamma$, $\mathfrak{N}=\mathfrak{N}_{k_F(\gamma)}(\gamma,\mathfrak{A})$, et choisissons une corestriction modérée $\bs{s}_\gamma: \mathfrak{g}\rightarrow K$ sur $\mathfrak{g}$ relativement  $K/F$. Alors d'après \cite[1.4.10]{BK}, pour $m\in {\Bbb Z}$, 
on a la suite exacte courte
$$
0 \rightarrow \mathfrak{p}_K^m\backslash \mathfrak{p}_K^m \mathfrak{N}\xrightarrow{-{\rm ad}_\gamma}\mathfrak{P}^{k_F(\gamma)+m} \xrightarrow{\bs{s}_\gamma}
\mathfrak{p}_K^{k_F(\gamma)+m}\rightarrow 0.\leqno{(5)}
$$
Rappelons qu'on a posé $\tilde{k}_F(\gamma)= k_F(\gamma)+n_F(\gamma)\geq 0$. Pour $m=1$, on en déduit la suite exacte courte
$$
0 \rightarrow \mathfrak{p}_K\backslash \mathfrak{p}_K \mathfrak{N}\xrightarrow{1-{\rm Ad}_\gamma}\mathfrak{P}^{\tilde{k}_F(\gamma)+1} \xrightarrow{\bs{s}_\gamma}
\mathfrak{p}_K^{\tilde{k}_F(\gamma)+1}\rightarrow 0.\leqno{(6)}
$$
Supposons l'extension $K/F$ séparable, et notons $\sigma_\gamma={\delta_\gamma \over f_\gamma} -(e_\gamma-1)$ son exposant de Swan. D'après \cite[1.3.8.(ii)]{BK}, on a la décomposition $\mathfrak{g}= {\rm ad}_\gamma(\mathfrak{g})\oplus K$, et notant $\bs{p}_K:\mathfrak{g}\rightarrow K$ la projection orthogonale par rapport à cette décomposition, on peut prendre pour $\bs{s}_\gamma$ l'application $y\mapsto \varpi_K^{\sigma_\gamma}p_K$, où $\varpi_K$ est une uniformisante de $K$. Puisque $K^\times\cap (1+ \mathfrak{p}_K\mathfrak{N})=U_K^1$ et $K^\times \cap U^{\tilde{k}_F(\gamma)+1}(\mathfrak{A})= U_K^{\tilde{k}_F(\gamma)+1}$, de la suite exacte (6), on déduit l'égalité
$$
\vert D_F(\gamma)\vert {\rm vol}\!\left(K^\times(1+ \mathfrak{p}_K\mathfrak{N}_{k_F(\gamma)}(\gamma, \mathfrak{A})),
\textstyle{dg\over dg_\gamma}\right)= q^{-f_\gamma \sigma_\gamma}
{\rm vol}\!\left(K^\times U^{\tilde{k}_F(\gamma)+1}(\mathfrak{A})),
\textstyle{dg\over dg_\gamma}\right).
$$
Or $q^{-f_\gamma \sigma_\gamma}= q^{-(\delta_\gamma - (N-f_\gamma))}$, d'où l'égalité du lemme dans le cas où $\gamma$ est séparable (d'après (4)). Si $\gamma$ n'est pas séparable, ce qui n'est possible que si $F$ est de caractéristique $p$, on déduit le résultat de la caractéristique nulle via la théorie des corps proches \cite{D,L3}. 
\end{proof}

Pour $\gamma\in G_{\rm qre}$, notons $f\mapsto I^G(\gamma,f)$ la distribution normalisée sur $G$ définie par
$$
I^G(\gamma,f) =\eta_G(\gamma)^{1\over 2}\,\ES{O}_\gamma(f).\leqno{(7)}
$$
Tout comme $\ES{O}_\gamma$, elle dépend du choix d'une mesure $G$--invariante $\textstyle{dg\over dg_\gamma}$ sur l'espace quotient $G_\gamma \backslash G$, et comme on a normalisé $dg_\gamma$ par ${\rm vol}(F^\times \backslash F[\gamma]^\times , \textstyle{dg_\gamma \over dz})=1$, elle ne dépend en fait que des choix de $dg$ et $dz$.

Pour $\gamma\in \mathfrak{g}_{\rm qre}$, on définira plus loin (\ref{variante sur l'algèbre de Lie}.(9)) une distribution normalisée
$$
\mathfrak{f}\mapsto I^\mathfrak{g}(\gamma,\mathfrak{f})=\eta_\mathfrak{g}(\gamma)^{1\over 2} \ES{O}_\gamma(\mathfrak{f})
$$
sur $\mathfrak{g}$, qui est une variante naturelle de la distribution normalisée sur $G$ définie par (7). On définira aussi (\ref{variante sur l'algèbre de Lie}.(1)) 
une variante additive $\mu_F^+(\gamma)$ de la constante définie par (1), qui vérifie l'analogue de l'égalité du 
lemme 1 (\ref{variante sur l'algèbre de Lie}.(5)): $\eta_\mathfrak{g}(\gamma) \mu_F^+(\gamma)=1$. Notons que si $\gamma\in G_{\rm qre}$ (i.e. si $N>1$, ou si $N=1$ et $\gamma\neq 0$), les constantes 
$\mu_F(\gamma)$ et $\mu_F^+(\gamma)$ sont liées par la formule (\ref{variante sur l'algèbre de Lie}.(2)): $\mu_F^+(\gamma)= q_{F[\gamma]}^{n_F(\gamma)(1-N)}\mu_F(\gamma)$.

\vskip1mm
Reprenons maintenant les hypothèses et les notations de \ref{le principe de submersion}: $\beta$ est un élément pur de $G$ tel que $E=F[\beta]\neq F$, $\mathfrak{b}={\rm End}_E(V)$ et $H={\rm Aut}_E(V)$. On pose $d= {N\over [E:F]}$. En particulier, on a forcément $N>1$, mais on peut avoir $d=1$ (si $\beta$ est quasi--régulier elliptique). Pour $b\in \mathfrak{b}_{\rm qre}$, on définit comme en 
\ref{variante sur l'algèbre de Lie} la distribution 
normalisée $\mathfrak{f}\mapsto I^{\mathfrak{b}}(b,\mathfrak{f})$ sur $\mathfrak{b}$, i.e. on pose
$$
I^{\mathfrak{b}}(b,\mathfrak{f})= \eta_{\mathfrak{b}}(b)^{1\over 2}\,\ES{O}_b^{\mathfrak{b}}(b,\mathfrak{f}).\leqno{(8)}
$$
Puisque $\beta\in A(E)_{\rm qre}^\times$ et $E\neq F$, on peut définir $\mu_F(\beta)$ tout comme on a défini $\mu_F(\gamma)$ pour $\gamma\in G_{\rm qre}$, \cad en posant
$$
\mu_F(\beta)= {{\rm vol}(E^\times (1+ \mathfrak{p}_{E}\mathfrak{N}_{k_F(\beta)}(\beta,\mathfrak{A}(E))),\textstyle{d\bar{g}_E})\over {\rm vol}(E^\times U^{k_F(\beta)+ n_F(\beta)+1}(\mathfrak{A}(E)),\textstyle{d\bar{g}_E})},
$$
où $d\bar{g}_E$ est une mesure de Haar sur $E^\times \backslash A(E)^\times$. On définit aussi la variante additive $\mu_F^+(\beta)$ de $\mu_F(\beta)$ en posant
$$
\mu_F^+(\beta)= {{\rm vol}(E+ \mathfrak{N}_{k_F(\beta)}(\beta,\mathfrak{A}(E)),\textstyle{\mathfrak{d}\bar{g}_E})\over {\rm vol}(E + \mathfrak{P}^{k_F(\beta)}(E)),\textstyle{\mathfrak{d}\bar{g}_E})},
$$
où $\mathfrak{d}\bar{g}_E$ est une mesure de Haar sur $A(E)/E$. Les constantes $\mu_F(\beta)$ et $\mu_F^+(\beta)$ sont liées par la formule 
(\ref{variante sur l'algèbre de Lie}.(2)): $\mu_F^+(\beta)=q_E^{n_F(\beta)(1-[E:F])}\mu_F(\beta)$.

% lemme 2
\begin{monlem2}Pour $b\in \mathfrak{b}_{\rm qre}\cap \underline{\mathfrak{Q}}^{\underline{k}_0+1}$, on a
$$
\mu_F(\beta + \bs{x}_0\otimes b) = q_E^{n_F(\beta)(d^2-d)}\mu_F(\beta)^{d^2} \mu_E^+(b).
$$
En particulier, la fonction $b\mapsto \eta_G(\beta + \bs{x}_0\otimes b)\eta_{\mathfrak{b}}(b)^{-1}$ 
est constante sur $\mathfrak{b}_{\rm qre}\cap \underline{\mathfrak{Q}}^{\underline{k}_0+1}$.
\end{monlem2}

\begin{proof}
Soit $b\in \mathfrak{b}_{\rm qre}\cap \underline{\mathfrak{Q}}^{\underline{k}_0+1}$, et soit $\gamma= \beta + \bs{x}_0\otimes b$. Quitte  remplacer $b$ 
par un conjugu dans $H$, on peut supposer que $\mathfrak{B}_b$ contient $\underline{\mathfrak{B}}$. On a donc 
l'identification $\mathfrak{A}_{\beta,b}= \mathfrak{A}(E)\otimes_{\mathfrak{o}_E}\mathfrak{B}_b$. Posons $\mathfrak{A}=\mathfrak{A}_{\beta,b}$, $\mathfrak{B}=\mathfrak{B}_b$, $\mathfrak{P}=\mathfrak{P}_{\beta,b}$, $\mathfrak{Q}=\mathfrak{Q}_b$ et $n=n_F(b,\beta)$.  Posons aussi $k=k_F(b,\beta)$ et $\tilde{k}= \tilde{k}_F(b,\beta)\;(=k+n)$. Enfin posons $\mathfrak{N}_\gamma= \mathfrak{N}_k(\gamma, \mathfrak{N})$ et (si $d>1$) $\mathfrak{N}_b= \mathfrak{N}_k(b, \mathfrak{B})$. Il s'agit de prouver que la quantité $\lambda= \mu_F(\gamma)\mu_E^+(b)^{-1}$ ne dépend que de $\beta$ (et pas de $b$), et de la calculer. Posons aussi $\lambda^+= 
\mu_F^+(\gamma)\mu_E^+(b)^{-1}$.

Supposons $d=1$ (il faut bien commencer!). Alors $\mu_E^+(b)=1$ et
$$
\lambda={ {\rm vol}\!\left(F[\gamma]^\times (1+ \mathfrak{p}_{F[\gamma]}\mathfrak{N}_\gamma),\textstyle{dg\over dg_\gamma}\right)
\over {\rm vol}\!\left(F[\gamma]^\times U^{\tilde{k}+1}(\mathfrak{A}),\textstyle{dg\over dg_\gamma}\right)}
= {{\rm vol}(1+ \mathfrak{p}_{F[\gamma]}\mathfrak{N}_\gamma,dg)\over {\rm vol}(U^{\tilde{k}+1}(\mathfrak{A}),dg)}
{ {\rm vol}(U_{F[\gamma]}^{\tilde{k}+1},dg_\gamma)\over 
{\rm vol}(U_{F[\gamma]}^1,dg_\gamma) }.
$$
Puisque $k=k_F(\beta)$ et $\tilde{k}-k= n = n_F(\beta)$, et que $\mathfrak{A}$ est l'unique $\mathfrak{o}$--ordre héréditaire dans $\mathfrak{g}$ normalisé par $E^\times$, le terme ${\rm vol}(U^{\tilde{k}+1}(\mathfrak{A}),dg)$ ne dépend pas de $b$. D'autre part, puisque les strates $[\mathfrak{A}, n, -k-1, \beta]$ et $[\mathfrak{A}, n , -k-1, \gamma]$ dans $\mathfrak{g}$ sont simples et équivalentes, posant $\mathfrak{N}_\beta = \mathfrak{N}_k(\beta,\mathfrak{A})$, on a $\mathfrak{p}_{F[\gamma]}\mathfrak{N}_\gamma = \mathfrak{p}_E\mathfrak{N}_\beta$ \cite[2.1.3]{BK}, et le terme ${\rm vol}(1+ \mathfrak{p}_{F[\gamma]}\mathfrak{N}_\gamma,dg)$ ne dépend pas de $b$. Enfin on a
$$
{  {\rm vol}(U_{F[\gamma]}^{\tilde{k}+1},dg_\gamma)\over 
{\rm vol}(U_{F[\gamma]}^1,dg_\gamma)}=[\mathfrak{o}_{F[\gamma]}: \mathfrak{p}_{F[\gamma]}^{\tilde{k}}]^{-1}= [\mathfrak{o}_E:\mathfrak{p}_E^{\tilde{k}}]^{-1}.
$$
En définitive, on a montré que $\lambda= \mu_F(\beta)$.

Supposons maintenant $d>1$. Posons $E_0=E[b]$ et $K=F[\gamma]$, et supposons tout d'abord que $b$ est $E$--minimal. Alors on a $\mathfrak{N}_b= \mathfrak{B}$, $k=-n_E(b)$ et 
$$
\mu_E^+(b)={[\mathfrak{B}: \mathfrak{Q}^k]\over [\mathfrak{o}_{E_0}: \mathfrak{p}_{E_0}^{k}]}=  q_{E_0}^{k(d-1)}.
$$
D'autre part, posant $k_0= k_0(\beta, \mathfrak{A})$ et $\tilde{k}_0= k_0 +n$, on a \cite[1.4.9]{BK}
$$
\mathfrak{N}_\gamma = \mathfrak{N}_k(\beta,\mathfrak{A})= \mathfrak{B} + \mathfrak{Q}^{k-k_0}\mathfrak{N}_{k_0}(\beta,\mathfrak{A}),
$$
d'où
$$
\mu_F(\gamma)= {[\mathfrak{N}_\gamma: \mathfrak{P}^{\tilde{k}}]\over [\mathfrak{o}_K: \mathfrak{p}_K^{\tilde{k}}]}= 
{[\mathfrak{B}:\mathfrak{Q}^{k-k_0}][\mathfrak{N}_{k_0}(\beta,\mathfrak{A}):\mathfrak{P}^{\tilde{k}_0}] \over [\mathfrak{o}_K: \mathfrak{p}_K^{\tilde{k}}]}.
$$
On obtient
$$
\lambda = [\mathfrak{N}_{k_0}(\beta,\mathfrak{A}):\mathfrak{P}^{\tilde{k}_0}]{[\mathfrak{B}:\mathfrak{Q}^{k-k_0}]\over [\mathfrak{B}:\mathfrak{Q}^k]}{[\mathfrak{o}_{E_0}: \mathfrak{p}_{E_0}^{k}]\over [\mathfrak{o}_K: \mathfrak{p}_K^{\tilde{k}}]}.
$$
Or on a $[\mathfrak{B}:\mathfrak{Q}^{k-k_0}]={[\mathfrak{B}: \mathfrak{Q}^{k+n}]\over [\mathfrak{B}:\mathfrak{Q}^{\tilde{k}_0}]}$, d'où
$$
\lambda = {[\mathfrak{N}_{k_0}(\beta,\mathfrak{A}):\mathfrak{P}^{\tilde{k}_0}] \over [\mathfrak{B}: \mathfrak{Q}^{\tilde{k}_0}]}
{[\mathfrak{B}:\mathfrak{Q}^n]\over [\mathfrak{o}_{E_0}:\mathfrak{p}_{E_0}^n]}.
$$
Puisque $\tilde{k}_0= \tilde{k}_F(\beta)e(\mathfrak{B}\vert \mathfrak{o}_E)$ avec $\tilde{k}_F(\beta)=k_F(\beta) + n_F(\beta)$, on a
$$
\mathfrak{P}^{\tilde{k}_0}= \mathfrak{A}(E)\otimes_{\mathfrak{o}_E}\mathfrak{Q}^{\tilde{k}_0}= \mathfrak{P}^{\tilde{k}_F}(\beta)(E)\otimes_{\mathfrak{o}_E}\mathfrak{B}.
$$
Or on a aussi
$$
\mathfrak{N}_{k_0}(\beta,\mathfrak{A})= \mathfrak{N}_{k_F(\beta)}(\beta,\mathfrak{A}(E))\otimes_{\mathfrak{o}_E}\mathfrak{B},
$$
et comme $\mathfrak{B}$ est un $\mathfrak{o}_E$--module libre de rang $d^2$, le terme 
${[\mathfrak{N}_{k_0}(\beta,\mathfrak{A}):\mathfrak{P}^{\tilde{k}_0}] \over [\mathfrak{B}: \mathfrak{Q}^{\tilde{k}_0}]}$ vaut
$$
{[\mathfrak{N}_{k_F(\beta)}(\beta,\mathfrak{A}(E)): \mathfrak{P}^{\tilde{k}_F(\beta)}(E)]^{d^2}\over [\mathfrak{o}_E:\mathfrak{p}_E^{\tilde{k}_F(\beta)}]^{d^2}}
=\mu_F(\beta)^{d^2}.
$$ 
Quant au terme ${[\mathfrak{B}:\mathfrak{Q}^n]\over [\mathfrak{o}_{E_0}:\mathfrak{p}_{E_0}^n]}$, il vaut 
$q_{E_0}^{n(d-1)}= q_E^{n_F(\beta)(d^2-d)}$. On obtient donc (dans le cas où $b$ est $E$--minimal)
$$
\lambda = q_E^{n_F(\beta)(d^2-d)}\mu_F(\beta)^{d^2}.\leqno{(9)}
$$
On en déduit (toujours dans le cas où $b$ est $E$--minimal)
$$
\lambda^+ = q_{F[\gamma]}^{n_F(\gamma)(1-N)} \lambda = q_{F[\gamma]}^{n_F(\gamma)(1-N)}
q_E^{n_F(\beta)(d^2-d)}q_E^{n_F(\beta)([E:F]-1)d^2}\mu_F^+(\beta)^{d^2}.
$$
Or $q_{F[\gamma]}^{n_F(\gamma)}= q_E^{n_F(\beta)d}$, d'où (toujours dans le cas où $b$ est $E$--minimal)
$$
\lambda^+ =\mu_F^+(\beta)^{d^2}.\leqno{(10)}
$$

Supposons maintenant que $b$ n'est pas $E$--minimal (on a donc forcément $E[b]\neq E$, i.e. $d>1$). Alors on reprend la première partie de la preuve de la proposition de \ref{le résultat principal}. On pose $r= n_E(b)$ --- on a donc $k=k_E(b)>-r$ --- et on considère la strate pure $[\mathfrak{B},r,-k,b]$ dans $\mathfrak{b}$. On écrit écrit $b= b_1+ \bs{y}_1\otimes c$ comme dans loc.~cit. En particulier, la strate $[\mathfrak{B},r,-k,b_1]$ dans $\mathfrak{b}$ est simple et équivalente à $[\mathfrak{B},r,-k,b]$, l'extension $E_1=E[b_1]$ de $E$ vérifie $[E_1:E]<d$, et $c$ est un élément quasi--régulier elliptique et $E_1$--minimal de $\mathfrak{b}_1= {\rm End}_{E_1}(V)$. De plus, avec les identifications de loc.~cit., l'élément $\gamma$ s'écrit
$$
\gamma = \beta + \bs{x}_0(b_1 + \bs{y}_1 \otimes c)= \beta_1 + \bs{x}_1\otimes c,
$$
et on a
$$
k_F(\beta_1)= \left\{\begin{array}{ll}
k_E(b_1) & \mbox{si $E_1\neq E$}\\
k_F(\beta) & \mbox{sinon}
\end{array}\right..
$$
et
$$k\;(=k_E(b))=k_{E_1}(c)=-n_{E_1}(c).
$$
Par récurrence sur la dimension, on peut supposer que le résultat que l'on veut démontrer est vrai pour le couple $(\beta,b_1)$: posant $d_1= [E_1:E]$, d'après (9), on a
$$
\mu_F(\beta_1)\mu_E^+(b_1)^{-1}= \mu_F(\beta)^{d_1^2}q_E^{n_F(\beta)(d_1^2-d_1)}.\leqno{(11)}
$$
Puisque $c$ est $E_1$--minimal, le résultat est vrai pour le couple $(\beta_1,c)$: posant $d'= {d\over d_1}$, on a
$$
\mu_F(\gamma)\mu_{E_1}^+(c)^{-1}= \mu_F(\beta_1)^{d'^2}q_{E_1}^{n_F(\beta_1)(d'^2-d')},
$$
d'où, puisque $q_{E_1}^{n_F(\beta_1)}= q_E^{n_F(\beta)d_1}$ et $d_1d'=d$,
$$
\mu_F(\gamma)\mu_{E_1}^+(c)^{-1}= \mu_F(\beta_1)^{d'^2}q_{E}^{n_F(\beta)d(d'-1)}.\leqno{(12)}
$$
En fait, a priori le couple $(\beta_1,c)$ n'est pas exactement de la forme voulue, puisque l'élément $c$ n'appartient pas à ${\rm End}_{E'_1}(V)$, $E'_1=F[\beta_1]$. Mais d'après le lemme 1 de \ref{approximation}, on peut toujours se ramener à un élément de la forme voulue en conjugant $\gamma$ dans $U^1(\mathfrak{A})$. On obtient un élément  de la forme voulue $\beta_1+ \bs{x}'_1\otimes c'$, avec $c'$ quasi--régulier elliptique dans $\mathfrak{b}'_1= {\rm End}_{E'_1}(V)$ et $E'_1$--minimal (d'après le corollaire 1 de \ref{raffinement}), et 
cette opération n'affecte pas la valeur de $\mu_{E_1}^+(c)$: on a $\mu_{E'_1}^+(c')= \mu_{E_1}^+(c)$. Comme on a aussi $q_{E_1}=q_{E'_1}$, on en déduit (12). 
Reste à traiter le couple $(b_1,c)$. Si $E_1\neq E$, alors puisque $c$ est $E$--minimal, d'après (10), on a
$$
\mu_E^+(b)\mu_{E_1}^+(c)^{-1}= \mu_E^+(b_1)^{d'^2}.\leqno{(13)}
$$
Si $E_1=E$, i.e. si $d_1=1$, alors l'égalité (13) reste vraie, puisque dans ce cas on a $\mu_E^+(b_1)=1$, $k_E(c)=k_E(b)=k$ et 
$\mathfrak{N}_k(b,\mathfrak{B})=\mathfrak{N}_k(c,\mathfrak{B})$, d'où $\mu_E^+(b)\mu_{E_1}^+(c)^{-1}=1$. 
En rassemblant les égalités (11), (12) et (13), on obtient la formule cherchée pour $\lambda$.
\end{proof}

Compte--tenu du lemme 2, la proposition suivante et son corollaire sont de simples conséquences de la proposition de \ref{le principe de submersion}.

% proposition
\begin{mapropo}
Il existe une constante $\lambda >0$ telle que pour 
tout $b\in \mathfrak{b}_{\rm qre}\cap \underline{\mathfrak{Q}}^{\underline{k}_0+1}$, posant $\gamma= \beta + \bs{x}_0\otimes b$, on a
$$
\theta_{\ES{O}_\gamma}= \lambda \ES{O}_b^{\mathfrak{b}}.
$$
Si de plus les mesures $\mathfrak{d}b'$ sur $\mathfrak{b}$ et $dh$ sur $H$ sont associées, \cad vérifient $dh= \mathfrak{d}^\times b'$, 
et si les mesures de Haar $dz$ sur $Z=F^\times$ et $dz_E$ sur $Z_E=E^\times$ vérifient ${\rm vol}(F^\times\backslash E^\times, {dz_E\over dz})=1$, 
alors cette constante $\lambda$ vaut
$$
(q_E^{n_F(\beta)} \mu_F(\beta))^{d^2}=\vert \beta \vert_E^d \textstyle{\eta_{\mathfrak{b}}(b)\over \eta_G(\gamma)}.
$$
\end{mapropo}

\begin{proof}
Soit $b\in \mathfrak{b}_{\rm qre}\cap \underline{\mathfrak{Q}}^{\underline{k}_0+1}$, et soit $\gamma= \beta + \bs{x}_0\otimes b$. D'après la proposition de \ref{le principe de submersion}, on a
$$
\theta_{\ES{O}_\gamma} = I_G^H(\beta,b){v_F(\gamma)\over v_E(b)} \ES{O}_{b}^{\mathfrak{b}}.
$$
D'après le lemme 2, il suffit de montrer qu'il existe une constante $\mu$ (indépendante de $b)$, telle que 
$$
I_G^H(\beta,b){v_F(\gamma)\over v_E(b)}=\mu {\mu_F(\gamma)\over \mu_E^+(b)}.
$$
Posons $\mathfrak{A}=\mathfrak{A}_{\beta,b}$, $\mathfrak{B}=\mathfrak{B}_b$, $\mathfrak{P}=\mathfrak{P}_{\beta,b}$, $\mathfrak{Q}=\mathfrak{Q}_b$ et $n=n_F(b,\beta)$. Posons aussi $k=k_F(b,\beta)$ et $\tilde{k}= \tilde{k}_F(b,\beta)$, $K=F[\gamma]$ et $E_0=E[b]$. On a
$$
{v_F(\gamma)\over {\rm vol}(U^{\tilde{k}+1}(\mathfrak{A}),dg)}= {\mu_F(\gamma)\over {\rm vol}(U^{\tilde{k}+1}_K,dg_\gamma)}
= \mu_F(\gamma){[\mathfrak{o}_K:\mathfrak{p}_K^{\tilde{k}}]\over {\rm vol}(U^1_K,dg_\gamma)}.\leqno{(14)}
$$
Posons
$$
c= {e(E/F){\rm vol}(U_E,dz_E)\over {\rm vol}(U_F,dz)}\;\left(={\rm vol}(F^\times\backslash E^\times ,\textstyle{dz_E\over dz})\right).
$$

Commen\c{c}ons par supposer $d=1$. Alors $\mathfrak{b}=E$, $\mu_E^+(b)=1$ et
$$
{v_E(b)\over {\rm vol}(\mathfrak{Q}^{k+1},\mathfrak{d}b')}={v_E(b)\over {\rm vol}(\mathfrak{p}_E^{k+1},\mathfrak{d}b')}
=\mu_E^+(b){{\rm vol}(U_E,dh)\over {\rm vol}(U_E,dz_E)}{[\mathfrak{o}_E:\mathfrak{p}_E^k]\over {\rm vol}(\mathfrak{p}_E,\mathfrak{d}b')}.\leqno{(15)}
$$ 
On a $e(K/F)=e(E/F)$ et $f(K/F)=f(E/F)$, par conséquent le volume 
${\rm vol}(U^1_K,dg_\gamma)$ est égal à
$$
(q_E-1)^{-1}{\rm vol}(U_K,dg_\gamma)= (q_E-1)^{-1}c^{-1}{\rm vol}(U_E,dz_E).
$$
D'autre part comme $k=k_F(\beta)$, $n=n_F(\beta)$, ${[\mathfrak{o}_K:\mathfrak{p}_K^{\tilde{k}}] \over [\mathfrak{o}_E:\mathfrak{p}_E^k]}=[\mathfrak{o}_E:\mathfrak{p}_E^n]= q_E^n$ et
$$
{\rm vol}(U_E,dh)=(q_E-1){\rm vol}(U_E^1,dh),
$$
en combinant (14) et (15), on obtient
$$
I_G^H(\beta,b){v_F(\gamma)\over v_E(b)}= c\,q_E^n{{\rm vol}(\mathfrak{p}_E,\mathfrak{d}b')\over {\rm vol}(U_E^1,dh)}{\mu_F(\gamma)\over \mu_E^+(b)}.
$$

Supposons maintenant $d>1$, et posons $\mathfrak{N}_b= \mathfrak{N}_k(b,\mathfrak{B})$. Alors on a
$$
{v_E(b)\over {\rm vol}(\mathfrak{Q}^{k+1},\mathfrak{d}b')}={{\rm vol}(1+ \mathfrak{p}_{E_0}\mathfrak{N}_b,dh)\over {\rm vol}(U_{E_0}^1, dh_b){\rm vol}(\mathfrak{Q}^{k+1},\mathfrak{d}b')}={[\mathfrak{p}_{E_0}\mathfrak{N}_b: \mathfrak{Q}^{k+1}]\over {\rm vol}(U_{E_0}^1,dh_b)}
{{\rm vol}(1+ \mathfrak{p}_{E_0}\mathfrak{N}_b,dh)\over {\rm vol}(\mathfrak{p}_{E_0}\mathfrak{N}_b,\mathfrak{d}b')}.
$$
Or $[\mathfrak{p}_{E_0}\mathfrak{N}_b: \mathfrak{Q}^{k+1}]=[\mathfrak{N}_b: \mathfrak{Q}^k]=\mu_E^+(b)[\mathfrak{o}_{E_0}:\mathfrak{p}_{E_0}^k]$ et ${{\rm vol}(1+ \mathfrak{p}_{E_0}\mathfrak{N}_b,dh)\over {\rm vol}(\mathfrak{p}_{E_0}\mathfrak{N}_b,\mathfrak{d}b')}={{\rm vol}(U^1(\mathfrak{B}),dh)\over {\rm vol}(\mathfrak{Q},\mathfrak{d}b')}$, d'où
$$
{v_E(b)\over {\rm vol}(\mathfrak{Q}^{k+1},\mathfrak{d}b')}=\mu_E^+(b){[\mathfrak{o}_{E_0}:\mathfrak{p}_{E_0}^k]\over {\rm vol}(U_{E_0}^1,dh_b)}
{{\rm vol}(U^1(\mathfrak{B}),dh)\over {\rm vol}(\mathfrak{Q},\mathfrak{d}b')}.\leqno{(16)}
$$
Puisque $e(E_0/F)=e(K/F)$ et $f(E_0/F)=e(K/F)$, on a (comme dans le cas $d=1$)
$$
{\rm vol}(U^1_K,dg_\gamma)= (q_{E_0}-1)^{-1} {\rm vol}(U_K,dg_\gamma)
$$
et
$$
{\rm vol}(U_K,dg_\gamma)= e(E_0/F)^{-1}{\rm vol}(U_F,dz)= c^{-1} e(E_0/E)^{-1}{\rm vol}(U_E,dz_E),
$$
d'où
$$
{\rm vol}(U^1_K,dg_\gamma)=c^{-1}{\rm vol}(U^1_{E_0},dh_b).
$$
On en déduit que
$$
{[\mathfrak{o}_{K}:\mathfrak{p}_{K}^{\tilde{k}}]\over {\rm vol}(U_{K}^1,dg_\gamma)}=c{[\mathfrak{o}_{E_0}:\mathfrak{p}_{E_0}^{\tilde{k}}]\over {\rm vol}(U_{E_0}^1,dh_b)}
= c{[\mathfrak{o}_{E_0}:\mathfrak{p}_{E_0}^k]\over {\rm vol}(U_{E_0}^1,dh_b)}[\mathfrak{p}_K^k:\mathfrak{p}_K^{\tilde{k}}]
$$
avec 
$$
[\mathfrak{p}_K^{k}:\mathfrak{p}_K^{\tilde{k}}]=q_K^{n}= q^{f(K/F)n_F(\beta)e(E_0/E)}=q_E^{n_F(\beta)d}.
$$ D'autre part, si $\underline{\mathfrak{B}}$ est un  
$\mathfrak{o}_E$--ordre héréditaire minimal $\mathfrak{b}$ tel que $\underline{\mathfrak{B}}\subset \mathfrak{B}$, posant $\underline{\mathfrak{Q}}= {\rm rad}(\underline{\mathfrak{B}})$, on a 
${{\rm vol}(U^1(\mathfrak{B}),dh)\over {\rm vol}(\mathfrak{Q},\mathfrak{d}b')}={{\rm vol}(U^1(\underline{\mathfrak{B}}),dh)\over {\rm vol}(\underline{\mathfrak{Q}},\mathfrak{d}b')}$, et cette quantité ne dépend pas de $b$. En combinant (14) et (16), on obtient
$$
I_G^H(\beta,b){\mu_F(\gamma)\over \mu_E(b)}= c\,q_E^{n_F(\beta)d}{{\rm vol}(\underline{\mathfrak{Q}},\mathfrak{d}b')\over 
{\rm vol}(U^1(\underline{\mathfrak{B}}),dh)}
{\mu_F(\gamma)\over \mu_E^+(b)}.
$$

La dernière assertion résulte de la formule du lemme 2, puisque si $dh= \mathfrak{d}^\times b'$, on a ${\rm vol}(U^1(\underline{\mathfrak{B}}),dh)= {\rm vol}(\underline{\mathfrak{Q}},\mathfrak{d}b')$, avec $U^1(\underline{\mathfrak{B}})= U_E^1=1+ \mathfrak{p}_E$ et $\underline{\mathfrak{Q}}=\mathfrak{p}_E$ si $d=1$. 
Cela achève la démonstration de la proposition.
\end{proof}

% remarque 1
\begin{marema1}
{\rm 
La mesure de Haar $\mathfrak{d}b'$ sur $\mathfrak{b}$ est utilisée pour \og descendre \fg une distribution $G$--invariante au voisinage de $\beta$ dans $G$ en une distribution 
$H$--invariante au voisinage de $0$ dans $\mathfrak{b}$, \cad pour définir l'application $T\mapsto \theta_T$, alors que la mesure de Haar $dh$ sur $H= \mathfrak{b}^\times$ est 
utilisée pour définir les intégrales orbitales quasi--régulières elliptiques sur $\mathfrak{b}$. Pour définir l'application 
$T\mapsto \theta_T$, on est passé du groupe $G$ à l'algèbre de Lie de $H$ via l'application $X\mapsto 1+X$ au voisinage de $0$ dans $\mathfrak{b}$. Il est donc naturel d'imposer que les mesures de Haar $\mathfrak{d}b'$ sur $\mathfrak{b}$ et $dh$ sur $H$ soient associées. D'autre part, on a normalisé les intégrales orbitales quasi--régulières elliptiques sur $G$ à l'aide d'une mesure de Haar $dz$ sur $Z=F^\times$. La normalisation naturelle de cette mesure est ${\rm vol}(U_F,dz)=1$. De même pour $H$, {\it vu comme groupe réductif connexe sur $E$}, la normalisation naturelle des intégrales orbitales quasi--régulières elliptique sur $\mathfrak{b}$ est celle donnée par la mesure de Haar $dz_E$ sur $Z_H=E^\times$ telle que ${\rm vol}(U_E,dz_E)=1$. En ce cas on a l'égalité 
${\rm vol}(F^\times\backslash E^\times,{dz_E\over dz})=e(E/F)$. En revanche, imposer la condition ${\rm vol}(F^\times\backslash E^\times,{dz_E\over dz})=1$ consiste à voir $H$ comme le groupe des points $F$--rationnels d'un groupe algébrique défini sur $F$ (obtenu comme restriction à la Weil d'un groupe algébrique réductif connexe défini et déployé sur $E$), et $Z$ comme la composante $F$--déployée de $Z_H$: pour $b\in \mathfrak{b}_{\rm qre}$, on demande que $dh_b$ soit la mesure de Haar sur $H_b= E[b]^\times$ telle que ${\rm vol}(F^\times\backslash H_b, {dh_b\over dz})=1$, \cad telle que
$$
e([E[b]/F){\rm vol}(U_{E[b]},dh_b)={\rm vol}(U_F,dz)\;(=1). \eqno{\blacksquare}
$$
}
\end{marema1}
 
% corollaire
\begin{moncoro}
\begin{enumerate}
\item[(i)]
Pour toute fonction $f\in C^\infty_{\rm c}(G)$, il existe une fonction $f^{\mathfrak{b}}\in C^\infty_{\rm c}(\underline{\mathfrak{Q}}^{\underline{k}_0+1})$ telle que pour tout 
$b\in \mathfrak{b}_{\rm qre}\cap \underline{\mathfrak{Q}}^{\underline{k}_0+1}$, on a l'égalité
$$
I^G(\beta + \bs{x}_0\otimes b,f) = I^{\mathfrak{b}}(b,f^\mathfrak{b}).
$$
\item[(ii)] Pour toute fonction $f\in C^\infty_{\rm c}(G)$, la fonction $G_{\rm qre}\rightarrow {\Bbb C},\,\gamma\mapsto I^G(\gamma,f)$ est localement constante. 
\end{enumerate}
\end{moncoro}

\begin{proof}
On obtient le point (i) comme dans la remarque 2 de \ref{le principe de submersion} (en utilisant le fait que ${\rm Im}(\delta)$ est ouvert fermé et $G$--invariant dans $G$: pour $f\in C^\infty_{\rm c}(G)$, la fonction $f\vert_{{\rm Im}(\delta)}$ appartient à $C^\infty_{\rm c}({\rm Im}(\delta))$. Quant au point (ii), il suffit de voir que si $\beta\in G_{\rm qre}$ (i.e. si $d=1$), alors pour tout 
$b\in \mathfrak{b}_{\rm qre}\cap \underline{\mathfrak{Q}}^{\underline{k}_0+1}$, posant $\gamma = \beta + \bs{x}_0\otimes b$, 
on a $\mu_F(\beta + \bs{x}_0\otimes b)=\mu_F(\beta)$ et $\mu_E^+(b)=1$. On peut alors appliquer le point (i), en remarquant que pour tout fonction $\mathfrak{f}\in C^\infty_{\rm c}(\mathfrak{b})$ et tout élément $b\in \mathfrak{b}=E$, on a $I^{\mathfrak{b}}(b,\mathfrak{f})= \ES{O}^\mathfrak{b}_b(\mathfrak{f})=
{{\rm vol}(U_E,dh)\over {\rm vol}(U_E,dz_E)}\mathfrak{f}(b)$. 
\end{proof}

% remarque 2
\begin{marema2}
{\rm Les résultats de cette section \ref{descente centrale au voisinage d'un élément pur} ne concernent que les élements quasi--réguliers {\it elliptiques} au voisinage d'un élément pur. On verra en \ref{descente centrale au voisinage d'un élément pur (suite)} et \ref{descente centrale au voisinage d'un élément fermé} qu'on peut les étendre à tous les éléments quasi--réguliers au voisinage d'un élément fermé. \hfill $\blacksquare$
}
\end{marema2}

\subsection{Variante sur l'algbre de Lie}\label{variante sur l'algbre de Lie}
Pour $\gamma\in \mathfrak{g}_{\rm qre}$, on peut définir comme en \ref{IO normalises}.(7) une distribution normalisée $\mathfrak{f}\mapsto I^{\mathfrak{g}}(\gamma, \mathfrak{f})$ sur $\mathfrak{g}$. On commence par définir la distribution $\mathfrak{f}\mapsto \ES{O}_\gamma(\mathfrak{f})=\ES{O}_\gamma^{\mathfrak{g}}(\mathfrak{f})$ sur $\mathfrak{g}$ comme on a défini la distribution $\ES{O}_\gamma$ sur $G$, \cad en intégrant la fonction $\mathfrak{f}\in C^\infty_{\rm c}(\mathfrak{g})$ sur l'orbite $\ES{O}_G(\gamma)$ grâce à la mesure $\textstyle{dg\over dg_\gamma}$ sur $G_\gamma \backslash G$, où $dg_\gamma$ est la mesure de Haar sur $F[\gamma]^\times$ telle que ${\rm vol}(F^\times \backslash F[\gamma]^\times, \textstyle{dg_\gamma\over dz})=1$. Notons que si $F[\gamma]=\mathfrak{g}$ (ce qui n'est possible que si $N=1$, i.e. si $\mathfrak{g}=F$), alors on a $\ES{O}_\gamma(\mathfrak{f})= {\rm vol}(U_F,dg)\mathfrak{f}(\gamma)$. On pose
$$
\mu^+_F(\gamma)= \left\{\begin{array}{ll}
{{\rm vol}(F[\gamma] + \mathfrak{N}_{k_F(\gamma)}(\gamma,\mathfrak{A}_\gamma),d\bar{y}) \over {\rm vol}(F[\gamma]+  \mathfrak{P}_\gamma^{k_F(\gamma)},d\bar{y})}&
\mbox{si $N>1$}\\
1 & \mbox{sinon}
\end{array}\right.,\leqno{(1)}
$$
où $d\bar{y}$ est une mesure de Haar sur $\mathfrak{g}/F[\gamma]$. La quantité $\mu_F^+(\gamma)$ peut être vue comme une version \og additive\fg de $\mu_F(\gamma)$. Notons (si $N>1$) que l'exposant $\tilde{k}_F(\gamma)$ dans l'expression au dénominateur de $\mu_F(\gamma)$ a été remplacé dans celle au dénominateur de $\mu_F^+(\gamma)$ par un exposant $k_F(\gamma)= \tilde{k}_F(\gamma)-n_F(\gamma)$, ce qui traduit l'égalité $\gamma U^{\tilde{k}_F(\gamma) +1}(\mathfrak{A}_\gamma)= \gamma + \mathfrak{P}_\gamma^{k_F(\gamma)+1}$. On vérifie (si $N>1$) 
que
$$
\mu_F^+(\gamma)= q_{F[\gamma]}^{n_F(\gamma)(1-N)}\mu_F(\gamma)\leqno{(2)}
$$
avec
$$
q_{F[\gamma]}^{n_F(\gamma)(1-N)}=q^{f_\gamma n_F(\gamma)(1-N)}= \vert \det (\gamma)\vert^{1-N}.
$$
Soit $\eta_{\mathfrak{g}}: \mathfrak{g}_{\rm qre}\rightarrow {\Bbb R}_{>0}$ la fonction définie par
$$
\eta_{\mathfrak{g}}(\gamma)= \left\{ \begin{array}{ll} q^{-f_\gamma (c_F(\gamma) + e_\gamma -1)} & \mbox{si $N>1$}\\
1 & \mbox{sinon}
\end{array}\right.,\leqno{(3)}
$$
où l'invariant $c_F(\gamma)$ a été défini en \ref{l'invariant c}. Rappelons que si $\gamma \neq 0$ (donc en particulier si $N>1$), on a 
$c_F(\gamma)= \tilde{c}_F(\gamma)- (N-1)n_F(\gamma)$. On a donc (si $N>1$, et même si $N=1$ et $\gamma\neq 0$, \cad si $\gamma\in G_{\rm qre}$)
$$
\eta_\mathfrak{g}(\gamma)= q^{f_\gamma n_F(\gamma)(N -1)}\eta_G(\gamma)=\vert \det(\gamma)\vert^{N-1}\eta_G(\gamma).\leqno{(4)}
$$
D'après le lemme 1 de \ref{IO normalisées}, pour $\gamma \in G_{\rm qre}$, on a
$$
\eta_{\mathfrak{g}}(\gamma)\mu_F^+(\gamma)= \eta_G(\gamma)\mu_F(\gamma)=1.\leqno{(5)}
$$
Enfin, pour $\gamma\in \mathfrak{g}_{\rm qre}$, on pose
$$
D_F^+(\gamma) = \left\{
\begin{array}{ll}\det_F(-{\rm ad}_\gamma; \mathfrak{g}/\mathfrak{g}_\gamma)& \mbox{si $N>1$}\\
1 & \mbox{sinon}
\end{array}\right..\leqno{(6)}
$$
Si $N>1$, on a donc
$$
D_F^+(\gamma)= {\rm det}_F(y \mapsto y \gamma; \mathfrak{g}/\mathfrak{g}_\gamma){\rm det}_F(1-{\rm Ad}_\gamma; \mathfrak{g}/\mathfrak{g}_\gamma)= 
\det(\gamma)^{N-1}D_F(\gamma).\leqno{(7)}
$$
D'après (4), on a (pour tout $\gamma \in \mathfrak{g}_{\rm qre}$):
\begin{enumerate}[leftmargin=17pt]
\item[(8)]si l'extension $F[\gamma]/F$ est séparable, alors $\eta_{\mathfrak{g}}(\gamma)= \vert D_F^+(\gamma)\vert q^{\delta_\gamma -(N-f_\gamma)}$.
\end{enumerate}
D'ailleurs pour tout $\gamma\in \mathfrak{g}_{\rm qre}$ (séparable ou non), on peut en déduire l'égalité 
$\eta_\mathfrak{g}(\gamma)\mu_F^+(\gamma)=1$ comme dans la preuve du lemme 1 de \ref{IO normalises}, grâce à la suite 
exacte \ref{IO normalises}.(5) pour $m=1$. Pour $\gamma\in \mathfrak{g}_{\rm qre}$, on note $\mathfrak{f}\mapsto I^{\mathfrak{g}}(\gamma,\mathfrak{f})$ la distribution normalisée sur $\mathfrak{g}$ définie par
$$
I^{\mathfrak{g}}(\gamma,\mathfrak{f})= \eta_{\mathfrak{g}}(\gamma)^{1\over 2}\,\ES{O}_\gamma(\mathfrak{f}).\leqno{(9)}
$$

\vskip1mm
D'après (4), pour $\gamma\in G_{\rm qre}$, on a $\eta_\mathfrak{g}(\gamma)= \vert \det(\gamma)\vert^{N-1}\eta_G(\gamma)$. 
Pour $b\in \mathfrak{b}_{\rm qre}\cap \underline{\mathfrak{Q}}^{\underline{k}_0+1}$, on a $\beta + \bs{x}_0\otimes b= \beta(1+ \beta^{-1}(\bs{x}_0\otimes b))$ 
avec $\beta^{-1}(\bs{x}_0\otimes b)\in \mathfrak{P}_\gamma$, par conséquent
$$
\vert \det(\gamma)\vert = \vert \det(\beta)\vert.
$$
En voyant ${\rm Im}(\delta)\subset G$ comme un ensemble ouvert fermé et $G$--invariant {\it dans $\mathfrak{g}$}, on en dduit la variante sur $\mathfrak{g}$ du corollaire de \ref{IO normalises}:

% corollaire
\begin{moncoro}
\begin{enumerate}
\item[(i)] Pour toute fonction $\mathfrak{f}\in C^\infty_{\rm c}(\mathfrak{g})$, il existe une fonction $\mathfrak{f}^{\mathfrak{b}}\in C^\infty_{\rm c}(\underline{\mathfrak{Q}}^{\underline{k}_0+1})$ telle que pour tout $b\in \mathfrak{b}_{\rm qre}\cap \underline{\mathfrak{Q}}^{\underline{k}_0+1}$, on a l'égalité
$$
I^\mathfrak{g}(\beta + \bs{x}_0\otimes b,\mathfrak{f}) = I^{\mathfrak{b}}(b,f^\mathfrak{b}).
$$
\item[(ii)] Pour toute fonction $\mathfrak{f}\in C^\infty_{\rm c}(\mathfrak{g})$, la fonction $\mathfrak{g}_{\rm qre}\rightarrow {\Bbb C},\,\gamma\mapsto I^\mathfrak{g}(\gamma,\mathfrak{f})$ est localement cons\-tante. 
\end{enumerate}
\end{moncoro}

%%%%%%%%%%%%%%%%%%%%%%%%%%%%%%%%%%%%%%
\section{Descente centrale: le cas général}\label{descente centrale: le cas général}

%%%%%%%%%%%%%%%%%
\subsection{Descente parabolique}\label{descente parabolique}Pour $\gamma\in G$, on note $D_G(\gamma)$ le coefficient de $t^N$ dans le polynôme ${\rm det}_F(t+1-{\rm Ad}_\gamma; \mathfrak{g})$, et on pose
$$
G_{\rm r}=\{\gamma \in G: D_G(\gamma)\neq 0\}.
$$
Un élément $\gamma\in G$ est dans $G_{\rm r}$ si et et seulement si son centralisateur $G_\gamma$ est un tore, \cad si et seulement s'il est (semisimple) régulier. On a l'inclusion $G_{\rm r}\subset G_{\rm qr}$, et $G_{\rm r}$ est l'ensemble des éléments quasi--réguliers {\it séparables} de $G$, \cad ceux tels que $F[\gamma]$ est un produit $E_1\times \cdots \times E_r$ d'extensions séparables $E_i/F$. On pose $G_{\rm re}= G_{\rm r}\cap G_{\rm qre}$. Pour $\gamma \in G_{\rm qre}$, on a donc $D_G(\gamma)=D_F(\gamma)$. De la même manière, pour $\gamma\in \mathfrak{g}$, on note $D_{\mathfrak{g}}(\gamma)$ le coefficient de $t^N$ dans le polynôme ${\rm \det}_F(t-{\rm ad}_\gamma; \mathfrak{g})$, et on note
$$
\mathfrak{g}_{\rm r}= \{\gamma \in \mathfrak{g}: D_\mathfrak{g}(\gamma) \neq 0\}\subset \mathfrak{g}_{\rm qr}
$$
l'ensemble des éléments (semisimples) réguliers de $\mathfrak{g}$. On pose $\mathfrak{g}_{\rm re}= \mathfrak{g}\cap \mathfrak{g}_{\rm qre}$. Pour $\gamma\in \mathfrak{g}_{\rm re}$, on a donc $D_\mathfrak{g}(\gamma) =D_F^+(\gamma)$. Si $N=1$, on a $G_{\rm re}= G_{\rm r}=G_{\rm qr}= G$ et $\mathfrak{g}_{\rm re}=\mathfrak{g}_{\rm r}= \mathfrak{g}_{\rm qr}= \mathfrak{g}$. En général, les inclusions $G_{\rm r}\subset G_{\rm qr}$, resp. $G_{\rm re}\subset G_{\rm qre}$, et $\mathfrak{g}_{\rm r}\subset \mathfrak{g}_{\rm qr}$, resp. $\mathfrak{g}_{\rm re}\subset \mathfrak{g}_{\rm qre}$, sont strictes. L'une d'elle est une égalité si et seulement si les trois autres le sont, ce qui n'est possible que si l'une des deux conditions suivantes est vérifiée:
\begin{itemize}
\item ${\rm car}(F)=0$;
\item ${\rm car}(F)=p>0$ et $p$ ne divise pas $N$.
\end{itemize} 

On étend naturellement ces définitions à tout groupe $H$ isomorphe à un produit fini de groupes linéaires $GL(d_i,E_i)$ pour des extensions finies $E_i/F$ --- rappelons qu'un élément $\gamma\in G$ est fermé, au sens où son 
orbite $\ES{O}_G(\gamma)$ fermée dans $G$ pour la topologie $\mathfrak{p}$--adique, 
si et seulement son centralisateur $H=G_\gamma$ est de cette forme (cf. la remarque 1 de \ref{éléments qre})) ---, donc en particulier à toute composante de Levi $H=M$ d'un sous--groupe parabolique de $G$.

Pour $\gamma\in G_{\rm qre}$, on a défini en \ref{IO normalisées}.(7) une distribution normalisée $f\mapsto I^G(\gamma,f)$. On étend comme suit cette définition à tout élément $\gamma\in G_{\rm qr}$. Pour un tel $\gamma$, on a $F[\gamma]=E_1\times \cdots \times E_r$ pour des extensions $E_i/F$ telles que $\sum_{i=1}^r [E_i:F]=N$. Soit $A_\gamma = F^\times \times \cdots \times F^\times \subset F[\gamma]^\times$ le sous--tore déployé maximal du centralisateur $G_\gamma =F[\gamma]^\times$ de $\gamma$ dans $G$, et soit $M=M(\gamma)$ le centralisateur $Z_G(A_\gamma)$ de $A_\gamma$ dans $G$. On note $da_\gamma$ la mesure de Haar sur $A_\gamma$ qui donne le volume $1$ au sous--groupe compact maximal $U_F\times\cdots\times U_F$ de $A_\gamma$, et $dg_\gamma$ la mesure de Haar sur $G_\gamma$ telle que ${\rm vol}(A_\gamma \backslash G_\gamma, \textstyle{dg_\gamma\over da_\gamma})=1$. On pose
$$
\ES{O}_\gamma(f)= \int_{G_\gamma \backslash G}f(g^{-1}\gamma g)\textstyle{dg\over dg_\gamma},\quad f\in C^\infty_{\rm c}(G).
$$
\'Ecrivons $\gamma=(\gamma_1,\ldots ,\gamma_r)$ avec $\gamma_i\in E_i^\times$, et posons
$$
\eta_{M}(\gamma)= \prod_{i=1}^r \eta_{G_i}(\gamma_i),\quad G_i={\rm Aut}_F(E_i).
$$
Notons $\mathfrak{m}=\mathfrak{m}(\gamma)$ le centralisateur $Z_\mathfrak{g}(A(\gamma))$ de $A(\gamma)$ dans $\mathfrak{g}$, et posons
$$
\eta_G(\gamma) = \vert D_{M\backslash G}(\gamma)\vert \eta_M(\gamma)
$$
avec
$$
D_{M\backslash G}(\gamma) = {\rm det}_F(1-{\rm Ad}_\gamma; \mathfrak{g}/\mathfrak{m}).
$$
Enfin on pose
$$
I^G(\gamma,f)= \eta_G(\gamma)^{1\over 2}\ES{O}_\gamma(f),\quad f\in C^\infty_{\rm c}(G).
$$
% remarque
\begin{marema}
{\rm 
D'après \ref{IO normalisées}.(4), si $\gamma$ est {\it séparable}, \cad si $\gamma\in G_{\rm r}$, on a
$$
\eta_G(\gamma) = \vert D_G(\gamma)\vert q^{\sum_{i=1}^r \delta_{\gamma_i} -  f_{\gamma_i}(e_{\gamma_i}-1)}
$$
avec
$$
D_G(\gamma) = {\rm det}_F(1-{\rm Ad}_\gamma; \mathfrak{g}/\mathfrak{g}_\gamma).
$$
En ce cas, on a
$$
D_{M\backslash G}(\gamma) = D_G(\gamma)D_M(\gamma)^{-1}. \eqno{\blacksquare}
$$
}
\end{marema}

Soit $P=MU$ un sous--groupe parabolique de $G$ de composante de Levi $M$ et de radical unipotent $U$, et soit $K$ un sous--groupe ouvert compact maximal de $G$ en \og bonne position\fg{} par rapport à $(P,A)$, $A=Z(M)$, \cad tel que
$$
P\cap K= (M\cap K)(U\cap K).
$$
Soient $dm$, $du$, $dk$, des mesures de Haar sur $M$, $U$, $K$, normalisées de telle manière que
$$
\int_Gf(g)dg= \int\!\!\!\int\!\!\!\int_{M\times U\times K} f(muk)dmdu dk, \quad f\in C^\infty_{\rm c}(G).\leqno{(1)}
$$
Pour $f\in C^\infty_{\rm c}(G)$, on note $f_P\in C^\infty_{\rm c}(M)$ le terme $K$--invariant (ou terme constant) de $f$ suivant $P$ défini par
$$
f_P(m)= \delta_P(m)^{1\over 2}\int\!\!\!\int_{K\times U}f(k^{-1}muk)dkdu, \quad m\in M,\leqno{(2)}
$$
où $\delta_P: P \rightarrow q^{\Bbb Z}\subset {\Bbb Q}^\times$ est le caractère module usuel défini par $d(p'pp'^{-1})= \delta_P(p')dp$ pour une (i.e. pour toute) mesure de Haar, à gauche ou à droite, $dp$ sur $P$. On note $\ES{O}_\gamma^M$ la distribution sur $M$ définie par
$$
\ES{O}_\gamma^M(f)= \int_{G_\gamma\backslash M}f(m^{-1}\gamma m)\textstyle{dm\over dg_\gamma}, \quad f\in C^\infty_{\rm c}(M),
$$
et on note $I^M(\gamma, \cdot)$ la distribution normalisée sur $M$ donnée par
$$
I^M(\gamma,\cdot )= \eta_M(\gamma)^{1\over 2}\ES{O}_\gamma^M.
$$
Alors on a la formule de descente bien connue
$$
\ES{O}_\gamma(f) = \vert D_{M\backslash G}(\gamma)\vert^{-{1\over 2}}\ES{O}_\gamma^M(f_P),\quad f\in C^\infty_{\rm c}(G).\leqno{(3)}
$$
D'où la formule de descente entre intégrales orbitales normalisées
$$
I^G(\gamma,f)= I^M(\gamma,f_P),\quad f\in C^\infty_{\rm c}(G).\leqno{(4)}
$$
Par construction, l'élément $\gamma$ est quasi--régulier elliptique dans $M$. On va voir plus loin que la formule (4) reste vraie pour tout tout élément $\gamma\in M\cap G_{\rm qr}$ 
(d'ailleurs on pourrait le voir tout de suite, aprs avoir tendu la fonction $\eta_M$  $M_{\rm qr}$ comme on l'a fait pour $G$: 
il suffit de voir que la formule de descente (3) reste vraie si $Z_G(A_\gamma)\subset M$).

\vskip1mm
On fixe une paire de Borel $(P_0,A_0)$ de $G$, et on note $U_0$ le radical unipotent de $P_0$. On fixe aussi un sous--groupe compact maximal $K$ de $G$ en bonne position par rapport à $(P_0,A_0)$. On suppose désormais, ce qui est loisible, que la mesure de Haar 
$dg$ sur $G$ vérifie
$$
{\rm vol}(K,dg)=1,
$$
et on note $dk$ la restriction de $dg$ à $K$.

Un sous--groupe parabolique $P$ de $G$ est dite {\it standard} s'il contient $P_0$. 
On note $\ES{P}=\ES{P}_G$ l'ensemble des sous--groupes paraboliques standards de $G$, et pour $P\in \ES{P}$, 
on note $U_P$ le radical unipotent de $P$, $M_P$ la composante de Levi de $P$ contenant $A_0$, et $A_P=Z(M_P)\subset A_0$ le centre de $M_P$. Pour chaque $P\in \ES{P}$, le groupe $K$ est en bonne position par rapport à la paire parabolique $(P,A_P)$ de $G$, \cad 
qu'on a la décomposition
$$
K\cap P=(K\cap M_P)(K\cap U_P).
$$
On prend comme mesures de Haar $dm=dm_P$ sur $M_P$ et $du=du_P$ sur $U_P$ celles qui donnent le volume $1$ à $M_P\cap K$ et à $U_P\cap K$. Alors on a la formule (1) pour ces mesures. Le groupe $M_P$ est un produit de groupes linéaires sur $F$. Comme pour $G$, on définit l'ensemble $(M_P)_{\rm qr}$ des éléments quasi--réguliers de $M_P$, le sous--ensemble $(M_P)_{\rm qre}\subset (M_P)_{\rm qr}$ des éléments quasi--réguliers elliptiques de $M_P$, et le facteur de normalisation $\eta_{M_P}: (M_P)_{\rm qr}\rightarrow {\Bbb R}_{>0}$. Pour $\gamma\in (M_P)_{\rm qr}$, on note $I^{M_P}(\gamma,\cdot)$ la distribution normalisée $\eta_{M_P}(\gamma)^{1\over 2}\ES{O}_\gamma^{M_P}$ sur $M_P$ définie par la mesure $\textstyle{dm\over dm_{P,\gamma}}$ sur $M_{P,\gamma} \backslash M_P$ dfinie comme suit. Soit $A_{P,\gamma}$ le sous--tore dploy maximal du centralisateur $M_{P,\gamma}$ de $\gamma$ dans $M_P$, et soit 
$da_{P,\gamma}$ la mesure de Haar sur $A_{P,\gamma}$ qui donne le volume $1$ au sous--groupe compact maximal de $A_{P,\gamma}$. Alors 
$dm_{P,\gamma}$ est la mesure de Haar sur $M_{P,\gamma}$ telle que
$$
{\rm vol}(A_{P,\gamma}\backslash M_{P,\gamma}, \textstyle{dm_{P,\gamma} \over da_{P,\gamma}})=1.
$$
On a l'inclusion
$$
M_P\cap G_{\rm qr}\subset (M_P)_{\rm qr}
$$
et l'égalité
$$
(M_P)_{\rm qre}\cap G_{\rm qr}=\{\gamma \in G_{\rm qr}: M(\gamma)=M_P\}.
$$
Pour $f\in C^\infty_{\rm c}(G)$, on note $f_P\in C^\infty_{\rm c}(M_P)$ le terme $K$--invariant de $f$ suivant $P$ défini comme en (2) par $du_P$ et $dk$. Plus généralement, pour $Q\in \ES{P}$ tel que $P\subset Q$ et $f\in C^\infty_{\rm c}(M_Q)$, on note 
$f_{P\cap M_Q}\in C^\infty_{\rm c}(M_P)$ le terme $(K\cap M_Q)$--invariant  de $f$ suivant $P\cap M_Q$, défini de la même manière en utilisant les mesures de Haar normalisées par le sous--groupe compact maximal $K\cap M_Q$ de $M_Q$. On a la propriété de transitivité
$$
f_P = (f_Q)_{P\cap M_Q}, \quad f\in C^\infty_{\rm c}(G).\leqno{(5)}
$$
D'après (4) et (5), pour $f\in C^\infty_{\rm c}(G)$, on a la formule de descente
$$
I^G(\gamma,f)= I^{M_P}(\gamma,f_P),\quad \gamma \in M_P\cap G_{\rm qr}.\leqno{(6)}
$$
D'ailleurs, plus généralement, pour $Q\in \ES{P}$ tel que $P\subset Q$ et $f\in C^\infty_{\rm c}(M_Q)$, 
on a la formule de descente
$$
I^{M_Q}(\gamma,f)= I^{M_P}(\gamma,f_{P\cap M_Q}),\quad \gamma \in M_P\cap (M_Q)_{\rm qr}.\leqno{(7)}
$$ 

Puisque
$$
G_{\rm qr}= \bigcup_{P\in \ES{P}}{^G((M_P)_{\rm qre}\cap G_{\rm qr})},\leqno{(8)}
$$
d'après (6) --- ou d'après (4) ---, l'étude des intégrales orbitales quasi--régulières normalisées de $G$ se ramène à celle des intégrales orbitales quasi--régulières elliptiques des sous--groupes de Levi $M_P$ de $G$. Pour $P\in \ES{P}$, l'ensemble 
${^G((M_P)_{\rm qre}\cap G_{\rm qr})}$ est ouvert dans $G$. On en déduit, d'après (8), la formule de descente (6), et 
le point (ii) du corollaire de \ref{IO normalisées}, que pour toute fonction $f\in C^\infty_{\rm c}(G)$, on a:
\begin{enumerate}[leftmargin=17pt]
\item[(9)] la fonction $G_{\rm qr}\rightarrow {\Bbb C},\, \gamma \mapsto I^G(\gamma,f)$ est localement 
constante.
\end{enumerate}

Pour $P\in \ES{P}$, le groupe $M_P$ est un produit de groupes linéaires sur $F$. On peut donc définir, par produit comme on l'a fait pour $G$, la filtration $\{(M_P)_{\rm qre}^k: k\in {\Bbb R}\}$ de $(M_P)_{\rm qre}$. On pose
$$
G_{\rm qr}^k= \bigcup_{P\in \ES{P}}{^G((M_P)_{\rm qre}^k\cap G_{\rm qr})},\quad k\in {\Bbb R}.\leqno{(10)}
$$

%%%%%%%%%%%%%%%%%%%%%%
\subsection{Variante sur l'algèbre de Lie (suite)}\label{variante sur l'algèbre de Lie (suite)}
On a bien sûr la variante sur $\mathfrak{g}$ des constructions précédentes. Pour $\gamma\in \mathfrak{g}_{\rm qr}$, on pose
$$
\ES{O}_\gamma(\mathfrak{f})= \int_{G_\gamma \backslash G}\mathfrak{f}(g^{-1}\gamma g)\textstyle{dg\over dg_\gamma},\quad 
\mathfrak{f}\in C^\infty_{\rm c}(\mathfrak{g}),
$$
où les mesures $dg$ sur $G$ et $dg_\gamma$ sur $G_\gamma =F[\gamma]^\times$ sont normalisées comme en 
\ref{descente parabolique}. Notons que si $F[\gamma]= \mathfrak{g}$, ce qui n'est possible que si $N=1$, alors $dg_\gamma=dg$ et $\ES{O}_\gamma = \bs{\delta}_\gamma$ (mesure de Dirac au point $\gamma$). On note encore $A_\gamma$ le sous--tore déployé maximal de $F[\gamma]^\times$, $M=M(\gamma)$ le centralisateur de $A(\gamma)$ dans $G$, et $\mathfrak{m}=\mathfrak{m}(\gamma)$ le centralisateur de $A(\gamma)$ dans $\mathfrak{g}$. \'Ecrivons $F[\gamma]=E_1\times \cdots \times E_r$, $\gamma =(\gamma_1,\ldots ,\gamma_r)$ avec $\gamma_i\in E_i$, et 
posons
$$
\eta_\mathfrak{m}(\gamma) = \prod_{i=1}^r \eta_{\mathfrak{g}_i}(\gamma_i),\quad \mathfrak{g}_i= {\rm End}_F(E_i),
$$
et
$$
\eta_\mathfrak{g}(\gamma) = \vert D_{\mathfrak{m}\backslash \mathfrak{g}}(\gamma)\vert \eta_\mathfrak{m}(\gamma)
$$
avec
$$
D_{\mathfrak{m}\backslash \mathfrak{g}}(\gamma) = {\rm det}_F(-{\rm ad}_\gamma; \mathfrak{g}/\mathfrak{m}).
$$
Enfin, on pose
$$
I^\mathfrak{g}(\mathfrak{f},\gamma)= \eta_\mathfrak{g}(\gamma)^{1\over 2}\ES{O}_\gamma(\mathfrak{f}),\quad \mathfrak{f}\in C^\infty_{\rm c}(\mathfrak{g}).
$$

% remarque 1
\begin{marema1}
{\rm 
D'après \ref{variante sur l'algbre de Lie}.(8), si $\gamma$ est {\it séparable}, \cad si $\gamma\in \mathfrak{g}_{\rm r}$, on a
$$
\eta_\mathfrak{g}(\gamma) = \vert D_\mathfrak{g}(\gamma)\vert q^{\sum_{i=1}^r \delta_{i} -  f_i(e_i-1)}
$$
avec $?_i=?(E_i/F)$ et
$$
D_\mathfrak{g}(\gamma) = {\rm det}_F(-{\rm ad}_\gamma; \mathfrak{g}/\mathfrak{g}_\gamma).
$$
En ce cas, on a
$$
D_{\mathfrak{m}\backslash \mathfrak{g}}(\gamma) = D_\mathfrak{g}(\gamma)D_\mathfrak{m}(\gamma)^{-1}. \eqno{\blacksquare}
$$
}
\end{marema1}

% remarque 2
\begin{marema2}
{\rm 
Pour $\gamma\in G_{\rm qr}$, on a
$$
\eta_\mathfrak{g}(\gamma) = \vert \det(\gamma)\vert^{N-1} \eta_G(\gamma).
$$
En effet, si $\gamma\in G_{\rm qre}$, cela résulte des définitions (cf. \ref{variante sur l'algbre de Lie}). En général, posant $M=M(\gamma)$ et $\mathfrak{m}= \mathfrak{m}(\gamma)$, on a $\gamma\in M_{\rm qre}$. Précisément, $M = G_1\times \cdots \times G_r$ avec $G_i={\rm Aut}_F(V_i)$, et $\mathfrak{m}=\mathfrak{g}_1\times\cdots \times \mathfrak{g}_r$ avec $\mathfrak{g}_i = {\rm End}_F(V_i)$, pour une décomposition $V=V_1\times\cdots \times V_r$. L'élément $\gamma$ s'écrit $\gamma=(\gamma_1,\ldots ,\gamma_r)$ avec $\gamma_i\in (G_i)_{\rm qre}$. Pour $i=1,\ldots ,r$, on a donc
$$
\eta_{\mathfrak{g}_i}(\gamma_i) = \vert \det(\gamma_i)\vert^{N_i-1} \eta_{G_i}(\gamma_i), \quad N_i= \dim_F(V_i).
$$
Ici $\det(\gamma_i)$ est le déterminant $\det_F(v\mapsto \gamma_iv;V_i)$, et $\det(\gamma_i)^{N_i-1}= \det_F(y\mapsto y\gamma_i; \mathfrak{g}_i/ (\mathfrak{g}_i)_{\gamma_i})$. Par produit, on obtient
$$
\eta_\mathfrak{m}(\gamma)= \vert {\rm \det}_F(y\mapsto y \gamma; \mathfrak{m}/\mathfrak{m}_\gamma)\vert \eta_G(\gamma).
$$
D'autre part, on a
$$
D_{\mathfrak{m}\backslash \mathfrak{g}}(\gamma)=  {\rm \det}_F(y\mapsto y\gamma; \mathfrak{g}/\mathfrak{m})D_{M\backslash G}(\gamma).
$$
Comme $\mathfrak{g}_\gamma = \mathfrak{m}_\gamma$, on a
$$
\vert {\rm \det}_F(y\mapsto y \gamma; \mathfrak{m}/\mathfrak{m}_\gamma)\vert \vert {\rm \det}_F(y\mapsto y\gamma; \mathfrak{g}/\mathfrak{m})\vert = \vert {\rm \det}_F(y\mapsto y \gamma; \mathfrak{g}/\mathfrak{g}_\gamma)\vert = \vert \det(\gamma)\vert^{N-1},
$$
d'où l'égalité cherchée: $\eta_\mathfrak{g}(\gamma) = \vert \det(\gamma)\vert^{N-1} \eta_G(\gamma)$. \hfill $\blacksquare$
}
\end{marema2}

Soit $P=MU$ un sous--groupe parabolique de $G$ de composante de Levi $M$ et de radical unipotent $U$, et soit $K'$ un sous--groupe compact maximal de $G$ en bonne position par rapport à $(P,A)$, $A=Z(M)$. Soit $dk'$ la mesure de Haar normalisée sur $K'$ (\cad que ${\rm vol}(K',dk')=1$), et soit 
$dm$, resp. $du$, la mesure de Haar sur $M$, resp. $U$, telle que ${\rm vol}(K'\cap M,dm)= 1$ , resp. ${\rm vol}(K'\cap U,du)=1$. Rappelons que d'après la normalisation de $dg$, puisque tous les sous--groupes compacts maximaux de $G$ sont conjugués dans $G$, on a ${\rm vol}(K',dg)=1$ (i.e. $dk'= dg\vert_{K'}$). La formule (1) de \ref{descente parabolique} est donc valable pour ces mesures. Notons $\mathfrak{p}$, $\mathfrak{m}$, $\mathfrak{u}$, les algèbres de Lie de $P$, $M$, $U$, naturellement identifiées à des sous--$F$--algèbres de $\mathfrak{g}$. On a la décomposition $\mathfrak{p}=\mathfrak{m}\oplus \mathfrak{u}$. Soit $\mathfrak{d}u$ la mesure de Haar sur $\mathfrak{u}$ image de $du$ par l'isomorphisme de variétés $\mathfrak{p}_F$--adiques $U\rightarrow \mathfrak{u},\, u\mapsto u-1$. En d'autres termes, $\mathfrak{d}u$ est la mesure de Haar sur $\mathfrak{u}$ telle que ${\rm vol}(\mathfrak{u}\cap \mathfrak{A},\mathfrak{d}u)=1$, où $ \mathfrak{A}$ l'$\mathfrak{o}$--ordre héréditaire dans $\mathfrak{g}$ tel que $K= \mathfrak{A}^\times$. Pour $\mathfrak{f}\in C^\infty_{\rm c}(\mathfrak{g})$, on note $\mathfrak{f}_\mathfrak{p}\in C^\infty_{\rm c}(\mathfrak{m})$ le terme $K'$--invariant (ou terme constant) de $\mathfrak{f}$ suivant $\mathfrak{p}$ défini par
$$
\mathfrak{f}_\mathfrak{p}(m)= \int\!\!\!\int_{K'\times \mathfrak{u}}\mathfrak{f}(k'^{-1}(m+u)k')dk' \mathfrak{d}u,\quad m\in M.\leqno{(1)}
$$
Par rapport à \ref{descente parabolique}.(2), noter l'absence du facteur $\delta_P$ (le groupe $\mathfrak{p}$ est unimodulaire).
Alors on a la variante sur $\mathfrak{g}$ de la formule de descente \ref{descente parabolique}.(3):
$$
\ES{O}_\gamma(\mathfrak{f})= \vert D_{\mathfrak{m}\backslash \mathfrak{g}}(\gamma)\vert^{-{1\over 2}}\ES{O}_\gamma^{\mathfrak{m}}(\mathfrak{f}_\mathfrak{p}),\quad \mathfrak{f}\in C^\infty_{\rm c}(\mathfrak{g}),\leqno{(2)}
$$
où $\ES{O}_\gamma^\mathfrak{m}$ est la distribution sur $\mathfrak{m}$ définie par
$$
\ES{O}_\gamma^{\mathfrak{m}}(\mathfrak{f}')= \int_{G_\gamma \backslash M}\mathfrak{f}'(m^{-1}\gamma m)\textstyle{dm\over dg_\gamma},\quad \mathfrak{f}'\in C^\infty_{\rm c}(\mathfrak{m}).
$$
Posant $I^\mathfrak{m}(\gamma ,\cdot) = \eta_\mathfrak{m}(\gamma)^{1\over 2}\ES{O}_\gamma^\mathfrak{m}$, 
on obtient la formule de descente entre intégrales orbitales normalisées
$$
I^\mathfrak{g}(\gamma,\mathfrak{f})= I^\mathfrak{m}(\gamma, \mathfrak{f}_\mathfrak{p}),\quad \mathfrak{f}\in C^\infty_{\rm c}(\mathfrak{g}).\leqno{(3)}
$$
Par construction (comme pour $G$), l'élément $\gamma$ est quasi--régulier elliptique dans $\mathfrak{m}$. 

Pour étendre la formule (3) à tout élément $\gamma\in \mathfrak{m}\cap \mathfrak{g}_{\rm qr}$, on procède exactement comme en \ref{descente parabolique}. On a déjà fixé un ensemble $\ES{P}=\ES{P}_G$ de sous--groupes paraboliques standards de $G$, et un sous--groupe compact maximal $K$ de $G$ en bonne position par rapport à $(P,A_P)$ pour tout $P\in \ES{P}$. Pour $P\in \ES{P}$, on note $\mathfrak{p}=\mathfrak{p}_P$, $\mathfrak{m}_P$, $\mathfrak{u}_P$, les algèbres de Lie de $P$, $M_P$, $U_P$, naturellement identifiées à des sous--$F$--algèbres de $\mathfrak{g}$. On a donc la décomposition $\mathfrak{p}= \mathfrak{m}_P\oplus \mathfrak{u}_P$. Pour $P=P_0$, on écrit $\mathfrak{p}_0$, $\mathfrak{m}_0$ et $\mathfrak{u}_0$ au lieu de $\mathfrak{p}_{P_0}$, $\mathfrak{m}_{P_0}$, $\mathfrak{u}_{P_0}$. Pour $P\in \ES{P}$, on a déjà fixé des mesures normalisées $dm_P$, $du_P$, $dm_{P,\gamma}$ pour $\gamma \in \mathfrak{m}_P\cap \mathfrak{g}_{\rm qr}$, sur les groupes $M_P$, $U_P$, $M_{P,\gamma} =G_\gamma$. On note $\mathfrak{d}u_P$ la mesure de Haar sur $\mathfrak{u}_P$ image de $du_P$ par l'isomorphisme de variétés $\mathfrak{p}_F$--adiques $U_P\rightarrow \mathfrak{u}_P,\, u\mapsto u-1$, et pour $\mathfrak{f}\in C^\infty_{\rm c}(\mathfrak{g})$, on note $\mathfrak{f}_{\mathfrak{p}_P}\in C^\infty_{\rm c}(\mathfrak{m}_P)$ le terme $K$--invariant de $\mathfrak{f}$ suivant $\mathfrak{p}$ défini à l'aide des mesures normalisées $dk$ sur $K$ et $\mathfrak{d}u_P$ sur $\mathfrak{u}_P$. Pour $\mathfrak{f}\in C^\infty_{\rm c}(\mathfrak{g})$, on a la formule de descente
$$
I^\mathfrak{g}(\gamma,\mathfrak{f})= I^{\mathfrak{m}_P}(\gamma,\mathfrak{f}_{\mathfrak{p}}),\quad \gamma\in \mathfrak{m}_P\cap \mathfrak{g}_{\rm qr}.\leqno{(4)}
$$
Plus généralement, pour $P,\, Q\in \ES{P}$ tels que $P\subset Q$ et $\mathfrak{f}\in C^\infty_{\rm c}(\mathfrak{m}_Q)$, on note $\mathfrak{f}_{\mathfrak{p}\cap \mathfrak{m}_Q}\in C^\infty_{\rm c}(\mathfrak{m}_P)$ le terme $(K\cap M_Q)$--invariant de $\mathfrak{f}$ suivant $\mathfrak{p}\cap \mathfrak{m}_Q$, défini comme en (1) à l'aide des mesures normalisées sur $K\cap M_Q$ et sur $\mathfrak{u}_P\cap \mathfrak{m}_Q$. Alors posant $\mathfrak{q}= \mathfrak{p}_Q$, on a la propriété de transitivité
$$
(\mathfrak{f}_\mathfrak{q})_{\mathfrak{p}\cap \mathfrak{m}_Q}= \mathfrak{f}_\mathfrak{p},\quad \mathfrak{f}\in C^\infty_{\rm c}(\mathfrak{g}),\leqno{(5)}
$$
et pour $\mathfrak{f}\in C^\infty_{\rm c}(\mathfrak{g})$, on a la formule de descente
$$
I^{\mathfrak{m}_Q}(\gamma,\mathfrak{f})=I^{\mathfrak{m}_P}(\gamma, \mathfrak{f}_{\mathfrak{p}\cap \mathfrak{m}_Q}),\quad 
\gamma \in \mathfrak{m}_P\cap (\mathfrak{m}_Q)_{\rm qr}.\leqno{(6)}
$$

Comme pour $G$, on a
$$
\mathfrak{g}_{\rm qr} = \bigcup_{P\in \ES{P}}{^G((\mathfrak{m}_P)_{\rm qre}\cap G_{\rm qr})}.\leqno{(7)}
$$
D'après (7), la formule de descente (5) et le point (ii) du corollaire de \ref{variante sur l'algèbre de Lie}, 
pour toute fonction $\mathfrak{f}\in C^\infty_{\rm c}(\mathfrak{g})$, on a la variante sur $\mathfrak{g}$ de \ref{descente parabolique}.(9):
\begin{enumerate}[leftmargin=17pt]
\item[(8)] la fonction $\mathfrak{g}_{\rm qr}\rightarrow {\Bbb C},\, \gamma \mapsto I^\mathfrak{g}(\gamma,\mathfrak{f})$ est localement 
constante.
\end{enumerate}

Comme on l'a fait pour $G$, on pose
$$
\mathfrak{g}_{\rm qr}^k= \bigcup_{P\in \ES{P}}{^G((\mathfrak{m}_P)_{\rm qre}^k\cap \mathfrak{g}_{\rm qr})},\quad k\in {\Bbb R}.
\leqno{(9)}
$$

%%%%%%%%%%%%%%%%%%%%%%
\subsection{Descente centrale au voisinage d'un élément pur (suite)}\label{descente centrale au voisinage d'un élément pur (suite)}Soit $\beta\in G$ un élément pur. On pose $E=F[\beta]$, $d= {N\over [E:F]}$, $\mathfrak{b}= {\rm End}_E(V)$ et $H= \mathfrak{b}^\times \;(=G_\beta)$. On suppose $E\neq F$. Rappelons que d'après la section \ref{descente centrale au voisinage d'un élément pur} (corollaire de \ref{IO normalisées}), il existe un élément $\bs{x}\in \mathfrak{g}$ dans l'image réciproque de $1$ par une corestriction modérée $\bs{s}: \mathfrak{g}\rightarrow \mathfrak{b}$ sur $\mathfrak{g}$ relativement à $E/F$ et un voisinage ouvert fermé et $G$--invariant $\Xi$ de $\beta$ dans $G$, tels que pour toute fonction $f\in C^\infty_{\rm c}(G)$, il existe une fonction $f^\mathfrak{b}_\Xi\in C^\infty_{\rm c}(\mathfrak{b})$ telle que
$$
I^G(\beta + \bs{x}b,f)= I^\mathfrak{b}(b,f^\mathfrak{b}_\Xi)\leqno{(1)}
$$
pour tout $b\in \mathfrak{b}_{\rm qre}\cap \ES{V}$ o $\ES{V}$ est un voisinage ouvert ferm de $0$ dans $\mathfrak{b}$ tel que $\beta + \bs{x} \ES{V} \subset \Xi$, 
ces conditions impliquant que $\beta+ \bs{x}b$ est quasi--régulier elliptique (dans $G$). 
Précisément, $\bs{x}=\bs{x}_0\otimes 1$ et $\bs{s}=\bs{s}_0\otimes {\rm id}_\mathfrak{b}$ pour un élément $\bs{x}_0\in \mathfrak{A}(E)$ dans l'image réciproque de $1$ par une corestriction modérée $\bs{s}_0:A(E)\rightarrow E$ sur $A(E)$ relativement à $E/F$ et une $(W,E)$--décomposition $\mathfrak{g}=A(E)\otimes_E\mathfrak{b}$ de $\mathfrak{g}$ induite par une $(W,E)$--décomposition $\underline{\mathfrak{A}}= \mathfrak{A}(E)\otimes_{\mathfrak{o}_E}\underline{\mathfrak{B}}$ de $\underline{\mathfrak{A}}$, où $\underline{\mathfrak{A}}$ est un $\mathfrak{o}$--ordre héréditaire dans $\mathfrak{g}$ normalisé par $E^\times$ tel que $\underline{\mathfrak{B}}= \mathfrak{b}\cap \underline{\mathfrak{A}}$ est un $\mathfrak{o}_E$--ordre héréditaire minimal dans $\mathfrak{b}$. Posant $\underline{k}_0= k_F(\beta)d \;(=k_0(\beta,\underline{\mathfrak{A}}))$ et $\mathfrak{Q}={\rm rad}(\underline{\mathfrak{B}})$, l'application partout submersive
$$
\delta: G\times \bs{x}\underline{\mathfrak{Q}}^{\underline{k}_0+1}\rightarrow G,\, (g,\bs{x}b)\mapsto g^{-1}(\beta + \bs{x}b)g\leqno{(2)}
$$
permet de \og descendre\fg une distribution $G$--invariante $T$ au voisinage de $\beta$ dans $G$ en une distribution $H$--invariante $\theta_T$ au voisinage de $0$ dans $\mathfrak{b}$ (cf. \ref{le principe de submersion}). En particulier, pour tout 
$b\in \mathfrak{b}_{\rm qre}\cap \underline{\mathfrak{Q}}^{\underline{k}_0+1}$, l'intégrale orbitale $\ES{O}_{\beta + \bs{x}b}$ sur $G$ se descend en une distribution $\theta_{\ES{O}_{\beta + \bs{x}b}}$ sur $\mathfrak{b}$, qui est un multiple de l'intégrale orbitale $\ES{O}_b^\mathfrak{b}$ sur $\mathfrak{b}$. On suppose ici que la mesure de Haar $\mathfrak{d}b'$ sur $\mathfrak{b}$ utilisée pour définir l'application $T\mapsto \theta_T$ est celle qui est associée à la mesure de Haar $dh$ sur $H$ utilisée pour définir les distributions $\ES{O}_b^\mathfrak{b}$ ($b\in \mathfrak{b}_{\rm qre}$), \cad que l'on a $dh = \mathfrak{d}^\times b'$. On suppose aussi que la mesure de Haar $dz_E$ sur $Z_H=E^\times$ utilisée pour normaliser les intégrales orbitales $\ES{O}_b^\mathfrak{b}$ ($b\in \mathfrak{b}_{\rm qre}$) est celle qui vérifie ${\rm vol}(F^\times\backslash E^\times, {dz_E\over dz})=1$. 
Alors d'après la proposition de \ref{IO normalisées}, pour $b\in \mathfrak{b}_{\rm qre}\cap \underline{\mathfrak{Q}}^{\underline{k}_0+1}$, posant $\gamma = \beta + \bs{x}b$, on a
$$
\theta_{\ES{O}_\gamma}= \lambda \ES{O}_b^\mathfrak{b}\leqno{(3)}
$$
avec
$$
\lambda =(q_E^{n_F(\beta)} \mu_F(\beta))^{d^2}=\vert \beta \vert_E^d \textstyle{\eta_\mathfrak{b}(b)\over \eta_G(\gamma)},
$$
ce qui, en termes des intégrales orbitales normalisées, équivaut à
$$
\theta_{I^G(\gamma,\cdot)}= \vert \beta \vert_E^d \left(\textstyle{\eta_G(\gamma)\over \eta_\mathfrak{b}(b)}\right)^{-\textstyle{1\over 2}}I^\mathfrak{b}(b,\cdot).\leqno{(4)}
$$

% remarque 1
\begin{marema1}
{\rm 
Pour généraliser ces formules (3) et (4) aux éléments $b\in \mathfrak{b}_{\rm qr}$ qui ne sont pas elliptiques, on est donc ramené à relier l'application $T\mapsto \theta_T$ à l'application terme constant $f\mapsto f_P$. Notons que les choses se présentent plutôt bien, puisque le rang $d$ de $H$ sur $E$ est un invariant stable par passage aux sous--groupes de Levi de $H$, et que, pour un sous--groupe de Levi $M_H$ de $H$, notant $M$ le plus petit sous--groupe de Levi de $G$ contenant $M_H$ (\cad le centralisateur dans $G$ de la composante $F$--déployée de $Z(M_H)$), $\mathfrak{m}$ l'algèbre de Lie de $M$, et $\mathfrak{m}_\beta = \mathfrak{m}\cap \mathfrak{b}$ l'algèbre de Lie de $M_H$, pour $b \in (\mathfrak{m}_\beta)_{\rm qre} \cap \mathfrak{b}_{\rm qr}$ suffisamment proche de $0$ dans $\mathfrak{m}_\beta$, 
l'élément $\gamma = \beta + \bs{x}b$ appartient à $M_{\rm qre}\cap G_{\rm qr}$, et par définition, on a
$$
\textstyle{\eta_G(\gamma)\over \eta_\mathfrak{b}(b)}= \textstyle{\eta_M(\gamma)\over \eta_{\mathfrak{m}_\beta}(b)}\textstyle{\vert D_{M\backslash G}(\gamma)\vert \over \vert D_{\mathfrak{m}_\beta\backslash \mathfrak{b}}(b)\vert }.
$$
On voit donc apparaître les Jacobiens des applications \og terme constant\fg sur $G$ et sur $\mathfrak{b}$. \hfill$\blacksquare$
}
\end{marema1}

Comme pour $G$, on fixe une paire parabolique minimale $(P_{H,0},A_{H,0})$ de $H$, et un sous--groupe compact maximal $K_H$ de $H$ en bonne position par rapport à $(P_{H,0},A_{H,0})$. 
On suppose que la mesure de Haar $dh$ sur $H$ est celle qui donne le volume $1$ à $K_H$. On note $\ES{P}_H$ l'ensemble des sous--groupes paraboliques {\it standards} de $H$, \cad ceux qui contiennent $P_{H,0}$, et pour $P\in \ES{P}_H$, 
on note $U_{P_H}$ le radical unipotent de $P_H$, $M_{P_H}$ la composante de Levi de $P_H$ contenant $A_{H,0}$, et $A_{P_H}=Z(M_{P_H})\subset A_{H,0}$ le centre de $M_{P_H}$. On suppose que les ensembles $\ES{P}$ et $\ES{P}_H$ sont {\it compatibles}, au sens où en posant $U_{H,0}=U_{P_{H,0}}$ et en notant $A^G_{H,0}\;(\simeq (F^\times)^d)$ le tore $F$--déployé maximal de $A_{H,0}\;(\simeq (E^\times)^d)$, on a les inclusions
$$
U_{H,0}\subset U_0,\quad A^G_{H,0}\subset A_0.
$$
Pour $P_H\in \ES{P}_H$, on note $A^G_{P_H}$ le tore $F$--déployé maximal de $A_{P_H}$, $M^G_{P_H}$ le centralisateur de $A^G_{P_H}$ dans $G$, et $P^G_H$ le sous--groupe parabolique standard de $G$ défini par $P^G_H = M^G_{P_H}U_0$, ou, de manière équivalente, par $A_{P_H^G}= A_{P_H}^G$, resp. par $M_{P_H^G}= M_{P_H}^G$. L'application
$$
\ES{P}_H\rightarrow \ES{P}_G,\, P_H \mapsto P^G_H\leqno{(5)}
$$
est injective. De plus, on a $P_H = H \cap P^G_H$, $M_{P_H}= H \cap M_{P_H^G}$ et $A_{P_H}= Z_H(A^G_{P_H})$. Notant $\ES{P}(H)=\ES{P}_G(H)$ l'image de (5), la bijection inverse est donnée par
$$
\ES{P}(H)\rightarrow \ES{P}_H,\, P \mapsto P_H=H\cap P. 
$$

Pour $P\in \ES{P}$, on note $\mathfrak{p}= \mathfrak{p}_P$, $\mathfrak{m}_P$, $\mathfrak{u}_P$, les algèbres de Lie de $P$, $M_P$, $U_P$, identifiées à des sous--$F$--algèbres de $\mathfrak{g}$. On a la décomposition $\mathfrak{p}= \mathfrak{m}_P\oplus \mathfrak{u}_P$. De même, pour $P_H\in \ES{P}_H$, on note $\mathfrak{p}_H= \mathfrak{p}_{P_H}$, $\mathfrak{m}_{P_H}$, $\mathfrak{u}_{P_H}$, les algèbres de Lie de $P_H$, $M_{P_H}$, $U_{P_H}$, identifiées à des sous--$E$--algèbres de $\mathfrak{b}$. On a la décomposition $\mathfrak{p}_H= \mathfrak{m}_{P_H}\oplus \mathfrak{u}_{P_H}$. Rappelons qu'on a fixé une $(W,E)$--décomposition $\mathfrak{g}=A(E)\otimes_E\mathfrak{b}$. L'hypothèse de compatibilité entre $\ES{P}$ et $\ES{P}_H$ assure que pour $P_H\in \ES{P}_H$ et $P=P_H^G\in \ES{P}$, on a les décompositions
$$
\mathfrak{p}= A(E)\otimes_E\mathfrak{p}_H,\quad 
\mathfrak{m}_P= A(E)\otimes_E\mathfrak{m}_{P_H},\quad 
\mathfrak{u}_P= A(E)\otimes_E\mathfrak{u}_{P_H}.\quad \leqno{(6)} 
$$
Bien sûr on a aussi
$$
\mathfrak{p}_H = \mathfrak{p}\cap \mathfrak{b},\quad
\mathfrak{m}_{P_H}= \mathfrak{m}_P\cap \mathfrak{b},\quad 
\mathfrak{u}_{P_H}= \mathfrak{u}_P\cap \mathfrak{b}.
$$

Pour $b\in \mathfrak{b}_{\rm qr}$, on définit comme en \ref{variante sur l'algèbre de Lie (suite)} les distributions $\ES{O}_b^\mathfrak{b}$ et $I^\mathfrak{b}(b,\cdot)=\eta_\mathfrak{b}(b)^{1\over 2}\ES{O}_b^\mathfrak{b}$ sur $\mathfrak{b}$, l'intégrale orbitale $\ES{O}_b^\mathfrak{b}$ sur $\mathfrak{b}$ étant normalisée de la manière suivante: le groupe $H_b=E[b]^\times$ est un produit $E_1^\times \times\cdots \times E_r^\times$ pour des extensions $E_i/E$; il contient le tore $F$--déployé maximal $A_b^G=(F^\times)^r$ que l'on munit de la mesure de Haar $da$ telle que 
${\rm vol}((\mathfrak{o}^\times)^s,da)=1$; alors on utilise pour définir $\ES{O}_b^\mathfrak{b}$ la mesure de Haar $dh_b$ sur $H_b$ telle que ${\rm vol}(A_b^G\backslash H_b, {dh_b\over da})=1$. Pour $b\in \mathfrak{b}_{\rm qre}$, cette normalisation co\?{\i}ncide avec celle introduite plus haut.

Pour $P_H\in \ES{P}_H$ et $b\in 
(\mathfrak{m}_{P_H})_{\rm qr}$, posant $\mathfrak{p}_*= \mathfrak{p}_{P_H}$, $\mathfrak{m}_*=\mathfrak{m}_{P_H}$ et $\mathfrak{u}_*=\mathfrak{u}_{P_H}$, on définit comme en \ref{variante sur l'algèbre de Lie (suite)} --- avec la normalisation ci--dessus --- les distributions $\ES{O}_b^{\mathfrak{m}_*}$ et $I^{\mathfrak{m}_*}(b,\cdot)= \eta_{\mathfrak{m}_*}(b)^{1\over 2}\ES{O}_b^{\mathfrak{m}_*}$ sur $\mathfrak{m}_*$. Pour $\mathfrak{f}\in C^\infty_{\rm c}(\mathfrak{b})$, on note $\mathfrak{f}_{\mathfrak{p}_*}\in C^\infty_{\rm c}(\mathfrak{m}_*)$ le terme $K_H$--invariant 
de $\mathfrak{f}$ suivant $\mathfrak{p}_*$, défini à l'aide des mesures normalisées sur $K_H$ et sur $\mathfrak{u}_*$. Pour $\mathfrak{f}\in C^\infty_{\rm c}(\mathfrak{b})$, d'après \ref{variante sur l'algèbre de Lie (suite)}.(4), 
on a la formule de descente
$$
I^\mathfrak{b}(b,\mathfrak{f})= I^{\mathfrak{m}_*}(b, \mathfrak{f}_{\mathfrak{p}_*}),\quad b\in \mathfrak{m}_*\cap \mathfrak{b}_{\rm qr}.\leqno{(7)}
$$

Soit $P_H\in \ES{P}_H$. Posons $P=P_H^G$, $M=M_P$, $U=U_P$, et notons $\mathfrak{p}=\mathfrak{p}_P$, $\mathfrak{m}=\mathfrak{m}_P$, $\mathfrak{u}=\mathfrak{u}_P$, les algèbres de Lie de $P$, $M$, $U$. De même, posons $P_*= P_H$, $M_* = M_{P_H}$, $U_* = U_{P_H}$, et notons  $\mathfrak{p}_*= \mathfrak{p}_{P_H}$, $\mathfrak{m}_*= \mathfrak{m}_{P_H}$, $\mathfrak{u}_*= \mathfrak{u}_{P_H}$, les algèbres de Lie de $P_*$, $M_*$, $U_*$. Soit aussi $P_*^-$ le sous--groupe parabolique de $H$ opposé à $P_*$ par rapport à $M_*$, et soit $U_*^-$ son radical unipotent. Notons $\mathfrak{u}_*^-\subset \mathfrak{b}$ l'algèbre de Lie de $\mathfrak{u}_*^-$. 
Rappelons que la $(W,E)$--décomposition $\mathfrak{g}= A(E)\otimes_E \mathfrak{b}$ de $\mathfrak{g}$ est induite à partir d'une $(W,E)$--décomposition $\underline{\mathfrak{A}}= \mathfrak{A}(E)\otimes_{\mathfrak{o}_E}\underline{\mathfrak{B}}$ de $\underline{\mathfrak{A}}$. On suppose que la sous--$\mathfrak{o}_E$--algèbre d'Iwahori $\underline{\mathfrak{B}}$ de $\mathfrak{b}$ est associée à une chambre de l'immeuble affine de $H$ contenue dans l'appartement associé au tore $E$--déployé maximal $A_{H,0}$ de $H$. Alors on a la décomposition
$$
\underline{\mathfrak{Q}}^k =(\underline{\mathfrak{Q}}^k\cap \mathfrak{u}_*^-)\oplus (\underline{\mathfrak{Q}}^k\cap \mathfrak{m}_*)\oplus (\underline{\mathfrak{Q}}^k\cap \mathfrak{u}_*),\quad k\in {\Bbb Z}.\leqno{(8)}
$$
Pour $k\in {\Bbb Z}$, posons
$$
\underline{\mathfrak{Q}}^{k}_{\mathfrak{m}}=\underline{\mathfrak{Q}}^{k}\cap 
\mathfrak{m}\;(=\underline{\mathfrak{Q}}^k\cap \mathfrak{m}_*).
$$
Pour $k,\, j\in {\Bbb Z}$, on a donc
$$\underline{\mathfrak{Q}}^{dk+j}_{\mathfrak{m}}= \varpi_E^k(\underline{\mathfrak{Q}}_ \mathfrak{m}^j).\leqno{(9)}
$$ 
o $\varpi_E$ est une uniformisante de $E$. La $E$--algbre $\mathfrak{m}_*$ se dcompose en $\mathfrak{m}_*= \mathfrak{b}_1\times \cdots \times 
\mathfrak{b}_s$, $\mathfrak{b}_i= {\rm End}_E(V_i)$, pour une dcomposition du $E$--espace vectoriel $V$ en 
$V= V_1\times \cdots \times V_s$. La $\mathfrak{o}_E$--algbre $\underline{\mathfrak{B}}_{\mathfrak{m}}=\underline{\mathfrak{B}}\cap \mathfrak{m}$ se dcompose en $\underline{\mathfrak{B}}_{\mathfrak{m}}= \underline{\mathfrak{B}}_1\times\cdots \times \underline{\mathfrak{B}}_s$ o $\underline{\mathfrak{B}}_i$ est un $\mathfrak{o}_E$--ordre hrditaire minimal (i.e. d'Iwahori) dans $\mathfrak{b}_i$, et posant $\underline{\mathfrak{Q}}_\mathfrak{m} = \underline{\mathfrak{Q}} \cap \mathfrak{m} \;(=\underline{\mathfrak{Q}}_\mathfrak{m}^1)$ et $\underline{\mathfrak{Q}}_i={\rm rad}(\underline{\mathfrak{B}}_i)$ pour $i=1,\ldots ,s$, on a 
$\underline{\mathfrak{Q}}_\mathfrak{m}= \underline{\mathfrak{Q}}_1\times \cdots \times \underline{\mathfrak{Q}}_s$. On en dduit que pour $k\in {\Bbb Z}$, 
posant $d_i= \dim_E(V_i)$, on a les galits
$$
\underline{\mathfrak{Q}}_\mathfrak{m}^{dk}= \underline{\mathfrak{Q}}_1^{d_1k}\times \cdots \times \underline{\mathfrak{Q}}_s^{d_sk},\quad 
\underline{\mathfrak{Q}}_\mathfrak{m}^{dk+1}= \underline{\mathfrak{Q}}_1^{d_1k+1}\times \cdots \times \underline{\mathfrak{Q}}_s^{d_sk +1}.\leqno{(10)}
$$
De mme, on a la dcomposition $\mathfrak{m}= \mathfrak{g}_1\times \cdots \times \mathfrak{g}_s$, $\mathfrak{g}_i= {\rm End}_F(V_i)$, et la $\mathfrak{o}$--algbre $\underline{\mathfrak{A}}_\mathfrak{m} = \underline{\mathfrak{A}}\cap \mathfrak{m}$ se dcompose en $\underline{\mathfrak{A}}_\mathfrak{m}= \underline{\mathfrak{A}}_1\times \cdots \times \underline{\mathfrak{A}}_s$ o $\underline{\mathfrak{A}}_i$ est l'unique $\mathfrak{o}_E$--ordre hrditaire dans $\mathfrak{g}_i$ normalis par $E^\times$ tel que $\underline{\mathfrak{A}}_i\cap \mathfrak{b}_i = \underline{\mathfrak{B}}_i$. L'élément $\bs{x}=\bs{x}_0\otimes 1$ appartient  $\underline{\mathfrak{A}}_\mathfrak{m} $, et d'aprs ce qui prcde --- rappelons que $\underline{k}_0= dk_F(\beta)$ ---, l'application
$$
\delta_M: M \times \underline{\mathfrak{Q}}^{\underline{k}_0+1}_{\mathfrak{m}}\rightarrow M,\, (m,b)\mapsto m^{-1}(\beta + \bs{x}b)m\leqno{(11)}
$$
est partout submersive. On peut donc, comme on l'a fait sur $G$ à l'aide de $\delta$ (\ref{le principe de submersion}), descendre une distribution $M$--invariante au voisinage de $\beta$ dans $M$ en une distribution $M_*$--invariante au voisinage de $0$ dans $\mathfrak{m}_*$:  toute distribution $M$--invariante $T_M$ sur $M$ est associe une distribution $\wt{\vartheta}_{T_M}$ sur $\bs{x}\underline{\mathfrak{Q}}^{\underline{k}_0+1}_\mathfrak{m}$ telle que pour toute fonction $\varphi\in 
C^\infty_{\rm c}(M\times \bs{x}\underline{\mathfrak{Q}}^{\underline{k}_0+1}_\mathfrak{m})$, on a
$$
\langle \varphi_{\delta_M} , \wt{\vartheta}_{T_M}\rangle = \langle \varphi^{\delta_M},T_M\rangle,
$$
o les fonctions $\varphi^{\delta_M}\in C^\infty_{\rm c}({\rm Im}(\delta_M))$ et $\phi_{\delta_M}\in C^\infty_{\rm c}(\bs{x}\underline{\mathfrak{Q}}^{\underline{k}_0+1}_\mathfrak{m})$ sont dfinies 
 l'aide de la mesure normalise $dm$ sur $M$ et de la mesure $\mathfrak{d}m_*$ sur $\mathfrak{m}_*$ associe  la mesure 
normalise sur $M_*$. \`A partir de $\wt{\vartheta}_{T_M}$, on construit comme en \ref{le principe de submersion} une distribution $H$--invariante $\vartheta_{T_M}$ sur l'ouvert $H$--invariant ${^{M_*}(\underline{\mathfrak{Q}}^{\underline{k}_0+1}_\mathfrak{m})}$ de $M_*$. \'Ecrivons $M_* =H_1\times \cdots \times H_s$, 
$H_i={\rm Aut}_E(V_i)$, et pour $i=1,\ldots ,s$, posons $\underline{k}_{i,0}= d_i k_F(\beta)$. D'aprs (10), on a
$$
{^{M_*}(\underline{\mathfrak{Q}}^{\underline{k}_0+1}_\mathfrak{m})}=
 {^{H_1}(\underline{\mathfrak{Q}}_1^{\underline{k}_{1,0}+1})}\times\cdots \times  {^{H_s}(\underline{\mathfrak{Q}}_s^{\underline{k}_{s,0}+1})},\leqno{(12)}
 $$
et d'aprs \ref{lments qre}.(7), ${^{M_*}(\underline{\mathfrak{Q}}^{\underline{k}_0+1}_\mathfrak{m})}$ est une partie (ouverte et $M_*$--invariante) {\it ferme} dans $\mathfrak{m}_*$. Le support ${\rm Supp}(\vartheta_{T_M})$ de $\vartheta_{T_M}$, qui est une partie fermée de ${^{M_*}(\underline{\mathfrak{Q}}^{\underline{k}_0+1}_\mathfrak{m})}$, est donc fermé dans $\mathfrak{m}_*$, et on peut prolonger cette distribution $\vartheta_{T_M}$ par $0$ sur 
$\mathfrak{m}_*\smallsetminus {\rm Supp}(\vartheta_{T_M})$. On obtient ainsi une distribution $M_*$--invariante sur $\mathfrak{m}_*$, de support ${\rm Supp}(\vartheta_{T_M})$, que l'on note $\theta_{T_M}$. On peut aussi restreindre la distribution $\vartheta_{T_M}$ à un voisinage ouvert fermé et $M_*$--invariant $\Omega'$ de $0$ dans $\mathfrak{m}_*$ contenu dans ${^{M_*}(\underline{\mathfrak{Q}}^{\underline{k}_0+1}_\mathfrak{m})}$, et dfinir une distribution $M_*$--invariante $\theta_{T_M}^{\Omega'}$ sur $\mathfrak{m}_*$, à support dans $\Omega'$, en posant
$$
\langle \mathfrak{f}, \theta_{T_M}^{\Omega'}\rangle = \langle \mathfrak{f}\vert_{\Omega'}, \vartheta_{T_M}\rangle,\quad \mathfrak{f}\in C^\infty_{\rm c}(\mathfrak{m}_*).
$$

D'autre part, pour $k\in {\Bbb Z}$, on a l'égalité
$$
{^H\!(\underline{\mathfrak{Q}}^{dk+1})} \cap \mathfrak{m} =  {^{M_*}(\underline{\mathfrak{Q}}^{dk+1}_\mathfrak{m})}.\leqno{(13)}
$$
En effet, puisque ${^H\!(\underline{\mathfrak{Q}}^{dk+1})}= \varpi_E^k({^H\!\underline{\mathfrak{Q}}})$ et ${^{M_*}(\underline{\mathfrak{Q}}^{dk+1}_{\mathfrak{m}})}= \varpi_E^k({^{M_*}(\underline{\mathfrak{Q}}_\mathfrak{m})})$, il suffit de vrifier (13) pour $k=0$. 
L'inclusion $\supset$ dans (13) est claire. Pour $b\in \mathfrak{b}$, notant $\zeta_{b}^\mathfrak{b}(t)= \sum_{j=0}^{d} a_{b,j}^{\mathfrak{b}}t^j\in E[t]$ le polynme caractristique du $E$--endomorphisme $b$ de $V$, d'aprs la remarque 2 de \ref{lments qre}, on a
$$
\{b\in \mathfrak{b}: \nu_E(a_{b,j}^\mathfrak{b})\geq 1,\, j=0,\ldots ,d-1\}= {^H\!\underline{\mathfrak{Q}}}.
$$
Pour $b=(b_1,\ldots ,b_s)\in \mathfrak{m}_*$, $b_i\in \mathfrak{b}_i$, le polynme caractristique $\zeta_b^\mathfrak{b}$ s'crit $\zeta_b^\mathfrak{b}= \prod_{i=1}^s \zeta_{b_i}^{\mathfrak{b}_i}$, et si les coefficients $a_{b,j}^\mathfrak{b}$ ($j=0,\ldots d-1$) appartiennent  $\mathfrak{p}_E$, alors pour $i=1,\ldots ,s$, 
le polynme caractristique $\zeta_{b_i}^{\mathfrak{b}_i}$ appartient  $\mathfrak{o}_E[t]$ et $\zeta_{b_i}^{\mathfrak{b}_i}\;({\rm mod}\,\mathfrak{p}_E)= t^{d_i}$. 
\`A nouveau d'aprs la remarque 2 de \ref{lments qre} (appliqu  chaque $\mathfrak{b}_i)$, on obtient l'inclusion $\subset$ dans (13) pour $k=0$.

% remarque 2
\begin{marema2}
{\rm 
D'aprs (13), on a l'galit
$$
{^H\!(\underline{\mathfrak{Q}}^{\underline{k}_0+1})} \cap \mathfrak{m} =  {^{M_*}(\underline{\mathfrak{Q}}^{\underline{k}_0+1}_\mathfrak{m})}.\leqno{(14)}
$$
En particulier si $\Omega$ est un voisinage ouvert fermé et $H$--invariant de $0$ dans $\mathfrak{b}$ contenu dans ${^H(\underline{\mathfrak{Q}}^{\underline{k}_0+1})}$, alors $\Omega_{\mathfrak{m}}= \Omega \cap \mathfrak{m}\;(= \Omega \cap \mathfrak{m}_*)$ est un voisinage ouvert fermé et $M_*$--invariant de $0$ dans $\mathfrak{m}_*$ contenu dans ${^{M_*}(\underline{\mathfrak{Q}}^{\underline{k}_0+1}_{\mathfrak{m}})}$. 
\hfill $\blacksquare$
}
\end{marema2}

Comme on l'a fait pour $\mathfrak{b}$, pour chaque $r\in {\Bbb R}$, on peut définir par produit le sous--ensemble $(\mathfrak{m}_*)_{\rm qre}^{r}$ de $(\mathfrak{m}_*)_{\rm qre}$: on pose
$$
(\mathfrak{m}_*)_{\rm qre}^{r} = (\mathfrak{b}_1)_{\rm qre}^{r}\times \cdots \times (\mathfrak{b}_s)_{\rm qre}^{r}.
$$
Notons que si $r= k+ {1\over d}$ pour un entier $k$, alors puisque pour $i=1,\ldots ,s$, les \og sauts \fg{} de la filtration $r\mapsto (\mathfrak{b}_i)_{\rm qre}^r$ de $(\mathfrak{b}_i)_{\rm qre}$ sont les lments de ${1\over d_i}{\Bbb Z}$, on a l'galit
$$
(\mathfrak{m}_*)_{\rm qre}^{r}=  (\mathfrak{b}_1)_{\rm qre}^{k+{1\over d_1}}\times \cdots \times (\mathfrak{b}_s)_{\rm qre}^{k+ {1\over d_s}}.
$$
D'aprs (12) et le lemme 2 de \ref{lments qre}, on a donc
$$
(\mathfrak{m}_*)_{\rm qre}^{k_F(\beta)+{1\over d}}= (\mathfrak{m}_*)_{\rm qre}\cap {^{M_*}\!(\underline{\mathfrak{Q}}^{\underline{k}_0+1}_\mathfrak{m})},$$
d'o  (grce  (14))
$$
(\mathfrak{m}_*)_{\rm qre}^{k_F(\beta)+{1\over d}}= (\mathfrak{m}_*)_{\rm qre}\cap {^H\!(\underline{\mathfrak{Q}}^{\underline{k}_0+1})}.\leqno{(15)}
$$
Pour $b\in (\mathfrak{m}_*)_{\rm qre}\cap \underline{\mathfrak{Q}}^{\underline{k}_0+1}$, l'élément 
$\gamma = \beta + \bs{x}b$ appartient à $M_{\rm qre}$, et d'après (3), on a l'égalité
$$
\theta_{\ES{O}_\gamma^M}= \lambda_M \ES{O}_b^{\mathfrak{m}_*}\leqno{(16)}
$$
avec
$$
\lambda_M =(q_E^{n_F(\beta)}\mu_F(\beta))^{\dim_E(\mathfrak{m}_*)}=\vert \beta \vert_E^d \textstyle{\eta_{\mathfrak{m}_*}(b)\over \eta_M(\gamma)},
$$
ce qui, en termes d'intégrales orbitales normalisées, équivaut à
$$
\theta_{I^{M}(\gamma,\cdot)} = \vert \beta \vert_E^d \left(\textstyle{\eta_M(\gamma)\over \eta_{\mathfrak{m}_*}(b)}\right)^{-{1\over 2}} I^{\mathfrak{m}_*}(b,\cdot ).\leqno{(17)}
$$

% lemme 1
\begin{monlem1}
Pour $b\in (\mathfrak{m}_*)_{\rm qre}\cap \mathfrak{b}_{\rm qr}\cap \underline{\mathfrak{Q}}^{\underline{k}_0+1}$, l'élément $\gamma = \beta + \bs{x}b$ appartient à $M_{\rm qre}\cap G_{\rm qr}$, et on a
$$
\textstyle{\vert D_{M\backslash G}(\gamma)\vert \over \vert D_{\mathfrak{m}_*\!\backslash \mathfrak{b}}(b)\vert_E}=
(q_E^{n_F(\beta)}\mu_F(\beta))^{\dim_E(\mathfrak{m}_*)- \dim_E(\mathfrak{b})}\;(=\lambda_M\lambda^{-1}).
$$
\end{monlem1}

\begin{proof}
\'Ecrivons $b=(b_1,\ldots ,b_s)$, $b_i\in (\mathfrak{b}_i)_{\rm qre}\cap \underline{\mathfrak{Q}}_i^{\underline{k}_{i,0}+1}$. Posons $M= G_1\times \cdots \times G_s$, $G_i= {\rm Aut}_F(V_i)$, et crivons $\gamma = (\gamma_1,\ldots ,\gamma_s)$, $\gamma_i=\beta + \bs{x}b_i\in G_i$. 
Pour $i=1,\ldots , s$, l'élément $\gamma_i$ est quasi--régulier elliptique dans $G_i$, donc $\gamma$ est quasi--régulier elliptique dans $M$. De plus, 
si $\gamma$ n'est pas quasi--régulier dans $G$, 
alors $b$ n'est pas quasi--régulier dans $\mathfrak{b}$, contradiction. Donc $\gamma\in M_{\rm qre}\cap G_{\rm qr}$.

On a $D_{M\backslash G}(\gamma)\det_F(y\mapsto y\gamma; \mathfrak{g}/\mathfrak{m})= D_{\mathfrak{m}\backslash \mathfrak{g}}(\gamma)$. Pour $i=1,\ldots ,s$, posons $N_i= \dim_F(V_i)$ et $d_i= {N_i\over [E:F]}$. On a
$$
{\rm det}_F(y\mapsto y\gamma; \mathfrak{g}/\mathfrak{m})= \det(\gamma)^{N-1}\textstyle{\prod_{i=1}^s} \det(\gamma_i)^{-(N_i-1)},
$$
d'où
$$
\vert {\rm det}_F(y\mapsto y\gamma; \mathfrak{g}/\mathfrak{m})\vert = \vert \beta\vert_E^{d(N-1)}\textstyle{\prod_{i=1}^s} \vert \beta\vert_E^{-d_i(N_i-1)}
$$
et, puisque $\sum_{i=1}^sd_i=d= {N\over [E:F]}$, 
$$
\vert {\rm det}_F(y\mapsto y\gamma; \mathfrak{g}/\mathfrak{m})\vert = \vert \beta \vert_E^{[E:F](d^2 - \sum_{i=1}^s d_i^2)}= \vert \beta\vert_E^{[E:F](\dim_E(\mathfrak{b})-\dim_E(\mathfrak{m}_*))}.
$$
D'autre part, on a $\mu_F^+(\beta)= \vert \beta\vert_E^{1-[E:F]}\mu_F(\beta)$, par conséquent il s'agit de prouver l'égalité
$$
\textstyle{\vert D_{\mathfrak{m}\backslash \mathfrak{g}}(\gamma)\vert \over \vert D_{\mathfrak{m}_*\!\backslash \mathfrak{b}}(b)\vert_E}
=\mu_F^+(\beta)^{-(\dim_E(\mathfrak{b})- \dim_E(\mathfrak{m}_*))}.
$$

Pour $1\leq i, j \leq s$ tels que $i\neq j$, posons $\mathfrak{g}_{i,j}= {\rm End}_F(V_i,V_j)$ et $\mathfrak{b}_{i,j}= {\rm End}_E(V_i,V_j)$. Le $F$--endomorphisme $-{\rm ad}_\gamma$ de $\mathfrak{g}$ se restreint en un $F$--automorphisme $g \mapsto g\gamma_i-\gamma_j g$ de $\mathfrak{g}_{i,j}$, que l'on note $\bs{\gamma}_{i,j}$. De la même manière, le $E$--endomorphisme $- {\rm ad}_b$ de $\mathfrak{b}$ se restreint en un $E$--automorphisme $y\mapsto yb_i - b_j y$, que l'on note $\bs{b}_{i,j}$. Puisque
$$
\textstyle{\vert D_{\mathfrak{m}\backslash \mathfrak{g}}(\gamma)\vert \over \vert D_{\mathfrak{m}_*\!\backslash \mathfrak{b}}(b)\vert_E}
= \prod_{i\leq i,  j \leq s,\, i\neq j} {\vert \det_F(\bs{\gamma}_{i,j}; \mathfrak{g}_{i,j})\vert \over \vert \det_E(\bs{b}_{i,j}; \mathfrak{b}_{i,j})\vert_E}
$$
et
$$
\dim_E(\mathfrak{b})-\dim_E(\mathfrak{m}_*)= \textstyle{\sum_{1\leq i, j\leq s,\, i\neq j}}\dim_E(\mathfrak{b}_{i,j}),
$$
il suffit de prouver que pour $1\leq i,j \leq s$ tels que $i\neq j$, on a
$$
 \textstyle{\vert \det_F(\bs{\gamma}_{i,j}; \mathfrak{g}_{i,j})\vert \over \vert \det_E(\bs{b}_{i,j}; \mathfrak{b}_{i,j})\vert_E}= \mu_F^+(\beta)^{-\dim_E(\mathfrak{b}_{i,j})}.
$$
Fixons un tel couple $(i,j)$ et prouvons l'égalité ci--dessus. Rappelons que la $(W,E)$--décompo\-sition $\mathfrak{g}= A(E)\otimes_E\mathfrak{b}$ est induite par une $(W,E)$--décomposition $\underline{\mathfrak{A}}= \mathfrak{A}(E)\otimes_{\mathfrak{o}_E}\underline{\mathfrak{B}}$. On en déduit (par restriction) une décomposition $\mathfrak{g}_{i,j}= A(E)\otimes_E \mathfrak{b}_{i,j}$. Pour $k\in {\Bbb Z}$, posons  $\mathfrak{X}^k= 
\underline{\mathfrak{P}}^k\cap \mathfrak{g}_{i,j}$ et $\mathfrak{Y}^k= \underline{\mathfrak{Q}}^k \cap \mathfrak{b}_{i,j}\;(= \underline{\mathfrak{Q}}^k\cap \mathfrak{g}_{i,j})$. On a $\mathfrak{X}^k\cap \mathfrak{b}_{i,j}= \mathfrak{Y}^k$ et $\mathfrak{X}^k= \mathfrak{A}(E)\otimes_{\mathfrak{o}_E}\mathfrak{Y}^k$. Soit $\mathfrak{N}= \mathfrak{N}_{\underline{k}_0}(\beta, \underline{\mathfrak{A}})$, et pour $k\in {\Bbb Z}$, posons $\mathfrak{Z}^k= \underline{\mathfrak{Q}}^k\mathfrak{N}\cap  \mathfrak{g}_{i,j}$. Puisque $\underline{\mathfrak{Q}}^k\subset \underline{\mathfrak{Q}}^k\mathfrak{N}$, on a $\mathfrak{Y}^k\subset \mathfrak{Z}^k$. On pose $\smash{\overline{\mathfrak{Z}}}^k= \mathfrak{Z}^k/\mathfrak{Y}^k$. Pour 
$z\in \mathfrak{Z}^k$, on a
$$
\bs{\gamma}_{i,j}(z) = -{\rm ad}_\beta(z) + z\bs{x}b_i - \bs{x} b_j z
$$
avec $z\bs{x}b_i - \bs{x}b_j z\in \mathfrak{X}^{\underline{k}_0 + k +1}$. \'Ecrivons
$$
z\bs{x}b_i - \bs{x}b_j z= -{\rm ad}_\beta(z') + \bs{x}z''
$$
avec $z'\in \mathfrak{Z}^{k+1}$ et $z''\in \mathfrak{Y}^{\underline{k}_0+k+1}_{i,j}$. On a donc
$$
\bs{\gamma}_{i,j}(z)= -{\rm ad}_\beta(z+z')+ \bs{x}z''.
$$
L'élément $z''$ est uniquement déterminé par $z$, puisqu'on a $z''= \bs{s}(z\bs{x}b_i - \bs{x}b_j z)$. Quant à l'élément $z'$, il n'est pas défini de manière unique, mais sa projection $\overline{z}'$ sur $\smash{\overline{\mathfrak{Z}}}^{k+1}$ l'est. On a donc défini deux applications $\mathfrak{o}$--linéaires
$$
\eta^k : \mathfrak{Z}^k \rightarrow \smash{\overline{\mathfrak{Z}}}^{k+1},\, z \mapsto z',\quad \nu^k: \mathfrak{Z}^k \rightarrow \mathfrak{Y}^{\underline{k}_0+ k+1},\, z\mapsto z'', 
$$
telles que
$$
\bs{\gamma}_{i,j}(z) = -{\rm ad}_\beta(z + \eta^k(z))+ \bs{x}\nu^k(z),\quad z\in \mathfrak{Z}^k.\leqno{(18)}
$$
Remarquons que pour $z=y\in \mathfrak{Z}^k\cap \mathfrak{b}_{i,j}= \mathfrak{Y}^k$, on a
$$
\nu^k(y)= yb_i - b_jy = \bs{b}_{i,j}(y).\leqno{(19)}
$$
Pour $k,\,k'\in {\Bbb Z}$ tels que $k'\geq k$, on a $\eta^k\vert_{\mathfrak{Z}^{k'}}= \eta^{k'}$ et $\nu^k\vert_{\mathfrak{Z}^{k'}}= \nu^{k'}$. On obtient deux applications $F$--linéaires
$$
\eta: \mathfrak{g}_{i,j} \rightarrow \overline{\mathfrak{g}}_{i,j}= \mathfrak{g}_{i,j}/\mathfrak{b}_{i,j},\quad \nu: \mathfrak{g}_{i,j}\rightarrow \mathfrak{b}_{i,j},
$$
telles que pour $k\in {\Bbb Z}_{\geq 1}$, on a $\eta^k= \eta\vert_{\mathfrak{Z}^k}$ et $\nu^k= \nu\vert_{\mathfrak{Z}^k}$. 
D'autre part, pour $m\in {\Bbb Z}$, d'après \cite[1.4.10]{BK}, on a la suite exacte courte
$$
0 \rightarrow \smash{\overline{\mathfrak{Z}}}^m \xrightarrow{-{\rm ad}_\beta}
\mathfrak{X}^{\underline{k}_0+ m} \buildrel \bs{s}\over{\longrightarrow} \mathfrak{Y}^{\underline{k}_0+m}\rightarrow 0.\leqno{(20)}
$$
Puisque d'après \cite[1.4.13]{BK}, on a $\underline{\mathfrak{Q}}^m\mathfrak{N}= \mathfrak{N}_0\otimes_{\mathfrak{o}_E}\underline{\mathfrak{Q}}^m$ avec $\mathfrak{N}_0= \mathfrak{N}_{k_F(\beta)}(\beta,\mathfrak{A}(E))$, elle se déduit par l'application $-\otimes_{\mathfrak{o}_E}\mathfrak{Y}^m$ de la suite exacte courte
$$
0 \rightarrow \mathfrak{N}_0 \xrightarrow{-{\rm ad}_\beta}
\mathfrak{P}^{k_F(\beta)}(E) \buildrel \bs{s}_0\over{\longrightarrow} \mathfrak{p}_E^{k_F(\beta)}\rightarrow 0.\leqno{(21)}
$$
D'après (20) et (21), pour tout $k\in {\Bbb Z}$ et toute mesure de Haar $\mathfrak{d}\bar{g}_{i,j}$ sur $\mathfrak{g}_{i,j}/\mathfrak{b}_{i,j}$, on a
$$
{{\rm vol}(\smash{\overline{\mathfrak{Z}}}^k, \mathfrak{d}\bar{g}_{i,j}) \over {\rm vol}(\mathfrak{X}^{\underline{k}_0+k}/ \mathfrak{Y}^{\underline{k}_0+k}, \mathfrak{d}\bar{g}_{i,j})} = \mu_F^+(\beta)^{\dim_E(\mathfrak{b}_{i,j})}.\leqno{(22)}
$$
De (18), (19) et (22), on déduit (voir par exemple \cite[5.3.3, 5.3.4]{L2}) que 
$$
\vert {\rm \det}_F(\bs{\gamma}_{i,j}; \mathfrak{g}_{i,j})\vert = 
\mu_F^+(\beta)^{-\dim_E(\mathfrak{b}_{i,j})}
\vert {\rm det}_E(\bs{b}_{i,j}; \mathfrak{b}_{i,j})\vert_E,
$$
ce qui est l'égalité cherchée.
\end{proof}

L'application terme constant  $C^\infty_{\rm c}(G)\rightarrow C^\infty_{\rm c}(M),\, f\mapsto f_P$ permet dualement d'associer à une distribution $T_M$ sur $M$ une distribution $\iota_P^G(T_M)$ sur $G$: pour $f\in C^\infty_{\rm c}(G)$, on pose
$$
\langle f, i_P^G(T_M)\rangle = \langle f_P,T_M\rangle. 
$$
De la même manière, l'application terme constant $C^\infty_{\rm c}(\mathfrak{b})\rightarrow C^\infty_{\rm c}(\mathfrak{m}_*),\, \mathfrak{f}\mapsto \mathfrak{f}_{\mathfrak{p}_*}$ permet dualement d'associer à une distribution $T_{\mathfrak{m}_*}$ sur $\mathfrak{m}_*$ une distribution $i_{\mathfrak{p}_*}^\mathfrak{b}(T_{\mathfrak{m}_*})$ sur $\mathfrak{b}$: pour $\mathfrak{f}\in C^\infty_{\rm c}(\mathfrak{b})$, on pose
$$
\langle \mathfrak{f}, i_{\mathfrak{p}_*}^\mathfrak{b}(T_{\mathfrak{m}_*})\rangle = \langle \mathfrak{f}_{\mathfrak{p}_*},T_{\mathfrak{m}_*}\rangle. 
$$

% lemme 2
\begin{monlem2}
Pour $b\in (\mathfrak{m}_*)_{\rm qre}\cap \mathfrak{b}_{\rm qr}\cap \underline{\mathfrak{Q}}^{\underline{k}_0+1}$, posant $\gamma = \beta + \bs{x}b$, 
on a
$$
i_{\mathfrak{p}_*}^{\mathfrak{b}}(\theta_{\ES{O}_\gamma^M})= \left( \textstyle{\vert D_{M\backslash G}(\gamma) \vert
\over \vert D_{\mathfrak{m}_*\! \backslash \mathfrak{b}}(b)\vert_E}\right)^{1\over 2}\theta_{i_P^G(\ES{O}_\gamma^M)}.
$$
\end{monlem2}

\begin{proof}Pour toute fonction $\phi\in C^\infty_{\rm c}(G\times \bs{x}\underline{\mathfrak{Q}}^{\underline{k}_0+1})$, on note $\phi_\delta^{\bs{x}}\in C^\infty_{\rm c}(\underline{\mathfrak{Q}}^{\underline{k}_0+1})$ la fonction dfinie par $\phi_\delta^{\bs{x}}(y)= \phi_\delta(\bs{x}y)$ pour $y\in \underline{\mathfrak{Q}}^{\underline{k}_0+1}$. De mme, pour toute fonction $\varphi\in C^\infty_{\rm c}(M\times \bs{x}\underline{\mathfrak{Q}}^{\underline{k}_0+1}_\mathfrak{m})$, on note $\varphi_{\delta_M}^{\bs{x}}\in C^\infty_{\rm c}(\underline{\mathfrak{Q}}^{\underline{k}_0+1}_\mathfrak{m})$ la fonction dfinie par $\varphi_{\delta_M}^{\bs{x}}(y)= \varphi_{\delta_M}(\bs{x}y)$ pour $y\in \underline{\mathfrak{Q}}^{\underline{k}_0+1}_\mathfrak{m}$. Il suffit de montrer que pour toute fonction $\phi\in C^\infty_{\rm c}(G\times \bs{x}\underline{\mathfrak{Q}}^{\underline{k}_0+1})$, il existe une fonction $\phi_{P,\mathfrak{p}_*}\in C^\infty_{\rm c}(M\times \bs{x}\underline{\mathfrak{Q}}^{\underline{k}_0+1}_\mathfrak{m})$ vrifiant les conditions $(i)$ et $(ii)$ suivantes, pour tout 
$b\in (\mathfrak{m}_*)_{\rm qre}\cap \mathfrak{b}_{\rm qr}\cap \underline{\mathfrak{Q}}^{\underline{k}_0+1}$:
\begin{enumerate}
\item[(i)] $\ES{O}_b^{\mathfrak{m}_*}((\phi_\delta^{\bs{x}})_{\mathfrak{p}_*})= \ES{O}_b^{\mathfrak{m}_*}((\phi_{P,\mathfrak{p}_*})_{\delta_M}^{\bs{x}})$;
\item[(ii)] $\ES{O}_{\beta + \bs{x}b}^M((\phi^\delta)_P) =  \left({\vert D_{M\backslash G}(\beta + \bs{x}b)\vert \over \vert D_{\mathfrak{m}_*\!\backslash \mathfrak{b}}(b)\vert}\right)^{-{1\over 2}}\ES{O}_{\beta + \bs{x}b}^M((\phi_{P,\mathfrak{p}_*})^{\delta_M})$.
\end{enumerate}
En effet, soit $\phi\in C^\infty_{\rm c}(G\times \bs{x}\underline{\mathfrak{Q}}^{\underline{k}_0+1})$, 
et supposons qu'une telle fonction $\phi_{P,\mathfrak{p}_*}$ existe. Soit $b\in (\mathfrak{m}_*)_{\rm qre}\cap \mathfrak{b}_{\rm qr}\cap \underline{\mathfrak{Q}}^{\underline{k}_0+1}$, et 
posons $\gamma = \beta + \bs{x}b\in M_{\rm qre}\cap G_{\rm qr}$. En appliquant (ii)  la distribution $\ES{O}_\gamma^M$ sur $M$, on obtient
$$
\theta_{i_P^G(\ES{O}_\gamma)}(\phi_\delta^{\bs{x}})= i_P^G(\ES{O}_\gamma)(\phi^\delta)= \ES{O}_\gamma^M((\phi^\delta)_P)= 
\left(\textstyle{\vert D_{M\backslash G}(\gamma)\vert \over \vert D_{\mathfrak{m}\backslash \mathfrak{g}}(b)\vert} \right)^{-{1\over 2}}\ES{O}_\gamma^M((\phi_{P,\mathfrak{p}_*})^{\delta_M}).
$$
En appliquant (i)  la distribution $\theta_{\ES{O}_\gamma^M}=\lambda_M \ES{O}_b^{\mathfrak{m}_*}$ sur $\mathfrak{m}_*$, on obtient
$$
\ES{O}_\gamma^M((\phi_{P,\mathfrak{p}_*})^{\delta_M})= \theta_{\ES{O}_\gamma^M}((\phi_{P,\mathfrak{p}_*})_{\delta_M}^{\bs{x}})
= \theta_{\ES{O}_\gamma^M}((\phi_\delta^{\bs{x}})_{\mathfrak{p}_*}).
$$
Or on a (par dfinition)
$$
\theta_{\ES{O}_\gamma^M}((\phi_\delta^{\bs{x}})_{\mathfrak{p}_*})= i_{\mathfrak{p}^*}^{\mathfrak{b}}(\theta_{\ES{O}_\gamma^M})(\phi_\delta^{\bs{x}}),
$$
d'o
$$
\theta_{i_P^G(\ES{O}_\gamma)}(\phi_\delta^{\bs{x}}) = \left( \textstyle{\vert D_{M\backslash G}(\gamma) \vert
\over \vert D_{\mathfrak{m}_* \backslash \mathfrak{b}}(b)\vert_E}\right)^{-{1\over 2}} i_{\mathfrak{p}^*}^{\mathfrak{b}}(\theta_{\ES{O}_\gamma^M})(\phi_\delta^{\bs{x}}).
$$
Comme l'galit ci--dessus est vraie pour toute fonction $\phi\in C^\infty_{\rm c}(G\times \bs{x}\underline{\mathfrak{Q}}^{\underline{k}_0+1})$ --- pourvu qu'il existe une fonction $\phi_{P,\mathfrak{p}_*}\in C^\infty_{\rm c}(M\times \bs{x}\underline{\mathfrak{Q}}^{\underline{k}_0+1}_{\mathfrak{m}})$ vrifiant les conditions (i) et (ii) ci--dessus --- cela dmontre le lemme. Reste  prouver l'existence de $\phi_{P,\mathfrak{p}_*}$.

Soit $\phi\in C^\infty_{\rm c}(G\times \bs{x}\underline{\mathfrak{Q}}^{\underline{k}_0+1})$. Puisque $C^\infty_{\rm c}(G\times \bs{x}\underline{\mathfrak{Q}}^{\underline{k}_0+1})= C^\infty_{\rm c}(G)\otimes C^\infty_{\rm c}(\bs{x}\underline{\mathfrak{Q}}^{\underline{k}_0+1})$, on peut par linarit supposer que $\phi$ est de la forme $\phi = f \otimes \xi$ avec $f\in C^\infty_{\rm c}(G)$ et $\xi\in C^\infty_{\rm c}(\bs{x}\underline{\mathfrak{Q}}^{\underline{k}_0+1})$. Soit $\bar{f}\in C^\infty_{\rm c}(M)$ la fonction dfinie par $\bar{f}(m)= \int_{U_P\times K}f(muk)dudk$. Notons $\mathfrak{f}$ la fonction $\xi^{\bs{x}}\in C^\infty_{\rm c}(\underline{\mathfrak{Q}}^{\underline{k}_0+1})$, et prenons le terme $K_H$--invariant $\mathfrak{f}_{\mathfrak{p}_*}\in C^\infty_{\rm c}(\mathfrak{m}_*)$ de $\mathfrak{f}$ suivant $\mathfrak{p}_*$. Puisque
$$
\underline{\mathfrak{Q}}^{\underline{k}_0+1} + \mathfrak{u}_* \subset {^H(\underline{\mathfrak{Q}}^{\underline{k}_0+1})},
$$
le support de $\mathfrak{f}_{\mathfrak{p}_*}$ est contenu dans ${^H\!(\underline{\mathfrak{Q}}^{\underline{k}_0+1})}\cap \mathfrak{m}_*$, et 
d'aprs (14), cette intersection co\"{\i}ncide avec ${^{M_*}\!(\underline{\mathfrak{Q}}^{\underline{k}_0+1}_{\mathfrak{m}})}$. On peut donc dcomposer $\mathfrak{f}_{\mathfrak{p}_*}$ en
$$
\mathfrak{f}_{\mathfrak{p}_*}= \sum_{h\in M_*} \mathfrak{f}_{\mathfrak{p}_*,h}
$$ 
avec $\mathfrak{f}_{\mathfrak{p}_*,h}\in C^\infty_{\rm c}(h\underline{\mathfrak{Q}}^{\underline{k}_0+1}_{\mathfrak{m}}h^{-1})$ et $\mathfrak{f}_{\mathfrak{p}_*,h}=0$ sauf pour un nombre fini de $h$. Comme en \ref{le principe de submersion}, on obtient une fonction  
$\sum_{h\in M_*} \mathfrak{f}_{\mathfrak{p}_*,h}\circ {\rm Ad}_h \circ \bs{s}$ sur $\bs{x}\underline{\mathfrak{Q}}^{\underline{k}_0+1}_{\mathfrak{m}}$, que l'on note 
$\xi_{\mathfrak{p}_*}\in C^\infty_{\rm c}(\bs{x}\underline{\mathfrak{Q}}^{\underline{k}_0+1}_{\mathfrak{m}})$. Bien sr la fonction $\xi_{\mathfrak{p}_*}$ n'est pas vraiment dfinie, puisqu'elle dpend de la dcomposition de $\mathfrak{f}_{\mathfrak{p}_*}$ choisie, mais pour toute distribution $M$--invariante $T_{M}$ sur $M$, la quantit $\langle \xi_{\mathfrak{p}_*}, \wt{\vartheta}_{T_M}\rangle$
est bien dfinie (elle ne dpend pas de la dcomposition de $\mathfrak{f}_{\mathfrak{p}_*}$ choisie, cf. \ref{le principe de submersion}). On pose
$$
\phi_{P,\mathfrak{p}_*}= \bar{f}\otimes \xi_{\mathfrak{p}_*}\in C^\infty_{\rm c}(M\times \bs{x}\underline{\mathfrak{Q}}^{\underline{k}_0+1}_{\mathfrak{m}}).
$$
Pour $b\in (\mathfrak{m}_*)_{\rm qre}\cap \mathfrak{b}_{\rm qr}\cap \underline{\mathfrak{Q}}^{\underline{k}_0+1}$, posant $\gamma = \beta + \bs{x}b$, on a
$$
\theta_{\ES{O}_\gamma^M}((\phi_{P,\mathfrak{p}_*})_{\delta_M}^{\bs{x}})= \theta_{\ES{O}_\gamma^M}(c\, \xi_{\mathfrak{p}_*}^{\bs{x}})
= \theta_{\ES{O}_\gamma^M}( (\phi_\delta^{\bs{x}})_{\mathfrak{p}_*})
$$
avec $c= \int_M\bar{f}(m)dm = \int_G f(g)dg$. Puisque $\theta_{\ES{O}_\gamma^M}= \lambda_M \ES{O}_b^{\mathfrak{m}_*}$, cela prouve que 
la fonction $\phi_{P,\mathfrak{p}_*}$ vrifie la condition (i). Quant  la condition (ii), pour $m\in M$ et $u\in U$, posons $v_m(u)= m^{-1}umu\in U$. Pour $m\in M$, $u\in U$, $k\in K$ et $b\in \mathfrak{m}_*\cap \mathfrak{b}_{\rm qr}$ tel que $\gamma = \beta + \bs{x}b$ appartient  $G_{\rm qr}$, l'lment  $\delta_M(m,\bs{x} b)= m^{-1}\gamma m$ appartient  $M \cap G_{\rm qr}$, et on a
$$
\delta(muk,\bs{x}b)= k^{-1}\delta_M(m,\bs{x}b) v_{\delta_M(m,\bs{x}b)}(u)k.\leqno{(23)}
$$
De plus, l'application $U\rightarrow U,\, u\mapsto v_{\delta_M(m,\bs{xb})}(u)$ est un automorphisme de varit $\mathfrak{p}$--adique de Jacobien constant 
gal 
$$
\vert {\rm \det}_F(1-{\rm Ad}_{\gamma^{-1}}; \mathfrak{u})\vert_F= \delta_P(\gamma)^{-{1\over 2}}\vert D_{M\backslash G}(\gamma)\vert^{1\over 2}.\leqno{(24)}
$$
Pour $m_*\in M_*$ et $u_*\in U_*$, posons $\mathfrak{v}_{m^*}(u_*)=u_*^{-1}m_*u_* - b_* \in \mathfrak{u}_*$.
Alors pour $m_*\in M_*$, $u_*\in U_*$, $k\in K_H$ et $b\in \mathfrak{m}_*\cap \mathfrak{b}_{\rm qr}$, on a
$$
k_*^{-1}u_*^{-1}m_*^{-1} b m_*u_*k_*= k_*^{-1}(m_*^{-1}b m_* + \mathfrak{v}_{m_*^{-1}bm_*}(u_*))k_*,
$$
et l'application $\mathfrak{u}_*\rightarrow \mathfrak{u}_*,\, y\mapsto \mathfrak{v}_{m_*^{-1}bm_*}(1+y)$ est un automorphisme de varit $\mathfrak{p}$--adique de Jacobien constant gal 
$$
\vert {\rm det}_F({\rm ad}_b; \mathfrak{u_*})\vert=\vert D_{\mathfrak{m}_*\!\backslash \mathfrak{b}}(b)\vert^{1\over 2}.\leqno{(25)}
$$
La formule (23), et les calculs des Jacobiens (24) et (25), entranent que la fonction $\phi_{P,\mathfrak{p}_*}$ vrifie la condition (ii).
\end{proof}

% proposition
\begin{mapropo}
Pour $b\in (\mathfrak{m}_*)_{\rm qre}\cap \mathfrak{b}_{\rm qr}\cap \underline{\mathfrak{Q}}^{\underline{k}_0+1}$, posant $\gamma = \beta + \bs{x}b$, 
on a
$$
\theta_{\ES{O}_\gamma} = \lambda \ES{O}_b^\mathfrak{b},
$$
où la constante $\lambda>0$ est celle de la proposition de \ref{IO normalisées}, \cad 
$\lambda = (q_E^{n_F(\beta)} \mu_F(\beta))^{d^2}$. 
\end{mapropo}

\begin{proof}
D'après la formule de descente \ref{descente parabolique}.(6), on a $i_P^G(I^M(\gamma,\cdot)= I^G(\gamma,\cdot)$, 
et d'après la formule de descente \ref{variante sur l'algèbre de Lie (suite)}.(4), on a $i_{\mathfrak{p}_*}^{\mathfrak{b}}(I^{\mathfrak{m}_*}(b,\cdot))= 
I^\mathfrak{b}(\gamma,\cdot)$. En termes des intégrales orbitales non normalisées, on a
donc $i_P^G(\ES{O}_\gamma^M)= \vert D_{M\backslash G}(\gamma)\vert^{1\over 2}\ES{O}_\gamma$ et $i_{\mathfrak{p}_*}^\mathfrak{b}(\ES{O}_b^{\mathfrak{m}_*})=\vert D_{\mathfrak{m}_*\!\backslash \mathfrak{b}}(b)\vert_E^{1\over 2}\ES{O}_b^\mathfrak{b}$.
D'aprs le lemme 1, on a
$$
\lambda = \textstyle{\vert D_{\mathfrak{m}_*\!\backslash \mathfrak{b}}(b)\vert_E \over \vert D_{M\backslash G}(\gamma)\vert}\lambda_M,
$$
et d'aprs le lemme 2 et la relation (16), on a
$$
\theta_{i_P^G(\ES{O}_\gamma^M)}=\left(\textstyle{D_{\mathfrak{m}_*\!\backslash \mathfrak{b}}(b)\vert_E \over \vert D_{M\backslash G}(\gamma)\vert }\right)^{1\over 2}i_{\mathfrak{p}_*}^{\mathfrak{b}}(\lambda_M \ES{O}_b^{\mathfrak{m}_*}).
$$
D'o la proposition.
\end{proof}

% corollaire 1
\begin{moncoro1}
Pour $b\in \mathfrak{b}_{\rm qr}\cap \underline{\mathfrak{Q}}^{\underline{k}_0+1}$, l'élément $\gamma = \beta + \bs{x}b$ appartient 
à $G_{\rm qr}$, et on a l'égalité
$$
\theta_{I^G(\gamma,\cdot)}= \vert \beta \vert_E^d \left( \textstyle{\eta_G(\gamma)\over \eta_{\mathfrak{b}}(b)}\right)^{-{1\over 2}}I^\mathfrak{b}(b,\cdot),
$$
avec
$$
\vert \beta \vert_E^d \left(\textstyle{\eta_G(\gamma)\over 
\eta_\mathfrak{b}(b)}\right)^{-{1\over 2}}= \vert \beta\vert_E^{d\over 2}\lambda^{1\over 2}.
$$
\end{moncoro1}

\begin{proof}
Soit $b\in \mathfrak{b}_{\rm qr} \cap  \underline{\mathfrak{Q}}^{\underline{k}_0+1}$. 
Puisque $\mathfrak{b}_{\rm qr} = \bigcup_{P\in \ES{P}_H} {^H\!((\mathfrak{m}_{P_H})_{\rm qre})}\cap \mathfrak{b}_{\rm qr}$, 
quitte à remplacer $b$ par $h^{-1}bh$ pour un $h\in H$, on peut (d'aprs le lemme de \ref{le rsultat principal}) 
supposer que $b\in (\mathfrak{m}_{P_H})_{\rm qre}\cap \mathfrak{b}_{\rm qr} \cap \underline{\mathfrak{Q}}^{\underline{k}_0+1}$ pour un $P_H\in \ES{P}_H$. Posons $P=P_H^G$, $M= M_P$ et $\mathfrak{m}_* = \mathfrak{m}_{P_H}$. L'élément $\gamma$ appartient à $M_{\rm qre}\cap G_{\rm qr}$, et on a 
$\theta_{\ES{O}_\gamma}= \lambda \ES{O}_b^{\mathfrak{b}}$ avec 
$$
\lambda = \lambda_M\textstyle{\vert D_{\mathfrak{m}_*\backslash \mathfrak{b}}(b)\vert_E \over \vert D_{M\backslash G}(\gamma)\vert}
= \vert \beta \vert_E^d \textstyle{\eta_{\mathfrak{m}_*}(b)\over \eta_M(\gamma)}\textstyle{\vert D_{\mathfrak{m}_*\backslash \mathfrak{b}}(b)\vert_E \over \vert D_{M\backslash G}(\gamma)\vert}
= \vert \beta \vert_E^d \textstyle{
\eta_\mathfrak{b}(b)\over \eta_G(\gamma)}.
$$
En termes des intégrales orbitales normalisées, on a donc
$$
\theta_{I^G(\gamma,\cdot)}= \vert \beta \vert_E^d \left(\textstyle{\eta_G(\gamma)\over 
\eta_\mathfrak{b}(b)}\right)^{-{1\over 2}} I^{\mathfrak{b}}(b,\cdot)
$$
avec $\vert \beta \vert_E^d \left(\textstyle{\eta_G(\gamma)\over 
\eta_\mathfrak{b}(b)}\right)^{-{1\over 2}}= \vert \beta	\vert_E^d (\vert \beta \vert_E^{-d}\lambda)^{1\over 2}
= \vert \beta\vert_E^{d\over 2}\lambda^{1\over 2}$.
\end{proof}

On en déduit que dans le corollaire de \ref{IO normalisées}, le point (i) reste vrai pour tous les 
éléments de $\underline{\mathfrak{Q}}^{\underline{k}_0+1}$ qui sont quasi--rguliers (dans $\mathfrak{b}$), et pas seulement 
pour ceux qui sont elliptiques:

% corollaire 2
\begin{moncoro2}
Pour toute fonction $f\in C^\infty_{\rm c}(G)$, il existe une fonction $f^{\mathfrak{b}}\in C^\infty_{\rm c}(\underline{\mathfrak{Q}}^{\underline{k}_0+1})$ telle que pour tout $b\in \mathfrak{b}_{\rm qr}
\cap \underline{\mathfrak{Q}}^{\underline{k}_0+1}$, on a l'égalité
$$
I^G(\beta + \bs{x}_0\otimes b,f) = I^{\mathfrak{b}}(b,f^\mathfrak{b}).
$$
\end{moncoro2}

% remarque 2
\begin{marema2}
{\rm On a bien sûr aussi la variante sur $\mathfrak{g}$ du corollaire 2 (le point (i) du corollaire de \ref{variante sur l'algbre de Lie} reste vrai pour tous les éléments de $\underline{\mathfrak{Q}}^{\underline{k}_0+1}$ qui sont quasi--réguliers, et pas seulement 
pour ceux qui sont elliptiques):
{\it --- Pour toute fonction $\mathfrak{f}\in C^\infty_{\rm c}(\mathfrak{g})$, il existe une fonction $\mathfrak{f}^{\mathfrak{b}}\in C^\infty_{\rm c}(\underline{\mathfrak{Q}}^{\underline{k}_0+1})$ telle que pour tout $b\in \mathfrak{b}_{\rm qr}\cap \underline{\mathfrak{Q}}^{\underline{k}_0+1}$, on a l'égalité
$$
I^\mathfrak{g}(\beta + \bs{x}_0\otimes b,\mathfrak{f}) = I^{\mathfrak{b}}(b,\mathfrak{f}^\mathfrak{b}).\eqno{\blacksquare}
$$
}}
\end{marema2}

%%%%%%%%%%%%%%%%%%%%%%%%%%
\subsection{Descente centrale au voisinage d'un élément fermé}\label{descente centrale au voisinage d'un élément fermé}
Soit $\beta\in G$ un élément fermé (cf. \ref{éléments qre}). Notons $\mathfrak{b}$ le centralisateur $\mathfrak{g}_\beta={\rm End}_{F[\beta]}(V)$ de $\beta$ dans $\mathfrak{g}$. La $F$--algèbre $F[\beta]$ se décompose en $F[\beta]= E_1\times \cdots \times E_r$ pour des extensions $E_i/F$. Pour $i=1,\ldots ,r$, notons $e_i$ l'idempotent de $F[\gamma]$ associé à $E_i$, et posons $V_i= e_i(V)$, $\mathfrak{g}_i={\rm End}_F(V_i)$ et $\mathfrak{b}_i= {\rm End}_{E_i}(V_i)$. On a donc la décomposition $\mathfrak{b}= \mathfrak{b}_1\times \cdots \times \mathfrak{b}_r$, et l'élément $\beta=(\beta_1,\ldots ,\beta_r)$ est ($F$--){\it pur} dans $\mathfrak{m}= \mathfrak{g}_1\times \cdots \times \mathfrak{g}_r$, au sens où pour $i=1,\ldots ,r$, l'élément $\beta_i$ est pur dans $\mathfrak{g}_i$. Notons $H$ le centralisateur $G_\beta =\mathfrak{b}^\times$ de $\beta$ dans $G$, $A_\beta$ le tore déployé maximal du centre $Z(H)=F[\beta]^\times$ de $H$, et $M=M(\beta)$ le centralisateur $Z_G(A_\beta)$ de $A_\beta$ dans $G$. On a donc $H= H_1\times \cdots \times H_r$ avec $H_i={\rm Aut}_{E_i}(V_i)$, et $M= G_1\times \cdots \times G_r$ avec $G_i= {\rm Aut}_F(V_i)$. Quitte à remplacer $\beta$ par un conjugué par $g$, on peut supposer que $\beta$ est en \og position standard\fg, \cad que $A_\beta= A_P$ pour un $P\in \ES{P}$. Alors on a $M= M_P$ et $\mathfrak{m}=\mathfrak{m}_P$. Rappelons que pour toute fonction $f\in C^\infty_{\rm c}(G)$, on a la formule de descente (\ref{descente parabolique}.(6))
$$
I^G(\gamma,f)= I^M(\gamma,f_P), \quad \gamma \in M\cap G_{\rm qr}.\leqno{(1)}
$$

Pour $i\in \{1,\ldots ,r\}$ tel que $E_i\neq F$, on peut descendre les intégrales orbitales normalisées 
au voisinage de $\beta_i$ dans $G_i$ en des intégrales orbitales normalisées au voisinage de $0$ dans $\mathfrak{b}_i$ 
(corollaire 2 de \ref{descente centrale au voisinage d'un élément pur (suite)}): il existe un élément $\bs{x}_i\in \mathfrak{g}_i$ et un voisinage ouvert compact $\ES{V}_i$ de $0$ dans $\mathfrak{b}_i$, tels que pour toute fonction $f_i\in C^\infty_{\rm c}(G_i)$, il existe une fonction $f_i^{\mathfrak{b}_i}\in C^\infty_{\rm c}(\ES{V}_i)$ telle que
$$
I^{G_i}(\beta_i + \bs{x}_i b_i,f_i)=  I^{\mathfrak{b}_i}(b_i,f_i^{\mathfrak{b}_i}),\quad b_i \in (\mathfrak{b}_i)_{\rm qr} \cap \ES{V}_i.
$$
Précisément, l'élément $\bs{x}_i$ est de la forme $\bs{x}_i=\bs{x}_{i,0}\otimes 1$ pour un élément $\bs{x}_{i,0}\in \mathfrak{A}(E_i)$ dans l'image réciproque de $1$ par une corestriction modérée $\bs{s}_{i,0}:A(E_i)\rightarrow E_i$ sur $A(E_i)$ relativement à $E_i/F$ et une $(W_i,E_i)$--décomposition $\mathfrak{g}_i=A(E_i)\otimes_{E_i}\mathfrak{b}_i$ de $\mathfrak{g}_i$ induite par une $(W_i,E_i)$--décomposition $\underline{\mathfrak{A}_i}= \mathfrak{A}(E_i)\otimes_{\mathfrak{o}_{E_i}}\underline{\mathfrak{B}}_i$ de $\underline{\mathfrak{A}}_i$, où $\underline{\mathfrak{A}}_i$ est un $\mathfrak{o}$--ordre héréditaire dans $\mathfrak{g}_i$ normalisé par $E_i^\times$ tel que $\underline{\mathfrak{B}}_i= \mathfrak{b}_i\cap \underline{\mathfrak{A}}_i$ est un $\mathfrak{o}_{E_i}$--ordre héréditaire minimal dans $\mathfrak{b}_i$. Le voisinage $\ES{V}_i$ de $0$ dans $\mathfrak{b}_i$ est donné par $\ES{V}_i = \underline{\mathfrak{Q}}_i^{\smash{\underline{k}_{i,0}+1}}$, où $\underline{\mathfrak{Q}}_i ={\rm rad}(\underline{\mathfrak{B}}_i)$ et $\underline{k}_{i,0} = k_0(\beta_i, \underline{\mathfrak{A}}_i)= d_i k_F(\beta_i)$ avec $d_i= \dim_{E_i}(V_i)\;(=e(\underline{\mathfrak{B}}_i\vert \mathfrak{o}_{E_i}))$. Rappelons que d'après \ref{éléments qre}.(7), ${^{H_i}(\ES{V}_i)}$ est une partie ouverte fermée et $H_i$--invariante dans $\mathfrak{b}_i$, et que d'après \ref{le résultat principal}.(3), $\omega_i={^{G_i}(\beta_i+ \bs{x}_i \ES{V}_i)}$ est une partie ouverte fermée et $G$--invariante dans $G$. 

Pour $i\in \{1,\ldots ,r\}$ tel que $E_i=F$, l'élément $\beta_i$ appartient à $F^\times$, et on a $\mathfrak{b}_i=\mathfrak{g}_i$ et $H_i=G_i$. On choisit un $\mathfrak{o}$--ordre héréditaire minimal $\underline{\mathfrak{B}}_i$ dans $\mathfrak{b}_i$. On pose $\bs{x}_i=1$,  $\underline{\mathfrak{Q}}_i = {\rm rad}(\underline{\mathfrak{B}}_i)$, $d_i=\dim_F(V_i)$ et $\ES{V}_i = \underline{\mathfrak{Q}}_i^{d_i\nu(\beta_i)+1}$. D'après \ref{éléments qre}.(7), l'ensemble 
$\omega_i= {^{G_i}(\beta_i + \bs{x}_i\ES{V}_i)}= \beta_i +  {^{G_i}(\ES{V}_i)}$ 
est un voisinage ouvert fermé et $G_i$--invariant de $\beta_i$ dans $G_i$. Pour toute fonction $f_i\in C^\infty_{\rm c}(G_i)$, on choisit une décomposition
$$
f_i = \sum_{g_i\in G_i} f_{i,g_i}
$$
avec $f_{i,g_i}\in C^\infty_{\rm c}(\beta_i + g_i\ES{V}_i g_i^{-1})$ et $f_{i,g_i}=0$ pour presque tout $g_i\in G_i$, et on note $f_i^{\mathfrak{b}_i}\in C^\infty_{\rm c}(\ES{V}_i)$ la fonction définie par
$$
f_i^{\mathfrak{b}_i}(b_i) = \sum_{g_i\in G_i} f_i\circ {\rm Int}_{g_i}(\beta_i + b_i), \quad b_i\in \ES{V}_i.
$$ On a donc
$$
I^{G_i}(\beta_i + \bs{x}_i b_i,f_i)=  I^{\mathfrak{b}_i}(b_i,f_i^{\mathfrak{b}_i}),\quad b_i \in 
(\mathfrak{b}_i)_{\rm qr}\cap \ES{V}_i.
$$

Soit $\bs{x}=(\bs{x}_1,\ldots ,\bs{x}_r)\in \mathfrak{m}$, et soit $\ES{V}$ le voisinage ouvert compact de $0$ dans $\mathfrak{b}$ défini par $\ES{V}= \ES{V}_1\times \cdots \times \ES{V}_r$. Notons que ${^H\ES{V}}= {^{H_1}(\ES{V}_1)}\times\cdots\times {^{H_r}(\ES{V}_r)}$ est un voisinage ouvert fermé et $H$--invariant de $0$ dans $\mathfrak{b}$, et que ${^M(\beta + \bs{x}\ES{V})}= \omega_1\times\cdots \times\omega_r$ est un voisinage ouvert fermé et $M$--invariant de $\beta$ dans $M$. D'après ce qui précède, pour toute fonction $f\in C^\infty_{\rm c}(M)$, il existe une fonction $f^\mathfrak{b}\in C^\infty_{\rm c}(\ES{V})$ telle que
$$
I^M(\beta + \bs{x}b,f)= I^{\mathfrak{b}}(b,f^\mathfrak{b}),\quad b\in \mathfrak{b}_{\rm qr} \cap \ES{V}.\leqno{(2)}
$$
En effet, puisque $C^\infty_{\rm c}(M)= C^\infty_{\rm c}(G_1)\otimes\cdots \otimes C^\infty_{\rm c}(G_r)$, la fonction $f$ est une combinaison linéaire de fonctions du type $f_1\otimes \cdots \otimes f_r$ avec $f_i \in C^\infty_{\rm c}(G_i)$. Pour $f= f_1\otimes \cdots \otimes f_r$, la fonction $f^\mathfrak{b}= f_1^{\mathfrak{b}_1}\otimes \cdots \otimes f_r^{\mathfrak{b}_r}$ convient. D'où le résultat par linéarité. 

Posons
$$
\ES{V}'= \{b\in \ES{V}: D_{M\backslash G}(\beta + \bs{x}b)\neq 0\}.
$$
Puisque $\mathfrak{b}= \mathfrak{g}_\beta$ est contenu dans $\mathfrak{m}$, on a $D_{M\backslash G}(\beta) \neq 0$, et $\ES{V}'$ est un voisinage ouvert de $0$ dans $\mathfrak{b}$. On a l'inclusion
$$
\{\beta + \bs{x}b: b\in \mathfrak{b}_{\rm qr}\cap \ES{V}'\} \subset M\cap G_{\rm qr}.\leqno{(3)}
$$
En effet, d'après \ref{descente centrale au voisinage d'un élément pur (suite)}, pour tout $b\in \mathfrak{b}_{\rm qr}\cap \ES{V}$, 
l'élément $\beta + \bs{x}b$ appartient à $M_{\rm qr}$. Par conséquent pour tout $b\in \mathfrak{b}_{\rm qr}\cap \ES{V}'$, 
puisque $\beta + \bs{x}b\in M_{\rm qr}$ et $D_{M\backslash G}(\beta + \bs{x}b)\neq 0$, 
l'élément $\beta + \bs{x}b$ appartient à $M\cap G_{\rm qr}$.

L'application
$$
\delta_M: M\times \bs{x} \ES{V} \rightarrow M, \, (m, \bs{x}b)\mapsto m^{-1}(\beta + \bs{x}b)m
$$
est partout submersive. 
Par définition de $\ES{V}'$, l'application
$$
\delta' : G \times \bs{x}\ES{V}' \rightarrow G, \, (g,\bs{x}b)\mapsto g^{-1}(\beta + \bs{x}b) g
$$
est elle aussi partout submersive. Pour $g\in G$ et $b\in \mathfrak{b}$, puisque $G=KP=PK$, on peut écrire $g=muk$ avec $m\in M$, $u\in U_P$ et $k\in K$. Alors pour $b\in \ES{V}'$, posant $\gamma = m^{-1}(\beta + \bs{x}b)m$, on a
$$
g^{-1}(\beta + \bs{x}b)g= k^{-1} \gamma (\gamma^{-1} u^{-1}\gamma u)k.
$$
Puisque $D_{M\backslash G}(\gamma)\neq 0$, l'application $U_P\rightarrow U_P,\, u\mapsto \gamma^{-1}u^{-1}\gamma u$ est un automorphisme de variété $\mathfrak{p}$--adique. D'où l'inclusion
$$
{\rm Im}(\delta') \subset {^K({\rm Im}(\delta_M)U_P)}.\leqno{(4)}
$$

% remarque 1
\begin{marema1}
{\rm L'inclusion (4) n'implique pas que si $f\in C^\infty_{\rm c}(G)$ est à support contenu dans ${\rm Im}(\delta')$, alors la fonction $f_P\in C^\infty_{\rm c}(M)$ est à support contenu dans ${\rm Im}(\delta_M)$. D'autre part, d'après le lemme 3 de \ref{parties compactes modulo conjugaison}, il existe un voisinage ouvert fermé et $M$--invariant $\Xi_M$ de $\beta$ dans $M$ tel que $\Xi_M \subset {^M(\beta + \bs{x} \ES{V}')}$. Pour un tel $\Xi_M$, on a ${^G(\Xi_M)}= {^K(\Xi_M U_P)}$, et $\Xi={^G(\Xi_M)}$ est un voisinage ouvert fermé et $G$--invariant de $\beta$ dans $G$.\hfill$\blacksquare$
}
\end{marema1}

D'après le lemme 3 de \ref{parties compactes modulo conjugaison}, on peut choisir un voisinage ouvert fermé et $M$--invariant $\Xi_M$ de $\beta$ dans $M$ tel que $\Xi_M \subset {^M(\beta + \bs{x} \ES{V}')}$. Pour toute fonction $f\in C^\infty_{\rm c}(G)$, la fonction $f_P\vert_{\Xi_M}$ est à support dans ${\rm Im}(\delta_M)= {^M(\beta + \bs{x}\ES{V})}$, et d'après (2), il existe une fonction $f_{\Xi_M}^\mathfrak{b} = (f_P\vert_{\Xi_M})^{\mathfrak{b}}\in C^\infty_{\rm c}(\ES{V})$ telle que
$$
I^M(\beta + \bs{x}b,f_P\vert_{\Xi_M}) = I^{\mathfrak{b}}(b,f_{\Xi_M}^{\mathfrak{b}}),\quad b\in \mathfrak{b}_{\rm qr}\cap \ES{V}.
$$
D'après (1), on en déduit que pour tout $b\in \mathfrak{b}_{\rm qr}\cap \ES{V}$ tel que $\beta + \bs{x}b\in \Xi_M$, on a l'égalité
$$
I^G(\beta +\bs{x}b,f)= I^{\mathfrak{b}}(b,f_{\Xi_M}^{\mathfrak{b}}). \leqno{(4)}
$$
D'après la remarque 1, ${^G(\Xi_M)}= {^K(\Xi_M U_P)}$ est un voisinage ouvert fermé et $G$--invariant de $\beta$ dans $G$. 

On peut choisir $\Xi_M$ de la forme
$$
\Xi_M= {^M(\beta + \bs{x}\ES{V}^{(k)})}, \quad \ES{V}^{(k)} = \underline{\mathfrak{Q}}_1^{d_1k_1 + 1}\times \cdots \times \underline{\mathfrak{Q}}_r^{d_rk_r+1},
$$
pour un entier un $r$--uplet $(k)= (k_1,\ldots ,k_r)\in {\Bbb Z}^r$ tel que $\ES{V}^{(k)}\subset \ES{V}'$ (pour $i=1,\ldots ,r$, on a 
donc forcément $k_i \geq \max\{ k_F(\beta_i),\nu_{E_i}(\beta_i)\}$). D'où la

% proposition
\begin{mapropo}
Soit un $r$--uplet $(k)\in {\Bbb Z}^r$ tel que $\ES{W}=\ES{V}^{(k)} \subset \ES{V}'$. Pour toute fonction $f\in C^\infty_{\rm c}(G)$, il existe une fonction $f^{\mathfrak{b}}\in C^\infty_{\rm c}(\ES{\ES{W}})$ telle que pour tout $b\in \mathfrak{b}_{\rm qr}\cap \ES{W}$, on a l'égalité
$$
I^G(\beta + \bs{x} b,f) = I^{\mathfrak{b}}(b,f^\mathfrak{b}).
$$
\end{mapropo}

% remarque 2
\begin{marema2}
{\rm On a aussi la variante sur $\mathfrak{g}$ de la proposition. Soit $\beta\in \mathfrak{g}$ un élément fermé. On reprend à l'identique les constructions précédentes, la seule différence étant que dans la décomposition $\beta =(\beta_1,\ldots ,\beta_r)$, les éléments $\beta_j$ tels que $F[\beta_j]=F$ sont dans $F$ (et pas forcément dans $F^\times$). On suppose toujours que $M=M(\beta)$ est de la forme $M=M_P$ pour un $P\in \ES{P}$. On définit $\ES{V}$ de la même manière --- c'est un voisinage ouvert compact de $\beta$ dans $\mathfrak{b}$ ---, et on pose
$$
\ES{V}'= \{b\in \ES{V}: D_{\mathfrak{m}\backslash \mathfrak{g}}(\beta + \bs{x}b)\neq 0\}.
$$
Alors $\ES{V}'$ est un voisinage ouvert de $\beta$ dans $\mathfrak{b}$. L'application
$$
\delta_M: M\times \bs{x} \ES{V} \rightarrow \mathfrak{m}, \, (m, \bs{x}b)\mapsto m^{-1}(\beta + \bs{x}b)m
$$
est partout submersive. 
Par définition de $\ES{V}'$, l'application
$$
\delta' : G \times \bs{x}\ES{V}' \rightarrow \mathfrak{g}, \, (g,\bs{x}b)\mapsto g^{-1}(\beta + \bs{x}b) g
$$
est elle aussi partout submersive. Pour un $r$--uplet $(k)\in {\Bbb Z}^r$, on définit le voisinage ouvert compact $\ES{V}^{(k)}$ de 
$0$ dans $\mathfrak{b}$ comme plus haut. 
On obtient de la même manière: {\it --- Soit un $r$--uplet $(k)\in {\Bbb Z}^r$ tel que $\ES{W}=\ES{V}^{(k)}\subset \ES{V}'$. 
Pour toute fonction $\mathfrak{f}\in C^\infty_{\rm c}(\mathfrak{g})$, il existe une fonction $\mathfrak{f}^\mathfrak{b}\in C^\infty_{\rm c}(\ES{W})$ tel que pour tout 
$b\in \mathfrak{b}_{\rm qr}\cap \ES{W}$, on a l'égalité
$$
I^\mathfrak{g}(\beta + \bs{x}b,\mathfrak{f})= I^\mathfrak{b}(b,\mathfrak{f}^\mathfrak{b}).\eqno{\blacksquare}
$$
}
}
\end{marema2}

%%%%%%%%%%%%%%%%%%%%%%%%%%
\section{Germes de Shalika et résultats sur l'algèbre de Lie}\label{germes de shalika et résultats sur l'algèbre de lie}

%%%%%%%%%%%%%%%%%%%%%%%%%%%
\subsection{Théorie des germes de Shalika}\label{théorie des germes de shalika}On reprend dans ce numéro les principaux éléments de la théorie des germes de Shalika au voisinage de $0$ dans $\mathfrak{g}_{\rm qr}$. Elle est exactement la mme qu'en caractristique nulle. Soit $\ES{N}$ l'ensemble des éléments nilpotents de $\mathfrak{g}$. On sait que $\ES{N}$ est réunion d'un nombre fini de classes de $G$--conjugaison, paramétrisées par les partitions de $N$. Soient $\ES{O}_0,\ldots ,\ES{O}_{d_\ES{N}}\subset \ES{N}$ l'ensemble de ces classes de $G$--conjugaison, ordonnées de telle manière que $\dim(\ES{O}_i)\leq \dim(\ES{O}_{i+1})$, où $\dim(X)$ désigne la dimension d'une variété $\mathfrak{p}$--adique $X$. On a donc $\ES{O}_0=\{0\}$, et $\ES{O}_{d_\ES{N}}$ est l'orbite nilpotente régulière, \cad celle de dimension $N^2-N$. Pour $k=0,\ldots ,d_{\ES{N}}$, 
la partie $\ES{N}_i=\coprod_{i=0}^k\ES{O}_i$ est fermée dans $G$, et l'orbite $\ES{O}_i$ est ouverte dans $\ES{N}_i$.

Pour $i=0,\ldots ,d_{\ES{N}}$, choisissons un élément $x_i\in \ES{O}_i$. On sait que le centralisateur $G_{x_i}$ de $x_i$ dans $G$ est unimodulaire. On peut donc fixer une mesure de Haar $dg_{x_i}$ sur $G_{x_i}$. Pour une fonction $\mathfrak{f}\in C^\infty_{\rm c}(\mathfrak{g})$, on pose
$$
\ES{O}_{x_i}(\mathfrak{f})= \int_{G_{x_i}\backslash G}\mathfrak{f}(g^{-1}x_i g)\textstyle{dg\over dg_{x_i}},.
$$
D'après \cite{Ho}, cette intégrale est absolument convergente. Elle définit donc une distribution $\ES{O}_{x_i}$ sur $\mathfrak{g}$, de support la fermeture $\overline{\ES{O}}_i$ de l'orbite $\ES{O}_i$ dans $\mathfrak{g}$ (pour la topologie $\mathfrak{p}$--adique). \`A l'élément nilpotent $x_i$ est associé comme suit un sous--groupe parabolique $P_{x_i}$ de $G$. Pour chaque entier $k\geq 0$, on note $V_{x_i}^k$ le noyau de l'endomorphisme $x_i^k$ de $V$. Alors
$$
P_{x_i}= \{g\in G: g(V_{x_i}^k)\subset V_{x_i}^k,\, k\geq 1\}.
$$
Le radical unipotent $U_{x_i}$ de $P_{x_i}$ est donné par
$$
U_{x_i}= \{g\in G: g(V_{x_i}^k)\subset V_{x_i}^{k-1}, k\geq 1\}.
$$
Soit $r_i$ le rang de $x_i$, \cad le plus petit entier $k\geq 1$ tel que $x_i^{k-1}\neq 0$. L'élément $x_i$ appartient à $U_{x_i}$, 
et pour $i=2,\ldots ,r_i$, il induit par passage aux quotients une application injective $V_{x_i}^k/ V_{x_i}^{k-1} \rightarrow V_{x_i}^{k-1}/V_{x_i}^{k-2}$, ce qui signifie (d'après le lemme 2 de \cite{Ho}) que l'orbite $\ES{O}_{P_{x_i}}(x_i)=\{p^{-1}x_ip: p\in P_{x_i}\}$ est dense dans $U_{x_i}$ pour la topologie $\mathfrak{p}$--adique. C'est d'ailleurs ce résultat qui permet de montrer la convergence absolue de l'intégrale orbitale $\ES{O}_{x_i}(\mathfrak{f})$. 

Pour $z\in F^\times$ et $\mathfrak{f}\in C^\infty_{\rm c}(\mathfrak{g})$, on note $\mathfrak{f}^z\in C^\infty_{\rm c}(\mathfrak{g})$ la fonction définie par $\mathfrak{f}^z(y)= \mathfrak{f}(zy)$. Pour $i=0, \ldots , d_{\ES{N}}$, l'orbite $\ES{O}_i$ vérifie $z\ES{O}_i=\ES{O}_i$, et la distribution $\ES{O}_{x_i}$ vérifie (cf. \cite[3.6.1]{L2})
$$
\ES{O}_{x_i}(\mathfrak{f}^z)= \vert z \vert^{{1\over 2}\dim(\ES{O}_i)}\ES{O}_{x_i}(\mathfrak{f}),\quad \mathfrak{f}\in C^\infty_{\rm c}(\mathfrak{g}).\leqno{(1)}
$$

Soient $\{\mathfrak{f}_i: i=0\ldots ,d_{\ES{N}}\}\subset C^\infty_{\rm c}(\mathfrak{g})$ un ensemble de fonctions vérifiant les deux conditions suivantes (cf. \cite[3.5.1]{L2}):
\begin{enumerate}
\item[(i)] pour $i=0,\ldots , m$, le support de la restriction de $\mathfrak{f}_i$ à $\ES{N}_i$ est contenu dans $\ES{O}_i$; 
\item[(ii)] pour $1\leq i,\,j \leq d_{\ES{N}}$, on a $\ES{O}_{x_i}(\mathfrak{f}_j) = \delta_{i,j}$ (symbole de Kronecker).
\end{enumerate} 
Pour toute fonction $\phi:\mathfrak{g}_{\rm qr}\rightarrow {\Bbb C}$, on note $[\phi]^\mathfrak{g}_0$ le germe de fonctions au voisinage de $0$ dans $\mathfrak{g}_{\rm qr}$ définie par $\phi$. Deux fonctions $\phi, \, \phi': \mathfrak{g}_{\rm qr}\rightarrow {\Bbb C}$ définissent le même germe 
$[\phi]^\mathfrak{g}_0=[\phi']^\mathfrak{g}_0$ si et seulement s'il existe un voisinage $\ES{V}$ de $0$ dans $\mathfrak{g}$ tel que $(\phi-\phi')\vert_{\ES{V} \cap \mathfrak{g}_{\rm qr}}=0$. D'après \cite{Ho} (cf. \cite[3.5.2]{L2}), pour toute fonction $\mathfrak{f}\in C^\infty_{\rm c}(\mathfrak{g})$, la fonction $\gamma \mapsto \ES{O}_\gamma(\mathfrak{f})$ sur $\mathfrak{g}_{\rm qr}$ admet le développement en germes
$$
[\ES{O}_\gamma(\mathfrak{f})]^\mathfrak{g}_0= \sum_{i=0}^{d_{\ES{N}}} \ES{O}_{x_i}(\mathfrak{f})[\ES{O}_\gamma(\mathfrak{f}_i)]^\mathfrak{g}_0.\leqno{(2)}
$$
En d'autres termes, il existe un voisinage $\ES{V}_{\mathfrak{f}}$ de $0$ dans $\mathfrak{g}$ tel que pour tout $\gamma \in \ES{V}_{\mathfrak{f}}\cap \mathfrak{g}_{\rm qr}$, on a l'égalité
$$
\ES{O}_\gamma(\mathfrak{f})= \sum_{i=0}^{d_{\ES{N}}} \ES{O}_{x_i}(\mathfrak{f})\ES{O}_\gamma(\mathfrak{f}_i).
$$
De plus, les germes de fonctions $[\gamma \mapsto \ES{O}_\gamma (\mathfrak{f}_i)]^\mathfrak{g}_0$ au voisinage de $0$ dans $\mathfrak{g}_{\rm qr}$ sont uniquement déterminés par le développement (2) pour toute fonction $\mathfrak{f}\in C^\infty_{\rm c}(\mathfrak{g})$. En particulier, si $\{\mathfrak{f}'_i: i=0,\ldots ,d_{\ES{N}}\}\subset C^\infty_{\rm c}(\mathfrak{g})$ est un autre ensemble de fonctions vérifiant les conditions (i) et (ii), alors pour $i=0,\ldots , m$, on a l'égalité des germes
$[\gamma \mapsto \ES{O}_\gamma(\mathfrak{f}'_i)]^\mathfrak{g}_0= [\gamma \mapsto \ES{O}_\gamma(\mathfrak{f}_i)]^\mathfrak{g}_0$. On note $\bs{a}_i$ ce germe 
$[\gamma \mapsto \ES{O}_\gamma(\mathfrak{f}_i)]_0$. On l'appelle {\it germe de Shalika} associé à l'orbite $\ES{O}_i$. 

Pour $z\in F^\times$ et $\gamma\in \mathfrak{g}_{\rm qr}$, on a $z\gamma\in \mathfrak{g}_{\rm qr}$. On peut donc, pour tout germe de fonctions $\bs{a}$ au voisinage de $0$ dans $\mathfrak{g}_{\rm qr}$, 
définir le germe $\bs{a}^z$: si $\bs{a}= [\varphi]^\mathfrak{g}_0$ pour une fonction $\varphi:\mathfrak{g}_{\rm qr}\rightarrow {\Bbb C}$, on pose $\bs{a}^z = [\varphi^z]^\mathfrak{g}_0$ avec $\varphi^z(\gamma)= \varphi(z\gamma)$, $\gamma\in \mathfrak{g}_{\rm qr}$. Pour $i=0,\ldots , d_{\ES{N}}$, d'après (2) et la propriété d'unicité des germes de Shalika, le germe $\bs{a}_i$ vérifie la formule d'homogénéité
$$
\bs{a}_i^z= \vert z\vert^{- {1\over 2}\dim(\ES{O}_i)}\bs{a}_i,\quad z\in F^\times.\leqno{(3)}
$$
Grâce à cette propriété d'homogénéité, on peut remplacer les germes de Shalika par des fonctions canoniques sur $\mathfrak{g}_{\rm qr}$, induisant les mêmes germes au voisinage de $0$ dans $\mathfrak{g}_{\rm qr}$. En effet (cf. \cite[17.8]{K}), pour $i=0,\ldots , d_{\ES{N}}$, il existe une unique fonction $\tilde{\bs{a}}_i: \mathfrak{g}_{\rm qr} \rightarrow {\Bbb C}$ telle que $[\tilde{\bs{a}}_i]^\mathfrak{g}_0= \bs{a}_i$ et $\tilde{\bs{a}}_i(z\gamma)= \vert z\vert^{-{1\over 2}\dim(\ES{O}_i)}\tilde{\bs{a}}_i(\gamma)$ pour tout 
$\gamma\in \mathfrak{g}_{\rm qr}$ et tout $z\in F^\times$. De plus (loc.~cit.), la fonction $\tilde{\bs{a}}_i$ est à valeurs dans ${\Bbb R}$ (car on peut choisir les fonctions $\mathfrak{f}_i$, pour $i=0,\ldots ,d_\ES{N}$, à valeurs dans ${\Bbb R}$), elle est invariante par translation par les éléments du centre $\mathfrak{z}=F$ de $\mathfrak{g}$, et invariante par conjugaison par les éléments de $G$. Bien sûr ces fonctions $\tilde{\bs{a}}_i$ dépendent de la normalisation des distributions $\ES{O}_\gamma$ pour 
$\gamma\in \mathfrak{g}_{\rm qr}$, et aussi de celle des distributions $\ES{O}_{x_i}$ pour $i=0,\dots ,d_{\ES{N}}$.

\vskip1mm
Les germes de Shalika $\bs{a}_i$, ou ce qui revient au même, les fonctions $\tilde{\bs{a}}_i$, sont en général très difficiles à calculer (voir par exemple \cite{R1, R2}). On dispose cependant d'un résultat crucial, concernant le germe $\bs{a}_0$ associé à l'orbite nulle: il induit un germe de fonctions constant au voisinage de $0$ dans $\mathfrak{g}_{\rm re}$. Précisément, en imposant la condition $\ES{O}_{x_0}= \bs{\delta}_{x_0}$ (mesure de Dirac en $x_0=0$), on a \cite[A.3.3]{He}
$$
\tilde{\bs{a}}_0(\gamma) = (-1)^{N-1}d({\rm St}_G)^{-1}, \quad \gamma \in \mathfrak{g}_{\rm re},\leqno{(4)}
$$
où $d({\rm St}_G)$ est le degré formel de la représentation de Steinberg ${\rm St}_G$ de $G$. Pour définir 
ce degré formel, on utilise bien sûr ici la mesure $\textstyle{dg\over dz}$ sur l'espace quotient $Z \backslash G$, où $dz$ est la mesure de Haar sur $Z=F^\times$ qui donne le volume $1$ à $U_F$. D'ailleurs, puisque
$$
\bs{a}_0= [y \mapsto \ES{O}_y(\mathfrak{f}_0)]^\mathfrak{g}_0
$$
et que l'application $\gamma \mapsto \ES{O}_\gamma(\mathfrak{f}_0)$ est localement constante sur $\mathfrak{g}_{\rm qre}$, l'égalité (4) est vraie pour tout $\gamma\in \mathfrak{g}_{\rm qre}$. Grâce à (4) et à la formule d'homogénéité (3), on obtient (cf. \cite[5.6.1]{L2}) que les fonctions $\tilde{\bs{a}}_i: \mathfrak{g}_{\rm qr}\rightarrow {\Bbb R}$ pour $i=0,\ldots , d_{\ES{N}}$, et même leurs restrictions à l'ouvert 
$\mathfrak{g}_{\rm r}\subset \mathfrak{g}_{\rm qr}$ des éléments (quasi--réguliers) séparables, sont linéairement indépendantes sur ${\Bbb C}$: si 
$$
\sum_{i=0}^{d_{\ES{N}}}\mu_i \tilde{\bs{a}}_i\vert_{\mathfrak{g}_{\rm r}}= 0
$$ pour des nombres complexes $\mu_i$, alors on a forcément $\mu_0=\cdots = \mu_{\ES{N}}=0$. 

%%%%%%%%%%%%%%%%%%%%%%%%%%%%%%%%
\subsection{Germes de Shalika normalisés}\label{germes de shalika normalisés}On peut dans le développement en germes 
\ref{théorie des germes de shalika}.(2), remplacer les distributions $\ES{O}_\gamma$ ($\gamma\in \mathfrak{g}_{\rm qr}$) par les distributions normalisées $I^\mathfrak{g}(\gamma,\cdot)$. On obtient de la même manière, pour toute fonction $\mathfrak{f}\in C^\infty_{\rm c}(\mathfrak{g})$, le développement en germes
$$
[I^\mathfrak{g}(\gamma,\mathfrak{f})]^\mathfrak{g}_0= \sum_{i=0}^{d_{\ES{N}}}\ES{O}_{x_i}(f) [I^\mathfrak{g}(\gamma, \mathfrak{f}_i)]^\mathfrak{g}_0.\leqno{(1)}
$$
Pour $i=0,\ldots ,d_{\ES{N}}$, le germe $[\gamma \mapsto I^\mathfrak{g}(\gamma, \mathfrak{f}_i]^\mathfrak{g}_0$ au voisinage de $0$ dans $\mathfrak{g}_{\rm qr}$ est appelé {\it germe de Shalika normalisé} associé à l'orbite nilpotente $\ES{O}_i$, et noté $\bs{b}_i$. Comme pour les germes $\bs{a}_i$, les germes de Shalika normalisés $\bs{b}_i$ sont uniquement déterminé par le développement (1) pour toute fonction $\mathfrak{f}\in C^\infty_{\rm c}(\mathfrak{g})$. De plus, par définition des distributions normalisées $I^\mathfrak{g}(\gamma,\cdot) $ ($\gamma \in \mathfrak{g}_{\rm qr}$), on a
$$
\bs{b}_i= \eta_\mathfrak{g}^{1\over 2}\bs{a}_i, \quad i=0,\ldots , d_{\ES{N}}.\leqno{(2)}
$$
Pour $\gamma\in \mathfrak{g}_{\rm qre}$ et $z\in F^\times$, on a
$$
\eta_\mathfrak{g}(z\gamma) =q^{-f_\gamma(c_F(z\gamma)+ e_\gamma -1)}= q^{-f_\gamma(e_\gamma(N-1)\nu(z) +c_F(\gamma)+ e_\gamma-1)}= \vert z\vert^{N(N-1)}\eta_{\mathfrak{g}}(\gamma).
$$
En en déduit que pour $\gamma\in \mathfrak{g}_{\rm qr}$ et $z\in F^\times$, on a encore
$$
\eta_{\mathfrak{g}}(z\gamma)= \vert z\vert^{N(N-1)}\eta_{\mathfrak{g}}(\gamma). \leqno{(3)}
$$
En effet, d'après la définition de $\eta_\mathfrak{g}(\gamma)$, posant $\mathfrak{m}= \mathfrak{m}(\gamma)$, on a $\eta_\mathfrak{g}(z\gamma)=  \vert D_{\mathfrak{m}\backslash \mathfrak{g}}(z\gamma)\vert  \eta_\mathfrak{m}(z\gamma)$ et $\eta_\mathfrak{g}(\gamma)=  \vert D_{\mathfrak{m}\backslash \mathfrak{g}}(\gamma)\vert  \eta_\mathfrak{m}(\gamma)$. \'Ecrivons $\mathfrak{m}= \mathfrak{g}_1\times \cdots \mathfrak{g}_r$ avec $\mathfrak{g}_i={\rm End}_F(V_i)$. Les éléments $\gamma$ et $z\gamma$ appartiennent à $\mathfrak{m}_{\rm qre}= (\mathfrak{g}_i)_{\rm qre}\times \cdots \times (\mathfrak{g}_r)_{\rm qre}$, et on a
$$
\eta_\mathfrak{m}(z\gamma)= \vert z\vert^{\sum_{i=1}^r N_i(N_i-1)} \eta_\mathfrak{m}(\gamma), \quad N_i= \dim_F(V_i).
$$
D'autre part, on a
$$
\vert D_{\mathfrak{m}\backslash \mathfrak{g}}(z\gamma)\vert =  \vert z \vert^{\dim_F(\mathfrak{g})-\dim_F(\mathfrak{m})}\vert D_{\mathfrak{m}\backslash \mathfrak{g}}(\gamma)\vert. 
$$
Or $\dim_F(\mathfrak{m})= \sum_{i=1}^r N_i^2$ et $\sum_{i=1}^rN_i = N$, d'où l'égalité (3).
Pour $i=0,\ldots ,d_{\ES{N}}$, d'après \ref{théorie des germes de shalika}.(3), le germe $\bs{b}_i$ vérifie donc la formule d'homogénéité 
$$
\bs{b}_i^z = \vert z\vert^{{1\over 2}N(N-1)- {1\over 2} \dim (\ES{O}_i)}\bs{b}_i
= \vert z \vert^{{1\over 2}(\dim (G_{x_i}) -N)} \bs{b}_i.\leqno{(4)}
$$
Notons que l'exposant ${1\over 2}(\dim(G_{x_i})- N)$ est toujours $\geq 0$. Comme pour les germes $\bs{a}_i$, on peut grâce à la formule d'homogénéité (4) remplacer les germes de Shalika normalisés $\bs{b}_i$ par de vraies fonctions: pour $i=0,\ldots , d_{\ES{N}}$, il existe une unique fonction $\tilde{\bs{b}}_i: \mathfrak{g}_{\rm qr} \rightarrow {\Bbb C}$ telle que $[\tilde{\bs{b}}_i]^\mathfrak{g}_0= \bs{b}_i$ et $\tilde{\bs{b}}_i(z\gamma)= \vert z\vert^{{1\over 2}(\dim(G_{x_i})-N}\tilde{\bs{b}}_i(\gamma)$ pour tout 
$\gamma\in \mathfrak{g}_{\rm qr}$ et tout $z\in F^\times$. Comme pour $\tilde{\bs{a}}_i$, la fonction $\tilde{\bs{b}}_i$ est à valeurs dans ${\Bbb R}$, elle est invariante par translation par les éléments du centre $\mathfrak{z}$ de $\mathfrak{g}$, et invariante par conjugaison par les éléments de $G$.

\vskip1mm
Soit un élément $z\in \mathfrak{z}\smallsetminus \{0\} \;(=F^\times)$. On peut, comme on l'a fait pour $z=0$, s'intéresser aux intégrales orbitales normalisées au voisinage de $z$ dans $\mathfrak{g}$. Les classes de $G$--conjugaison dans $\mathfrak{g}$ qui contiennent $z$ dans leur fermeture sont exactement les $\ES{O}_G(z+ x_i)= z + \ES{O}_i$ pour $i=0,\ldots ,d_{\ES{N}}$. Pour $\mathfrak{f}\in C^\infty_{\rm c}(\mathfrak{g})$, le germe de fonctions $[\gamma \mapsto I^\mathfrak{g}(\gamma,\mathfrak{f})]_z^\mathfrak{g}$ au voisinage de $z$ dans $\mathfrak{g}_{\rm qr}$ est donné par le développement en germes
$$
[I^\mathfrak{g}(\gamma,\mathfrak{f})]_z^\mathfrak{g}= \sum_{i=0}^{d_{\ES{N}}} \ES{O}_{z+ x_i}(\mathfrak{f})\bs{b}_i(\gamma),\quad \gamma \in \mathfrak{g}_{\rm qr},\leqno{(5)}
$$ 
où la distribution $\ES{O}_{z+x_i}$ sur $\mathfrak{g}$ est définie à l'aide de la mesure $dg_{z+x_i}=dg_{x_i}$ sur $G_{z+x_i}=G_{x_i}$. Il suffit, pour obtenir (5), d'appliquer (1) à la fonction $\mathfrak{f}'\in C^\infty_{\rm c}(\mathfrak{g})$ définie par $\mathfrak{f}'(\gamma)= \mathfrak{f}(z+\gamma)$ et d'utiliser la propriété d'invariance des fonctions $\tilde{\bs{b}}_i$ par translation par les éléments 
de $\mathfrak{z}$.

\vskip1mm
Soit $\beta\in \mathfrak{g}$ un élément fermé. Reprenons les notations de \ref{descente centrale au voisinage d'un élément fermé}. Notons $\mathfrak{b}$ le centralisateur $\mathfrak{g}_\beta={\rm End}_{F[\beta]}(V)$ de $\beta$ dans $\mathfrak{g}$. \'Ecrivons $F[\beta]= E_1\times \cdots \times E_r$ pour des extensions $E_i/F$. Pour $i=1,\ldots ,r$, notons $e_i$ l'idempotent de $F[\gamma]$ associé à $E_i$, et posons $V_i= e_i(V)$, $\mathfrak{g}_i={\rm End}_F(V_i)$ et $\mathfrak{b}_i= {\rm End}_{E_i}(V_i)$. On a la décomposition $\mathfrak{b}= \mathfrak{b}_1\times \cdots \times \mathfrak{b}_r$, et l'élément $\beta=(\beta_1,\ldots ,\beta_r)$ est ($F$--){\it pur} dans $\mathfrak{m}= \mathfrak{g}_1\times \cdots \times \mathfrak{g}_r$. Notons $H$ le centralisateur $G_\beta =\mathfrak{b}^\times$ de $\beta$ dans $G$, $A_\beta$ le tore déployé maximal du centre $Z(H)=F[\beta]^\times$ de $H$, et $M=M(\beta)$ le centralisateur $Z_G(A_\beta)$ de $A_\beta$ dans $G$. On a $H=H_1\times \cdots \times H_r$ avec $H_i={\rm Aut}_{E_i}(V_i)$, et $M= G_1\times \cdots \times G_r$ avec $G_i={\rm Aut}_F(V_i)$. Quitte à remplacer $\beta$ par $g^{-1}\beta g$ pour un $g\in G$, on peut supposer que $A_\beta= A_P$ pour un $P\in \ES{P}$ --- en ce cas, on a $M= M_P$ et $\mathfrak{m}=\mathfrak{m}_P$ ---, mais ce  n'est pas indispensable ici. Soit $\bs{x}=(\bs{x}_1,\ldots ,\bs{x}_r)\in \mathfrak{m}$ un élément comme en \ref{descente centrale au voisinage d'un élément fermé}. D'après la remarque 2 de \ref{descente centrale au voisinage d'un élément fermé}, pour toute fonction $\mathfrak{f}\in C^\infty_{\rm c}(\mathfrak{g})$, il existe une fonction $\mathfrak{f}^{\mathfrak{b}}\in C^\infty_{\rm c}(\mathfrak{b})$ telle qu'on a l'égalité des germes
$$
[I^\mathfrak{g}(\beta + \bs{x}b, \mathfrak{f})]_\beta^\mathfrak{g} = [I^\mathfrak{b}(b, \mathfrak{f}^\mathfrak{b})]_0^\mathfrak{b}, 
\quad b\in \mathfrak{b}_{\rm qr}.\leqno{(6)}
$$
À droite de l'égalité (6), $[I^\mathfrak{b}(b, \mathfrak{f}^\mathfrak{b})]_0^\mathfrak{b}$ est le germe de la fonction $b\mapsto I^\mathfrak{b}(b, \mathfrak{f}^\mathfrak{b})$ au voisinage de $0$ dans $\mathfrak{b}_{\rm qr}$, et à gauche, $[I^\mathfrak{g}(\gamma,\mathfrak{f})]_\beta^\mathfrak{g}$ est le germe de la fonction $\gamma \mapsto I^\mathfrak{g}(\gamma,\mathfrak{f})$ au voisinage de $\beta$ dans $\mathfrak{g}_{\rm qr}$. L'égalité (6) a un sens car pour $b\in \mathfrak{b}_{\rm qr}$ suffisamment proche de $0$, l'élément $\beta + \bs{x}b$ appartient à $\mathfrak{g}_{\rm qr}$. On peut bien sûr, comme on l'a fait pour $\mathfrak{g}$, écrire le développement en germes (1) de 
$[I^\mathfrak{b}(b,\mathfrak{f}_*)]_0$ pour toute fonction $\mathfrak{f}_*\in C^\infty_{\rm c}(\mathfrak{b})$:
$$
[I^\mathfrak{b}(b, \mathfrak{f}_*)]_0^\mathfrak{b}= \sum_{j=0}^{d_{\mathfrak{N}(\mathfrak{b})}} \ES{O}_{y_{j}}^\mathfrak{b}(\mathfrak{f}_*)[I^\mathfrak{b}(b,\mathfrak{f}_{*,j})]_0^\mathfrak{b},\leqno{(7)}
$$
où
\begin{itemize}
\item $\ES{N}(\mathfrak{b})$ est l'ensemble des éléments nilpotents de $\mathfrak{b}$;
\item $y_1,\ldots ,y_{d_{\ES{N}(\mathfrak{b})}}$ est un système de représentants des $H$--orbites dans $\ES{N}(\mathfrak{b})$, ordonnées de telle manière que $\dim(\ES{O}_H(y_j))\leq \dim(\ES{O}_H(y_{j+1}))$;
\item $\ES{O}_{y_{j}}^\mathfrak{b}$ est l'intégrale orbitale nilpotente sur $\mathfrak{b}$ associée à l'orbite $\ES{O}_H(y_j)$, définie par le choix d'une mesure de Haar sur $H_{y_j}$;
\item $\{\mathfrak{f}_{*,j}: j=0\ldots ,d_{{\ES{N}(\mathfrak{b})}}\}\subset C^\infty_{\rm c}(\mathfrak{\mathfrak{b}})$ est un ensemble de fonctions vérifiant les conditions (i) et (ii) de \ref{théorie des germes de shalika} pour ces orbites $\ES{O}_H(y_j)$.
\end{itemize}
Pour $j=0,\ldots , d_{\ES{N}(\mathfrak{b})}$, on note $\bs{b}_j^\mathfrak{b}$ le germe de Shalika normalisé $[b \mapsto I^\mathfrak{b}(b, \mathfrak{f}_{*,j})]_0^\mathfrak{b}$. D'après (6) et (7), pour toute fonction $\mathfrak{f}\in C^\infty_{\rm c}(\mathfrak{g})$, il existe une fonction $\mathfrak{f}^\mathfrak{b}\in C^\infty_{\rm c}(\mathfrak{b})$ telle que le germe de fonctions $[\gamma \mapsto I^\mathfrak{g}(\gamma,\mathfrak{f})]_\beta^\mathfrak{g}$ au voisinage de $\beta$ dans $\mathfrak{g}_{\rm qr}$ est donné par
$$
[I^\mathfrak{g}(\beta + \bs{x}b,\mathfrak{f})]_\beta^\mathfrak{g}= \sum_{j=0}^{d_{\ES{N}(\mathfrak{b})}} \ES{O}_{y_j}^\mathfrak{b}(\mathfrak{f}^\mathfrak{b}) \bs{b}_j^\mathfrak{b}(b),\quad b\in \mathfrak{b}_{\rm qr}.\leqno{(8)}
$$
Pour $j=0,\ldots ,d_{\ES{N}(\mathfrak{b})}$, la $H$--orbite $\ES{O}_{y_j}^\mathfrak{b}$ se décompose en 
$\ES{O}_{y_j}^\mathfrak{b}= \ES{O}_{y_{j,1}}^{\mathfrak{b}_1}\times \cdots \times \ES{O}_{y_{j,r}}^{\mathfrak{b}_r}$, où on a posé $y_j=(y_{j,1},\ldots ,y_{j,r})$, et le germe de Shalika normalisé $\bs{b}_j^\mathfrak{b}$ (au voisinage de $0$ dans $\mathfrak{b}_{\rm qr}$) associé à l'orbite nilpotente $\ES{O}_{y_j}^\mathfrak{b}$ dans $\mathfrak{b}$ se décompose en $\bs{b}_j^\mathfrak{b} = \bs{b}_{j,1}^{\mathfrak{b}_1}\otimes \cdots \otimes \bs{b}_{j,r}^{\mathfrak{b}_r}$, où, pour $k=1,\ldots ,r$, 
$\bs{b}_{j,k}^{\mathfrak{b}_k}$ est le germe de Shalika normalisé (au voisinage de $0$ dans $(\mathfrak{b}_k)_{\rm qr}$) associé à l'orbite nilpotente $\ES{O}_{y_{j,k}}^{\mathfrak{b}_k}$ dans $\mathfrak{b}_k$. Ce germe $\bs{b}_j^\mathfrak{b}$ vérifie donc la formule d'homogénéité: pour $z=(z_1,\ldots ,z_r)\in E_1^\times\times \cdots \times E_r^\times$, notant $c_k$ la dimension de la variété $\mathfrak{p}_{E_k}$--adique $(H_k)_{y_{j,k}}$ et 
posant $d_k=\dim_{E_k}(V_k)$, on a
$$
(\bs{b}_j^\mathfrak{b})^z= \vert z_1\vert_{E_1}^{{1\over 2}(c_1-d_1)}\cdots \vert z_1\vert_{E_r}^{{1\over 2}(c_r-d_r)}\bs{b}_j^\mathfrak{b}\leqno{(9)}  
$$
La formule (9) permet comme plus haut d'associer au germe $\bs{b}_j^\mathfrak{b}$ une fonction $\smash{\tilde{\bs{b}}}_j^\mathfrak{b}: \mathfrak{b}_{\rm qr}\rightarrow {\Bbb C}$. Cette fonction est à valeurs dans ${\Bbb R}$, elle est invariante par translation par les éléments du centre $F[\beta]= E_1\times\cdots \times E_r$ de $\mathfrak{b}$, et invariante par conjugaison par les éléments de $H$. 
D'après (9), pour $z\in F^\times$ identifié à $(z,\ldots ,z)\in E_1^\times\cdots \times E_r^\times$, on a
$$
(\bs{b}_j^\mathfrak{b})^z = \vert z\vert^{{1\over 2}(\dim(H_{y_j})- N)}\bs{b}_j^\mathfrak{b}.\leqno{(10)}
$$
Ici $\dim(H_{y_j})= \sum_{k=1}^r [E_k:F]c_k$ est la dimension de $H_{y_j}=(H_1)_{y_1}\times\cdots \times (H_r)_{y_r}$ en tant que 
{\it variété $\mathfrak{p}$--adique}, et $N= \sum_{k=1}^r [E_k:F]d_k$. On a bien sûr toujours 
${1\over 2}(\dim(H_{y_j})- N)\geq 0$. D'ailleurs pour définir la fonction $\tilde{\bs{b}}_j^\mathfrak{b}$ à partir du germe $\bs{b}_j^\mathfrak{b}$, on peut tout aussi bien utiliser la formule d'homogénéité (10) (au lieu de (9)). 

Au voisinage de $\beta$ dans $\mathfrak{g}_{\rm qr}$, les fonctions $\tilde{\bs{b}}_i:\mathfrak{g}_{\rm qr}\rightarrow {\Bbb R}$ ($i= 0,\ldots ,d_\ES{N}$) associées aux germes de Shalika normalisés $\bs{b}_i$ pour $\mathfrak{g}$ admettent un développement 
en termes des fonctions $\smash{\tilde{\bs{b}}}_j^\mathfrak{b}:\mathfrak{b}_{\rm qr}\rightarrow {\Bbb R}$ ($j=0,\ldots ,d_{\ES{N}(\mathfrak{b})}$) associées aux germes de Shalika normalisés $\bs{b}_j^\mathfrak{b}$ pour $\mathfrak{b}$:

% lemme 
\begin{monlem}
Il existe un  voisinage ouvert compact $\ES{V}_\beta$ de $0$ dans $\mathfrak{b}$ tel que:
\begin{itemize}
\item pour $b\in \mathfrak{b}_{\rm qr}\cap \ES{V}_\beta$, l'élément $\beta + \bs{x} b$ appartient à $\mathfrak{g}_{\rm qr}$;
\item pour $i=0,\ldots , d_{\ES{N}}$, il existe des constantes $\lambda_{i,j}\in {\Bbb C}$ ($j=0,\ldots ,d_{\ES{N}(\mathfrak{b})}$) telles que
$$
\tilde{\bs{b}}_i(\beta + \bs{x} b)=\sum_{j=0}^{d_{\ES{N}(\mathfrak{b})}}\lambda_{i,j}\smash{\tilde{\bs{b}}}_j^\mathfrak{b}(b),\quad b\in  \mathfrak{b}_{\rm qr}\cap \ES{V}_\beta.
$$
\end{itemize}
\end{monlem}

\begin{proof} C'est la mme qu'en caractristique nulle (cf. \cite[lemma 17.7]{K}), compte--tenu du dveloppement (8) et des formules d'homognit (4) et (10). 
\end{proof}

%%%%%%%%%%%%%%%%%%%%%
\subsection{Les germes de Shalika normalisés sont localement bornés}\label{les germes de Shalika normalisés sont localement bornés}
On commence par un résultat sur les fonctions $\tilde{\bs{b}}_i:\mathfrak{g}_{\rm qr}\rightarrow {\Bbb R}$ associées aux germes de Shalika normalisés $\bs{b}_i$ pour $i=0,\ldots ,d_\ES{N}$:

% proposition
\begin{mapropo}
Les fonctions $\tilde{\bs{b}}_i: \mathfrak{g}_{\rm qr}\rightarrow {\Bbb R}$ ($i=0,\ldots ,d_{\ES{N}}$) sont 
localement bornées sur $\mathfrak{g}$, au sens où pour tout élément fermé $\beta\in \mathfrak{g}$, il existe un voisinage $\omega_\beta$ de $\beta$ dans $\mathfrak{g}$ --- que l'on peut supposer vérifiant ${^G(\mathfrak{z}+ \omega_\beta)} = \omega_\beta$ --- tel que
$$
\sup\{\tilde{\bs{b}}_i(\gamma): \gamma\in \mathfrak{g}_{\rm qr}\cap  \omega_\beta,\, i=0,\ldots ,d_{\ES{N}}\}<+\infty.
$$
\end{mapropo}

\begin{proof}Puisque les fonctions $\tilde{\bs{b}}_i$ sont invariantes par translations par les éléments du centre $\mathfrak{z}$ de $\mathfrak{g}$, on peut les considérer comme des fonctions sur $\mathfrak{g}_{\rm qr}/\mathfrak{z}$. 
Il s'agit donc de montrer que les fonctions 
$\tilde{\bs{b}}_i: \mathfrak{g}_{\rm qr}/ \mathfrak{z}\rightarrow {\Bbb R}$ ($i=0,\ldots ,d_{\ES{N}}$) sont 
localement bornées sur $\overline{\mathfrak{g}}=\mathfrak{g}/\mathfrak{z}$. 
Pour tout élément $\gamma\in \mathfrak{g}$, on note $\bar{\gamma}$ l'élément $\gamma + \mathfrak{z}$ de $\overline{\mathfrak{g}}$. 

Soit un élément fermé $\beta\in \mathfrak{g}$. On pose $\mathfrak{b}=\mathfrak{g}_\beta$, et on note $\mathfrak{z}_*= F[\beta]$ le centre de $\mathfrak{b}$. On pose aussi $\overline{\mathfrak{b}}= \mathfrak{b}/\mathfrak{z}_*$. Rappelons que $\mathfrak{z}_* = E_1\times\cdots \times E_r$ pour des extensions finies $E_i/F$, et que $\mathfrak{b}= \mathfrak{b}_1\times \cdots \times \mathfrak{b}_r$ avec $\mathfrak{b}_i= {\rm End}_{E_i}(V_i)$, $V= V_1\times \cdots \times V_r$. On procède par récurrence sur la dimension de $\mathfrak{b}$ sur $F$. Si $\overline{\beta}\neq 0$ (i.e. si $\beta \not\in \mathfrak{z}$), alors $\dim_F(\mathfrak{b}) < \dim_F(\mathfrak{g})$, et d'après le lemme de 
\ref{germes de shalika normalisés}, il existe un élément $\bs{x}\in \mathfrak{g}$ et un voisinage ouvert compact $\ES{V}_\beta$ de $0$ dans $\mathfrak{b}$ tels que 
${^G(\beta + \bs{x} \ES{V}_\beta)}$ est ouvert dans $\mathfrak{g}$ et pour $i=0,\ldots ,d_{\ES{N}}$, 
on a
$$
\tilde{\bs{b}}_i(\beta + \bs{x} b)=\sum_{j=0}^{d_{\ES{N}(\mathfrak{b})}}\lambda_{i,j}\smash{\tilde{\bs{b}}}_j^\mathfrak{b}(b),\quad b\in \mathfrak{b}_{\rm qr}\cap \ES{V}_\beta .
$$
Par hypothèse de récurrence, les fonctions $\tilde{\bs{b}}_j^\mathfrak{b}: \mathfrak{b}_{\rm qr}/\mathfrak{z}_* \rightarrow {\Bbb R}$ sont bornées sur  $\mathfrak{z}_*+\ES{V}_\beta$. Par conséquent les fonctions $\tilde{\bs{b}}_i: \mathfrak{g}_{\rm qr}/\mathfrak{z}\rightarrow {\Bbb R}$ sont bornées sur $\omega_\beta =  \mathfrak{z} + {^G(\beta + \bs{x}\ES{V}_\beta)}$. En d'autres termes, les fonctions $\tilde{\bs{b}}_i: \mathfrak{g}_{\rm qr}/\mathfrak{z}\rightarrow {\Bbb R}$ sont localement bornées sur $\overline{\mathfrak{g}}\smallsetminus \{0\}$, et on est ramené à prouver qu'elles sont bornées au voisinage de $0\in \overline{\mathfrak{g}}$. Soit $\Lambda$ un $\mathfrak{o}$--réseau dans $\mathfrak{g}$. Alors $\bar{\Lambda}= \Lambda + \mathfrak{z}$ est un $\mathfrak{o}$--réseau dans $\overline{\mathfrak{g}}$, et puisque $\bar{\Lambda} \smallsetminus \mathfrak{p}\bar{\Lambda}$ est une partie compacte de $\overline{\mathfrak{g}}\smallsetminus \{0\}$, pour $i=0,\ldots , d_\ES{N}$, il existe une constante $c_i>0$ telle que
$$
 \sup \{\vert \tilde{\bs{b}}_i(\bar{\gamma})\vert: \gamma \in \mathfrak{g}_{\rm qr},\, \bar{\gamma} \in \bar{\Lambda}\smallsetminus \mathfrak{p}\bar{\Lambda}\} \leq c_i.
$$
Pour $\gamma\in \mathfrak{g}_{\rm qr}$ tel que $\bar{\gamma}\in \mathfrak{p}\Lambda$, il existe un (unique) entier $k\geq 1$ tel que $\varpi^{-k}\bar{\gamma}\in \Lambda\smallsetminus \mathfrak{p}\bar{\Lambda}$ pour une uniformisante $\varpi$ de $F$. 
Alors pour $i=0,\ldots ,d_\ES{N}$, d'après la formule d'homogénité \ref{germes de shalika normalisés}.(4), on a
$$
\tilde{\bs{b}}_i(\bar{\gamma})= \tilde{\bs{b}}_i^{\varpi^k}(\varpi^{-k}\gamma)=q^{-{k\over 2}(\dim(G_{x_i})-N)}\tilde{\bs{b}}_i(\varpi^{-k}\bar{\gamma}),
$$
d'où
$$
 \sup \{\vert \tilde{\bs{b}}_i(\bar{\gamma})\vert: \gamma \in \mathfrak{g}_{\rm qr},\, \bar{\gamma} \in \bar{\Lambda}\} \leq c_i.
$$
Cela achève la démonstration de la proposition.
\end{proof}

% corollaire 1
\begin{moncoro1}
Pour toute fonction $\mathfrak{f}\in C^\infty_{\rm c}(\mathfrak{g})$, la fonction $
\mathfrak{g}_{\rm qr}\rightarrow {\Bbb C},\, \gamma \mapsto I^\mathfrak{g}(\gamma,\mathfrak{f})$
est bornée sur $\mathfrak{g}$: il existe une constante $c_\mathfrak{f}>0$ telle que
$$
\sup\{\vert I^\mathfrak{g}(\gamma,\mathfrak{f})\vert: \gamma\in \mathfrak{g}_{\rm qr}\}\leq c_\mathfrak{f}.
$$
\end{moncoro1}

\begin{proof}
D'après la proposition, la fonction $\mathfrak{g}_{\rm qr}\rightarrow {\Bbb C},\, \gamma \mapsto I^\mathfrak{g}(\gamma,\mathfrak{f})$
est localement bornée sur $\mathfrak{g}$, et puisque d'après le lemme 1 de \ref{parties compactes modulo conjugaison}, il existe une partie compacte $\Omega$ dans $\mathfrak{g}$ telle que $I^\mathfrak{g}(\gamma,\mathfrak{f})=0$ pour tout $\gamma\in \mathfrak{g}_{\rm qr}$ tel que $\gamma\not\in {^G\Omega}$ --- il suffit de choisir $\Omega$ de telle manière que $\overline{^G{\rm Supp}(f)}\subset {^G\Omega}$ ---, elle est bornée sur $\mathfrak{g}$.
\end{proof}

% remarque
\begin{marema}
{\rm 
Les fonctions $\tilde{\bs{b}}_i$ associées aux germes de Shalika normalisés $\bs{b}_i$ sont des fonctions (à valeurs réelles) sur $\mathfrak{g}_{\rm qr}/\mathfrak{z}$. On peut aussi s'intéresser aux intégrales orbitales quasi--régulières des 
fonctions $\bar{\mathfrak{f}}\in C^\infty_{\rm c}(\overline{\mathfrak{g}})$, où (comme dans la preuve de la proposition) on a posé $\overline{\mathfrak{g}}= \mathfrak{g}/\mathfrak{z}$. Pour $\bar{\mathfrak{f}}\in C^\infty_{\rm c}(\overline{\mathfrak{g}})$, le support de $\bar{\mathfrak{f}}$ est une partie ouverte compacte de $\overline{\mathfrak{g}}$, que l'on peut voir comme une partie ouverte fermée de $\mathfrak{g}$ invariante par translation par $\mathfrak{z}$; on la note $\mathfrak{S}(\bar{\mathfrak{f}})\subset \mathfrak{g}$. Pour $\gamma\in \mathfrak{g}_{\rm qr}$ et $\bar{\mathfrak{f}}\in C^\infty_{\rm c}(\overline{\mathfrak{g}})$, l'ensemble $ \ES{O}_G(\gamma)\cap 
\mathfrak{S}(\bar{\mathfrak{f}})$ est compact --- car l'ensemble $\{\gamma'\in \mathfrak{S}(\bar{\mathfrak{f}}): \det(\gamma')=\det(\gamma)\}$ l'est ---, et on peut définir comme en \ref{variante sur l'algèbre de Lie (suite)} l'intégrale orbitale $\ES{O}_\gamma(\bar{\mathfrak{f}})$ et l'intégrale orbitale normalisée $I^\mathfrak{g}(\gamma, \bar{\mathfrak{f}})= \eta_{\mathfrak{g}}^{1\over 2}(\gamma)\ES{O}_\gamma(\bar{\mathfrak{f}})$. On a clairement
$$
I^\mathfrak{g}(z+\gamma,\bar{\mathfrak{f}})= I^\mathfrak{g}(\gamma, \bar{\mathfrak{f}}),\quad z\in \mathfrak{z}.
$$

Pour $\mathfrak{f}\in C^\infty_{\rm c}(G)$, on pose $\bar{\mathfrak{f}}(g)= \int_\mathfrak{z} \mathfrak{f}(z+g)\mathfrak{d}z$ ($g\in G$), où $\mathfrak{d}z$ est une mesure de Haar sur $\mathfrak{z}=F$. On obtient une application linéaire surjective
$$
C^\infty_{\rm c}(\mathfrak{g})\rightarrow C^\infty_{\rm c}(\overline{\mathfrak{g}}),\, \mathfrak{f}\mapsto \bar{\mathfrak{f}},
$$
et pour $\mathfrak{f}\in C^\infty_{\rm c}(\mathfrak{g})$, on a
$$
I^\mathfrak{g}(\gamma,\overline{\mathfrak{f}})= \int_\mathfrak{z} I^\mathfrak{g}(z+\gamma, \mathfrak{f})\mathfrak{d}z,\quad \gamma\in \mathfrak{g}_{\rm qr}.\leqno{(1)}
$$
L'intégrale (1) est absolument convergente, car l'ensemble $\{z\in \mathfrak{z}: \ES{O}_G(z+\gamma) \cap {\rm Supp}(f)\}$ est compact.
On obtient la variante suivante du corollaire 1:
\begin{enumerate}[leftmargin=17pt]
\item[(2)] pour tout fonction $\bar{\mathfrak{f}}\in C^\infty_{\rm c}(\overline{\mathfrak{g}})$, 
la fonction $\mathfrak{g}_{\rm qr}/\mathfrak{z}\rightarrow {\Bbb C},\, \gamma \mapsto I^\mathfrak{g}(\gamma,\mathfrak{f})$ est bornée sur $\overline{\mathfrak{g}}$.
\end{enumerate}
En effet, choisissons une fonction $\mathfrak{f}\in C^\infty_{\rm c}(\mathfrak{g})$ se projetant sur $\bar{\mathfrak{f}}$, et un ouvert compact $\Omega$ dans $\mathfrak{g}$ tel que $\overline{^G{\rm Supp}(f)}\subset {^G\Omega}$. Puisque $I^\mathfrak{g}(\gamma,\mathfrak{f})=0$ pour tout $\gamma\in\mathfrak{g}_{\rm qr}$ tel que $\gamma \notin {^G\Omega}$, on a $I^\mathfrak{g}(\gamma,\bar{\mathfrak{f}})=0$ pour tout $\gamma\in \mathfrak{g}_{\rm qr}$ tel que $\gamma\notin \mathfrak{z} +{^G\Omega}$, et il suffit de voir que la fonction $\gamma \mapsto I^\mathfrak{g}(\gamma, \bar{\mathfrak{f}})$ est bornée sur $\Omega$ (i.e sur $\mathfrak{z}+ \Omega$). L'ensemble $\omega = \{z\in \mathfrak{z}: (z + {^G\Omega})\cap {\rm Supp}(f)\}$ est ouvert compact, et d'après (1), pour $\gamma \in \mathfrak{g}_{\rm qr} \cap \Omega$, on a 
$$
\vert I^\mathfrak{g}(\gamma, \bar{\mathfrak{f}})\vert 
\leq \int_{\mathfrak{z}}\vert I^\mathfrak{g}(z + \gamma,\mathfrak{f})\vert \mathfrak{d}z \leq {\rm vol}(\omega, \mathfrak{d}z) 
\sup\{ \vert I^\mathfrak{g}(\gamma',\mathfrak{f})\vert: \gamma'\in \mathfrak{g}_{\rm qr}\cap \Omega\}.
$$
D'où le point (2). \hfill $\blacksquare$
}
\end{marema}

% corollaire 2
\begin{moncoro2}
La fonction $\eta_\mathfrak{g}^{-{1\over  2}}: \mathfrak{g}_{\rm qr}\rightarrow {\Bbb R}_{>0}$ est localement intégrable 
(par rapport à une mesure de Haar $\mathfrak{d}g$) sur $\mathfrak{g}$: pour toute fonction $\mathfrak{f}\in C^\infty_{\rm c}(\mathfrak{g})$, l'intégrale
$$
\int_\mathfrak{g} \eta_\mathfrak{g}(g)^{-{1\over  2}} \mathfrak{f}(g)\mathfrak{d}g
$$
est absolument convergente.
\end{moncoro2}

\begin{proof}Si $T$ est un tore maximal de $G$ --- \cad le groupe des points $F$--rationnels d'un tore maximal de $GL(N)$ défini sur $F$ ---, on note $W^G(T)$ le groupe de Weyl $N_G(T)/T$, $\mathfrak{t}$ l'algèbre de Lie de $T$, et $\alpha=\alpha_T: (T\backslash G) \times \mathfrak{t}\rightarrow \mathfrak{g}$ l'application $(g,\gamma)\mapsto g^{-1}\gamma g$. Pour $\gamma\in \mathfrak{t}$, la différentielle $\delta\alpha_{(1,\gamma)}: (\mathfrak{g}/\mathfrak{t})\times \mathfrak{t}\rightarrow \mathfrak{g}$ de $\alpha$ au point $(1,\gamma)$ est l'application
$$
(x, y) \mapsto {\rm ad}_\gamma(x)+y.
$$
Pour $(g,\gamma)\in (T\backslash G)\times \mathfrak{t}$, puisque $\alpha(g,\gamma)= g^{-1}\alpha(1,\gamma)g$, on en déduit que le Jacobien de $\alpha$ au point $(g,\gamma)$ est égal à $\det_F({\rm ad}_\gamma; \mathfrak{g}/\mathfrak{t})= D_\mathfrak{g}(-\gamma)\;(=D_\mathfrak{g}(\gamma))$. Rappelons que pour $\gamma \in \mathfrak{t}\cap \mathfrak{g}_{\rm qr}$, la distribution $\ES{O}_\gamma$ sur $\mathfrak{g}$ est définie via la mesure $G$--invariante ${dg\over dt}$ sur $T\backslash G$, où $dg$ est la mesure de Haar sur $G$ qui donne le volume $1$ à $K=GL(N,\mathfrak{o})$, et $dt$ est la mesure de Haar sur $T$ normalisée par ${\rm vol}(A_T\backslash T, {dt \over da_T})=1$. Ici $A_T$ est le tore déployé maximal de $T$, et $da_T$ est la mesure de Haar sur $A_T$ qui donne le volume $1$ au sous--groupe compact maximal de $A_T$. On note $\mathfrak{d}t$ la mesure de Haar sur $\mathfrak{t}$ associée à $dt$. 
On suppose aussi, ce qui est loisible, que la mesure de Haar $\mathfrak{d}g$ sur $\mathfrak{g}$ est celle associée à $dg$.   

Soit une fonction 
$\mathfrak{f}\in C^\infty_{\rm c}(\mathfrak{g})$. Quitte à remplacer $\mathfrak{f}$ par $\vert \mathfrak{f}\vert$, on peut supposer $\mathfrak{f}\geq 0$. La formule d'intégration de Weyl donne
$$
\int_{\mathfrak{g}} \eta_\mathfrak{g}(g)^{-{1\over  2}} \mathfrak{f}(g)\mathfrak{d}g=
\sum_T \vert W^G(T)\vert^{-1} \int_{\mathfrak{t}\cap \mathfrak{g}_{\rm qr}}\vert D_\mathfrak{g}(\gamma)\vert 
\eta_\mathfrak{g}(\gamma)^{-{1\over  2}} \ES{O}_\gamma(\mathfrak{f})\mathfrak{d}\gamma,\leqno{(3)}
$$
où $T$ parcourt un système de représentants des classes de conjugaison de tores maximaux de $G$. On peut aussi regrouper les tores maximaux $T$ suivant les classes de conjugaison des sous--tores déployés maximaux $A_T\subset T$. Rappelons qu'on a fixé un ensemble $\ES{P}= \ES{P}_G$ de sous--groupes paraboliques santards de $G$. Deux éléments $P,\,P'\in \ES{P}$ sont dits {\it associés} s'il existe un élément $g\in G$ tel que $gA_Pg^{-1}= A_{P'}$, ou, ce qui revient au même, tel que $gM_Pg^{-1} = M_{P'}$. Fixons un ensemble de représentants $\ES{P}^*\subset \ES{P}$ des classes d'association. Pour $P\in \ES{P}^*$, fixons un ensemble de représentants $\ES{T}_P$ des classes de $M_P$--conjugaison de tores maximaux $T$ de $M_P$ tels que $A_T=A_P$. Tout tore maximal $T$ de $G$ est conjugué (dans $G$) à un unique élément de $\coprod_{P\in \ES{P}^*}\ES{T}_P$. D'autre part si $T,\,T'\in \ES{T}_P$ pour un $P\in \ES{P}^*$, et si $T'= gTg^{-1}$ pour un $g\in G$, alors on a $gA_Tg^{-1}=A_{T'}$, et $g$ définit un élément de 
$W^G(A_P)= N_G(A_P)/M_P$. Pour $P\in \ES{P}^*$, on a
$$
\vert W^G(T)\vert = \vert W^G(A_P)\vert \vert W^{M_P}(T)\vert.
$$
D'après (3), on obtient
$$
\int_{\mathfrak{g}} \eta_\mathfrak{g}(g)^{-{1\over  2}} \mathfrak{f}(g)\mathfrak{d}g=
\sum_{P\in \ES{P}^*}\vert W^G(A_P)\vert^{-1}\sum_{T\in \ES{T}_P} \vert W^{M_P}(T)\vert^{-1} \int_{\mathfrak{t}\cap \mathfrak{g}_{\rm qr}}\vert D_\mathfrak{g}(\gamma)\vert 
\eta_\mathfrak{g}(\gamma)^{-{1\over  2}} \ES{O}_\gamma(\mathfrak{f})\mathfrak{d}\gamma.\leqno{(4)}
$$
Pour $P\in \ES{P}^*$, $T\in \ES{T}_P$ et $\gamma\in \mathfrak{t}\cap \mathfrak{g}_{\rm qr}$, puisque $\vert D_\mathfrak{g}(\gamma)=
\vert D_{\mathfrak{m}_P}(\gamma)\vert \vert D_{\mathfrak{m}_P\backslash \mathfrak{g}}(\gamma)\vert$ et (par définition) 
$\eta_\mathfrak{g}(\gamma) = \eta_{\mathfrak{m}_P}(\gamma)\vert D_{\mathfrak{m}_P\backslash \mathfrak{g}}(\gamma)\vert$, 
d'après la formule de descente \ref{variante sur l'algèbre de Lie (suite)}.(3), on a
$$
\vert D_\mathfrak{g}(\gamma)\vert \eta_\mathfrak{g}(\gamma)^{-{1\over  2}} \ES{O}_\gamma(\mathfrak{f})=
\vert D_\mathfrak{g}(\gamma)\vert \eta_\mathfrak{g}(\gamma)^{-1}I^\mathfrak{g}(\gamma,\mathfrak{f})= 
\vert D_{\mathfrak{m}_P}(\gamma)\vert \eta_{\mathfrak{m}_P}(\gamma)^{-1}I^{\mathfrak{m}_P}(\gamma, \mathfrak{f}_{\mathfrak{p}_P}).
$$
Le produit $\vert D_{\mathfrak{m}_P}(\gamma)\vert \eta_{\mathfrak{m}_P}(\gamma)^{-1}$ ne dépend pas de l'élément $\gamma \in \mathfrak{t}\cap \mathfrak{g}_{\rm qr}$: d'après \ref{variante sur l'algbre de Lie}.(8), il vaut $q^{- \mu(T)}$ pour une constante $\mu(T)\geq 0$ calculée comme suit. Le tore $T$ s'écrit $T=E_1^\times \times \cdots \times E_r^\times$ pour des extensions séparables $E_i/F$. Pour $i=1,\dots ,r$, posons $e_i=e(E_i/F)$, $f_i= f(E_i/F)$, $N_i = [E_i:F]$ et $\delta_i = \delta(E_i/F)$. On a
$$
\mu(T)= \sum_{i=1}^r \delta_i - f_i(e_i-1).
$$
L'égalité (4) devient
$$
\int_{\mathfrak{g}} \eta_\mathfrak{g}(g)^{-{1\over  2}} \mathfrak{f}(g)\mathfrak{d}g=
\sum_{P\in \ES{P}^*}\vert W^G(A_P)\vert^{-1}\sum_{T\in \ES{T}_P} \vert W^{M_P}(T)\vert^{-1} q^{-\mu(T)} \int_{\mathfrak{t}\cap \mathfrak{g}_{\rm qr}}I^{\mathfrak{m}_P}(\gamma , \mathfrak{f}_{\mathfrak{p}_P})\mathfrak{d}\gamma.\leqno{(5)}
$$

Pour $P\in \ES{P}^*$, choisissons une mesure de Haar $\mathfrak{d}a_P$ sur l'algèbre de Lie $\mathfrak{a}_P$ de $A_P$ (qui est aussi le centre de $\mathfrak{m}_P$), posons $\overline{\mathfrak{m}}_P = \mathfrak{m}_P/\mathfrak{a}_P$, et notons $\bar{\mathfrak{f}}_{\mathfrak{m}_P}\in C^\infty_{\rm c}(\overline{\mathfrak{m}}_P)$ la fonction définie par
$$
\bar{\mathfrak{f}}_{\mathfrak{m}_P}(m)= \int_{\mathfrak{a}_P}\mathfrak{f}_{\mathfrak{m}_P}(a_P + m) \mathfrak{d}a_P, \quad m\in \mathfrak{m}_P.
$$
Posons $\overline{\mathfrak{t}}= \mathfrak{t}/ \mathfrak{a}_P$ et notons $\mathfrak{d}\bar{t}$ la mesure quotient $\textstyle{\mathfrak{d}t\over \mathfrak{d}a_P}$ sur $\overline{\mathfrak{t}}$. Puisque
$$
{\rm vol}((\mathfrak{t}\cap (\mathfrak{m}_P)_{\rm qr})\smallsetminus (\mathfrak{t} \cap \mathfrak{g}_{\rm qr}), \mathfrak{d}t)=0, 
$$
on a (d'après (1))
$$
\int_{\mathfrak{t}\cap \mathfrak{g}_{\rm qr}}I^{\mathfrak{m}_P}(\gamma , \mathfrak{f}_{\mathfrak{p}_P})\mathfrak{d}\gamma
= \int_{\overline{\mathfrak{t}}\cap ((\mathfrak{m}_P)_{\rm qr}/\mathfrak{a}_P)}I^{\mathfrak{m}_P}(\gamma , \bar{\mathfrak{f}}_{\mathfrak{p}_P})\mathfrak{d}\bar{\gamma}. 
$$
Or par construction $\overline{\mathfrak{t}}$ est compact --- si $T= E_1^\times\times\cdots \times E_r^\times$ comme plus haut, alors 
on a $\overline{\mathfrak{t}}= (E_1/F)\times\cdots \times (E_r/F)$ ---, et d'après la remarque, la fonction $\gamma \mapsto I^{\mathfrak{m}_P}(\gamma, \bar{\mathfrak{f}}_{\mathfrak{p}_P})$ est bornée sur $\overline{\mathfrak{m}}_P = \mathfrak{m}_P/\mathfrak{a}_P$. On en déduit qu'il existe une constante $c_P(\mathfrak{f})$ telle que pour tout $T\in \ES{T}_P$, on a
$$
\int_{\mathfrak{t}\cap \mathfrak{g}_{\rm qr}}I^{\mathfrak{m}_P}(\gamma , \mathfrak{f}_{\mathfrak{p}_P})\mathfrak{d}\gamma
\leq c_P(\mathfrak{f}).\leqno{(6)}
$$

Posons $c(\mathfrak{f})= \max\{c_P(\mathfrak{f}): P\in \ES{P}^*\}$. La majoration (6) injectée dans (5) donne
$$
\int_{\mathfrak{g}} \eta_\mathfrak{g}(g)^{-{1\over  2}} \mathfrak{f}(g)\mathfrak{d}g\leq c(\mathfrak{f})
\sum_{P\in \ES{P}^*}\vert W^G(A_P)\vert^{-1}\sum_{T\in \ES{T}_P} \vert W^{M_P}(T)\vert^{-1} q^{-\mu(T)}.\leqno{(7)}
$$
On est donc ramené à prouver que pour chaque $P\in \ES{P}^*$, la somme $\sum_{T\in \ES{T}_P} \vert W^{M_P}(T)\vert^{-1} q^{-\mu(T)}$ est finie. Il suffit de le faire pour $P= G$ (le cas des autres $P$ s'en déduisant par produit et récurrence sur la dimension de $G$). Or pour $P=G$, l'expression $\sum_{T\in \ES{T}_G} \vert W^{G}(T)\vert^{-1} q^{-\mu(T)}$ n'est autre que le terme à gauche de l'égalité \ref{la formule de masse de Serre}.(3), \cad la formule de masse de Serre étendue à toutes les extensions séparables de $F$ de degré $N$. 
\end{proof}

%%%%%%%%%%%%%%%%%%%%%%%%%%%%%%
\subsection{Intégrabilité locale de la fonction $\eta_\mathfrak{g}^{-{1\over 2}-\epsilon}$}\label{intégrabilité locale du eta Lie}On l'a dit dans l'introduction, pour établir une formule des traces locale, il est nécessaire d'obtenir un peu plus que l'intégrabilité locale de la fonction $\eta_\mathfrak{g}^{-{1\over 2}}$. C'est ce que nous faisons dans ce numéro.

Rappelons que la fonction $\eta_\mathfrak{g}: \mathfrak{g}_{\rm qr}\rightarrow {\Bbb R}_{>0}$ se factorise à travers $\mathfrak{g}_{\rm qr}/\mathfrak{z}$. De plus, elle vérifie la propriété d'homogénéité
$$
\eta_\mathfrak{g}(z\gamma) = \vert z\vert^{N(N-1)}\eta_\mathfrak{g}(\gamma),\quad \gamma\in \mathfrak{g}_{\rm qr},\, z\in F^\times.\leqno{(1)}
$$
En effet, pour $\gamma\in G_{\rm qre}$, elle résulte de l'égalité $\eta_\mathfrak{g}(\gamma) = \vert \det(\gamma)\vert^{N-1} \eta_G(\gamma)$ (remarque 2 de \ref{variante sur l'algèbre de Lie (suite)}) et du fait que l'application $\eta_G:G_{\rm qr}\rightarrow {\Bbb R}_{>0}$ se factorise à travers $G/Z$. Puisque $\mathfrak{g}_{\rm qre} = G_{\rm qre}$ si $N>1$, cette propriété (1) est vraie pour tout $\gamma \in \mathfrak{g}_{\rm qre}$. Pour $\gamma\in \mathfrak{g}_{\rm qr}$, posant $\mathfrak{m}= \mathfrak{m}(\gamma)$ comme en \ref{variante sur l'algèbre de Lie (suite)}, on a 
$\eta_\mathfrak{g}(\gamma) = \vert D_{\mathfrak{m}\backslash \mathfrak{g}}(\gamma)\vert \eta_\mathfrak{m}(\gamma)$. Pour $z\in F^\times$, on a $\vert D_{\mathfrak{m}\backslash \mathfrak{g}}(z\gamma)= \vert z\vert^{N^2- \dim_F(\mathfrak{m})}\vert 
D_{\mathfrak{m}\backslash \mathfrak{g}}(\gamma)\vert$, et puisque $\gamma\in \mathfrak{m}_{\rm qre}$, d'après la propriété 
(1) pour $\mathfrak{m}$, on a
$\eta_\mathfrak{m}(z\gamma)= \vert z\vert^{\dim_F(\mathfrak{m})-N}\eta_\mathfrak{m}(\gamma)$. Cela prouve (1) pour tout $\gamma\in \mathfrak{g}_{\rm qr}$.

% proposition
\begin{mapropo}
Pour tout $\epsilon>0$ tel que $(N^2-N)\epsilon <1$, la fonction $\eta_\mathfrak{g}^{-{1\over 2}-\epsilon}$ est localement intégrable sur $\mathfrak{g}$.
\end{mapropo}

\begin{proof}Soit $\epsilon$ comme dans l'énoncé. Puisque la fonction $\eta_\mathfrak{g}^{-{1\over 2}-\epsilon}: \mathfrak{g}_{\rm qr}\rightarrow {\Bbb R}_{>0}$ se factorise à travers $\mathfrak{g}_{\rm qr}/\mathfrak{z}$, il suffit de montrer qu'elle est localement intégrable sur $\overline{\mathfrak{g}}= \mathfrak{g}/ \mathfrak{z}$. 
Grâce à (1), on peut procéder comme pour l'étude des fonctions $\tilde{\bs{b}}_i$ (cf. \ref{les germes de Shalika normalisés sont localement bornés}). Il s'agit de montrer que pour chaque élément fermé $\beta\in \mathfrak{g}$, 
la fonction $\eta_\mathfrak{g}^{-{1\over 2}-\epsilon}: \mathfrak{g}_{\rm qr}/\mathfrak{z}\rightarrow {\Bbb R}_{>0}$ est localement intégrable au voisinage de $\bar{\beta}= \beta + \mathfrak{z}$ dans $\overline{\mathfrak{g}}$. On raisonne par récurrence sur $N$. Pour $N=1$, il n'y a rien à démontrer. On suppose donc $N>1$.

Soit $\beta\in \mathfrak{g}$ un élément fermé. On pose $\mathfrak{b}=\mathfrak{g}_\beta$, $\mathfrak{z}_*= F[\beta]$, et $\overline{\mathfrak{b}}= \mathfrak{b}/\mathfrak{z}_*$. Rappelons que $\mathfrak{z}_* = E_1\times\cdots \times E_r$ pour des extensions finies $E_i/F$, et que $\mathfrak{b}= \mathfrak{b}_1\times \cdots \times \mathfrak{b}_r$ avec $\mathfrak{b}_i= {\rm End}_{E_i}(V_i)$, $V= V_1\times \cdots \times V_r$. On procède par récurrence sur la dimension de $\mathfrak{b}$ sur $F$. Supposons $\overline{\beta}\neq 0$ (i.e. $\beta \not\in \mathfrak{z}$). Notons $A_\beta$ le tore déployé 
maximal $F^\times \times \cdots \times F^\times $ de $E_1^\times \times\cdots \times E_r^\times$, $M = M(\beta)$ le centralisateur de $A_\beta$ dans $\mathfrak{g}$, et $\mathfrak{m}= \mathfrak{m}(\beta)$ l'algèbre de Lie de $M$. On a l'inclusion $\mathfrak{b}\subset \mathfrak{m}$, et $\beta$ est pur dans $\mathfrak{m}$. L'ensemble $\omega= \{\gamma \in \mathfrak{m}: D_{\mathfrak{m}\backslash \mathfrak{g}}(\gamma)\neq 0\}$ est ouvert dans $\mathfrak{m}$, et l'application 
$$
G \times \omega\rightarrow \mathfrak{g},\, (g,\gamma)\mapsto g^{-1}\gamma g,
$$
est partout submersive. L'ensemble $\omega' = \{\gamma\in \mathfrak{m}: \vert D_{\mathfrak{m}\backslash \mathfrak{g}}(\gamma)\vert = \vert D_{\mathfrak{m}\backslash \mathfrak{g}}(\beta)\vert\}$ est contenu dans $\omega$, et c'est un voisinage ouvert ($M$--invariant par conjugaison, et $\mathfrak{z}$--invariant par translation) de $\beta$ dans $\mathfrak{m}$. Puisque pour $(g,\gamma)\in G\times (\mathfrak{g}_{\rm qr}\cap \omega')$, on a
$$
\eta_\mathfrak{g}(g^{-1} \gamma g)= \vert D_{\mathfrak{m}\backslash \mathfrak{g}}(\beta)\vert \eta_\mathfrak{m}(\gamma),
$$
l'intégrabilité locale de la fonction 
$\eta_\mathfrak{g}^{-{1\over 2} -\epsilon}: \mathfrak{g}_{\rm qr}\rightarrow {\Bbb R}_{>0}$ au voisinage de $\beta$ dans $\mathfrak{g}$ est impliquée par l'intégrabilité locale de la fonction $\eta_{\mathfrak{m}}^{-{1\over 2}-\epsilon}: \mathfrak{m}_{\rm qr}\rightarrow {\Bbb R}_{>0}$ au voisinage de $\beta$ dans $\mathfrak{m}$. L'hypothse de rcurrence est vrifie: 
en crivant $\mathfrak{m}= \mathfrak{g}_1\times \cdots \times \mathfrak{g}_r$ avec 
$\mathfrak{g}_i= {\rm End}_F(V_i)$ et en posant $N_i= \dim_F(V_i)$, on a $N_i(N_i-1)\epsilon <1$ pour tout $i$. On s'est donc ramené au cas où $\mathfrak{m}=\mathfrak{g}$, \cad au cas où $\beta$ est pur dans $\mathfrak{g}$. 

On suppose de plus que $\beta$ est pur dans $\mathfrak{g}$ (avec toujours $\beta\notin \mathfrak{z}$). Alors $\beta\in G$. Reprenons les notations de \ref{descente centrale au voisinage d'un élément pur (suite)}, en particulier l'application partout submersive \ref{descente centrale au voisinage d'un élément pur (suite)}.(2)
$$
\delta: G\times \bs{x}\underline{\mathfrak{Q}}^{\underline{k}_0+1}\rightarrow G,\, (g,\bs{x}b)\mapsto g^{-1}(\beta + \bs{x}b)g.
$$
Notons $\ES{V}_\beta$ le voisinage ouvert compact $\underline{\mathfrak{Q}}^{\underline{k}_0+1}$ de $0$ dans $\mathfrak{b}$. D'après le corollaire 1 de \ref{descente centrale au voisinage d'un élément pur (suite)}, pour $(g,b)\in G\times (\mathfrak{b}_{\rm qr}\cap \ES{V}_\beta)$, l'élément $\gamma = g^{-1}(\beta + \bs{x}b)g$ appartient à $G_{\rm qr}$, et on a
$$
{\eta_G(\gamma)\over \eta_\mathfrak{b}(b)}= (q_E^{n_F(\beta)} \mu_F(\beta))^{d^2} \vert \beta \vert_E^{-d} =\mu_F(\beta)^{d^2} \vert \beta \vert_E^{d^2-d} ,
$$
où (rappel) $d={N\over [E:F]}$. Puisque
$\eta_G(\gamma)= \vert \det(\gamma)\vert^{1-N}\eta_\mathfrak{g}(\gamma)$ avec 
$$
\vert \det(\gamma)\vert^{1-N}= \vert \det(\beta)\vert^{1-N}=\vert \beta \vert_E^{-d(N-1)}.
$$
On obtient 
$$
\eta_\mathfrak{g}(\gamma) = c_\beta \eta_\mathfrak{b}(b), \quad c_\beta= \mu_F(\beta)^{d^2}\vert \beta\vert_E^{d(d-1+ N-1)}.
$$
D'après le principe de submersion (cf. \ref{le principe de submersion}), pour toute fonction $\phi \in C^\infty_{\rm c}(G\times \bs{x}\underline{\mathfrak{Q}}^{\underline{k}_0+1}))$, on a
\begin{eqnarray*}
\int_G\phi^\delta(g) \eta_\mathfrak{g}^{-{1\over 2}-\epsilon}(g)dg &=& \int_{G\times \mathfrak{b}}\phi(g,b)\eta_\mathfrak{g}^{-{1\over 2}-\epsilon}(g^{-1}(\beta + \bs{x}b)g)dg\mathfrak{d}b \\
&=& c_\beta \int_\mathfrak{b} \phi_\delta(b) \eta_\mathfrak{b}^{-{1\over 2}-\epsilon}(b)\mathfrak{d}b.
\end{eqnarray*}
L'intégrabilité de la fonction $\eta_\mathfrak{g}^{-{1\over 2} -\epsilon}: \mathfrak{g}_{\rm qr}\rightarrow {\Bbb R}_{>0}$ au voisinage de $\beta$ dans $\mathfrak{g}$ est impliquée par celle de la fonction $\eta_{\mathfrak{b}}^{-{1\over 2}-\epsilon}: \mathfrak{b}_{\rm qr}\rightarrow {\Bbb R}_{>0}$ au voisinage de $0$ dans $\mathfrak{b}$, qui est supposée connue par hypothèse de récurrence (notons que $(d^2-d)\epsilon <1$).

Grâce à l'hypothèse de récurrence, on a prouvé que pour tout élément fermé $\beta\in \mathfrak{g}$ tel que $\beta \notin \mathfrak{z}$, l'application $\eta_{\mathfrak{g}}^{-{1\over 2}-\epsilon}: \mathfrak{g}_{\rm qr}\rightarrow {\Bbb R}_{>0}$ est localement intégrable au voisinage de $\beta$ dans $\mathfrak{g}$. En d'autres termes, l'application $\eta_\mathfrak{g}^{-{1\over 2}-\epsilon}: \mathfrak{g}_{\rm qr}/\mathfrak{z}\rightarrow {\Bbb R}_{>0}$ est localement intégrable sur $\overline{\mathfrak{g}}\smallsetminus \{0\}$. Reste à traiter le cas $\beta\in \mathfrak{z}$ (i.e. $\bar{\beta}=0$). Choisissons un $\mathfrak{o}$--réseau $\Lambda$ dans $\mathfrak{g}$, et posons $\bar{\Lambda}= \mathfrak{z} + \Lambda \subset \overline{\mathfrak{g}}$.  Il suffit de prouver que
$$
\int_{\bar{\Lambda}}\eta_\mathfrak{g}^{-{1\over 2}-\epsilon}(g) \mathfrak{d}\bar{g} < +\infty,
$$
où $\mathfrak{d}\bar{g}$ est une mesure de Haar sur $\overline{\mathfrak{g}}$. 
Puisque $\eta_\mathfrak{g}^{-{1\over 2}-\epsilon}$ est localement intégrable sur $\overline{\mathfrak{g}}\smallsetminus \{0\}$, on a
$$
\int_{\bar{\Lambda}\smallsetminus \varpi \bar{\Lambda}}\eta_\mathfrak{g}^{-{1\over 2}-\epsilon}(g) \mathfrak{d}\bar{g} = c < +\infty,
$$
où $\varpi$ est une uniformisante de $F$. Comme $\bar{\Lambda} = \coprod_{i\geq 0} \varpi^i (\bar{\Lambda}\smallsetminus \varpi \bar{\Lambda})$, on obtient
$$
\int_{\bar{\Lambda}}\eta_\mathfrak{g}^{-{1\over 2}-\epsilon}(g) \mathfrak{d}\bar{g}=\sum_{i\geq 0}
\int_{\varpi^i(\bar{\Lambda}\smallsetminus \varpi \bar{\Lambda})}\eta_\mathfrak{g}^{-{1\over 2}-\epsilon}(g) \mathfrak{d}\bar{g}.
$$
Or pour $i\geq 1$, d'après la formule d'homogénité (1), on a
\begin{eqnarray*}
\int_{\varpi^i(\bar{\Lambda}\smallsetminus \varpi \bar{\Lambda})}\eta_\mathfrak{g}^{-{1\over 2}-\epsilon}(g) \mathfrak{d}\bar{g}
&=& \vert \varpi^i\vert^{N^2-1} \int_{\bar{\Lambda}\smallsetminus \varpi \bar{\Lambda}}\eta_\mathfrak{g}^{-{1\over 2}-\epsilon}(\varpi^ig) \mathfrak{d}\bar{g}\\
&=& \vert \varpi^i\vert^{N^2-1}\vert \varpi^i\vert^{-N(N-1)({1\over 2} + \epsilon)} c\\
&=& q^{-i\left({N(N+1)\over 2} - 1 - N(N-1)\epsilon \right)}c. 
\end{eqnarray*}
Puisque $N>1$ et $N(N-1)\epsilon <1$, la constante $\alpha = {N(N+1)\over 2} - 1 - N(N-1)\epsilon$ vérifie $\alpha>1$, et on obtient
$$
\int_{\bar{\Lambda}}\eta_\mathfrak{g}^{-{1\over 2}-\epsilon}(g) \mathfrak{d}\bar{g}= 
\left(\sum_{i\geq 0} q^{-\alpha i} \right)c <+\infty.
$$
Cela achève la démonstration de la proposition.\end{proof}

% remarque
\begin{marema}
{\rm Considérons la fonction $\gamma \mapsto \lambda_\mathfrak{g}(\gamma)=\log (\max\{1, \eta_\mathfrak{g}(\gamma)^{-1}\})$ sur $\mathfrak{g}_{\rm qr}$. Cette fonction (à valeurs dans ${\Bbb R}_{>0}$) mesure d'une manière assez subtile la distance séparant un élément quasi--régulier de $\mathfrak{g}$ de l'ensemble $\mathfrak{g}\smallsetminus \mathfrak{g}_{\rm qr}$. En effet, pour $\gamma\in \mathfrak{g}_{\rm qr}$, posant $\mathfrak{m}=\mathfrak{m}(\gamma)$, on a $\eta_\mathfrak{g}(\gamma)= \vert D_{\mathfrak{m}\backslash \mathfrak{g}}(\gamma)\vert \eta_\mathfrak{m}(\gamma)$ avec $\gamma \in \mathfrak{m}_{\rm qre}$, et le facteur $\vert D_{\mathfrak{m}\backslash \mathfrak{g}}(\gamma)\vert^{-1}$ est d'autant plus grand que les valeurs propres de l'automorphisme ${\rm ad}_\gamma$ de $\mathfrak{g}/\mathfrak{m}$ sont proches les unes des autres. D'autre part pour $\gamma\in \mathfrak{g}_{\rm qre}$, on a $\eta_\mathfrak{g}^{-1}(\gamma)= \mu_F^+(\gamma)= \vert \det(\gamma)\vert^{1-N}\mu_F(\gamma)$, et le facteur $\mu_F(\gamma)$ est d'autant plus grand que $\gamma$ est loin d'être minimal (rappelons que $\mu_F(\gamma)\geq 1$ avec égalité si et seulement si $\gamma$ est minimal).

Pour tout $\alpha\in {\Bbb R}_{\geq 0}$, on a:
\begin{enumerate}[leftmargin=17pt]
\item[(2)] la fonction $\mathfrak{g}_{\rm qr}\rightarrow {\Bbb R}_{>0},\,\gamma \mapsto \lambda_\mathfrak{g}^\alpha(\gamma) \eta_\mathfrak{g}^{-{1\over 2}}$ est 
localement intégrable sur $\mathfrak{g}$.
\end{enumerate}
En effet, pour tout $\epsilon >0$, il existe une constante $c=c(\epsilon,\alpha)$ telle que
$$
\log (\max\{1, y\})^\alpha \leq cy^\epsilon, \quad y\in {\Bbb R}_{\geq 0}.
$$
On a donc
$$
\lambda_\mathfrak{g}^\alpha(\gamma) \eta_\mathfrak{g}^{-{1\over 2}}(\gamma) \leq c \eta_\mathfrak{g}^{-{1\over 2}-\epsilon}(\gamma),\quad \gamma\in \mathfrak{g}_{\rm qr}.
$$
Il suffit de prendre $\epsilon$ tel que $N(N-1)\epsilon <1$ et d'appliquer la proposition.\hfill $\blacksquare$
}
\end{marema}

%%%%%%%%%%%%%%%%%%%%%%%%%%%%%%%%
\section{Résultats sur le groupe}\label{résultats sur le groupe}

\subsection{Les intégrales orbitales normalisées sont bornées}On a prouvé que les intégrales orbitales quasi--régulières d'une fonction $\mathfrak{f}\in C^\infty_{\rm c}(\mathfrak{g})$ sont bornées sur $\mathfrak{g}$ (corollaire 1 de \ref{les germes de Shalika normalisés sont localement bornés}). L'analogue sur $G$ de ce résultat est la 

% proposition
\begin{mapropo}
Pour toute fonction $f\in C^\infty_{\rm c}(G)$, la fonction $G_{\rm qr}\rightarrow {\Bbb C},\, \gamma \mapsto I^G(\gamma,f)$
est bornée sur $G$: il existe une constante $c_f>0$ telle que
$$
\sup\{\vert I^G(\gamma,f)\vert: \gamma\in G_{\rm qr}\}\leq c_f.
$$
\end{mapropo}

\begin{proof}Soit $f\in C^\infty_{\rm c}(G)$. Il suffit de prouver que la fonction $G_{\rm qr}\rightarrow {\Bbb C},\, \gamma \mapsto I^G(\gamma,f)$ est bornée sur 
le compact ouvert ${\rm Supp}(f)$, ou, ce qui revient au même, qu'elle est localement bornée sur $G$. 

Soit un élément fermé $\beta\in G$. Notons $\mathfrak{b}$ le centralisateur $\mathfrak{g}_\beta = {\rm End}_{F[\beta]}(V)$ de $\beta$ dans $\mathfrak{g}$. 
Reprenons les notations de \ref{descente centrale au voisinage d'un élément fermé}. On a $F[\beta]=E_1\times\cdots \times E_r$, $V= V_1\times \cdots \times V_r$, et $\mathfrak{b}= \mathfrak{b}_1\times \cdots \times \mathfrak{b}_r$ avec $\mathfrak{b}_i = {\rm End}_{E_i}(V_i)$. On pose aussi $\mathfrak{m}= \mathfrak{g}_1\times\cdots \times \mathfrak{g}_r$ avec $\mathfrak{g}_i={\rm End}_F(V_i)$. On peut supposer que $\mathfrak{m}$ est standard, \cad de la forme $\mathfrak{m}= \mathfrak{m}_P$ pour un $P\in \ES{P}$. On pose $M=M_P$ et $H = \mathfrak{b}^\times$. D'après la proposition de \ref{descente centrale au voisinage d'un élément fermé}, il existe un voisinage ouvert compact $\ES{W}$ de $0$ dans $\mathfrak{b}$ et un élément $\bs{x}\in M$ tels que:
\begin{itemize}
\item l'application $G\times \bs{x}\ES{W} \rightarrow G, \, (g,\bs{x}b)\mapsto g^{-1}(\beta + \bs{x}b)g$ est partout submersive;
\item pour tout $b\in  \mathfrak{b}_{\rm qr}\cap \ES{W}$, l'élément $\beta + \bs{x}b$ appartient à $G_{\rm qr}$;
\item il existe une fonction $f^\mathfrak{b}\in C^\infty_{\rm c}(\ES{W})$ telle que pour tout $b\in \mathfrak{b}_{\rm qr}\cap \ES{W}$, on a l'égalité $I^G(\beta + \bs{x}b,f)=I^\mathfrak{b}(b,f^\mathfrak{b})$.
\end{itemize}
Puisque la fonction $\mathfrak{b}_{\rm qr}\rightarrow {\Bbb C},\, b\mapsto I^\mathfrak{b}(b,f^\mathfrak{b})$ est bornée sur $\mathfrak{b}$ (corollaire 1 de \ref{les germes de Shalika normalisés sont localement bornés}), on obtient que la fonction $G_{\rm qr}\rightarrow {\Bbb C},\, \gamma \mapsto I^G(\gamma,f)$ est bornée sur l'ouvert ${^G(\beta + \bs{x}\ES{W})}$ de $G$.

La fonction $G_{\rm qr}\rightarrow {\Bbb C},\, \gamma \mapsto I^G(\gamma,f)$ est donc localement bornée sur $G$, et la proposition est démontrée.
\end{proof}

\subsection{Intégrabilité locale de la fonction $\eta_G^{-{1\over 2}-\epsilon}$}\label{intégrabilité locale du eta}On a prouvé que pour $\epsilon>0$ tel que $N(N-1)\epsilon<1$, la fonction $\eta_\mathfrak{g}^{-{1\over 2}-\epsilon}$ est localement intégrable sur $\mathfrak{g}$ (proposition de \ref{intégrabilité locale du eta Lie}). On en déduit la

% proposition
\begin{mapropo}
Pour tout $\epsilon>0$ tel que $(N^2-N)\epsilon <1$, la fonction $\eta_G^{-{1\over 2}-\epsilon}$ est localement intégrable sur $G$.
\end{mapropo}

\begin{proof}
Il s'agit de montrer que pour chaque élément fermé $\beta\in G$, la fonction $\eta_G^{-{1\over 2}-\epsilon}:G_{\rm qr}\rightarrow {\Bbb R}_{>0}$ est localement intégrable au voisinage de $\beta$ dans $G$. Reprenons la démonstration de la proposition de 
\ref{intégrabilité locale du eta Lie}. On a défini un sous--groupe de Levi $M=M(\beta)$ de $G$ tel que $\beta$ est pur dans $M$. L'ensemble $\Omega= \{\gamma \in M: D_{M\backslash G}(\gamma)\neq 0\}$ est ouvert dans $M$, et l'application
$$
G\times \Omega\rightarrow G,\, (g,\gamma)\mapsto g^{-1}\gamma g
$$
est partout submersive. L'ensemble $\Omega' = \{\gamma\in M: \vert D_{M\backslash G}(\gamma)\vert = \vert D_{M\backslash G}(\beta)\vert\}\;(\subset \Omega)$ est un voisinage ouvert ($M$--invariant par conjugaison, et $Z$--invariant par translation) de $\beta$ dans $M$. Puisque pour $(g,\gamma)\in G\times (\Omega'\cap G_{\rm qr})$, on a
$$
\eta_G(g^{-1} \gamma g)= \vert D_{M\backslash G}(\beta)\vert \eta_M(\gamma),
$$
l'intégrabilité locale de la fonction 
$\eta_G^{-{1\over 2} -\epsilon}: G_{\rm qr}\rightarrow {\Bbb R}_{>0}$ au voisinage de $\beta$ dans $G$ est impliquée par l'intégrabilité locale de la fonction $\eta_{M}^{-{1\over 2}-\epsilon}: M_{\rm qr}\rightarrow {\Bbb R}_{>0}$ au voisinage de $\beta$ dans $M$. Notons que si $\mathfrak{m}=\mathfrak{m}(\beta)\subset \mathfrak{g}$ est l'algèbre de Lie de $M$, alors 
$\mathfrak{m}= \mathfrak{g}_1\times \cdots \times \mathfrak{g}_s$ avec $\mathfrak{g}_i={\rm End}_F(V_i)$ pour un sous--$F$--espace vectoriel $V_i$ de $V$ de dimension 
$N_i$, et l'hypothse de rcurrence est vrifie: on a $(N_i^2-N_i)\epsilon< 1$ pour tout $i$. On s'est donc ramené au cas où $\mathfrak{m}=\mathfrak{g}$, \cad au cas où $\beta$ est pur dans $\mathfrak{g}$. Si $\beta$ est un élément pur de $G$, alors en notant $\mathfrak{b}=\mathfrak{g}_\beta$ le centralisateur de $\beta$ dans $\mathfrak{g}$, on obtient 
comme dans la démonstration de \ref{intégrabilité locale du eta Lie} (par descente centrale) que l'intégrabilité locale de la fonction $\eta_G^{-{1\over 2}-\epsilon}:G_{\rm qr}\rightarrow {\Bbb R}_{>0}$ au voisinage de $\beta$ dans $G$ est impliquée par celle de la fonction $\eta_\mathfrak{b}^{-{1\over 2}-\epsilon}: \mathfrak{b}_{\rm qr}\rightarrow {\Bbb R}_{>0}$ au voisinage de $0$ dans $\mathfrak{b}$, 
laquelle est dsormais prouve. 
\end{proof}

%%%%%%%%%%%%%%%%%%%%%%%%%%%%%%%%

\end{document}